%% file: arxiv_v3.tex
\documentclass[11pt]{article}
\usepackage[top=1in, bottom=1in, left=1in, right=1in]{geometry}

\usepackage[linktocpage,colorlinks,linkcolor=blue,anchorcolor=blue,citecolor=blue,urlcolor=blue,pagebackref]{hyperref}
\usepackage{euscript,mdframed}
\usepackage{microtype,todonotes,relsize}
\usepackage{amsmath,amsthm}
\usepackage{algpseudocode}
\usepackage{bbding}
\usepackage{pifont}
\usepackage{diagbox}

\usepackage{accents}
\input{_macros}
\allowdisplaybreaks
\usepackage[authoryear,round]{natbib}
\renewcommand*{\backref}[1]{\ifx#1\relax \else Page #1 \fi}
\renewcommand*{\backrefalt}[4]{%
  \ifcase #1 \footnotesize{(Not cited.)}%
  \or        \footnotesize{(Cited on page~#2.)}%
  \else      \footnotesize{(Cited on pages~#2.)}%
  \fi
}

\newcommand*{\colorboxed}{}
\def\colorboxed#1#{%
  \colorboxedAux{#1}%
}
\newcommand*{\colorboxedAux}[3]{%
  \begingroup
    \colorlet{cb@saved}{.}%
    \color#1{#2}%
    \boxed{%
      \color{cb@saved}%
      #3%
    }%
  \endgroup
}
\numberwithin{equation}{section}
\usepackage{xcolor} 
\usepackage{marginnote}

\newcommand{\todol}[2][]{{%
 \let\marginpar\marginnote
 \reversemarginpar
 \renewcommand{\baselinestretch}{0.8}%
 \todo[color=yellow]{#2}}}

\title{A Barrier Function Approach for Bilevel Optimization with Coupled Lower-Level Constraints: Formulation, Approximation and Algorithms}

\author{Xiaotian Jiang, Jiaxiang Li, Mingyi Hong, Shuzhong Zhang\footnote{X.~Jiang and S.~Zhang are with Department of Industrial and System Engineering, University of Minnesota, Minneapolis, MN, USA. E-mails: \texttt{jian0851@umn.edu}, \texttt{zhangs@umn.edu}. J.~Li and M.~Hong are with Department of Electrical and Computer Engineering, University of Minnesota, Minneapolis, MN, USA. E-mails: \texttt{li003755@umn.edu}, \texttt{mhong@umn.edu}.}}

\date{\today}

\begin{document}
\maketitle

\begin{abstract}
    In this paper, we consider bilevel optimization problem where the lower-level has coupled constraints, i.e. the constraints depend both on the upper- and lower-level variables. In particular, we consider two settings for the lower-level problem. The first is when the objective is strongly convex and the constraints are convex with respect to the lower-level variable; The second is when the lower-level is a linear program. We propose to utilize a barrier function reformulation to translate the problem into an unconstrained problem. By developing a series of new techniques, we proved that both the hyperfunction value and hypergradient of the barrier reformulated problem (uniformly) converge to those of the original problem under minimal assumptions. Further, to overcome the non-Lipschitz smoothness of hyperfunction and lower-level problem for barrier reformulated problems, we design an adaptive algorithm that ensures a non-asymptotic convergence guarantee. We also design an algorithm that converges to the stationary point of the original problem asymptotically under certain assumptions. The proposed algorithms require minimal assumptions, and to our knowledge, they are the first with convergence guarantees when the lower-level problem is a linear program. Numerical experiments are conducted to show the effectiveness of the proposed method.
\end{abstract}

\section{Introduction}

Bilevel optimization (BLO) is drawing wide attention in the communities including machine learning, operations research and signal processing. It observes wide applications in various machine learning problems, such as hyperparameter optimization~\citep{Maclaurin2015GradientbasedHO,2018Bilevel}, meta learning~\citep{finn2017model,2018Bilevel,ji2021bilevel} and reinforcement learning~\citep{stadie2020learning, zeng2024demonstrations,li2024joint,li2024getting}. In the field of operations research, BLO finds its applications in pricing, transportation design, game theory, among others \citep{labbe2016bilevel,Silvrio2022ABO}. Standard BLO takes the form:
\begin{align}\label{eq:blo_unconstrained}
    \min_{x, y^*(x)}\ f(x,y^\ast(x))\quad \text{ s.t.}\quad x\in\mathcal{X},\ y^\ast(x)\in \argmin_{y\in \mathcal{Y}(x)} g(x,y),
\end{align}
where $f$ is called the upper-level objective and $g$ is called the lower-level objective; $\mathcal{X}\subseteq \mathbb{R}^n$, and $\mathcal{Y}(x)\subseteq \mathbb{R}^m$ are feasible sets for the upper- and lower-problems, respectively. In game theory, the bilevel problem can be thought of as a two-player (Stackelberg) game, the lower-level and upper-level problems are also called the follower and the leader, respectively; see e.g.~\cite{kohli2012optimality,liu2018pessimistic}. 

{{In BLO with {\it unconstrained} lower-level problems where $\mathcal{Y}\equiv \mathbb{R}^m$, a popular class of algorithms are referred to as the {\it implicit gradient descent} method \citep{ghadimi2018approximationmethodsbilevelprogramming,yang2021provablyfasteralgorithmsbilevel,hong2022twotimescaleframeworkbileveloptimization,pedregosa2022hyperparameteroptimizationapproximategradient}. Let us denote the {\it hyperfunction} $\phi$ as}} 
\begin{align}\label{eq:hyperfunction}
    \phi(x):=\min_{y^\ast(x)}f(x,y^\ast(x)), \text{ where }y^\ast(x)\in\argmin_{y\in\RR^m} g(x,y).
\end{align}
 The  {\it implicit gradient descent} method operates by using the implicit function theorem to compute (some approximated version of) the gradient of $\phi(x)$, which we refer to as the hypergradient $\nabla_x\phi(x)$\footnote{Note, such a function exists under certain conditions, which can be broader than the one stated above; We will discuss those conditions in detail later.}, thus allowing us to apply gradient descent-type of methods to $\phi(x)$. {{Indeed, as we will discuss shortly, there has been tremendous recent success in developing efficient implicit gradient decent-type methods and analytical approaches for {\it unconstrained} problems where $\mathcal{Y}(x) \equiv \mathbb{R}^m$. However, it is still unclear how to effectively leverage these developments for problems with lower-level constraints}}. In particular, the constrained lower-level solution $y^*(x)\in\argmin_{y\in\mathcal{Y}(x)} g(x,y)$ could be nonsmooth, discontinuous or even set-valued based on the specific form of $\mathcal{Y}(x)$ and $g(x,y)$ (see \citet[Page 3]{pmlr-v202-khanduri23a} and our Example \ref{counterexample1}).

In this work, we consider the following BLO problem with lower-level constraints:
\begin{align}\label{eq:blo_constrained}
    \min\ f(x,y^\ast(x))\quad \text{s.t.}\quad x\in\mathcal{X},\ y^\ast(x)\in\argmin_{\{y:h_i(x,y)\leq 0,\ i=1,...,k\}} g(x,y),
\end{align}
where $f(x,y):\RR^n\times\RR^m\to\RR$ is (possibly) a non-convex function, also with the lower-level objective function $g(x,y):\RR^n\times\RR^m\to\RR$ and the constraints $h_i(x,y):\RR^n\times\RR^m\to\RR$ and $\mathcal{X}$ is a compact set. We refer to lower-level constraints as {\it coupled}, meaning that (at least one) of the constraint functions $h_i(x,y)$, $i\in[k]$, depends on both $x$ and $y$. 
Otherwise, they are referred to as {\it uncoupled}. 

If the constraints of the lower-level problem are uncoupled, many algorithms have been studied to deal with such problems. However, these algorithms fail to solve coupled constrained problems. While there are studies addressing coupled constrained BLO, they typically rely on strong assumptions and often only guarantee asymptotic convergence. For more details about these existing works, see Section \ref{sec:relatedwork}. To avoid assumptions that are difficult to verify and to achieve non-asymptotic convergence results, in this work we propose a barrier function-based method for BLO. More specifically, we consider a smooth approximation of the original problem \eqref{eq:blo_constrained} as follows
\begin{align}\label{eq:blo_barrier_reformulation}
    \min\ f(x,y)\quad \text{ s.t. }\quad x\in\mathcal{X},\; y=\argmin_{y\in \RR^m} \widetilde{g}_t(x,y):=g(x,y)-t\sum_{i=1}^k\log(-h_i(x,y)),
\end{align}
where $t$ is a sufficiently small fixed constant. That is, we transform the lower-level constraints into a log-barrier added to the lower-level objective function, thereby converting the lower-level problem into an unconstrained one. Theoretically, this transformation makes the gradient of the hyperfunction computable, allowing us to use the standard implicit function approach to solve the resulting BLO problem efficiently and with theoretical guarantees. Practically, log-barrier reformulation models decision-making processes more realistically: instead of abruptly rejecting a decision at the boundary, it gradually decreases acceptance as the decision variable approaches the boundary, simulating a continuous and natural decline in feasibility. 

Next, we will review related works for solving various forms of BLO problems discussed so far.

\subsection{Related Works}\label{sec:relatedwork}

\begin{table}
  \centering
\resizebox{0.99\textwidth}{!}{%
\begin{tabular}{c|c|c|c|c}
    \hline
    Algorithm & $g(x,y)\&h(x,y)$ & Reformulation & Complexity & Nonstandard Assumptions \\
    \hline\hline
    IGBA$^\dag$ & S-C\&C & BR & non-asymptotic & needed \\
    \hline
    GAM & S-C\&C & OP & asymptotic & not needed \\
    \hline
    BSG$^\dag$ & S-C\&C & KR & asymptotic & needed \\
    \hline
    iP-DwCA$^\dag$ & J-W-C\&J-C & MVFR & asymptotic & needed \\
    \hline
    APPM$^\dag$ & C\&C & VFR & non-asymptotic & needed \\
    \hline
    BiC-GAFFA$^\dag$ & C\&C & LPR & non-asymptotic & needed \\
    \hline
    IGOGIC & S-C\&L & LPR & non-asymptotic & needed \\
    \hline
    BLOCC$^\dag$ & S-C\&C & LPR & non-asymptotic & needed \\
    \hline
    BFBM (proposed) & S-C\&C or L\&L & BR & non-asymptotic & not needed \\
    \hline
  \end{tabular}
}
  \caption{\footnotesize Comparison of BFBM algorithm (our proposed Algorithm \ref{algo:fo}) with different algorithms for addressing coupled constrained BLO problem:  IGBA \citep{tsaknakis2023implicit}, GAM (Algorithm 1 in \cite{xu2023efficient}), BSG \citep{giovannelli2021inexact}, iP-DwCA \citep{Gao2023MoreauEB}, APPM (Algorithm 4 in \cite{lu2024firstorderpenaltymethodsbilevel}), BiC-GAFFA \citep{yao2024overcominglowerlevelconstraintsbilevel}, IGOGIC (Algorithm 5 in \cite{kornowski2024firstordermethodslinearlyconstrained}), BLOCC (Algorithm 1 and 2 in \cite{jiang2024primal}). BR means barrier reformulation; KR means KKT-based reformulation; MVFR means relaxed Moreau envelope-based Value function reformulation; VFR means Value function reformulation; LPR means Lagrangian-based penalized reformulation; OP means original problem. C means convex in $y$; S-C means strongly convex in $y$; N-C means non-convex in $y$; J-C means jointly convex in $(x,y)$; J-W-C means jointly weakly convex in $(x,y)$; L means linear in $y$. \dag\ means that authors did not consider the relation of the stationary point between reformulation and original problem, or the relation of hypergradient between reformulation and original problem. Nonstandard assumptions refer to assumptions beyond LICQ, Lipschitz smoothness, compactness, and Slater condition.}
  \label{table:related_works}
\end{table}

\noindent \textbf{Bilevel Optimization without lower-level constraints.} BLO without lower-level constraints has seen remarkable advancements through gradient-based approaches such as approximate implicit differentiation (AID) and iterative differentiation (ITD). AID, utilizing the implicit function theorem, approximates the hypergradient and has demonstrated effective finite-time, and finite sample (when the problem is stochastic) convergence for unconstrained, strongly convex lower-level problems \citep{ghadimi2018approximationmethodsbilevelprogramming,yang2021provablyfasteralgorithmsbilevel,hong2022twotimescaleframeworkbileveloptimization,pedregosa2022hyperparameteroptimizationapproximategradient}. ITD further enhances convergence by differentiating the entire iterative algorithm used at the lower-level, supported by substantial subsequent works \citep{maclaurin2015gradientbasedhyperparameteroptimizationreversible,franceschi2017forwardreversegradientbasedhyperparameter,nichol2018firstordermetalearningalgorithms,ji2021bilevel}. Utilizing AID and momentum-based algorithms, \cite{khanduri2021near,yang2023achieving} achieves $\mathcal{O}(\epsilon^{-3})$ rate of convergence to achieve an $\epsilon$-stationary point for stochastic BLO. Additionally, penalty-based methods have become popular, simplifying computational complexities by transforming BLO into single-level problems using various penalty terms \citep{mehra2021penaltymethodinversionfreedeep,shen2023penaltybasedbilevelgradientdescent,kwon2024penaltymethodsnonconvexbilevel}. Both AID and ITD utilize the Hessian of the lower-level objective for estimating the hypergradients, and a recent line of works \citep{kwon2023fully,yang2023achieving} develops fully first order algorithms by leveraging finite-difference to estimate the Hessian-vector product needed to compute the hypergradients.

\noindent \textbf{Bilevel Optimization with lower-level constraints.} BLO with lower-level constraints is significantly more complex as compared with their unconstrained counterparts. Previous studies have typically addressed upper-level constraints in various works \citep{chen2022singletimescalemethodstochasticbilevel,chen2022fastconvergentproximalalgorithm}. Approaches for uncoupled lower-level constraints include the SIGD method \citep{pmlr-v202-khanduri23a}, which targets the constraint $A y\leq b$ and demonstrates asymptotic convergence. \cite{shi2024double} proposes a double-momentum based algorithm for lower-level constrained BLO and provides convergence analysis toward the so-called $(\delta,\epsilon)$-stationary point, where they use a finite-difference technique to avoid the computation of Hessian-vector product and yield a dimension-dependent rate of convergence. Works such as \cite{shen2023penaltybasedbilevelgradientdescent,kwon2024penaltymethodsnonconvexbilevel} utilize penalty reformulations to handle both upper and lower uncoupled constraints. In particular, if the lower-level constraint is a smooth manifold (for example, a sphere), \cite{li2024riemannian,han2024framework,dutta2024riemannian} consider manifold optimization techniques for efficiently solving the lower-level constrained BLO problems. However, none of these works are able to handle lower-level coupled constraints, i.e.\ the situation where the lower-level constraints vary according to $x$. For example, in \cite{shen2023penaltybasedbilevelgradientdescent}, the authors introduce a penalty-based bilevel gradient descent (PBGD) algorithm that reformulates the bilevel problem into a single-loop penalized problem solved using the projected gradient descent method. However, despite the convexity of the feasible set $\mathcal{Y}(x):=\{y : h_i(x,y)\leq 0,\, i=1,
...,k\}$ for each $x$, the product set $\mathcal{X} \times \mathcal{Y}(x)$ may not retain convexity. Therefore, in the case of coupled constraints, the PBGD algorithm that transforms the original problem into a single loop algorithm will no longer be applicable, as the domain may now be non-convex. Moreover, the method of computing the gradient of the value function, which is necessary for the PBGD algorithm, is also not applicable in the case of coupled constraints. In \cite{pmlr-v202-khanduri23a}, the authors employed an implicit gradient descent-based method, but the calculation of the hypergradient is no longer applicable in the coupled case. 

When we have lower-level coupled equality constraints which are affine in $y$ (e.g.\ $\{y:h(x)+Ay=c\}$), \cite{pmlr-v206-xiao23a} considers a direct extension of AID-based method for hypergradient estimation. The key in their analysis is that equality constraints will not affect the differentiability of the lower-level solver $y^*(x)$ (see Lemma 2 in \cite{pmlr-v206-xiao23a}). Another work that deals with lower-level equality constraints is \cite{kornowski2024firstordermethodslinearlyconstrained}, where the authors propose a finite-difference approximation method for the hypergradient, provided that the lower-level equality constraint is linear for both $x$ and $y$.

However, lower-level (coupled) inequality constraints are much harder to deal with since inequality constraints will result in a nonsmooth lower-level solver $y^*(x)$. In \cite{giovannelli2021inexact}, the authors consider stochastic BLO with lower-level coupled inequality constraints and provide an asymptotic convergence analysis to the stationary point under the assumption that for each $x$, there exists a lower-level solver $y^*(x)$ satisfying the lower-level KKT conditions such that LICQ, strict complementarity slackness condition (SCSC) and second-order sufficient condition (SOSC) are also satisfied. \cite{xu2023efficient} proposed a gradient approximation scheme for the hypergradient also under LICQ assumptions on the lower-level constraints and proved its asymptotic convergence to the stationary points. \cite{Gao2023MoreauEB} consider a Moreau Envelope-based algorithm and obtain an asymptotic convergence result. Recently, \cite{jiang2024primal} considered a penalty-based primal-dual reformulation to translate the BLO with lower-level coupled constraints problem into a single-level minimax optimization. The authors proved a non-asymptotic rate of convergence for the reformulated problem toward $\epsilon$-stationary points under the assumptions that LICQ is satisfied at each point and an upper bound for the dual variable (see Lemma 2 in \cite{jiang2024primal}). Beyond LICQ, they also need a curvature condition assumption on the multipliers to obtain this convergence result (see Assumption 5 in \cite{jiang2024primal}), {{and such an assumption is often not easy to verify}}. Another recent work \cite{yao2024overcominglowerlevelconstraintsbilevel} also considers a penalty-type reformulation and provides a non-asymptotic rate of convergence toward the $\epsilon$-stationary point of the reformulated problem. However, this algorithm is largely an infeasible algorithm and needs an explicit assumption on the upper bound of the function value at each iteration in order to achieve feasibility. \cite{lu2024firstorderpenaltymethodsbilevel} also considers a penalty reformulated problem of the lower-level constrained BLO and shows the asymptotic convergence toward stationary points under appropriate assumptions. When each of the coupled constraints is linear to both $x$ and $y$, \cite{kornowski2024firstordermethodslinearlyconstrained} proposes a penalty and finite-difference based algorithm which converges to $(\delta,\epsilon)$-stationary point with a non-asymptotic rate under the assumption of having access to the optimal dual variable, which is difficult in practice. We summarize these literature on BLO with lower-level coupled inequality constraints along with our proposed algorithm in Table \ref{table:related_works}.

A recent work \cite{tsaknakis2023implicit} proposed a certain barrier approximation approach for BLO problems. Despite the fact that the idea of using barrier approximation is close to the current work,  the relationship between the reformulation and the original problem is not fully investigated. Further, an explicit and unconventional assumption on the lower-level constraint function values is imposed to achieve theoretical convergence. Moreover, a black-box lower-level solver is required to avoid the non-Lipschitz smooth nature of the barrier reformulated lower-level problem.  

Additionally, it is worth mentioning that the coupled constrained problem has recently been investigated for a special subclass of the BLO problem: the minimax problem. \cite{tsaknakis2023minimax} defined a type of local minimum point for minimax problems with coupled constraints and designed an algorithm that converges to such points.

\subsection{Contributions}

As mentioned before, in this work we propose a reformulation \eqref{eq:blo_barrier_reformulation}, and show that it approximates well to the original problem \eqref{eq:blo_constrained}, in terms of both function values and gradients. Further, we propose an algorithm that optimizes the reformulated problem, and analyze its non-asymptotic convergence under mild conditions. Our main contributions are summarized below:
\begin{enumerate}
\vspace{-0.2cm}
    \item To our knowledge, this is the first work that systematically study the approximation error between the barrier reformulated problem \eqref{eq:blo_barrier_reformulation} and the original problem \eqref{eq:blo_constrained} in terms of both hyperfunction values and hypergradients\footnote{Note that the approximation error in terms of hyperfunction values when $g$ is strongly convex is first studied in \cite{tsaknakis2023implicit}, whereas the hypergradient is not studied before.}. In particular, we show that approximation errors of hyperfunction values can be properly quantified not only when the lower-level objective $g$ is strongly convex in $y$ and the constraints $h_i(x,y)$'s are convex in $y$, but also when the lower-level objective $g$ and constraints $h_i(x,y)$'s are all linear in $y$. To the best of our knowledge, we are the first to prove such a result for the latter setting. We also proved the asymptotic convergence of hypergradient at \textbf{SCSC} point (see Definition \ref{def:scsc_point}) as the barrier coefficient $t$ goes to 0 within these two settings. To show these results, we develop several novel proof techniques. In particular, we first show $\lim_{t\to 0}y_t^\ast(x)=y^\ast(x)$ where $y^\ast(x)$ and $y_t^\ast(x)$ are the optimal lower-level solution for \eqref{eq:blo_constrained} and \eqref{eq:blo_barrier_reformulation} respectively. We further show that differentiation and limits are interchangeable, i.e. $\lim_{t\to 0}\nabla_x y_t^\ast(x)=\nabla_x \lim_{t\to 0}y_t^\ast(x)=\nabla_x y^\ast(x)$, by establishing the uniform convergence of $\nabla_x y_t^\ast(x)$. This interchangeability is crucial for the convergence of the hypergradient.
    
    \vspace{-0.2cm}
    \item Based on the previous contribution, we develop a new adaptive algorithm that theoretically achieves non-asymptotic $\widetilde{\mathcal{O}}(1/(\epsilon^2 t^{4.5}))$ convergence rate to the $\epsilon$ stationary point of the barrier reformulated problem \eqref{eq:blo_barrier_reformulation} under mild assumptions. One challenge in our algorithm design and analysis is that $\widetilde{g}_t(x,y_t^\ast(x))$ explodes when $h_i(x,y_t^\ast(x))$ is close to zero, so the hyperfunction and lower-level problem of \eqref{eq:blo_barrier_reformulation} may not enjoy global Lipschitz gradient property. To overcome this difficulty, we present a key result (in Theorem \ref{lem:local_M}), which states that if $h_i(x,y^\ast_t(x))$ tends towards zero as $t$ approaches zero, it does so at most linearly with respect to $t$, and the linear coefficient is {\it independent} of both $t$ and $x$. This result allows us to transform the original non-Lipschitz smooth lower-level problem into a Lipschitz smooth one by shrinking the feasible set of the lower-level problem, and to design the step sizes for the outer loop adaptively. Additionally, under strongly convex setting, when assuming that the upper-bounds for some provably bounded terms are known, the convergence rate can be improved significantly to $\tilde{\mathcal{O}}(1/(\epsilon^2 t^{1.5}))$.
\vspace{-0.2cm}
    \item We provide numerical experiments on a class of strongly convex lower-level problems and a class of linear lower-level problems, both with linear inequality constraints. We compare our method with a number of existing works and verify the effectiveness of the proposed method. Our algorithm is the only one that always guarantees that the obtained solution is feasible for the lower-level problem, and it is the most effective algorithm among those we tested when lower-level is a linear program.
\end{enumerate}

\subsection{Notations and Terminology}
The basic notations used throughout this paper are introduced in this section. Additional notations can be found in Appendix \ref{sec:notations} for clarity and reference:
\begin{itemize}
    \item $\mathbb{R}^n$ is the Euclidean space of dimension $n$, $\left<\cdot,\cdot\right>$ and $\|\cdot\|$ are the canonical inner product and norm of Euclidean space;
    \item $B_x(r)$ is the ball centered at the point $x$ with radius $r$, i.e. $B_x(r)=\{z\in\mathbb{R}^n:\|z-x\|\leq r\}$;
    \item Let $y^\ast(x)$ and $y^\ast_t(x)$ be the optimal solution set of the original problem (\ref{eq:blo_unconstrained}) and the barrier reformulation problem (\ref{eq:blo_barrier_reformulation}) respectively; 
    \item Let $\phi(x)$ and $\widetilde{\phi}_t(x)$ be hyperfunction (see (\ref{eq:hyperfunction}) and (\ref{eq:barrierhyperfunction}) for definition) of the original problem (\ref{eq:blo_unconstrained}) and the barrier reformulation problem (\ref{eq:blo_barrier_reformulation}) respectively, and $\widetilde{g}_t(x,y):=g(x,y)-t\sum_{i=1}^k\log(-h_i(x,y))$ be the reformulated lower-level objective function; 
    \item Define $\mathcal{Y}(x):=\{y: h_i(x,y)\leq 0,\ i=1,...,k\}$ as the feasible set of lower-level problem, and $k$ is the number of constraints for lower-level problem;
    \item $\widetilde{\mathcal{O}}(\cdot)$ represents $\mathcal{O}(\cdot)$  with logarithmic term omitted.
\end{itemize}

\section{Motivating Application}\label{sec:MotivatingApplication}
Let us briefly discuss a few applications for the considered BLO problems. 

\noindent\textbf{Linear Setting:} In the first setting, both $g(x,y)$ and $h_i(x,y)$ are linear functions of $y$ for any $i$. One relevant application is the bilevel price-setting problem in the transportation context (see \cite{labbe2016bilevel}), which can be formulated as: 
\begin{align}\label{eq:blo_price_setting}
    \max_{T}&\quad T^\top x\\
    \rm{s.t.}\quad\ (x,y)\in\arg\;min_{x,y}&\quad  (c_1+T)^\top x + c_2^\top y  \nonumber\\
    \rm{s.t.}& \quad A_1(T) x+A_2(T) y \geq b(T),\quad  x,y\geq 0\nonumber
\end{align}
where $T\in\RR^{d_x}$ is the upper-level decision variable, representing the tax value, and $(x, y)\in\RR^{d_x}\times\RR^{d_y}$ are the lower-level decision variables, corresponding to the allocation of transportation activities across taxable and untaxable routes. The matrices $A_1(T)\in\RR^{k\times d_x}$ and $A_2(T)\in\RR^{k\times d_y}$ represent the constraints that govern the feasibility of transportation along taxed and untaxed routes.

In this specific transportation setting, the leader is a regulatory authority or city planner, who decides the tax value $T$ on transportation activities. The leader's goal is to maximize the revenue from the tax, represented by $T^\top x$. The follower is a freight transporter who needs to minimize their total transportation cost, which includes both the tax on taxable routes and the cost of using taxed and untaxed routes. The cost function is $(c_1+T)^\top x + c_2^\top y$, where $c_1$ and $c_2$ are the cost coefficients for the taxable and nontaxable routes, respectively, and $T$ adjusts the cost of the taxable routes.
The matrices $A_1(T)$ and $A_2(T)$ represent coefficients that influence the capacity or feasibility of using taxed and untaxed routes for transportation, and $b(T)$ means demands from different customers. These matrices depend on $T$ because the tax may impact the availability or cost of using taxed routes. The inequality $A_1(T)^\top x + A_2(T)^\top y \geq b(T)$ ensures that customer demands are met. The constraints $x, y \geq 0$ enforce that the quantities of goods transported along both types of routes cannot be negative, reflecting the real-world constraint that transportation volumes must be non-negative.

\noindent\textbf{Strongly Convex Setting:} In the second setting, $g(x,y)$ is strongly convex, and $h_i(x,y)$ are convex in $y$ for any $i$. The primary applications in this setting arise from machine learning. For example, support vector machines (SVM) inherently possess a bilevel optimization structure, where hyper-parameter optimization can be framed as a constrained BLO problem; we refer the readers to recent works such as \citep{jiang2024primal, xu2023efficient} for more detailed discussions.

\section{Barrier Reformulation of Bilevel Problems}\label{sec:convergence_of_hyperfunction}

In this section, we will explore the hyperfunction of the barrier reformulation (\ref{eq:blo_barrier_reformulation}), defined below, and its relationship with the hyperfunction of the original problem (\ref{eq:blo_constrained}):
    \begin{align}\label{eq:barrierhyperfunction}
        \widetilde{\phi}_t(x)=f(x,y_t^\ast(x)),\text{ where }y^\ast_t(x)=\argmin_{y\in \RR^m} \widetilde{g}_t(x,y):=g(x,y)-t\sum_{i=1}^k\log(-h_i(x,y)).
    \end{align}
 Specifically, we first demonstrate that the barrier reformulation is differentiable under certain conditions. This is mainly because the lower-level problem is strictly convex (see Proposition \ref{prop:diff_hyperfun}). Subsequently, we investigate the convergence of the function values and gradients of the hyperfunction as $t$ approaches 0, under various different assumptions of the lower-level problem.
\subsection{Preliminaries}

In this section, we study some basic properties of the barrier reformulation problem (\ref{eq:blo_barrier_reformulation}). We make the following assumption throughout this section.

\begin{assumption}\label{assumption:general1}The following holds for every $x\in\mathcal{X},y\in\mathcal{Y}(x)$, and $i\in\{1,...,k\}$
\begin{enumerate}
    \item $f(x,y)$ is once while $g(x,y)$ and $h_i(x,y)$ are twice continuously differentiable;\label{assumption:general1(1)}
    \item $\mathcal{X}$ is convex and compact, and for any $x\in\mathcal{X}$ there exists $y\in\mathcal{Y}(x)$ such that $h_i(x,y)<0$ for any $i\in\{1,2,...,k\}$;\label{assumption:general1(2)}
    \item Linear Independence Constraint Qualification (LICQ) is satisfied for any $(x, y)$ where $x \in \mathcal{X}$ and $y \in y^\ast(x) = \arg\min_{y \in \mathcal{Y}(x)} g(x, y)$, i.e. the gradients $\{\nabla_{y} h_i(x, y) : i \text{ is active}\}$ are linearly independent for any $x \in \mathcal{X}$ and $y \in y^\ast(x)$. \label{assumption:general1(3)}
\end{enumerate}
\end{assumption}

Assumption \ref{assumption:general1}(\ref{assumption:general1(1)}) is a basic assumption in the BLO literature \citep{ghadimi2018approximationmethodsbilevelprogramming,hong2022twotimescaleframeworkbileveloptimization,xu2023efficient}. Assumption \ref{assumption:general1}(\ref{assumption:general1(2)}) says that Slater's condition is satisfied for any $x$ which is a common assumption for studying BLO problems under constraints \citep{tsaknakis2023implicit,schmidt2023gentle,beck2023survey}. Assumption \ref{assumption:general1}(\ref{assumption:general1(3)}) is also a standard requirement, commonly used across various studies in BLO~\citep{jiang2024primal,kornowski2024firstordermethodslinearlyconstrained,pmlr-v202-khanduri23a}. 

Next, we have two different set of assumptions on $g(x,y)$ and $h_i(x,y)$'s.

\begin{assumption}\label{assumption:nonlinear} $g(x,y)$ is $\mu_g$-strongly convex in $y$ for any $x\in\mathcal{X}$; $h_i(x,y)$ is convex in $y$ for any $x\in\mathcal{X}$ and $i\in\{1,...,k\}$.
\end{assumption}

\begin{assumption}\label{assumption:linear}$g(x,y)$, $h_i(x,y)$'s are all linear in $y$ for for any $x\in\mathcal{X}$ and $i\in\{1,...,k\}$, i.e. the lower-level problem is a linear program, and  $\mathcal{Y}(x)$ is compact for any $x\in\mathcal{X}$.
\end{assumption}

\begin{remark}
    When the lower-level problem is a linear program, to make sure the feasible set $\mathcal{Y}(x)$ is compact, the number of constraints should be greater than the dimension of the lower-level problem, i.e. $k\geq m$.
\end{remark}

The strong convexity requirement in Assumption \ref{assumption:nonlinear} has been the focus of most recent research on BLO, such as \cite{xu2023efficient,pmlr-v202-khanduri23a,hong2022twotimescaleframeworkbileveloptimization,yao2024overcominglowerlevelconstraintsbilevel}. In contrast, Assumption \ref{assumption:linear}, which requires that the lower-level problem is a linear program, has not yet been explored. The compactness of the feasible set in the lower-level problem is a technical assumption made to ensure that the hyperfunction in the barrier reformulation is differentiable, which is observed in the following proposition. See Appendix \ref{proof:prop:diff_hyperfun} for proof.

\begin{proposition}\label{prop:diff_hyperfun}
    Suppose that Assumption \ref{assumption:general1}(\ref{assumption:general1(1)}) and Assumption \ref{assumption:general1}(\ref{assumption:general1(2)}) hold, and either Assumption \ref{assumption:nonlinear} or Assumption  \ref{assumption:linear} holds. 
    Then $\widetilde{g}_t(x,y)$ is strictly convex in $y$, and $\widetilde{\phi}_t(x)$ is differentiable for any $t>0$ and $x\in\mathcal{X}$.
\end{proposition}

We note that the above result holds regardless of the number of the lower-level optimal solutions for the original problem \eqref{eq:blo_constrained}. The differentiability of $\widetilde{\phi}_t(x)$ allows us to design algorithms for the barrier reformulated problem. However, before we design algorithms, we need to understand the relationship between $\widetilde{\phi}_t(x)$ and the original hyperfunction $\phi(x)$.

\subsection{Convergence of Hyperfunction Value for Barrier Reformulation}

In this section, we study the relationship between the hyperfunction of problems \eqref{eq:blo_constrained} and its barrier reformulation \eqref{eq:blo_barrier_reformulation}. Such a study is critical as it helps reveal the utility of the proposed reformulation.

Under Assumption \ref{assumption:nonlinear} where the lower-level problem is strongly convex in $y$, it can be shown that hyperfunction of the original problem and that of the barrier reformulated problem can be bounded uniformly (see \cite{tsaknakis2023implicit}):

\begin{theorem}[Lemma 1 in \cite{tsaknakis2023implicit}]\label{thm:Ioannis}
    Suppose that Assumption \ref{assumption:general1}(\ref{assumption:general1(1)}), \ref{assumption:general1}(\ref{assumption:general1(2)}), and Assumption \ref{assumption:nonlinear} hold. If $f$ is Lipschitz continuous with coefficient $L_f$, then the following holds:
     \begin{align*}
        |\widetilde{\phi}_t(x)-\phi(x)|\leq L_f \left(\frac{2kt}{\mu_g}\right)^{1/2}.
    \end{align*}
\end{theorem}

Under Assumption \ref{assumption:linear}, i.e.\ the linear setting, we aim to establish the error bound 
between the hyperfunction values as a function of $t$ when $y^\ast(x)$ is unique. The reason that we do not attempt to show global uniform bounds is that in the case where the lower-level problem is a linear program, the hyperfunction of the original problem may be discontinuous, and even for the continuous points, the bound may not be uniform. Below we provide an example.

\begin{example}\label{counterexample1}
    Consider a BLO with  $f(x,y)=y_1$ with $x\in\RR$, $y=(y_1,y_2)^\top\in\RR^2$. 
Further, the lower-level problem takes the following form:
\begin{align*}
    \min_{y}\ &g(x,y):= xy_1+y_2\\
    \rm{s.t.}\quad & y_2\geq 0, \; y_1\geq-1, \; y_1\leq 1.
\end{align*}
It is easy to see that the lower-level problem has the following solutions:
\begin{align*}
    \arg \min_y {g(x,y)} = \begin{cases} 
(-1, 0), & \text{if } x > 0 \\
[-1,1]\times\{0\}, & \text{if } x=0 \\
(1, 0), & \text{if } x < 0.
\end{cases}
\end{align*}
The barrier-penalized reformulation of $g(x,y)$ is:
\begin{align*}
    \widetilde{g}_t(x,y)=xy_1+y_2-t(\log(y_2)+\log(1-y_1)+\log(1+y_1)).
\end{align*}
Direct calculation shows that its stationary point is
\begin{align*}
    \begin{cases} 
y_1=\begin{cases} 
0, & \text{if } x = 0 \\
\frac{t-\sqrt{t^2+x^2}}{x}, & \text{if } x \ne 0 
\end{cases} & \\
y_2=t. &  
\end{cases}
\end{align*}
The hyperfunction of the original problem is
\begin{align*}
\phi(x) = \begin{cases} 
-1, & \text{if } x \geq 0 \\
1, & \text{if } x < 0 
\end{cases}
\end{align*}
and the hyperfunction corresponding to the barrier-penalized reformulation is
\begin{align*}
\widetilde\phi_t(x) = \begin{cases} 
0, & \text{if } x = 0 \\
\frac{t-\sqrt{t^2+x^2}}{x}, & \text{if } x \ne 0 . 
\end{cases}
\end{align*}
However, $\left|\widetilde\phi_t(0)-\phi(0)\right|=1$ for any $t$. Moreover, when $x\ne 0$, we have $\left|\widetilde\phi_t(x)-\phi(x)\right|=|\frac{t-\sqrt{t^2+x^2}}{x} +1|$ when $x>0$ or $\left|\widetilde\phi_t(x)-\phi(x)\right|=|\frac{t-\sqrt{t^2+x^2}}{x} -1|$ when $x<0$, and both of these terms do not admit a uniform bound with respect to $t$.
\end{example}  
Clearly, the discontinuity at $x=0$ arises because $y^\ast(x)$ experiences a jump when transitioning from $x>0$ to $x<0$, or, in other words, $y^\ast(0)$ is not unique. Further, it is important to note that even if we exclude the point $0$, there does not exist a uniform error bound as a function of $t$ and is independent of $x$. Therefore, under the linear setting, we will only analyze those $x$ that admits {\it unique} $y^\ast(x)$, and we seek to establish a {\it pointwise} convergence relationship, that is, $\lim_{t\to 0}\widetilde{\phi}_t(x)=\phi(x)$. In the case where the lower-level problem is simply a linear program the uniqueness of the solution then depends on its dual solution being non-degenerate. Since vector $x$ plays the role of the right-hand side of the dual constraints in this case, if $x$ is subject to small perturbations or imprecision then by the perturbation theory for linear programming, the dual optimal solution of the lower-level problem is likely to be non-degenerate after perturbation. Hence, the primal, which is the lower-level problem itself, will likely have a unique optimal solution in that case.

To proceed with our analysis, we first present the following known result (exercise 11.12(b) in Chapter 11 of \cite{boyd2004convex}).
\begin{lemma}[Optimality gap]\label{lem:optimalitygap}
    Suppose that Assumption \ref{assumption:general1}(\ref{assumption:general1(1)}), \ref{assumption:general1}(\ref{assumption:general1(2)}) hold, and $g(x,y)$, $h_i(x,y)$'s are convex in $y$, then we have the optimality gap
    \begin{align*}
        g(x,y^\ast_t(x))-g(x,y^\ast(x))\leq kt.
    \end{align*}
\end{lemma}

To use the above result under linear setting, note that if $y^\ast(x)$ is unique, then for any point $y$ within the feasible region of the lower-level problem, the cosine of the angle between $y-y^\ast(x)$ and the $\nabla_y g(x,y^\ast(x))$ has a positive lower bound for any fixed $x$, which we denote as $\tau(x)$, i.e.
$$
\tau(x) = \argmin_{y\in \mathcal{Y}(x)}\frac{\left\langle y-y^\ast(x),\nabla_y g(x,y^\ast(x))\right\rangle}{\|y-y^\ast(x)\|\|\nabla_y g(x,y^\ast(x))\|}.
$$
It follows that:
\begin{align*}    
    g(x,y)-g(x,y^\ast(x))=\left\langle y-y^\ast(x),\nabla_y g(x,y^\ast(x))\right\rangle\geq\tau(x)\|\nabla_y g(x,y^\ast(x))\|\cdot\|y-y^\ast(x)\|.
\end{align*}
By uniqueness of $y^\ast(x)$ and Slater condition, $\nabla_y g(x,y^\ast(x))\ne 0$, then we obtain:
\begin{align}\label{eq:lineargap}
    \|y-y^\ast(x)\|\leq\frac{1}{\tau(x)\|\nabla_y g(x,y^\ast(x))\|}\left(g(x,y)-g(x,y^\ast(x))\right).
\end{align}
Combining Lemma \ref{lem:optimalitygap}, we obtain
$$\|y-y^\ast(x)\|\leq\frac{kt}{\tau(x)\|\nabla_y g(x,y^\ast(x))\|}.$$ If we further assume that $f(x,y)$ is $L_f$-Lipschitz continuous, then
$$|\widetilde{\phi}_t(x)-\phi(x)|=|f(x,y^\ast_t(x))-f(x,y^\ast(x))|\leq L_f \frac{kt}{\tau(x)\|\nabla_y g(x,y^\ast(x))\|}.$$
To conclude, we establish the following theorem:
\begin{theorem}[Hyperfunction value convergence under Linear setting]
    Suppose that Assumption \ref{assumption:general1}(\ref{assumption:general1(1)}), \ref{assumption:general1}(\ref{assumption:general1(2)}), and Assumption \ref{assumption:linear} hold. If $f(x,y)$ is Lipschitz continuous with coefficient $L_f$ and $y^\ast(x)$ is unique at point $x$, then
     \begin{align*}
        |\widetilde{\phi}_t(x)-\phi(x)|\leq L_f \frac{kt}{\tau(x)\|\nabla_y g(x,y^\ast(x))\|},
    \end{align*}
    where $\tau(x):=\argmin_{y\in\mathcal{Y}(x)}\cos\left<\nabla_y g(x,y^\ast(x)),y-y^\ast(x)\right>>0$.
\end{theorem}

\subsection{Convergence of Hypergradient for Barrier Reformulation}

In this subsection, we explore the convergence behavior of the hypergradient of the barrier reformulation $\nabla_x\widetilde{\phi}_t(x)$ as $t$ goes to zero. Understanding this convergence is vital because optimization algorithms commonly converge to points where the gradient is zero. Therefore, it is crucial to ensure that when $\nabla_x\widetilde{\phi}_t(x)$ approaches zero, $\nabla_x\widetilde{\phi}_t(x)$ also becomes sufficiently small. Since the hyperfunction of the original problem may be discontinuous (under linear setting, see Example \ref{counterexample1}) or non-smooth (under strongly convex setting, see the example in \citet[Page 3]{pmlr-v202-khanduri23a}), we aim to show asymptotic convergence at the smooth points of $\phi(x)$.

We present the following definition as an essential characterization of the smoothness of the original problem.

\begin{definition}[SCSC point]\label{def:scsc_point}
    Suppose Assumption \ref{assumption:general1} holds. We say $x$ is an \textbf{SCSC (Strict Complementarity Slackness Condition) point} if for any $y\in y^\ast(x)$, we have that $h_{i}(x,y)=0$ implies $\lambda_i(x,y) >0$, where $y^\ast(x)$ is the set of optimal solutions for the lower-level problem, and $\lambda_i(x,y)$ is the optimal Lagrangian multiplier corresponding the the $i$-th constraint.
\end{definition}

We first provide a remark regarding the notation of the multipliers.

\begin{remark}\label{remark:welldefined}
    Under the strong convex setting, $y^\ast(x)$ is unique, so $\lambda_i(x,y^\ast(x))$ only depends on $x$. Under the linear setting, although $y^\ast(x)$ is not unique and forms a set, due to the LICQ assumption, the value of the multiplier $\lambda_i(x,y)$ actually is independent of the choice of $y$ in the optimal solution set $y^\ast(x)$. For the specific proof, see (a)$\Rightarrow$(b) part of the proof of the following Proposition \ref{prop:relation}. Therefore, in both cases we denote $\lambda_i(x,y^\ast(x))$ simply as $\lambda_i(x)$ without loss of generality.
\end{remark}

Note that $\nabla_x\phi(x)=\nabla_x f(x,y)+(\nabla_x y^\ast(x))^\top\nabla_y f(x,y)$, the existence of $\nabla_x y^\ast(x)$ is a sufficient condition of the differentiability of the hyperfunction $\phi(x)$. The following proposition shows the relation between \textbf{SCSC} point and the existence of $\nabla_x y^\ast(x)$. Please see Appendix \ref{proof:prop:relation} for the proof.

\begin{proposition}\label{prop:relation}
    Suppose Assumption \ref{assumption:general1} holds. We have the following results:
    \begin{enumerate}
        \item If Assumption \ref{assumption:nonlinear} holds, then $x$ is an \textbf{SCSC} point implies $\nabla_x y^\ast(x)$ exists, but not vice versa;
        \item If Assumption \ref{assumption:linear} holds, then following conditions are equivalent:
        \begin{enumerate}
            \item $x$ is an \textbf{SCSC} point;
            \item $y^\ast(x)$ is unique;
            \item $\nabla_x y^\ast(x)$ exists.
        \end{enumerate}
    \end{enumerate}
\end{proposition}

We point out that \citet[Page 3]{pmlr-v202-khanduri23a} provides an example where the hyperfunction is nonsmooth when \textbf{SCSC} is not satisfied, under the case that the lower-level objective function is strongly convex.

Next, we develop an analysis of the asymptotic behavior of the error based on the \textbf{SCSC} assumption, that is, we will show $\lim_{t\to 0}\nabla_x\widetilde{\phi}_t(x)=\nabla_x\phi(x)$. By direct computation, we obtain:
\begin{align*}
    \lim_{t\to 0}\nabla_x\widetilde{\phi}_t(x)&=\lim_{t\to 0}\left(\nabla_x f(x,y^\ast_t(x))+\left(\nabla_x y^\ast_t(x)\right)^\top\nabla_y f(x,y^\ast_t(x))\right)\\
    &=\nabla_x f(x,y^\ast(x))+\left(\lim_{t\to 0}\nabla_x y_t^\ast(x)\right)^\top\nabla_y f(x,y^\ast(x)),\\
    \nabla_x\phi(x)&=\nabla_x f(x,y^\ast(x))+\left(\nabla_x y^\ast(x)\right)^\top\nabla_y f(x,y^\ast(x)).
\end{align*}
Thus, the core issue is to prove that the Jacobian matrix converges, i.e.\ $\lim_{t\to 0}\nabla_x y_t^\ast(x)=\nabla_x y^\ast(x)$. We will prove this conclusion separately for the linear and strongly convex settings.

First, under the linear setting, the proof is relatively simple. This is because $\lim_{t\to 0}\nabla_x y_t^\ast(x)$ has a very straightforward form, and $\nabla_x y^\ast(x)$ also has a simple local expression (see (\ref{eq:localexpression})), which exactly equals $\lim_{t\to 0}\nabla_x y_t^\ast(x)$. In particular, we have the following result, whose proof can be found in Appendix \ref{proof:thm:mainlinear}.

\begin{theorem}[Jacobian convergence under Linear setting]\label{thm:mainlinear}
    Suppose Assumption \ref{assumption:general1}, \ref{assumption:linear} hold. If $x$ is an \textbf{SCSC} point of BLO problem (\ref{eq:blo_constrained}), then we have $\lim_{t\to 0}\nabla_x y_t^\ast(x)=\nabla_x y^\ast(x)$.
\end{theorem}

Under the strongly convex setting, the situation is more complicated. The reason is that $\lim_{t\to 0}\nabla_x y_t^\ast(x)$ becomes very difficult to determine. Even though $\nabla_x y^\ast(x)$ is computable \citep[Section 3]{Giorgi2018ATO}, proving that it equals $\lim_{t\to 0}\nabla_x y_t^\ast(x)$ becomes extremely challenging. Therefore, we use a different method to prove the conclusion.

Note that under the strongly convex setting, $y^\ast_t(x)$ uniformly converges to $y^\ast(x)$, which is a direct consequence of the optimality gap in Lemma \ref{lem:optimalitygap}. To prove $\lim_{t\to 0}\nabla_x y_t^\ast(x)=\nabla_x \lim_{t\to 0}y_t^\ast(x)=\nabla_x y^\ast(x)$, we essentially need to show that differentiation and limits are interchangeable. A common sufficient condition for this interchangeability is that the Jacobian matrix $\nabla_x y^\ast_t(x)$ converges uniformly in a neighborhood around $x$.
By proving the uniform convergence of $y^\ast_t(x)$ to $y^\ast(x)$ (see Lemma \ref{lem:optimalitygapforpoint} in Appendix \ref{proof:somebasiclemmas}) and local uniform convergence of $\nabla_x y^\ast_t(x)$ (see Lemma \ref{lem:uniformconvergeofgradient} in Appendix \ref{proof:thm:mainnonlinear}) of the Jacobian around the \textbf{SCSC} point, we are able to obtain the following conclusion.

\begin{theorem}[Jacobian convergence under Strongly convex setting]\label{thm:mainnonlinear}
    Suppose Assumption \ref{assumption:general1}, and Assumption \ref{assumption:nonlinear} hold, i.e. when $g(x,y)$ is  strongly convex  in $y$. If $x$ is an \textbf{SCSC} point of BLO problem (\ref{eq:blo_constrained}), then we have $\lim_{t\to 0}\nabla_x y_t^\ast(x)=\nabla_x y^\ast(x)$.
\end{theorem}

To conclude, we have the following main result for the hypergradient convergence in two cases.

\begin{theorem}[Hypergradient convergence]\label{thm:main_result}
    Suppose Assumption \ref{assumption:general1} holds, and either Assumption \ref{assumption:nonlinear} or Assumption \ref{assumption:linear} holds. If $x$ is an \textbf{SCSC} point of BLO problem (\ref{eq:blo_constrained}), then $\lim_{t\to 0}\nabla_x\widetilde{\phi}_t(x)=\nabla_x\phi(x)$,
    where $\nabla_x\widetilde{\phi}_t(x)$ is the hypergradient of the barrier reformulation (\ref{eq:blo_barrier_reformulation}).
\end{theorem}

The reader may wonder if the above convergence can be made uniform. However the remark below indicates that such a form of convergence is impossible.
\begin{remark}
    Denote the set of \textbf{SCSC} points as $\widetilde{\mathcal{X}}$. Under Assumption \ref{assumption:general1}, Assumption \ref{assumption:nonlinear} or Assumption \ref{assumption:linear}, it is impossible to get uniform error estimation between $\nabla_x\phi(x)$ and $\nabla_x\widetilde{\phi}_t(x)$ even on the set $\widetilde{\mathcal{X}}$. This is because if the convergence of hypergradient for the barrier reformulation would be uniform on $\widetilde{\mathcal{X}}$, then it is not hard to prove that $\nabla_x\phi(x)$ is well defined for any $x$ on the boundary of $\widetilde{\mathcal{X}}$, denoted as $\partial \widetilde{\mathcal{X}}$. However, the hypergradient of the original problem may not be well defined at some point on $\partial \widetilde{\mathcal{X}}$ (see Example \ref{counterexample1} and the example in \citet[Page 3]{pmlr-v202-khanduri23a}), which leads to a contradiction.
\end{remark}

\section{Algorithm Design}\label{section:Algorithm Design}

So far we have studied the relations between the original problem \eqref{eq:blo_constrained} and the barrier reformulated problem \eqref{eq:blo_barrier_reformulation}, including the function and gradient approximation errors. Now we turn to focus on the algorithm design for solving the barrier reformulated problem \eqref{eq:blo_barrier_reformulation} and provide both theoretical and numerical evidence to support its effectiveness. We aim at designing algorithm that works both for the linear and the strongly convex setting. To concisely state the assumptions required for this section, we present them as follows:

\begin{assumption}\label{assumption:nonlinear2} $g(x,y)$ is $\mu_g$-strongly convex in $y$ for any $x\in\mathcal{X}$; $h_i(x,y)$ is convex in $y$ for any $x\in\mathcal{X}$ and $i\in\{1,...,k\}$.
\end{assumption}

\begin{assumption}\label{assumption:linear2}$g(x,y)$, $h_i(x,y)$ are all linear in $y$ for any $x\in\mathcal{X}$ and $i\in\{1,...,k\}$, i.e.\ the lower-level problem is a a linear program.
\end{assumption}

Note that we only require one of the above two assumptions. Designing an algorithm that works for both settings is highly non-trivial since we lose the global Lipschitz gradient for the lower-level problem for both settings due to the existence of the log-barrier in \eqref{eq:blo_barrier_reformulation}, and the Lipschitz smoothness of the hyperfunction $\widetilde{\phi}_t$ is largely unknown. In \cite{tsaknakis2023implicit}, to avoid this difficulty, the authors impose an assumption (Assumption 3 in \cite{tsaknakis2023implicit}) over the function values of $h_i$ at the lower-level optimal points that directly results in the Lipschitz smoothness of the lower-level problem and hypergradient. However, this assumption is not verifiable and the level of difficulty drastically increases without such assumption. We thus design our algorithm by carefully controlling the local Lipschitz smoothness constant around the lower-level optimal solution at each step, and meticulously controlling the stepsize for each update point. To proceed, we first impose the following assumptions.

\begin{assumption}\label{assumption:general3}The following holds for every $x\in\mathcal{X},y\in\mathcal{Y}(x)$, and $i\in\{1,...,k\}$
\begin{enumerate}
    \item $f(x,y)$ is one time and $g(x,y)$, $h_i(x,y)$ are two times continuously differentiable;\label{assumption:general3(1)}
    \item $\mathcal{X}$ is convex and compact, and for any $x\in\mathcal{X}$ there exists $y\in\mathcal{Y}(x)$ such that $h_i(x,y)<0$ for any $i\in\{1,2,...,k\}$;\label{assumption:general3(2)}
    \item For any $x\in\mathcal{X}$, $\mathcal{Y}(x)$ is compact, and $\|y\|\leq R$ for any $y$ in $\mathcal{Y}(x)$;\label{assumption:general3(3)}
    \item  LICQ is satisfied for any $(x, y)$ where $x \in \mathcal{X}$ and $y \in y^\ast(x) = \arg\min_{y \in \mathcal{Y}(x)} g(x, y)$, i.e. the gradients $\{\nabla_{y} h_i(x, y) : i \text{ is active}\}$ are linearly independent for any $x \in \mathcal{X}$ and $y \in y^\ast(x)$.\label{assumption:general3(4)}
\end{enumerate}
\end{assumption}

Note that differently as Assumption \ref{assumption:nonlinear}, we have assumed the compactness of $\mathcal{Y}(x)$ for {\it both} the linear and the strongly convex setting. We also need the following smoothness assumptions.

\begin{assumption}\label{assumption:general2}The following holds for every $x,\overline{x}\in\mathcal{X},y,\overline{y}\in\mathcal{Y}(x)$, and $i\in\{1,...,k\}$
\begin{enumerate}
    \item $\|\nabla_{(x,y)} f(x,y)\|\leq L_f$;\label{assumption:general2(1)}
    \item $\|\nabla_{(x,y)} f(x,y)-\nabla_{(x,y)} f(\overline{x},\overline{y})\|\leq \overline{L}_f\|(x,y)-(\overline{x},\overline{y})\|$\label{assumption:general2(2)};
    \item $\|\nabla_{(x,y)} g(x,y)\|\leq L_g$; \label{assumption:general2(3)}
    \item $\|\nabla^2_{yy} g(x,y)\|\leq\overline{L}_g$; $\|\nabla^2_{xy} g(x,y)\|\leq \overline{L}_g$;\label{assumption:general2(4)}
    \item $\|\nabla^2_{xy} g(x,y)-\nabla^2_{xy} g(\overline{x},\overline{y})\|\leq \overline{\overline{L}}_{g}\|(x,y)-(\overline{x},\overline{y})\|$; $\|\nabla^2_{yy} g(x,y)-\nabla^2_{yy} g(\overline{x},\overline{y})\|\leq \overline{\overline{L}}_{g}\|(x,y)-(\overline{x},\overline{y})\|$\label{assumption:general2(5)};
    \item $\|\nabla_{(x,y)} h(x,y)\|\leq L_h$\label{assumption:general2(6)};
    \item $\|\nabla^2_{yy} h(x,y)\|\leq\overline{L}_h$; $\|\nabla^2_{xy} h(x,y)\|\leq \overline{L}_h$;\label{assumption:general2(7)}
    \item $\|\nabla^2_{xy} h(x,y)-\nabla^2_{xy} h(\overline{x},\overline{y})\|\leq \overline{\overline{L}}_{h}\|(x,y)-(\overline{x},\overline{y})\|$; $\|\nabla^2_{yy} h(x,y)-\nabla^2_{yy} h(\overline{x},\overline{y})\|\leq \overline{\overline{L}}_{h}\|(x,y)-(\overline{x},\overline{y})\|$.
\end{enumerate}
\end{assumption}

These assumptions, except Assumption \ref{assumption:general2}(\ref{assumption:general2(3)}), are standard in BLO literature such as \cite{ghadimi2018approximationmethodsbilevelprogramming,tsaknakis2023implicit,shen2023penaltybasedbilevelgradientdescent}. Assumption \ref{assumption:general2}(\ref{assumption:general2(3)}), which is an assumption of an upper bound for $\nabla g(x,y)$, can be implied by the compactness of $\mathcal{Y}(x)$ and $\mathcal{X}$.

\subsection{The Proposed Algorithm}

Now we are ready to establish the algorithm framework. We follow the standard bilevel implicit gradient method as in \cite{ghadimi2018approximationmethodsbilevelprogramming,ji2021bilevel,hong2022twotimescaleframeworkbileveloptimization}, where we estimate the hypergradient $\nabla_x\phi(x)$ at each update point $x$ and conduct a gradient descent update to get the next iterate. 

Specifically, we first find the approximate optimal solution $\hat{y}$ of the lower-level solver $y^*(x)$ for a fixed $x$, and such $\hat{y}$ is obtained by some careful design (see Algorithm \ref{algo:fo_lower_2}). In particular, since the lower-level objective $\widetilde{g}_t$ is not globally Lipschitz smooth in $y$, we need to carefully shrink the feasible set and estimate a local Lipschitz smoothness constant in the shrunk set, which is reflected in \textbf{Step 3} of Algorithm \ref{algo:fo_lower_2}. Then an accelerated gradient descent step is employed to approximate the lower-level optimal solution in \textbf{Step 4} of Algorithm \ref{algo:fo_lower_2}. 

Upon obtaining the approximate optimal lower-level solution, we further design the update for variable $x$ in Algorithm \ref{algo:fo}. In particular, we use the implicit gradient methods to compute the approximate gradient of the hyperfunction, which is
\begin{align*}
    \hat{\nabla}_x \widetilde{\phi}_t(x)=\nabla_x f(x,\hat{y})-\nabla^2_{xy}\widetilde{g}_t(x,\hat{y})\left(\nabla^2_{yy}\widetilde{g}_t(x,\hat{y})\right)^{-1} \nabla_y f(x,\hat{y}),
\end{align*}
where recall that the barrier reformulation  for the lower-level problem is defined as
\begin{align}\label{eq:barrierreformulationlowerlevel}
    \widetilde{g}_t(x,y)=g(x,y)-t\sum_{i=1}^k \log(-h_i(x,y)).
\end{align}
The {\bf main technical challenge} to analyze the proposed algorithm is that the hypergradient may suffer from non-Lipschitzness due to log barriers, similarly as the lower-level objective.  In addition, the inexact lower-level solution $\hat{y}$ introduces approximation errors to the exact lower-level solver $y^*(x)$ which further increases the difficulties of the analysis.  We overcome these difficulties and obtain an $\widetilde{\mathcal{O}}(1/\epsilon^{2})$ rate of convergence (Theorem \ref{thm:total_rate}) by carefully designing the upper-level stepsize $\eta_s$, which contains the information of local Lipschitz constants for the hypergradient. See Section \ref{sec:convergence_analysis} for details of the convergence analysis.
  
\begin{algorithm}[ht]
\caption{Hypergradient Based Bilevel Barrier Method for \eqref{eq:blo_barrier_reformulation}}\label{algo:fo}
\begin{algorithmic}
\State \textbf{Step 0: Initialization.} Given an initial point $x_0$, the accuracy level $\epsilon$, total iteration $S$.

\State \For{$s=0,1,...,S$}{
\State \textbf{Step 1: Solve the inner loop.} Solve the lower-level problem for $\hat{y}_s$ via Algorithm \ref{algo:fo_lower_2} such that $\|\hat{y}_s-y^*(x_s)\|\leq\epsilon_s$.
\State \textbf{Step 2: Update $x_s$.} 
\State Compute the following approximate gradient of $\widetilde{\phi}_t(x)$ at $x_s$
$$
\hat{\nabla}_x\widetilde{\phi}_t(x_s)=\nabla_x f(x_s, \hat{y}_s) - \nabla_{xy} \widetilde{g}_t(x_s,\hat{y}_s) (\nabla_{yy}^2 \widetilde{g}_t(x_s,\hat{y}_s))^{-1} \nabla_y f(x_s, \hat{y}_s)
$$
\State Update $x_s$ by
\begin{align*}
    x_{s+1}=\proj_{\mathcal{X}}\left(x_s - \eta_s\hat{\nabla}_x\widetilde{\phi}_t(x_s)\right)
\end{align*}
where the stepsize $\eta_s$ can be calculated according to (\ref{eq:stepsize}) in Theorem \ref{thm:total_rate}.
}

\State Output $x_s$ with $s=\argmin_{s=0,...,S-1}\{\frac{1}{\eta_s}\|x_{s} - x_{s+1}\|\}$
\end{algorithmic}
\end{algorithm}

\begin{algorithm}[ht]
\caption{lower-level solver for Algorithm \ref{algo:fo}}\label{algo:fo_lower_2}
\begin{algorithmic}
\State \textbf{Step 0: Initialization.} Given reference point $x_s$, stepsize $\gamma$, and constant $d=1$.
\State \textbf{Step 1: Find the initial $d$ such that $\mathcal{Y}_{d}$ is nonempty.} \\
\While{$\mathcal{Y}_{d}(x_s):=\{y:h_i(x_s,y)\leq -d,\ i=1,...,k\}$ is empty}{$d=d/2$} \\
\State Denote $d_s:=d$ and initialize $y_0\in \mathcal{Y}_{d_s}(x_s)$, 
\State \textbf{Step 2: Estimate the local Lipschitz smoothness constant.}
\begin{enumerate}
    \item Compute scalar $m_s$ by replacing $d$ in (\ref{eq:m}) with $d_s$ from \textbf{Step 1}. Thus $y^\ast_t(x_s)$ is contained in $\mathcal{Y}_{m_s}(x_s):=\{y:h_i(x_s,y)\leq -m_s,\ i=1,...,k\}$;
    \item Compute the Lipschitz smoothness constant $\overline{L}_{\widetilde{g}_t,m_s}$ of $\widetilde{g}_t(x,y)$ in $y$ on $\mathcal{Y}_{m_s}(x_s)$; $\overline{L}_{\widetilde{g}_t,m_s}$ can be computed by replacing $m$ in (\ref{eq:prop:lipschitz_smooth}) with $m_s$.
\end{enumerate}
\State \textbf{Step 3: Update based on current estimate of Lipschitz smoothness constant.}
\For{$j=0,1,...,J-1:=\mathcal{O}(\sqrt{\overline{L}_{\widetilde{g}_t,m_s}/\mu_{\widetilde{g}_t}} \log(1/\epsilon_s))$}{
\State Update
\begin{align*}
    \omega&=\begin{cases}
        y_0 &\text{ if }j=0\\
        y_j+\frac{1-\sqrt{\mu_{\widetilde{g}_t}/\overline{L}_{\widetilde{g}_t,m_s}}}{1+\sqrt{\mu_{\widetilde{g}_t}/\overline{L}_{\widetilde{g}_t,m_s}}}(y_j-y_{j-1}) &\text{ if }j\geq 1,
    \end{cases}\\
    y_{j+1}&=\proj_{\mathcal{Y}_{m_s}(x)}\left(y_{j}-\frac{1}{\overline{L}_{\widetilde{g}_t,m_s}}\nabla_y\widetilde{g}_t(x_s,\omega)\right),
\end{align*}
where the strongly convexity constant $\mu_{\widetilde{g}_t}$ is from Proposition \ref{prop:strongly_convex}.}
\State \textbf{Step 4: Output results.} Output $\hat{y}_s:=y_{J}$ and constant $d_s$, $m_s$.
\end{algorithmic}
\end{algorithm}

\subsection{Convergence Analysis}\label{sec:convergence_analysis}

Now we aim to develop an algorithm with non-asymptotic convergence rate. We first discuss the main difficulty.

\noindent\textbf{Technical Difficulty:} Algorithm \ref{algo:fo} leverages the following approximation for the hypergradient:
\begin{align*}
    \hat{\nabla}_x \widetilde{\phi}_t(x)=\nabla_x f(x,\hat{y})-\nabla^2_{xy}\widetilde{g}_t(x,\hat{y})\left(\nabla^2_{yy}\widetilde{g}_t(x,\hat{y})\right)^{-1} \nabla_y f(x,\hat{y}),
\end{align*}
and it is necessary to estimate the Lipschitz constant of $\hat{\nabla}_x \widetilde{\phi}_t(x)$, which further requires an upper bound estimate of $\nabla^2_{xy}\widetilde{g}_t(x,\hat{y})$. By directly computing the mixed Hessian (or referred to as Jacobian)
\begin{align*}
\nabla^2_{xy}\widetilde{g}_t(x,\hat{y})&=\nabla^2_{xy}g(x,\hat{y})+t\sum_{i=1}^k\left(\frac{\nabla^2_{xy}h_i(x,\hat{y})}{-h_i(x,\hat{y})}+\frac{\nabla_x h_i(x,\hat{y})\nabla_y h_i(x,\hat{y})^\top}{h_i^2(x,\hat{y})}\right),
\end{align*}
we need a negative upper bound estimate of $h_i(x,\hat{y})$, which can be achieved by upper bounding $h_i(x,y^\ast_t(x))$ if $h_i(x,\hat{y})$ and $h_i(x,y^\ast_t(x))$ are close to each other. However upper bounding $h_i(x,y^\ast_t(x))$ is hard: if $t$ is sufficiently small, for those indices $i$ that are active in the original problem, $h_i(x,y^\ast_t(x))$ will also approach zero. On the other hand, in order to make sure that $h_i(x,\hat{y})$ and $h_i(x,y^\ast_t(x))$ are as close as possible, we need to design an effective algorithm for the lower-level problem. It is challenging to design an effective algorithm: due to the log-barrier, the gradient of barrier reformulation $\widetilde{g}_t(x)$ tends to explode near the boundary, making it not globally Lipschitz smooth. Specifically, observe that the gradient of lower-level problem, given below
$$\nabla_y \widetilde{g}_t(x,y)=\nabla_y g(x,y)+t\sum_{i=1}^k\frac{\nabla_y h_i(x,y)}{-h_i(x,y)},$$
 may goes to infinity when $-h_i(x,y)$ goes to zero. To overcome this, we consider the lower-level problem in a shrunk set $\mathcal{Y}_{d}(x):=\{y:h_i(x,y)\leq -d,\ i=1,...,k\}$. In order to determine the size of $d$,  such that $y^\ast_t(x)$ is still in this shrunk set $\mathcal{Y}_{d}(x)$, we need to identify a negative value that upper bounds $h_i(x,y^\ast_t(x))$. The next result establishes such an upper bound. Please see Appendix \ref{proof:lem:local_M} for proof.

\begin{theorem}\label{lem:local_M}
    Suppose $0< t\leq T$, Assumption \ref{assumption:general3}(\ref{assumption:general3(1)})(\ref{assumption:general3(2)})(\ref{assumption:general3(3)}) and Assumption \ref{assumption:general2} all hold, and either Assumption \ref{assumption:nonlinear2} or Assumption \ref{assumption:linear2} holds. For a fixed $x$, if $\min_{y\in\mathcal{Y}(x)}\{\max_{i\in\{1,...,k\}}h_i(x,y)\}\leq -d<0$. Then for any $i\in\{1,...,k\}$, we have $h_i(x,y_t^\ast(x))\leq -m$, where
    \begin{align}\label{eq:m}
        m:=\min\left\{t\frac{d^2}{4dRL_g+4RTkL_h},\frac{d}{2}\right\}.
    \end{align}
\end{theorem}

\begin{remark}
    Theorem \ref{lem:local_M} still holds when $g(x,y)$ is non-convex in $y$, and $h_i(x,y)$ is convex in $y$ for any $i$. In this case, for $m$ computed in \eqref{eq:m},  $h_i(x,y)\leq -m$ holds for any $y$ in the set $y_t^\ast(x)$.
\end{remark}

Theorem \ref{lem:local_M} provides an interesting way to estimate the negative upper bound at the optimal point $h_i(x,y^\ast_t(x))$: as long as there exists a tuple $(y, d)$ in such that $y\in \mathcal{Y}(x)$ and $h_i(x,y)<-d$ for any $i$, it is possible to compute the desired negative upper bound. This enable us to solve the aforementioned difficulty without explicitly assuming a negative upper bound at the optimal point.

The next result computes the strong convexity constant of lower-level penalized objective $\widetilde{g}_t(x,y)$, allowing us to use accelerated gradient methods to obtain the global optimal solutions. The proof is provided in Appendix \ref{proof:prop:strongly_convex}.

\begin{proposition}\label{prop:strongly_convex}
    Suppose the assumptions in Theorem  \ref{lem:local_M} are satisfied, then for any fix $x$ and $t$, the barrier reformulation for the lower-level problem, expressed in (\ref{eq:barrierreformulationlowerlevel}), is $\mu_{\widetilde{g}_t}$-strongly convex in $y\in\{y:h_i(x,y)<0,\ i=1,...,k\}$. Further
    \begin{enumerate}
        \item $\mu_{\widetilde{g}_t}=\mu_g$ when Assumption \ref{assumption:nonlinear2} holds;
        \item $\mu_{\widetilde{g}_t}=t\frac{\sigma}{H^2}$ when Assumption \ref{assumption:linear2} holds, for some constant $H$ and $\sigma$ be defined as follows:
        \begin{itemize}
            \item Denote $\sigma(x)$ as the smallest eigenvalue of the following matrix (which is independent of $y$ due to the linearity of the lower-level problem)
            $$(\nabla_y h_1(x,y),\cdots,\nabla_y h_k(x,y))^\top (\nabla_y h_1(x,y),\cdots,\nabla_y h_k(x,y)).$$
            Then set $\sigma:=\min_{x\in\mathcal{X}}\sigma(x)$. We also have $\sigma>0$;
            \item $H:=\sup_{i\in\{1,...,k\},x\in\mathcal{X},y\in\mathcal{Y}_x}-h_i(x,y)$. We also have $H<\infty$.
        \end{itemize}
    \end{enumerate}
\end{proposition}

\begin{remark}\label{remark:sc}
    When Assumption \ref{assumption:linear2} holds, i.e. the lower-level problem is a linear program, the constant of strong convexity depends on $\sigma$, which is hard to compute for many applications. In practice, it can be made more computable by adding a constraint $h_{k+1}(x,y)=\|y\|^2\leq R^2$ where $R$ is from Assumption \ref{assumption:general3}. Adding this constraint will not affect the original problem because the feasible set of the original problem is already contained within that ball, and we can immediately obtain $\widetilde{g}_t(x,y)$ is $\frac{2t}{R}$-strongly convex.
\end{remark}

Before proceeding to resolve the technical difficulty we discussed before, we need to emphasize the following remark:

\begin{remark}\label{remark:non-emptyness}
    The assumed compactness of $\mathcal{X}$ and the Slater condition implies that there exists a positive constant $D$ such that $\{y:h_i(x,y)\leq -D,\ i=1,...,k\}\ne\emptyset$ for any $x$, because $\max_{y\in\mathcal{Y}(x)}h_i(x,y)<0$ for any $x\in\mathcal{X}$. In our analysis, we need to use $D$ when proving the convergence rate; however, we do not need to know the value of $D$ a priori during the design of the algorithm.
\end{remark}

In view of Theorem \ref{lem:local_M} and Remark \ref{remark:non-emptyness}, we can immediately obtain the following global estimate on the upper bound for the constraints $h_i(x,y^\ast_t(x))$ and $h_i(x,\hat{y}_s)$, which will also be used for the further convergence proof.

\begin{corollary}\label{remark:upper_bound_of_h_at_approximate_point}
    By Remark \ref{remark:non-emptyness} and \textbf{Step 1} of Algorithm \ref{algo:fo_lower_2}, the feasibility check must terminate when $d$ is between $D/2$ to $D$, thus $d_s$ is greater than $D/2$ for any $s$. By replacing $d$ with $D/2$ in Theorem \ref{lem:local_M}, we have a positive lower bound estimate of $m_s$, denoted as $M$. Note that by \textbf{Step 3} of Algorithm \ref{algo:fo_lower_2}, the projection guarantees that  $\hat{y}_s\in\mathcal{Y}_{m_s}(x_s)$. Therefore we have $h_i(x_s,\hat{y}_s)\leq -M$ for any $s$ and $i$. Since $y^\ast(x_s)$ is also in $\mathcal{Y}_{m_s}(x_s)$ by Theorem \ref{lem:local_M}, we have $h_i(x_s,y^\ast(x_s))\in\mathcal{Y}_{M}(x_s)$. This implies $h_i(x_s,y^\ast(x_s))\leq -M$ for any $s$ and $i$.
    
    In conclusion, $-M$ is the upper bound of both the exact lower-level solution $h_i(x_s,y^\ast(x_s))$, and the inexact lower-level solution $h_i(x_s,\hat{y}_s)$ at each update point .
\end{corollary}

We are now ready to analyze the lower-level convergence of Algorithm \ref{algo:fo_lower_2}. We first have the following Lipschitzness property for the barrier reformulated lower-level problem.

\begin{proposition}\label{prop:lipschitz_smooth}
    Suppose the assumptions in Theorem  \ref{lem:local_M} are satisfied, then the following barrier reformulated lower-level objective function
    \begin{align*}
        \widetilde{g}_t(x,y)=g(x,y)-t\sum_{i=1}^k \log(-h_i(x,y))
    \end{align*}
    is $\overline{L}_{\widetilde{g}_t,m}$-Lipschitz smooth with respect to $y$ in the set $\mathcal{Y}_m(x):=\{y:h_i(x,y)\leq-m,\ i=1,...,k\}$ for any $x$, and 
    \begin{align}\label{eq:prop:lipschitz_smooth}
        \overline{L}_{\widetilde{g}_t,m}=\overline{L}_g+\frac{tk\overline{L}_h}{m}+\frac{tkL_h^2}{m^2}.
    \end{align}
\end{proposition}

The proof of Proposition \ref{prop:lipschitz_smooth} is available in Appendix \ref{proof:prop:lipschitz_smooth}. Now we state the convergence of the lower-level Algorithm \ref{algo:fo_lower_2}.

\begin{theorem}\label{thm:convergence_rage_alg_2}
    Suppose Assumption \ref{assumption:general3}(\ref{assumption:general3(1)})(\ref{assumption:general3(2)})(\ref{assumption:general3(3)}) and Assumption \ref{assumption:general2} all hold, and either Assumption \ref{assumption:nonlinear2} or Assumption \ref{assumption:linear2} holds. For the output of Algorithm \ref{algo:fo_lower_2}, within $\mathcal{O}(\sqrt{\overline{L}_{\widetilde{g}_t,m_s}/\mu_{\widetilde{g}_t}} \log(1/\epsilon_s))$ number of iterations, we have
    \begin{align*}
        \|y_J-y^\ast_t(x)\|\leq \epsilon_s.
    \end{align*}
    Denote $\kappa_s=\overline{L}_{\widetilde{g}_t,m_s}/\mu_{\widetilde{g}_t}$, then $\kappa_s$ has a upper bound estimate $\kappa=\overline{L}_{\widetilde{g}_t,M}/\mu_{\widetilde{g}_t}$, where $\mu_{\widetilde{g}_t}$ is from Proposition \ref{prop:strongly_convex}, $\overline{L}_{\widetilde{g}_t,M}$ can be computed by replacing $m$ in (\ref{eq:prop:lipschitz_smooth}) with $M$ from Corollary \ref{remark:upper_bound_of_h_at_approximate_point}, and $\overline{L}_{\widetilde{g}_t,m_s}$ can be computed by replacing $m$ in (\ref{eq:prop:lipschitz_smooth}) with $m_s$ from the \textbf{Step 2} of Algorithm \ref{algo:fo_lower_2}.
    
\end{theorem}

The proof follows the standard proof of the convergence of the accelerated proximal/projected gradient descent (see \cite{Nesterov2012GradientMF}) since the lower objective $\widetilde{g}_t$ is both $\overline{L}_{\widetilde{g}_t,m_s}$ Lipschitz smooth and $\mu_{\widetilde{g}_t}$ strongly convex in the region $\mathcal{Y}_{m_s}(x)$. By Corollary  \ref{remark:upper_bound_of_h_at_approximate_point}, $m_s$ has a lower bound estimate $M$, so $\overline{L}_{\widetilde{g}_t,m_s}$ has a upper bound estimate $\overline{L}_{\widetilde{g}_t,M}$, which means $\kappa_s$ has a upper bound estimate $\overline{L}_{\widetilde{g}_t,M}/\mu_{\widetilde{g}_t}$, denoted as $\kappa$. The proof is complete.

\begin{remark}
    We remark that the feasibility check (Step 1 of Algorithm \ref{algo:fo_lower_2}) is achievable by the following scheme: let us assume that a solver is available, which is capable of computing the projection operation of a point onto a non-empty convex set. To determine whether the set $\{y:h_i(x,y)\leq-d,\ i=1,...,k\}$ is non-empty, we can employ the solver to project any arbitrary point onto the set. If this solver can return a solution, then we confirm that the set $\{y : h_i(x, y) \leq -d,\ i=1,...,k\}$ is indeed non-empty. Conversely, if the solver can not return a solution, then the set is empty. Therefore we believe that the feasibility check step is no more complicated than a projection onto a convex set.
\end{remark}

The upper-level problem is more challenging than the lower-level problem. This difficulty arises because solving the lower-level problem requires only the value $d_s$ at a point $x_s$, where there exists a $y$ such that $h_i(x_s, y) \leq -d_s$ for any $i$. In contrast, designing the stepsize for the upper-level problem necessitates the global value $D$, where for any $x$, there exists a $y$ such that $h_i(x, y) \leq -D$ for any $i$. Specifically, for lower-level problem, we only need the value $d_s$ at the point $x_s$ and then estimate the value of \(m_s\) in \textbf{Step 2} of Algorithm \ref{algo:fo_lower_2}, so we can obtain the optimal value within the shrunk feasible set $\{y:h_i(x,y)\leq -m_s,\ i=1,...,k\}$. However, when addressing the upper-level problem, we need to estimate the Lipschitz constant of hypergradient, which requires the global upper bound of $h_i(x, y^\ast_t(x))$, as stated in Corollary \ref{remark:upper_bound_of_h_at_approximate_point} as $-M$. To obtain the $-M$, the global value $D$ is necessary in view of Corollary \ref{remark:upper_bound_of_h_at_approximate_point}. Yet, we do not know the value of $D$ a priori.

Our approach begins with determining the value of $d_s$ at a specific reference point \(x_s\). Leveraging the Lipschitz continuity assumption of \(h\), it is clear that for any $x$ in the ball $B_{x_s}\left(d_s/2L_h\right)$ there exists a point $y$ such that $h_i(x, y) \leq -d_s/2$ for any $i$. Therefore we can estimate the local upper bound of $h_i(x, y^\ast_t(x))$ within the ball $B_{x_s}\left(d_s/2L_h\right)$. This enables us to design an appropriate stepsize for the upper-level problem. Furthermore, we constrain the stepsize to ensure that next reference point \(x_{s+1}\) remains within this ball. The process is shown in Figure \ref{fig:update_tikz_plots1} and \ref{fig:update_tikz_plots2}.

\begin{figure}[htbp]
\centering

\begin{minipage}{.45\textwidth}
\centering
\tikzset{every picture/.style={line width=0.75pt}} 
\begin{tikzpicture}[x=0.75pt,y=0.75pt,yscale=-1,xscale=1]
\draw   (70.6,152.2) .. controls (70.6,111.77) and (103.37,79) .. (143.8,79) .. controls (184.23,79) and (217,111.77) .. (217,152.2) .. controls (217,192.63) and (184.23,225.4) .. (143.8,225.4) .. controls (103.37,225.4) and (70.6,192.63) .. (70.6,152.2) -- cycle ;
\draw  [fill={rgb, 255:red, 0; green, 0; blue, 0 }  ,fill opacity=1 ] (142.6,150.7) .. controls (142.6,149.76) and (143.36,149) .. (144.3,149) .. controls (145.24,149) and (146,149.76) .. (146,150.7) .. controls (146,151.64) and (145.24,152.4) .. (144.3,152.4) .. controls (143.36,152.4) and (142.6,151.64) .. (142.6,150.7) -- cycle ;
\draw    (143.27,150.4) -- (72.27,137.4) ;
\draw   (144.47,151.4) .. controls (145.36,146.82) and (143.51,144.08) .. (138.93,143.19) -- (119.62,139.44) .. controls (113.07,138.17) and (110.25,135.24) .. (111.14,130.66) .. controls (110.25,135.24) and (106.53,136.89) .. (99.99,135.62)(102.93,136.2) -- (80.67,131.87) .. controls (76.09,130.98) and (73.36,132.82) .. (72.47,137.4) ;
\draw (117,155) node [anchor=north west][inner sep=0.75pt]    {$x_{s-1}$};
\draw (90,111) node [anchor=north west][inner sep=0.75pt]  [font=\small]  {$d_{s-1} /( 2L_{h})$};
\end{tikzpicture}
\caption{$d_{s-1}$ is an output of Algorithm \ref{algo:fo_lower_2}, which indicates that there exists one point $y_0$ such that $h_i(x_{s-1},y_0)\leq -d_{s-1}$ for any $i$. By Lipschitz continuity of $h_i(x,y)$, when $x$ is in the ball $B_{x_{s-1}}(d_{s-1}/(2L_h))$, we have $h_i(x,y_0)\leq -d_{s-1}/2$ for any $i$. In view of Theorem \ref{lem:local_M}, we have a negative upper bound of $h_{i}(x,y^\ast_t(x))$ for any $i\in\{1,..,k\}$ and $x$ in the ball $B_{x_{s-1}}(d_{s-1}/(2L_h))$.}\label{fig:update_tikz_plots1}
\end{minipage}%
\hspace{5mm}
\begin{minipage}{.45\textwidth}
\centering
\tikzset{every picture/.style={line width=0.75pt}} 

\begin{tikzpicture}[x=0.75pt,y=0.75pt,yscale=-1,xscale=1]

\draw  [dash pattern={on 4.5pt off 4.5pt}] (70.6,152.2) .. controls (70.6,111.77) and (103.37,79) .. (143.8,79) .. controls (184.23,79) and (217,111.77) .. (217,152.2) .. controls (217,192.63) and (184.23,225.4) .. (143.8,225.4) .. controls (103.37,225.4) and (70.6,192.63) .. (70.6,152.2) -- cycle ;
\draw  [fill={rgb, 255:red, 0; green, 0; blue, 0 }  ,fill opacity=1 ] (142.6,150.7) .. controls (142.6,149.76) and (143.36,149) .. (144.3,149) .. controls (145.24,149) and (146,149.76) .. (146,150.7) .. controls (146,151.64) and (145.24,152.4) .. (144.3,152.4) .. controls (143.36,152.4) and (142.6,151.64) .. (142.6,150.7) -- cycle ;
\draw    (144.27,150.4) -- (174.52,133.38) ;
\draw [shift={(176.27,132.4)}, rotate = 150.64] [color={rgb, 255:red, 0; green, 0; blue, 0 }  ][line width=0.75]    (10.93,-3.29) .. controls (6.95,-1.4) and (3.31,-0.3) .. (0,0) .. controls (3.31,0.3) and (6.95,1.4) .. (10.93,3.29)   ;
\draw  [fill={rgb, 255:red, 0; green, 0; blue, 0 }  ,fill opacity=1 ] (174.6,132.7) .. controls (174.6,131.76) and (175.36,131) .. (176.3,131) .. controls (177.24,131) and (178,131.76) .. (178,132.7) .. controls (178,133.64) and (177.24,134.4) .. (176.3,134.4) .. controls (175.36,134.4) and (174.6,133.64) .. (174.6,132.7) -- cycle ;
\draw   (124.6,131.33) .. controls (124.6,102.98) and (147.58,80) .. (175.93,80) .. controls (204.28,80) and (227.27,102.98) .. (227.27,131.33) .. controls (227.27,159.68) and (204.28,182.67) .. (175.93,182.67) .. controls (147.58,182.67) and (124.6,159.68) .. (124.6,131.33) -- cycle ;
\draw  [dash pattern={on 4.5pt off 4.5pt}]  (143.27,150.4) -- (72.27,137.4) ;
\draw  [dash pattern={on 4.5pt off 4.5pt}] (144.27,149.4) .. controls (145.18,144.83) and (143.35,142.08) .. (138.78,141.16) -- (120.45,137.5) .. controls (113.91,136.19) and (111.1,133.25) .. (112.01,128.67) .. controls (111.1,133.25) and (107.37,134.88) .. (100.83,133.57)(103.78,134.16) -- (82.5,129.91) .. controls (77.93,129) and (75.18,130.83) .. (74.27,135.4) ;
\draw    (176.27,132.4) -- (136.27,99.4) ;
\draw   (176.27,131.4) .. controls (179.18,127.75) and (178.82,124.47) .. (175.17,121.56) -- (168.45,116.18) .. controls (163.24,112.01) and (162.1,108.11) .. (165.01,104.47) .. controls (162.1,108.11) and (158.04,107.85) .. (152.83,103.69)(155.17,105.56) -- (146.11,98.31) .. controls (142.46,95.39) and (139.18,95.75) .. (136.27,99.4) ;
\draw    (176.27,132.4) -- (218.29,125.71) ;
\draw [shift={(220.27,125.4)}, rotate = 170.96] [color={rgb, 255:red, 0; green, 0; blue, 0 }  ][line width=0.75]    (10.93,-3.29) .. controls (6.95,-1.4) and (3.31,-0.3) .. (0,0) .. controls (3.31,0.3) and (6.95,1.4) .. (10.93,3.29)   ;
\draw  [fill={rgb, 255:red, 0; green, 0; blue, 0 }  ,fill opacity=1 ] (218.6,125.7) .. controls (218.6,124.76) and (219.36,124) .. (220.3,124) .. controls (221.24,124) and (222,124.76) .. (222,125.7) .. controls (222,126.64) and (221.24,127.4) .. (220.3,127.4) .. controls (219.36,127.4) and (218.6,126.64) .. (218.6,125.7) -- cycle ;

\draw (117,153) node [anchor=north west][inner sep=0.75pt]    {$x_{s-1}$};
\draw (171,141) node [anchor=north west][inner sep=0.75pt]    {$x_{s}$};
\draw (84,115) node [anchor=north west][inner sep=0.75pt]  [font=\tiny]  {$d_{s-1} /( 2L_{h})$};
\draw (153,92) node [anchor=north west][inner sep=0.75pt]  [font=\tiny]  {$d_{s} /( 2L_{h})$};
\draw (228,105) node [anchor=north west][inner sep=0.75pt]    {$x_{s+1}$};

\end{tikzpicture}
\caption{Utilizing the negative upper bound of $h_{i}(x,y^\ast_t(x))$ in the ball $B_{x_{s-1}}(d_{s-1}/(2L_h))$ obtained from left figure, we estimate the Lipschitz smoothness constant of $\widetilde{\phi}_t(x)$ in this ball, and design the stepsize $\eta_{s-1}$. The design of $\eta_{s-1}$ ensures that $x_s$ is still in $B_{x_{s-1}}(d_{s-1}/(2L_h))$. We update and get $x_{s}$ by \textbf{Step 2} of Algorithm \ref{algo:fo}. We also get the ball $B_{x_s}(d_s/(2L_h))$ by repeating the argument in the left figure, then update to the next point $x_{s+1}$ by repeating the above process.}\label{fig:update_tikz_plots2}
\end{minipage}

\end{figure}

We first establish the following lemma on the Lipschitz properties of the lower-level Hessian, whose proof can be found in Appendix \ref{proof:lem:lipHessian}.

\begin{lemma}\label{lem:lipHessian}
    Suppose the assumptions in Theorem  \ref{lem:local_M} are satisfied. If $h_i(x_1,y_1)\leq -m$, $h_i(x_2,y_2)\leq-m$ for any $i$, then  $\|\nabla^2_{yy}\widetilde{g}_t(x_1,y_1)-\nabla^2_{yy}\widetilde{g}_t(x_2,y_2)\|\leq \overline{\overline{L}}_{\widetilde{g}_t,m}\|(x_1,y_1)-(x_2,y_2)\|$; $\|\nabla^2_{xy}\widetilde{g}_t(x_1,y_1)-\nabla^2_{xy}\widetilde{g}_t(x_2,y_2)\|\leq \overline{\overline{L}}_{\widetilde{g}_t,m}\|(x_1,y_1)-(x_2,y_2)\|$, where
    \begin{align}\label{eq:liphessian!}
        \overline{\overline{L}}_{\widetilde{g}_t,m}=\overline{\overline{L}}_{g}+tk\left(\frac{\overline{\overline{L}}_h}{m}+\frac{\overline{L}_hL_h}{m^2}+\frac{2L^3_h}{m^3}+\frac{2L_h\overline{L}_h}{m^2}\right).
    \end{align}
\end{lemma}

The following lemma informs us that when we know, for a fixed $x$, the values of the constraints $h_i(x,y)$ for some interior point $y$, we can evaluate the Lipschitz smoothness constant of $\widetilde{\phi}_t$ within a spherical neighborhood $B_{x_s}\left(d_s/2L_h\right)$ of $x$ which can help us to design the stepsize. The proof is available in Appendix \ref{proof:lem:local_lip}.

\begin{lemma}\label{lem:local_lip}
    Suppose the assumptions in Theorem  \ref{lem:local_M} are satisfied. For any fixed $x$, if there exists $y\in\mathcal{Y}(x)$ such that $h_i(x,y)\leq -d$ for any $i\in\{1,...,k\}$, then 
    \begin{enumerate}
        \item For any $x+\Delta x$, where $\|\Delta x\|\leq\frac{d}{2L_h}$, the following inequality holds: $h_i(x+\Delta x,y^\ast_t(x+\Delta x))\leq -m^{loc}$, where $m^{loc}$ is a positive constant which can be computed through replacing $d$ by $d/2$ in (\ref{eq:m});
        \item The hyperfunction $\widetilde{\phi}_t(x)$ is local $\overline{L}^{loc}_{\widetilde{\phi}_t}$-Lipschitz smooth. That is, the following holds  for any $x_1,x_2$ satisfies $\max\{\|x_1-x\|,\|x_2-x\|\}\leq\frac{d}{2L_h}$:
        \begin{align*}
            \|\nabla_{x} \widetilde{\phi}_t(x_1)-\nabla_{x} \widetilde{\phi}_t(x_2)\|\leq \overline{L}^{loc}_{\widetilde{\phi}_t}\|x_1-x_2\|.
        \end{align*}
        $\overline{L}^{loc}_{\widetilde{\phi}_t}$ can be computed as follows:
        \begin{align*}
            \overline{L}^{loc}_{\widetilde{\phi}_t}=&\left(\overline{L}_f+\overline{L}_f\frac{1}{\mu_{\widetilde{g}_t}}(\overline{L}_g+\frac{tk\overline{L}_h}{m^{loc}}+\frac{tkL_h^2}{(m^{loc})^2})+L_f\frac{\overline{\overline{L}}_{\widetilde{g}_t,m^{loc}}}{(\mu_{\widetilde{g}_t})^2}(\overline{L}_g+\frac{tk\overline{L}_h}{m^{loc}}+\frac{tkL_h^2}{(m^{loc})^2})+L_f\frac{\overline{\overline{L}}_{\widetilde{g}_t,m^{loc}}}{\mu_{\widetilde{g}_t}}\right)\\
        &\times\left(1+\frac{1}{\mu_{\widetilde{g}_t}}(\overline{L}_g+\frac{tk\overline{L}_h}{{m^{loc}}}+\frac{tkL_h^2}{(m^{loc})^2})\right),
        \end{align*}
        where $\overline{\overline{L}}_{\widetilde{g}_t,m^{loc}}$ is from (\ref{eq:liphessian!}) by replacing $m$ with $m^{loc}$, and $\mu_{\widetilde{g}_t}$ is from Proposition \ref{prop:strongly_convex}. 
    \end{enumerate}
\end{lemma}

Replacing $d$ in Lemma \ref{lem:local_lip} with the current iteration $d_s$ from the output of Algorithm \ref{algo:fo_lower_2}, we have an estimate of local Lipschitz smoothness constant $\overline{L}_{\widetilde{\phi}_t, s}:=\overline{L}^{loc}_{\widetilde{\phi}_t}$ in the ball $B_{x_s}(d_s/(2L_h))$. Now we are able to establish the following lemma on the upper bound of $\overline{L}_{\widetilde{\phi}_t, s}$, whose proof is provided in Appendix \ref{proof:lem:lip_global_bound}:

\begin{lemma}\label{lem:lip_global_bound}
     Suppose the assumptions in Theorem  \ref{lem:local_M} are satisfied. The estimate $\overline{L}_{\widetilde{\phi}_t, s}$ has a upper bound $\overline{L}_{\widetilde{\phi}_t}=\mathcal{O}(1/t^4)$.
\end{lemma}

\begin{remark}
    Although $\overline{L}_{\widetilde{\phi}_t, s}$ is bounded by $\overline{L}_{\widetilde{\phi}_t}$, the constant $\overline{L}_{\widetilde{\phi}_t}$ is not computable if we do not know the constant $D$ from Remark \ref{remark:non-emptyness}, thus this cannot help us design the stepsize. $\overline{L}_{\widetilde{\phi}_t, s}$ can make sure that the stepsize $\eta_s$ has a lower bound, a fact which is essential for our main convergence result Theorem \ref{thm:total_rate} to hold.
\end{remark}

Now that in practice we are not able to calculate the hypergradient $\nabla_x\widetilde{\phi}_t(x)$ exactly, and in Algorithm \ref{algo:fo} we approximate the hypergradient by $\hat{\nabla}_x\widetilde{\phi}_t(x)$. We thus need the following lemma on the approximation error of the hypergradient, whose proof can be found in Appendix \ref{proof:lemma:approx_error_hypergrad}.

\begin{lemma}\label{lemma:approx_error_hypergrad}
    Suppose the assumptions in Theorem  \ref{lem:local_M} are satisfied. Let $x_s$ be the reference point, and $\hat{y}_s$ is from the output of Algorithm \ref{algo:fo_lower_2}. Define the following approximate gradient
    \begin{align}\label{eq:approximategradient}
        \hat{\nabla}_x\widetilde{\phi}_t(x_s)=\nabla_x f(x_s,\hat{y}_s)-\nabla_{xy}^2\widetilde{g}_t(x_s,\hat{y}_s)(\nabla_{yy}^2\widetilde{g}_t(x_s,\hat{y}_s))^{-1}\nabla_y f(x_s,\hat{y}_s).
    \end{align}
    Then the gradient approximation error can be upper bounded as $\|\nabla_x\widetilde{\phi}_t(x_s)-\hat{\nabla}_x\widetilde{\phi}_t(x_s)\|\leq\epsilon_s \overline{L}'_{\widetilde{\phi}_t,m_s}$, where $m_s$ is from the output of Algorithm \ref{algo:fo_lower_2} and
    \begin{equation}\label{eq:L'phims}
        \resizebox{0.99\textwidth}{!}{$\overline{L}'_{\widetilde{\phi}_t,m_s}=\overline{L}_f+\overline{L}_f\frac{1}{\mu_{\widetilde{g}_t}}\left(\overline{L}_g+\frac{tk\overline{L}_h}{m_s}+\frac{tkL_h^2}{m_s^2}\right)+L_f\left(\frac{1}{\mu_{\widetilde{g}_t}}\right)^2\overline{\overline{L}}_{\widetilde{g}_t,m_s}\left(\overline{L}_g+\frac{tk\overline{L}_h}{m_s}+\frac{tkL_h^2}{m_s^2}\right)+L_f\frac{1}{\mu_{\widetilde{g}_t}}\overline{\overline{L}}_{\widetilde{g}_t,m_s},$}
    \end{equation}
    where $\overline{\overline{L}}_{\widetilde{g}_t,m_s}$ is from Lemma \ref{lem:lipHessian} which can be computed by replacing $m$ with $m_s$.
\end{lemma}

Finally, we state the convergence of the upper-level problem as follows.

\begin{theorem}\label{thm:total_rate}
    Suppose Assumption \ref{assumption:general3}(\ref{assumption:general3(1)})(\ref{assumption:general3(2)})(\ref{assumption:general3(3)}) and Assumption \ref{assumption:general2} all hold, and either Assumption \ref{assumption:nonlinear2} or Assumption \ref{assumption:linear2} holds. We recall and define the following terms:
    \begin{itemize}
        \item $\overline{L}_{\widetilde{\phi}_{t},s}$ is computed by replacing $d$ in Lemma \ref{lem:local_lip} with $d_s$ from the output of Algorithm \ref{algo:fo_lower_2};
        \item $\hat{\nabla}_x\widetilde{\phi}_t(x_s)$ is from (\ref{eq:approximategradient});
        \item $\overline{L}'_{\widetilde{\phi}_t,m_s}$ is from (\ref{eq:L'phims});
        \item $M$ is from Corollary  \ref{remark:upper_bound_of_h_at_approximate_point}; 
        \item $\mu_{\widetilde{g}_t}$ is from Proposition \ref{prop:strongly_convex};
        \item $\overline{L}_{\widetilde{\phi}_t}$ is from Lemma \ref{lem:lip_global_bound};
        \item $D$ is from Remark \ref{remark:non-emptyness};
        \item $d_s$ is from the output of Algorithm \ref{algo:fo_lower_2};
        \item
        $\zeta = \min\left\{1,\frac{D}{4L_h}\times\frac{1}{L_f + \frac{L_f}{\mu_{\widetilde{g}_t}}\left(\overline{L}_g+ t k \frac{\overline{L}_h}{M} + t k \frac{L_h^2}{M^2}\right)}, \frac{1}{\overline{L}_{\widetilde{\phi}_t}}\right\}$.
    \end{itemize}
    We set the stepsize $\eta_s$ as
    \begin{align}\label{eq:stepsize}
        \eta_s=\min\left\{1,  \frac{d_s}{2 L_h\|\hat{\nabla}_x\widetilde{\phi}_t(x_s)\|}, \frac{1}{\overline{L}_{\widetilde{\phi}_{t},s}}\right\},
    \end{align}and set the stop criterion $\epsilon_s=\epsilon/(4\overline{L}'_{\widetilde{\phi}_t,m_s})$ for the inner loop. Then the the sequence $\{x_s\}$ generated by Algorithm \ref{algo:fo} satisfies
    $$
    \min_{s=0,...,S-1}\frac{1}{\eta_s}\|x_{s} - \proj_{\mathcal{X}}(x_s-\eta_s \nabla_x\widetilde{\phi}_t(x_s))\|\leq\epsilon
    $$
    within $\widetilde{\mathcal{O}}(1/(\zeta\epsilon^2))$ number of iterations.
\end{theorem}

 The proof of Theorem \ref{thm:total_rate} is in Appendix \ref{proof:thm:total_rate}.
\begin{remark}\label{remark:dependence}
    In terms of the dependency on $\epsilon$, Theorem \ref{thm:total_rate} achieves the rate $\widetilde{\mathcal{O}}(\epsilon^{-2})$. However, if higher order (3rd order or even higher for the lower-level objective) smoothness is assumed for both lower- and upper-level objectives, the convergence could be improved to $\mathcal{O}(\epsilon^{-7/4})$ using the techniques proposed in \cite{wang2024efficientmethodforsaddlepoint} for the special case of bilevel optimization, namely the minimax saddle point problems. We refer to \cite{wang2024efficientmethodforsaddlepoint} and omit the details to keep the conciseness of this work.
    
    In terms of the dependency on $t$, from Lemma \ref{lem:lip_global_bound} and Theorem \ref{thm:total_rate}, we can see that $\zeta=\mathcal{O}(t^4)$ with respect to $t$. For the inner loop, we have that the number of iterations is $\widetilde{\mathcal{O}}(t^{-0.5})$ (more details see Appendix \ref{explainationforremark:dependence}). Therefore, the total number of iterations is $\widetilde{\mathcal{O}}(t^{-4.5})$. Moreover, under the strong convex setting, if we already know the upper bounds of some provably bounded terms a priori, then by adjusting the stepsize, the number of iterations with respect to $t$ can be reduced to $\widetilde{\mathcal{O}}(t^{-1.5})$ (for details, see Appendix \ref{sec:improvedconvergencerate}).
\end{remark}

In the next section, we discuss how to recover a stationary point for the original lower-level constrained bilevel problem \eqref{eq:blo_constrained}.

\subsection{Asymptotic Convergence to the Original Problem}

The convergence result in Theorem \ref{thm:total_rate} indicates a non-asymptotic rate of convergence is achievable for solving the barrier reformulated problem \eqref{eq:blo_barrier_reformulation}. Based on our previous discussion, achieving non-asymptotic convergence toward the original problem \eqref{eq:blo_constrained} is intractable since the hypergradient of the original problem \eqref{eq:blo_constrained} can be non-differentiable or even discontinuous (Example \ref{counterexample1}). Therefore we inspect the asymptotic convergence toward the original problem when the \textbf{SCSC} assumption is satisfied (which will make the original problem differentiable). In particular, we consider Algorithm \ref{algo:fo_asymptotic} listed below, where we iteratively shrink the barrier parameter $t$ in our formulation \eqref{eq:blo_barrier_reformulation} and utilize Algorithm \ref{algo:fo} for each $t$.
\begin{algorithm}
\caption{Hypergradient Based Bilevel Barrier Method for \eqref{eq:blo_constrained}}\label{algo:fo_asymptotic}
\begin{algorithmic}
\State \textbf{Step 0: Initialization.} Given an initial point $x_0$, $t_0=1$, the accuracy level $\epsilon_0$.

\State \For{$i=0,1,2,...$}{
\State \textbf{Step 1: Solve the problem with current $t$.} Call Algorithm \ref{algo:fo} with initial point $x_i$, accuracy $\epsilon=\epsilon_i$ and $t=t_i$ and other parameters specified as in Theorem \ref{thm:total_rate}, obtaining the output $x_{i+1}$

\State \textbf{Step 2: Decrease the parameter.} Set $\epsilon_{i+1}=\epsilon_i/2$ and $t_{i+1}=t_{i}/2$
}

\end{algorithmic}
\end{algorithm}

Algorithm \ref{algo:fo_asymptotic} is a direct application of Algorithm \ref{algo:fo} with shrinking $t$, which is analogous to the classical path-following scheme. Different from the path-following scheme, we do not have self-concordant barriers, thus we have to solve each of the sub-step to a certain precision for each fix $t$. We have the following result for the asymptotic convergence of Algorithm \ref{algo:fo_asymptotic}.

\begin{theorem}\label{thm:asymptotic}
    Assume the same assumptions as Theorem \ref{thm:total_rate} and consider the output sequence $\{x_i\}$ of Algorithm \ref{algo:fo_asymptotic}. Suppose in addition Assumption \ref{assumption:general3}(\ref{assumption:general3(4)}) holds. If the limit point $x^*$ of $\{x_i\}\subset \mathcal{X}$ is \textbf{SCSC} points of original BLO problem \eqref{eq:blo_constrained}, then $x^*$ is the stationary point of \eqref{eq:blo_constrained}.
\end{theorem}

The proof of Theorem \ref{thm:asymptotic} is available in Section \ref{proof:thm:asymptotic}.

\section{Numerical Experiments}
In this section, we conduct numerical experiments to verify the effectiveness of the proposed Algorithm \ref{algo:fo}. Our experiments are conducted on a single PC with Intel Core 12400F CPU and 16GB of RAM. Our code is available at \url{https://github.com/jxxxxxxt/bilevel_ipm}.

\subsection{Generated Strongly Convex Problem}
In this experiment, we evaluate the performance of our proposed BFBM algorithm using a class of randomly generated problems. Both the upper-level and lower-level objective functions are designed to be strongly convex, and the constraints are all linear in $y$. The problem template is given as follows:
\begin{align}\label{eq:strongly_convex_problem}
    \min_{x,y}\ f(x, y) =& \frac{1}{2} \begin{pmatrix} x \\ y \end{pmatrix}^\top A \begin{pmatrix} x \\ y \end{pmatrix} + b^\top \begin{pmatrix} x \\ y \end{pmatrix} \nonumber\\
    \rm{s.t.}\quad\ y\in\argmin_{y}\ &g(x, y) = \frac{1}{2} y^\top C y + d^\top y + x^\top D y  \nonumber\\
    \rm{s.t.}&\quad\  x^\top E_i y \leq e_i \quad \text{for } i = 1, \dots, k, \quad -1\leq y_i\leq 1 \quad \text{for } i = 1, \dots, n\nonumber,
\end{align}
where \( x \in \mathbb{R}^n \) represents the upper-level decision variables, \( y \in \mathbb{R}^n \) represents the lower-level decision variables, \( A \in \mathbb{R}^{2n \times 2n} \) and \( C \in \mathbb{R}^{n \times n} \) are symmetric positive definite matrices, \( b \in \mathbb{R}^{2n} \) and \( d \in \mathbb{R}^n \) are vectors, \( D \in \mathbb{R}^{n \times n} \) and \( E_i \in \mathbb{R}^{n \times n} \) are matrices, $e_i\in\mathbb{R}^+$ are positive numbers for $i=1,..,k$. All these are randomly generated (see Appendix \ref{sc_details_of_experiments}).

\paragraph{Experimental Setup and Algorithm Comparison}

In this experiment, we set the problem dimension to \( n = 60 \) and the number of primary constraints to \( k = 20 \). We generated ten problem instances using seeds ranging from 0 to 9 to ensure variability and reproducibility. We allocated a two-minute time budget to each of the three algorithms—BFBM~(Our proposed Algorithm), BSG~\citep{giovannelli2021inexact}, and BLOCC~\citep{jiang2024primal}—and evaluated them under two stepsize strategies: constant and diminishing. (For details about the parameters, refer to Appendix \ref{sc_details_of_experiments}.)

After the allotted time, the quality of the obtained solutions \((x, y)\) was evaluated based on four metrics: the {\it upper-level function value} \( f(x,y) \), the {\it hyperfunction value} \( f(x, y_{\text{opt}}) \) where \( y_{\text{opt}} \) is the optimal solution of the lower-level problem given \( x \), the {\it lower-level optimality gap} \( g(x,y) - g(x, y_{\text{opt}}) \), and the {\it total constraint violation} defined as
\[
\sum_{i=1}^{k} \max(0, x E_i^{(1)} y - e_i^{(1)}) + \sum_{i=1}^{n} \max(0, y_i - 1) + \sum_{i=1}^{n} \max(0, -y_i - 1),
\]
which measures the overall lower-level feasibility of the solution. A higher total constraint violation indicates that the obtained solution deviates more from the feasible region of the lower-level problem. The total constraint violation equals zero if and only if the obtained solution is feasible in the lower-level problem. For each algorithm, the solution with the best hyperfunction value across both stepsize strategies was selected for comparison.

\begin{table}[ht]
\centering
\caption{Evaluate solution qualities using upper-level metrics: Hyperfunction Value (Upper-Level
Function Value)}
\label{tab:selected_stepsize}
\begin{tabular}{cccc}
\toprule
\textbf{Seed} & \textbf{BSG} & \textbf{BLOCC} & \textbf{BFBM} \\
\midrule
0 & \textbf{1.1138 (1.0546)} & 2.6983 (1.3556) & 1.5372 (1.2014) \\
1 & 1.1550 (1.0850) & 0.9524 (0.1207) & \textbf{0.1657 (-0.0186)} \\
2 & 1.9700 (1.8424) & 3.1455 (2.1752) & \textbf{1.8765 (1.6664)} \\
3 & \textbf{2.0662 (2.1330)} & 2.7364 (2.5312) & 2.1675 (1.9531) \\
4 & \textbf{0.6963 (17.6067)} & 1.4438 (1.4192) & 0.8288 (0.6353) \\
5 & 0.7596 (0.7157) & 0.8560 (1.4241) & \textbf{0.6675 (0.3143)} \\
6 & \textbf{1.0723 (18.5577)} & 1.2513 (0.2774) & 2.2743 (1.8548) \\
7 & 0.6451 (1.0544) & \textbf{0.3920 (0.5784)} & 2.8073 (2.2171) \\
8 & 17.2746 (146.8528) & 2.8086 (1.1682) & \textbf{2.2652 (2.0724)} \\
9 & 8.4697 (15.3909) & 5.8981 (5.9888) & \textbf{1.5066 (1.0964)} \\
\bottomrule
\end{tabular}
\label{tab:best_hyperfunction_solutions}
\end{table}

Table~\ref{tab:best_hyperfunction_solutions} compares the solutions obtained by the three algorithms across ten different seeds under a two-minute computational budget per run. Each solution is evaluated in terms of both its hyperfunction value $f(x, y_{\text{opt}})$ and its upper-level function value $f(x, y)$. As shown in the table, our proposed algorithm BFBM achieves the lowest hyperfunction value among the three algorithms in five out of the ten seeds.

\begin{table}[ht]
\centering
\caption{Evaluate solution qualities using lower-level metrics: Total Constraint Violation (Optimality Gap)}
\label{tab:gap_violation}
\begin{tabular}{cccc}
\toprule
\textbf{Seed} & 
\textbf{BSG} & 
\textbf{BLOCC} & 
\textbf{BFBM} \\
\midrule
0 & 0.0000(0.0289) & 0.0000(0.1354) & 0.0000(0.0515) \\
1 & 0.0084(0.0218) & 0.0000(0.2183) & 0.0000(0.0478) \\
2 & 0.0000(0.0425) & 0.0000(0.1248) & 0.0000(0.0643) \\
3 & 0.0100(-0.0204) & 0.0000(0.0863) & 0.0000(0.0443) \\
4 & 0.0122(-0.0392) & 0.0066(0.0285) & 0.0000(0.0527) \\
5 & 0.0000(0.0224) & 0.0192(0.0099) & 0.0000(0.0530) \\
6 & 0.0140(-0.0317) & 0.0000(0.1670) & 0.0000(0.0718) \\
7 & 0.0136(-0.0320) & 0.0050(0.0171) & 0.0000(0.0562) \\
8 & 1.2690(-1.0111) & 0.0000(0.7070) & 0.0000(0.0717) \\
9 & 0.0169(-0.0059) & 0.0000(0.1192) & 0.0000(0.0510) \\
\bottomrule
\end{tabular}
\end{table}

As shown in Table~\ref{tab:gap_violation}, we compare the solutions obtained by the three algorithms across ten different random seeds in terms of their optimality gaps and total violations. It is crucial to consider only solutions with a total violation of zero as feasible; therefore, the comparison of optimality gaps is meaningful exclusively for these feasible solutions.

Our proposed algorithm, BFBM, consistently achieves a total violation of zero for all ten seeds, ensuring the feasibility of its solutions. Compared to BLOCC, BFBM achieves lower total constraint violation and optimality gap. Relative to BSG, although both exhibit small optimality gaps, BFBM maintains a lower total constraint violation.

This stability underscores the reliability and high quality of the solutions generated by the BFBM algorithm.

\subsection{Price Setting Problem}

In this experiment, we assess the performance of our proposed BFBM algorithm in the context of a price setting problem (see Section \ref{sec:MotivatingApplication}). In this problem, both the upper-level and lower-level objective functions are linear, and the constraints are all linear in $y$. Here we consider the case that the lower-level problems are all uncoupled, i.e. they do not depend on $T$. All the matrices and vectors in (\ref{eq:blo_price_setting}) are randomly generated (see Appendix \ref{sc_details_of_experiments}).

\paragraph{Experimental Setup and Algorithm Comparison}

In this experiment, we set the problem dimension to \( n = 60 \) and the number of primary constraints to \( k = 20 \). We generated ten problem instances using seeds ranging from 0 to 9 to ensure variability and reproducibility. We allocated a four-minute time budget to each of the three algorithms—BFBM~(Our proposed Algorithm), BSG~\citep{giovannelli2021inexact}, and BLOCC~\citep{jiang2024primal}—and evaluated them under diminishing stepsize strategy. (For details about the parameters, refer to Appendix \ref{sc_details_of_experiments}.)

\begin{table}[htbp]
\centering
\caption{Hyperfunction (Evaluate solution qualities using upper-level metrics: Hyperfunction Value (Upper-Level Function Value))}
\label{tab:hyper_upper}
\begin{tabular}{cccc}
\toprule
\textbf{Seed} & \textbf{BSG} & \textbf{BLOCC} & \textbf{BFBM} \\
\midrule
0 & -0.1279 (-0.0710) & -0.0046 (-0.0047) & \textbf{-0.2825 (-0.2750)} \\
1 & -0.1337 (-0.0405) & -0.0034 (-0.0034) & \textbf{-0.2455 (-0.2382)} \\
2 & -0.1291 (-0.1417) & -0.0000 (0.0000) & \textbf{-0.3629 (-0.3552)} \\
3 & -0.0509 (-0.0112) & -0.0000 (0.0000) & \textbf{-0.2950 (-0.2868)} \\
4 & -0.0559 (-0.1086) & -0.0059 (-0.0060) & \textbf{-0.3034 (-0.2888)} \\
5 & -0.0499 (-0.0612) & -0.0000 (0.0000) & \textbf{-0.2992 (-0.2891)} \\
6 & -0.0045 (-0.0445) & -0.0044 (-0.0043) & \textbf{-0.1590 (-0.1506)} \\
7 & 0.0000 (0.0000) & -0.0000 (0.0000) & \textbf{-0.3070 (-0.2945)} \\
8 & -0.0314 (-0.1053) & -0.0000 (0.0000) & \textbf{-0.1790 (-0.1741)} \\
9 & -0.0748 (-0.0603) & -0.0100 (-0.0095) & \textbf{-0.3348 (-0.3248)} \\
\bottomrule
\end{tabular}
\end{table}

\begin{table}[htbp]
\centering
\caption{Evaluate solution qualities using lower-level metrics: Total Constraint Violation (Optimality Gap)}
\label{tab:violation_gap}
\begin{tabular}{cccc}
\toprule
\textbf{Seed} & \textbf{BSG} & \textbf{BLOCC} & \textbf{BFBM} \\
\midrule
0 & 2.1335 (1.8119) & 0.0676 (0.0088) & 0.0000 (0.0120) \\
1 & 4.4754 (-1.0703) & 0.0618 (0.0312) & 0.0000 (0.0119) \\
2 & 3.0165 (0.0653) & 0.1464 (0.0024) & 0.0000 (0.0120) \\
3 & 4.6530 (-0.8170) & 0.0875 (-0.0049) & 0.0000 (0.0119) \\
4 & 3.4429 (-1.0803) & 0.1834 (-0.0523) & 0.0000 (0.0120) \\
5 & 3.7905 (-0.2300) & 0.1748 (-0.0246) & 0.0000 (0.0119) \\
6 & 3.7567 (-1.3199) & 0.0917 (0.0293) & 0.0000 (0.0120) \\
7 & 6.6517 (-2.5038) & 0.1058 (0.0149) & 0.0000 (0.0120) \\
8 & 2.5576 (0.0622) & 0.2153 (-0.0415) & 0.0000 (0.0120) \\
9 & 3.4771 (0.5839) & 0.0814 (0.0316) & 0.0000 (0.0120) \\
\bottomrule
\end{tabular}
\end{table}

For the BSG and BLOCC algorithms, we conducted experiment on a regularized version of the price setting problem. In the regularized version, a small regularization term $0.001\times(\|x\|^2+\|y\|^2)$ was added to the lower-level objective function $g(T,x,y)$, thereby ensuring that it is strongly convex. From Table \ref{tab:hyper_upper} and \ref{tab:violation_gap}, it is evident that for the price-setting problem, the solution obtained by our BFBM algorithm is optimal in terms of hyperfunction quality. In addition, while maintaining strict feasibility (i.e., the total constraint violation is zero), the optimality gap remains consistently low, demonstrating the robustness and efficiency of our method when lower-level problem is a linear program.

\section{Conclusion}

In this paper, we proposed a barrier function reformulation approach for BLO optimization problems involving lower-level problems with coupled constraints, i.e. the constraints depend on both upper- and lower-level variables. We focused on two specific cases: one where the lower-level objective is strongly convex with convex constraints concerning the lower-level variable, and another where the lower-level problem is a linear program.

By developing a series of new techniques, we demonstrated that the barrier reformulated problem converges to the original bilevel problem in terms of both hyperfunction value and hypergradient. To solve the reformulated problem whose hyperfunction and lower-level problems are not Lipschitz smooth, we designed two effective algorithms: one that guarantees non-asymptotic convergence for the barrier reformulated problem for a fixed $t$, and the other that converges asymptotically to a stationary point of the original problem if this point is an \textbf{SCSC} point (see Definition \ref{def:scsc_point}). Notably, this is the first work providing convergence guarantees for BLO optimization problems where the lower-level problem is a linear program.  For the general case that $g(x,y)$ and $h_i(x,y)$ are all convex in $y$, we have the convergence for hyperfunction value, and  Algorithm \ref{algo:fo} still works in this case (see Appendix \ref{section:convex_case}).

We conducted experiments on strongly convex and linear lower-level problems with linear inequality constraints. Our algorithm is the only one that ensure the solution is feasible for lower-level problem. For the price-setting problem, it was the most effective algorithm, demonstrating its robustness and effectiveness for bilevel optimization.

Our work establishes a unified framework that operates under minimal assumptions, significantly advancing the methodology in bilevel optimization. Future efforts will focus on extending our algorithms to scenarios with stochastic upper- and lower-level objective functions.

\clearpage
\bibliographystyle{abbrvnat} 
\bibliography{reference}

\clearpage
\appendix
\section*{Appendix}

\section{Details of Experiments}\label{sc_details_of_experiments}

\subsection{Generated Strongly Convex Problem}

\paragraph{Data Generation Process}

The problem data is generated using the following procedure:

\begin{enumerate}
    \item \textbf{Random Seed Initialization:} 
    A fixed random seed is set using \texttt{np.random.seed(seed)} to guarantee reproducibility of the results across different runs.

    \item \textbf{Matrix Generation:}
    \begin{itemize}
        \item \textbf{Matrix \( A \):} 
        Generated as a symmetric positive definite matrix:
        \[
        A = 0.5 \times (A_{\text{init}} \times A_{\text{init}} ^\top) + I_{2n}.
        \]
        Here, each entry of the initial random matrix \( A_{\text{init}}  \) is drawn from a standard normal distribution \( \mathcal{N}(0, 1) \);

        \item \textbf{Matrix \( C \):} 
        Generated as a symmetric positive definite matrix:
        \[
        C = 0.5 \times (C_{\text{init}}  \times C_{\text{init}} ^\top) + I_n.
        \]
        Each entry of the initial random matrix \( C_{\text{init}}  \) is drawn from a standard normal distribution \( \mathcal{N}(0, 1) \);

        \item \textbf{Matrix \( D \):} 
        Generated as a random matrix:
        \[
        D = \text{randn}(n, n).
        \]
        Each entry of \( D \) is independently drawn from a standard normal distribution \( \mathcal{N}(0, 1) \).
    \end{itemize}
    
    \item \textbf{Vector Generation:}
    \begin{itemize}
        \item \textbf{Vector \( b \):} 
        Generated as a random vector:
        \[
        b = \text{randn}(2n).
        \]
        Each entry of \( b \) is drawn from a normal distribution \( \mathcal{N}(0, 1) \) and then scaled by a factor of 10 to increase the magnitude of the coefficients;

        \item \textbf{Vector \( d \):} 
        Generated as a random vector:
        \[
        d = \text{randn}(n).
        \]
        Each entry of \( d \) is drawn from a normal distribution \( \mathcal{N}(0, 1) \) and scaled by a factor of 10.
    \end{itemize}
    
    \item \textbf{Constraints Generation:}
    \begin{itemize}
        \item \textbf{Coupled Constraints:} 
        A total of $k$ coupled linear constraints are imposed on the lower-level variables:
        \[
        x^\top E_i y \leq e_i \text{, for } i = 1, \dots, k.
        \]
        \begin{itemize}
            \item Each matrix \( E_i \in \mathbb{R}^{n \times n} \) is generated with entries drawn independently from a standard normal distribution \( \mathcal{N}(0, 1) \);
            \item Each scalar \( e_i \in \mathbb{R} \) is drawn uniformly from the interval \([0, 1]\), i.e., \( e_i \sim \mathcal{U}(0, 1) \);
        \end{itemize}
        
        \item \textbf{Box Constraints:} 
        Simple Box constraints are imposed to make sure the lower-level feasible set is compact.
        \[
        -\mathbf{1} \leq y \leq \mathbf{1}.
        \]
    \end{itemize}
\end{enumerate}

\paragraph{Implementation Details}

\begin{enumerate}
    \item \textbf{BSG Algorithm~\cite{giovannelli2021inexact}:} 
    \begin{itemize}
        \item For inner loop, we use a max-min solver (Algorithm 2 in \cite{jiang2024primal}) to obtain the lower-level optimal solution $y$ and Lagrange multiplier $z_I$;
        \item We choose the initial point $x_{\text{init}}=0$, $y_{\text{init}}=0$, $z_{I,\text{init}}=0$; For the main loop, the stepsize is set to $\alpha_{k}=2\times 10^{-6}$ or $5\times 10^{-5}/\sqrt{k+1}$; For the max-min solver, the stepsizes for the first and second loops are set to $\eta_1=10^{-3}$ and $\eta_2=0.02$ respectively; The stopping criteria for the first and second loops in the max-min solver are $\epsilon_1=0.01$ and $\epsilon_2=0.01$.
    \end{itemize}
    \item \textbf{BLOCC Algorithm~\cite{jiang2024primal}:}
    \begin{itemize}
        \item We choose the initial point $x_{\text{init}}=0$, $y_{g,\text{init}}=0$, $y_{F,\text{init}}=0$, $\mu_{g,\text{init}}=0$, $\mu_{F,\text{init}}=0$, and the penalty coefficient $\gamma=10$; For the main loop, the stepsize is set to $\alpha_{k}=2\times10^{-5}$ or $10^{-3}/\sqrt{k+1}$; For both max-min solvers, the stepsizes for the first and second loops are set to $\eta_1=10^{-3}$ and $\eta_2=0.5$ respectively; The stopping criteria for the first and second loops in both max-min solvers are $\epsilon_1=0.05$ and $\epsilon_2=0.05$.
    \end{itemize}
    \item \textbf{BFBM Algorithm (Our proposed Algorithm):}
    \begin{itemize}
        \item For inner loop we use Netwon method;
        \item We set $t=0.01$, and $M=0.1$, which means we shrink the lower-level feasible set to $\{y|h_{i}(x,y)\leq-10^{-3}\text{ for }i=1,...,k\}$;
        \item We choose the initial point $x_{\text{init}}=0$, $y_{\text{init}}=0$; For the main loop, the stepsize is set to $\alpha_{k}=2\times10^{-5}$ or $5\times 10^{-4}/\sqrt{k+1}$; For inner loop the stepsizes is set to $\eta=0.1$; The stopping criteria for inner loop is $\epsilon_y=0.01$.
    \end{itemize}
\end{enumerate}

\subsection{Price Setting Problem}

\paragraph{Data Generation Process}

The problem data is generated using the following procedure:

\begin{enumerate}
    \item \textbf{Random Seed Initialization:} 
    A fixed random seed is set using \texttt{np.random.seed(seed)} to guarantee reproducibility of the results across different runs.
    \item \textbf{Matrix Generation:}
    \begin{itemize}
    \item \textbf{Matrix \( A_1 \) and \( A_2 \):} 
    Generated as a matrix with entries drawn from a normal distribution \( \mathcal{N}(0, 1) \) and then made non-negative by taking the absolute value. Additionally, a correction is applied: if all values in a row of \( A_1 \) or \( A_2 \) are less than 1, we set the maximum value in that row to be 1. This ensures that each city has at least one road leading to it. Finally, we multiply each element of \( A_2 \) by 0.2 to simulate the significantly lower carrying capacity of untaxed roads compared to taxed roads.
    
\end{itemize}

\item \textbf{Vector Generation:}
\begin{itemize}
    \item \textbf{Vector \( b \), \( c_1 \) and \( c_2 \):} 
    Generated as random vectors, with each entry drawn from \( \mathcal{N}(0, 1) \), and then made non-negative by taking the absolute value.

\end{itemize}

\end{enumerate}

\paragraph{Implementation Details}

\begin{enumerate}
\item \textbf{BSG Algorithm~\cite{giovannelli2021inexact}:} 
    \begin{itemize}
        \item For inner loop, we use a max-min solver (Algorithm 2 in \cite{jiang2024primal}) to obtain the lower-level optimal solution $y$ and Lagrange multiplier $z_I$;
        \item We choose the initial point $T_{\text{init}}=0$, $x_{\text{init}}=0$, $y_{\text{init}}=0$, $z_{I,\text{init}}=0$; For the main loop, the stepsize is set to $\alpha_{k}=0.01/\sqrt{k+1}$; For the max-min solver, the stepsizes for the first and second loops are set to $\eta_1=2^{-5}$ and $\eta_2=30$ respectively; The stopping criteria for the first and second loops in the max-min solver are $\epsilon_1=5\times 10^{-4}$ and $\epsilon_2=1$.
    \end{itemize}
    \item \textbf{BLOCC Algorithm~\cite{jiang2024primal}:}
    \begin{itemize}
        \item We choose the initial point $T_{\text{init}}=0$, $x_{g,\text{init}}=0$, $y_{g,\text{init}}=0$, $x_{F,\text{init}}=0$, $y_{F,\text{init}}=0$, $\mu_{g,\text{init}}=0$, $\mu_{F,\text{init}}=0$, and the penalty coefficient $\gamma=2$; For the main loop, the stepsize is set to $\alpha_{k}=0.01/\sqrt{k+1}$; For both max-min solvers, the stepsizes for the first and second loops are set to $\eta_1=2\times10^{-5}$ and $\eta_2=30$ respectively; The stopping criteria for the first and second loops in both max-min solvers are $\epsilon_1=10^{-4}$ and $\epsilon_2=1$.
    \end{itemize}
    \item \textbf{BFBM Algorithm (Our proposed Algorithm):}
    \begin{itemize}
        \item For inner loop we use Newton method;
        \item We set $t=10^{-4}$, and $M=0.1$, which means we shrink the lower-level feasible set to $\{y|h_{i}(x,y)\leq-10^{-5}\text{ for }i=1,...,k\}$;
        \item We choose the initial point $x_{\text{init}}=0$, $y_{\text{init}}=0$; For the main loop, the stepsize is set to $\alpha_{k}=0.02/\sqrt{k+1}$; For inner loop the stepsizes is set to $\eta=0.01$; The stopping criteria for inner loop is $\epsilon_{xy}=5\times 10^{-4}$.
    \end{itemize}
\end{enumerate}

\section{Preliminaries}
\subsection{Notations}\label{sec:notations}

\begin{itemize}
    \item Denote $[\lambda_1(x,y),...,\lambda_k(x,y)]:=\argmax_\lambda g(x,y)+\sum_{i=1}^k\lambda_i h_i(x,y)$ be the Lagrange multipliers at point $(x,y)$, where $y\in y^\ast(x)$. If $y^\ast(x)$ is unique, or $\lambda_i(x,y)$ is independent of the choice of $y\in y^\ast(x)$, for simplicity, we write $\lambda_i(x,y)=\lambda_i(x)$, and, under this case, define $\mathcal{I}^\ast(x):=\{i:\lambda_i(x)>0\}$ be the set of index that corresponding multiplier is positive;
    \item $\mathcal{Y}_m(x):=\{y:h_i(x,y)\leq -m,\ i=1,...,k\}$ is the shrinked feasible set;
    \item $\hat{y}_s$ is approximation of $y^\ast(x_s)$ from the output of Algorithm \ref{algo:fo_lower_2};
    \item $D$ introduced in Remark \ref{remark:non-emptyness} is a constant that for any $x$ there exists $y$ satisfies $h_i(x,y)\leq -D$ for any $i$;
    \item $\mu_{\widetilde{g}_t}$ introduced in Proposition \ref{prop:strongly_convex} is the strongly convexity constant of $\widetilde{g}_t(x,y)$ in $y$;
    \item $M$ defined in Corollary \ref{remark:upper_bound_of_h_at_approximate_point} is the positive lower bound of $-h_i(x,y^\ast(x))$ and $-h_i(x_s,\hat{y}_s)$; 
    \item $\overline{L}_{\widetilde{g}_t,m}$ introduced in Proposition \ref{prop:lipschitz_smooth} is the Lipschitz smoothness constant of $\widetilde{g}_t(x,y)$ in $y$ in the set $\mathcal{Y}_m(x)$; 
    \item $\overline{\overline{L}}_{\widetilde{g}_t,m}$ introduced in Proposition \ref{lem:lipHessian} is the Lipschitz constant of Hessian for $\widetilde{g}_t(x,y)$ in the set $\mathcal{Y}_m(x)$; 
    \item $\kappa_s$ utilized in Theorem \ref{thm:convergence_rage_alg_2} is the condition number of the inner loop at the $s$-th step of the outer loop;
    \item $\overline{L}_{\widetilde{\phi}_t,s}$ studied in Lemma \ref{lem:lip_global_bound} is the estimate of local Lipschitz smoothness constant at the $s$-th step of the outer loop;
    \item $\overline{L}_{\widetilde{\phi}_t}$ introduced in Lemma \ref{lem:lip_global_bound} is the global upper bound of $\overline{L}_{\widetilde{\phi}_t,s}$;
    \item $\hat{\nabla}_x\widetilde{\phi}_t(x_s):=\nabla_x f(x,y)+\nabla^2_{xy}\widetilde{g}_t(x_s,\hat{y}_s)(\nabla^2_{yy}\widetilde{g}_t(x_s,\hat{y}_s))^{-1}\nabla_y f(x_s,\hat{y}_s)$ is the approximation of the hypergradient $\nabla_x\widetilde{\phi}_t(x_s)$ at the $s$-th step of the outer loop;
    \item $\overline{L}_{\widetilde{\phi}_t,m_s}'$ studied in Lemma \ref{lemma:approx_error_hypergrad} is the approximation error between $\hat{\nabla}_x\widetilde{\phi}_t(x_s)$ and $\nabla_x\widetilde{\phi}_t(x_s)$, i.e. $\|\hat{\nabla}_x\widetilde{\phi}_t(x_s)-\nabla_x\widetilde{\phi}_t(x_s)\|\leq \overline{L}_{\widetilde{\phi}_t,m_s}'\|y^\ast(x_s)-\hat{y}_s\|$;
    \item $\overline{L}_{\widetilde{\phi}_t}'$ given in Lemma \ref{lem:upperboundofapproxgradient} is the upper bound of $\overline{L}_{\widetilde{\phi}_t,m_s}'$.
\end{itemize}

\subsection{Some basic lemmas}\label{proof:somebasiclemmas}

\begin{lemma}\label{lem:optimalitygapforpoint}
    Suppose Assumption \ref{assumption:general1}(\ref{assumption:general1(1)}), \ref{assumption:general1}(\ref{assumption:general1(2)}) and Assumption \ref{assumption:nonlinear} hold, then 
    \begin{align}
        \left\|y^\ast_t(x)-y^\ast(x)\right\|\leq \sqrt{\frac{2}{\mu_g}kt}.
    \end{align}
\end{lemma}

\begin{proof}
    Since $g(x,y)$ is $\mu_g$-strongly convex in $y$, we have
    \begin{align}\label{eq:sc}
        g(x,y)\geq g(x,y^\ast(x))+\nabla_y g(x,y^\ast(x))(y-y^\ast(x))+\frac{\mu_g}{2}\left\|y-y^\ast(x)\right\|^2.
    \end{align}
    By the optimal condition, $\nabla_y g(x,y^\ast(x))(y-y^\ast(x))\geq 0$ for any feasible point $y$. Replacing $y$ by $y^\ast_t(x)$ in Equation (\ref{eq:sc}), we get
    \begin{align*}
        \frac{\mu_g}{2}\left\|y^\ast_t(x)-y^\ast(x)\right\|^2&\leq g(x,y^\ast_t(x))-g(x,y^\ast(x))-\nabla_y g(x,y^\ast(x))(y^\ast_t(x)-y^\ast(x))\\
        &\leq g(x,y^\ast_t(x))-g(x,y^\ast(x)).
    \end{align*}
    From lemma \ref{lem:optimalitygap}, we know  $g(x,y^\ast_t(x))-g(x,y^\ast(x))\leq kt$, which implies
    \begin{align}\label{ineq:uniformlyconverge}
        \left\|y^\ast_t(x)-y^\ast(x)\right\|\leq \sqrt{\frac{2}{\mu_g}kt}.
    \end{align}
\end{proof}

\begin{lemma}\label{lem:continuityofyastt}
    Suppose Assumption \ref{assumption:general1} holds, and either Assumption \ref{assumption:linear} or \ref{assumption:nonlinear} holds, then $y^\ast_t(x)$ continuously depends on $t$ for any fixed $x$.
\end{lemma}

\begin{proof}
    Denote $G_x(t,y)=g_t(x,y)$. We fix a $x$. Note that
    \begin{align*}
        \nabla_{y} (\nabla_y G_x(t,y))=\nabla^2_{yy}\widetilde{g}_t(x,y)\succ 0
    \end{align*}
    for any $t$ and $y$, and $\nabla_y G_x(t,y^\ast_t(x))=0$. By implicit function theorem, we can find a neighborhood $U_t\times U_{y^\ast_t(x)}$ of $(t,y^\ast_t(x))$ such that there exists a continuous function $Y(t)$ satisfies $\nabla_y G_x(t,y)=0$ if and only if $y=Y(t)$. Since $G$ is strongly convex in $y$, $\nabla_y G_x(t,y)=0$ means $y=y^\ast_t(x)$. Therefore, by the uniqueness of the optimal solution, locally we have $y^\ast_t(x)=Y(t)$ continuously depending on $t$. As we already know $y^\ast_t(x)$ is a function of $t$, we conclude that globally, i.e. for any $t$, $y^\ast_t(x)$ continuously depends on $t$.
\end{proof}

\begin{lemma}\label{lem:con_of_multi}
    Suppose Assumption \ref{assumption:general1} holds, and either Assumption \ref{assumption:nonlinear} or Assumption \ref{assumption:linear} holds, then the Lagrange multiplier $\lambda_i(x)$ is continuous for any $i$. Furthermor, since $\mathcal{X}$ is compact, $\lambda_i(x)$ has upper bound independent of $i$ and $x$.
\end{lemma}

\begin{proof}
    When $g(x,y)$ is strongly convex in $y$ and $h_i(x,y)$ is convex in $y$, then $y^\ast(x)$ is unique. If $\lambda_i(x)$ is not continuous for some point $x_0$ and some index $i_0$, then we can find a sequence $\{x_j\}_{j=1}^\infty$ such that $\lim_{j\to\infty}x_j=x_0$, and $\lim_{j\to\infty}\lambda_{i_0}(x_j)\ne\lambda_{i_0}(x_0)$, or does not exist.\\
    
    \noindent\textbf{Case 1:} If $\lim_{j\to\infty}\lambda_{i_0}(x_j)$ does not exist. We consider separately the cases where $\lambda_{i_0}(x_j)$ is uniformly bounded and where there exists an unbounded subsequence. 
    
    If $\{\lambda_{i_0}(x_j)\}_{j=1}^\infty$ has an unbounded subsequence, still denoted as $\{\lambda_{i_0}(x_j)\}_{j=1}^\infty$, i.e. $\lim_{j\to\infty}|\lambda_{i_0}(x_j)|=\infty$, then set $\mu^\ast_j=\max_{i}|\lambda_i(x_j)|$. By the unboundedness assumption, $\lim_{j\to\infty}\mu^\ast_j=\infty$. We can choose a sub-sequence of $\{x_j\}_{j=1}^\infty$, still denoted as $\{x_j\}$, such that $i^\ast=\arg\max_{i}|\lambda_i(x_j)|$ for some fixed index $i^\ast$. Since $|\lambda_i(x_j)|/\mu^\ast_j\leq 1$ for any $i$ and $j$, we can choose a subsequence, still denoted as $\{x_j\}$, such that $\lim_{j\to\infty}\lambda_i(x_j)/\mu^\ast_j$ exists for any $i$. Denote $\mu_i(x_0):=\lim_{j\to\infty}\lambda_i(x_j)/\mu^\ast_j$. It is clear that $\mu_{i^\ast}(x_0)=1$.
    Replacing $x$ with $x_j$ in the KKT condition equation
    \begin{align}\label{eq:kkt}
        \nabla_y g(x,y^\ast(x))+\sum_{i=1}^k\lambda_i(x) \nabla_y h_i(x,y^\ast(x))=0,
    \end{align}
    dividing by $\mu^\ast_j$, and letting $j$ approach infinity, we get
    \begin{align}\label{eq:kkt2}
        \sum_{i=1}^k\mu_i(x_0)\nabla_y h_i(x_0,y^\ast(x_0))=0.
    \end{align}
    We will show that $\mu_i(x_0)\ne 0$ means $i$ is active at $x_0$. If $\mu_i(x_0)\ne 0$, then $\lambda_i(x_j)\ne 0$ except for at most finitely many $j$. So the index $i$ is active at the point $x_j$, which implies $h_i(x_j,y^\ast(x_j))=0$. Thus, $h_i(x_0,y^\ast(x_0))=\lim_{j\to\infty}h_i(x_j,y^\ast(x_j))=0$, implying that $i$ is active at the point $x_0$. (\ref{eq:kkt2}) can be written as
    \begin{align*}
        \sum_{i\in\mathcal{I}^\ast(x_0)}\mu_i(x_0)\nabla_y h_i(x_0,y^\ast(x_0))=0,
    \end{align*}
    where $\mathcal{I}^\ast(x_0)$ is the set of active index at $x_0$. By LICQ assumption, $\nabla_y h_i(x_0,y^\ast(x_0))$ are linear independent for $i\in\mathcal{I}^\ast(x_0)$. Note that $\mu_{i^\ast}(x_0)=1\ne 0$, this contradicts the LICQ assumption. Therefore, we know that $\{\lambda_{i_0}(x_j)\}_{j=1}^\infty$ is uniform bounded.\\
    
    \noindent\textbf{Case 2:} If $\lim_{j\to\infty}\lambda_{i_0}(x_j)$ exists but not equals $\lambda_{i_0}(x_0)$, we will show that this also contradicts LICQ assumption. Case 1 demonstrates that $\lambda_i(x_j)$ has a uniform bound for any $i$ and $j$. Thus, we can choose a subsequence of $\{x_j\}_{j=1}^\infty$, still denoted as $\{x_j\}_{j=1}^\infty$, such that $\lim_{j\to\infty}\lambda_i(x_j)=:\mu_i(x_0)$ exists for any $i$. Replacing $x$ with $x_j$ in (\ref{eq:kkt}), and letting $j$ approach infinity, we get
    \begin{align}\label{eq:kkt3}
        \nabla_y g(x_0,y^\ast(x_0))+\sum_{i=1}^k\mu_i(x_0)\nabla_y h_i(x_0,y^\ast(x_0))=0.
    \end{align}
    Replacing $x$ with $x_0$ in the KKT condition (\ref{eq:kkt}), and subtracting (\ref{eq:kkt3}), we obtain
     \begin{align*}
        \sum_{i=1}^k(\lambda_i(x_0)-\mu_i(x_0))\nabla_y h_i(x_0,y^\ast(x_0))=0.
    \end{align*}
    If $\lambda_i(x_0)-\mu_i(x_0)\ne 0$, then either $\lambda_i(x_0)=0$ or $\lambda_i(x_0)\ne 0$. $\lambda_i(x_0)=0$ means $\mu_i(x_0)\ne 0$, which implies $i$ is active at point $x_0$ similar as the argument in Case 1. $\lambda_i(x_0)\ne 0$ means that $i$ is active at point $x_0$. Therefore, $\lambda_i(x_0)-\mu_i(x_0)\ne 0$ always means is active at point $x_0$. Note that $\lambda_{i_0}(x_0)-\mu_{i_0}(x_0)\ne 0$, which contradicts the LICQ assumption.\\

    When $g(x,y)$ and $h_i(x,y)$ are linear in $y$ for any $i$, then by Remark \ref{remark:welldefined}, we know that the multipliers $\lambda_i(x)$ are well-defined in the sense that is independent of $y\in y^\ast(x)$. We can also denote $\nabla_y g(x,y)$ as $w(x)$, and $\nabla_y h_i(x,y^\ast(x))$ as $v_i(x)$ since $\nabla_y g(x,y)$ and $\nabla_y h_i(x,y)$ is independent of $y$. Then the proof is the same as the strongly convex case.
\end{proof}

\section{Improved Convergence Rate}\label{sec:improvedconvergencerate}

In this section, we will demonstrate how to modify the choice of stepsize so that the convergence rate of Algorithm \ref{algo:fo} with respect to $t$ can be improved to $\widetilde{\mathcal{O}}(t^{-1.5})$ under strongly convex case. However, it should be noted that we need to utilize the upper bounds of certain terms. We will prove that these terms are indeed bounded, but these bounds are not computable in practice. 

Specifically, in Lemma \ref{lem:local_lip}, when estimating the local Lipschitz constant, we frequently evaluate the term $\nabla_{xy}\widetilde{g}_t(x,y)(\nabla_{yy}\widetilde{g}_t(x,y))^{-1}$. For the upper bound estimation of $(\nabla_{yy}\widetilde{g}_t(x,y))^{-1}$, we directly employ the strongly convexity constant of $\nabla_{yy}\widetilde{g}_t(x,y)$ in $y$ from Proposition \ref{prop:strongly_convex}. However, it is important to note that, under strongly convex case, when $\nabla_{xy}\widetilde{g}_t(x,y)$ explodes, which is because $h$ approaches zero, $\nabla_{yy}\widetilde{g}_t(x,y)$ also explodes. Therefore, using the strong convexity of $\nabla_{yy}\widetilde{g}_t(x,y)$ for the upper bound estimation of $(\nabla_{yy}\widetilde{g}_t(x,y))^{-1}$ results in significant losses. This leads to an underestimate of the dependence of $\nabla_{xy}\widetilde{g}_t(x,y)(\nabla_{yy}\widetilde{g}_t(x,y))^{-1}$ on $t$. 

We assume that all values of $t$ are less than or equal to 1. This assumption is justified, as in practical applications, we often choose a very small $t$. To improve the estimate, we establish the following result:

\begin{lemma}\label{lem:boundednessofJabobian}
    Suppose Assumption \ref{assumption:general3}, Assumption \ref{assumption:general2}, and Assumption \ref{assumption:nonlinear2} hold, i.e. under the strongly convex setting. Then for any $0<t\leq 1$ and $x\in\mathcal{X}$, the following holds 
    $$\|\nabla_x y^\ast_t(x)\|=\|(\nabla_{yy} \widetilde{g}_t(x,y^\ast_t(x)))^{-1}\nabla_{yx} \widetilde{g}_t(x,y^\ast_t(x))\|\le J_1$$
    for some constant $J_1>0$.
\end{lemma}

The proof is provided in Section \ref{proof:lem:boundednessofJabobian}. In addition, we also need the following technical results:

\begin{lemma}\label{cor:bound1}
    Suppose Assumptions in Lemma \ref{lem:boundednessofJabobian} are satisfied. Then for any  $0<t\leq 1$, any $x\in\mathcal{X}$ and index $\hat{j}$  we have 
    $$\left\|(\nabla_{yy} \widetilde{g}_t(x))^{-1}\frac{t\nabla_y h_{\hat{j}}(x,y^\ast_t(x))}{(-h_{\hat{j}} (x,y^\ast_t(x)))^2}\right\|\le J_2$$
  for some constant $J_2>0$.
\end{lemma}

\begin{lemma}\label{cor:bound2}
    Suppose Assumptions in Lemma \ref{lem:boundednessofJabobian} are satisfied. There exists a constant $\beta$ such that for any $0<t\leq 1$, $x$ in $\mathcal{X}$ and $y\in\text{int}\mathcal{Y}(x)$, which is the interior of $\mathcal{Y}(x)$, such that $\|y-y^\ast(x)\|\leq\beta$, we have $\|(\nabla_{yy}\widetilde{g}_t(x,y))^{-1}\nabla_{yx}\widetilde{g}_t(x,y)\|\le J_3$, for some constant $J_3>0$.
\end{lemma}

Lemma \ref{cor:bound1} is a byproduct of Lemma \ref{lem:boundednessofJabobian}. We provide the proof in Section \ref{proof:cor:bound1}. The proof of Lemma \ref{cor:bound2} utilizes some technical details from Lemma \ref{lem:boundednessofJabobian}, as discussed in Section \ref{proof:cor:bound2}. Now we can estimate the improved local Lipschitz constant by the boundedness we proved before.

\begin{lemma}\label{lem:improvedglobalupperboundLip}
    Suppose the assumptions in Lemma \ref{lem:boundednessofJabobian} are satisfied. For any $0<t\leq 1$, and any fixed $x$, if there exists $y\in\mathcal{Y}(x)$ such that $h_i(x,y)\leq -d$ holds for any $i\in\{1,...,k\}$, then the hyperfunction $\widetilde{\phi}_t(x)$ is $\overline{L}^{loc}_{\widetilde{\phi}_t}$-Lipschitz smooth in the ball $B_x(\min\{\frac{d}{2L_h},\frac{m}{2L_h(1+J_1)}\})$ which centered at $x$ with radius $\min\{\frac{d}{2L_h},\frac{m}{2L_h(1+J_1)}\}$, where $J_1$ is from Lemma \ref{lem:boundednessofJabobian} and $m$ is computed by (\ref{eq:m}). That is, the following holds for any $x_1,x_2$ satisfies $\max\{\|x_1-x\|,\|x_2-x\|\}\leq\min\{\frac{d}{2L_h},\frac{m}{2L_h(1+J_1)}\}$:
    \begin{align*}
            \|\nabla_{x} \widetilde{\phi}_t(x_1)-\nabla_{x} \widetilde{\phi}_t(x_2)\|\leq \overline{L}^{loc}_{\widetilde{\phi}_t}\|x_1-x_2\|.
    \end{align*}
    $\overline{L}^{loc}_{\widetilde{\phi}_t}$ can be computed as follows
    
    \resizebox{\textwidth}{!}{
    \begin{minipage}{\textwidth}
    \begin{align*}
        \overline{L}^{loc}_{\widetilde{\phi}_t}=&J_1\left(\overline{L}_f+\overline{L}_fJ_1+L_fJ_1\left(\frac{\overline{\overline{L}}_g}{\mu_g}+\frac{tk\overline{\overline{L}}_h}{\mu_gm^{loc}}+\frac{tk\overline{L_h}L_h}{\mu_g\left(m^{loc}\right)^2}+2\frac{tkL_h\overline{L}_h}{\mu_g\left(m^{loc}\right)^2}+J_2\frac{kL_h^2}{m^{loc}}+J_2\frac{3kL_h^2}{m^{loc}}\right)+\overline{\overline{L}}_{\widetilde{g}_t,m^{loc}}J_1\right).
    \end{align*}
    \end{minipage}
    }
     Here $m^{loc}$ is a positive constant which can be computed through replacing $d$ with $d/2$ in (\ref{eq:m}), and $\mu_{\widetilde{g}_t}$ is from Lemma \ref{prop:strongly_convex}, $J_1$ is from Lemma \ref{lem:boundednessofJabobian}, $J_2$ is from Lemma \ref{cor:bound1}. $\overline{L}^{loc}_{\widetilde{\phi}_t}$ has a upper bound $\overline{L}_{\widetilde{\phi}_t}=\mathcal{O}(\frac{1}{t})$. 
\end{lemma}

The proof of Lemma \ref{lem:improvedglobalupperboundLip} can be found in Section \ref{proof:lem:improvedglobalupperboundLip}. Given our refined estimation of the Lipschitz constant, we are now able to optimize the stepsize selection. This adjustment significantly enhances the convergence rate with respect to $t$, leading to the following convergence results, whose proof is in (\ref{proof:thm:total_rate2}):

\begin{theorem}\label{thm:total_rate2}
    Suppose Assumption \ref{assumption:general3}, Assumption \ref{assumption:general2}, and Assumption \ref{assumption:nonlinear2} hold, i.e. under the strongly convex setting. We recall and define the following terms:
    \begin{itemize}
        \item $\beta$ is from Lemma \ref{cor:bound2};
        \item $\overline{L}_{\widetilde{\phi}_{t},s}$ is computed by replacing $d$ in Lemma \ref{lem:improvedglobalupperboundLip} with $d_s$ from the output of Algorithm \ref{algo:fo_lower_2};
        \item $\overline{L}'_{\widetilde{\phi}_t,m_s}$ is from (\ref{eq:L'phims});
        \item $\hat{\nabla}_x\widetilde{\phi}_t(x_s)$ is from (\ref{eq:approximategradient});
        \item $J_1$ is from Lemma \ref{lem:boundednessofJabobian};
        \item $J_3$ is from Lemma \ref{cor:bound2};
        \item $M$ is from Corollary  \ref{remark:upper_bound_of_h_at_approximate_point}; 
        \item $D$ is from Remark \ref{remark:non-emptyness};
        \item $m_s$, $d_s$ are from the output of Algorithm \ref{algo:fo_lower_2};
        \item $\mu_{\widetilde{g}_t}$ is from Proposition \ref{prop:strongly_convex};
        \item $\overline{L}_{\widetilde{\phi}_t}$ is from Lemma \ref{lem:improvedglobalupperboundLip};
        \item
        $\zeta = \min\left\{1, \frac{D}{4L_hL_f(1+J_3)}, \frac{M}{4L_hL_f(1+J_1)(1+J_3)}, \frac{1}{\overline{L}_{\tilde{\phi}_t}}\right\}$.
    \end{itemize}

    Assume $0<t\leq\min\{ \frac{\mu_g\beta^2}{2k},1\}$. We set the stepsize $\eta_s$ as
    \begin{align}
        \eta_s=\min\{1,  \frac{d_s}{2 L_h}\times \frac{1}{\|\hat{\nabla}_x\tilde{\phi}(x_s)\|}, \frac{m_s}{2 L_h(1+J_1)}\times\frac{1}{\|\hat{\nabla}_x\tilde{\phi}(x_s)\|},\frac{1}{\overline{L}_{\tilde{\phi}_{t},s}}\},
    \end{align} 
    and set the stop criterion $\epsilon_s=\min\{\epsilon,\beta/2\}/\overline{L}'_{\widetilde{\phi}_t,m_s}$ for the inner loop. Then the sequence $\{x_s\}$ generated by Algorithm \ref{algo:fo} satisfies
    $$
    \min_{s=0,...,S-1}\frac{1}{\eta_s}\|x_{s} - \proj_{\mathcal{X}}(x_s-\eta_s \nabla_x\tilde{\phi}(x_s))\|\leq\epsilon
    $$
    within $\widetilde{\mathcal{O}}(1/(\zeta\epsilon^2))$ number of iterations.
\end{theorem}

\begin{remark}
     From Lemma \ref{lem:improvedglobalupperboundLip} and Theorem \ref{thm:total_rate2}, we can see that $\zeta=\mathcal{O}(t)$ with respect to $t$. Note that for inner loop we have $\kappa=\widetilde{\mathcal{O}}(t^{-0.5})$, the total number of iterations is $\widetilde{\mathcal{O}}(t^{-1.5})$.
\end{remark}

\section{Convex Case}\label{section:convex_case}

When $g(x,y)$ is convex in $y$, some intrinsic difficulties arise in attempting to establish a connection between the barrier reformulation and the original problem. First, we will provide some counterexamples to illustrate these issues.

Before the discussion, we establish the following basic assumption for convex case:

\begin{assumption}\label{assumption:general4}The following holds
\begin{enumerate}
    \item $f(x,y)$ is one time and $g(x,y)$, $h_i(x,y)$ are two times continuously differentiable for every $x\in\mathcal{X},y\in\mathcal{Y}(x)$, and $i\in\{1,...,k\}$;\label{assumption:general4(1)}
    \item $\mathcal{X}$ is convex and compact, and for any $x\in\mathcal{X}$ there exists $y\in\mathcal{Y}(x)$ such that $h_i(x,y)<0$ for some $i\in\{1,2,...,k\}$;\label{assumption:general4(2)}
    \item For any $x\in\mathcal{X}$, $\mathcal{Y}(x)$ is compact, and $\|y\|\leq R$ for any $y$ in $\mathcal{Y}(x)$.\label{assumption:general4(3)}
\end{enumerate}
\end{assumption}
When $g(x,y)$ and $h_i(x,y)$ are all convex in $y$, the lower-level objective function $\widetilde{g}_t(x)$ of the barrier reformulation \ref{eq:blo_barrier_reformulation} may not have the unique optimal solution. Consider the following counterexample:

\begin{example}
    Define the lower-level objective function $g(x,y)$ as follows
    \begin{align*}
        g(x,y)=
        \begin{cases}
            (y-1)^2&\text{ if }1\leq y\leq 2\\
            0&\text{ if }-1\leq y< 1\\
            (y+1)^2&\text{ if }-2\leq y< -1,
        \end{cases}
    \end{align*}
    and the lower-level constraint $h(x,y):=g(x,y)-1$. Then $g(x,y)$ and $h(x,y)$ are both convex, and
    \begin{align*}
        \widetilde{g}_t(x,y)=g(x,y)-t\log(1-g(x,y))\geq 0-t\log(1)=0,
    \end{align*} 
    The optimal solution set of the barrier reformulation is $[-1,1]$ for any $t$, which means $y^\ast_t(x)$ is not unique.
\end{example}

Because the lower-level problem does not have a unique solution, the hypergradient cannot be determined. Therefore, in this situation, we cannot use the implicit gradient methods for the barrier reformulation.

When $g(x,y)$ is convex in $y$, but there exist at least one strongly convex constraint $h_i(x,y)$ in $y$, it is not hard to see the lower-level problem of barrier reformulation is strongly convex. Therefore, the hyperfunction of barrier reformulation is well-defined. From this, we can explore the connection between the hyperfunctions of the original problem and the barrier reformulation. Like the counterexample \ref{counterexample1} under the linear setting, we can still only attempt to prove the convergence of the hyperfunction when $y^\ast(x)$ is unique. Before proceeding, let us first present an assumption for the convex case.

\begin{assumption}\label{assumption:convex}
    $g(x,y)$ is convex in $y$ for any $x$, and $h_i(x,y)$ are convex in $y$ for any $i$ and $x$. There exists at least one strongly convex constraint $h_i(x,y)$ in $y$.
\end{assumption}
We give the following convergence result for the convex case:

\begin{theorem}
    Suppose Assumption \ref{assumption:general4} and Assumption \ref{assumption:convex} hold. If $x$ satisfies that $y^\ast(x)$ is unique, then $\lim_{t\to 0}\widetilde{\phi}_t(x)=\phi(x)$.
\end{theorem}

The proof can be directly obtained by optimal condition 
\begin{align*}
    g(x,y_t^\ast(x))-g(x,y^\ast(x))\leq kt,
\end{align*}
which means $\lim_{t\to 0}g(x,y_t^\ast(x))=g(x,y^\ast(x))$. Note that $y^\ast(x)$ is unique, it not hard to see $\lim_{t\to 0}y_t^\ast(x)=y^\ast(x)$, which implies $\lim_{t\to 0}\widetilde{\phi}_t(x)=\phi(x)$.\\

However, establishing a relationship between the hypergradients may still be challenging, as the hypergradient of the original problem might not be determinable through the expression:
\begin{align*}
    \nabla_x\phi(x)=\nabla_x f(x,y^\ast(x))-\nabla_{xy}^2 g(x,y^\ast(x))(\nabla_{yy}^2g(x,y^\ast(x))^{-1}\nabla_y f(x,y^\ast(x))
\end{align*}
for any $x$. Let's see the following example:

\begin{example}
    Define the lower-level objective function $g(x,y)=y^4$, and the lower-level constraint $h(x,y):=y^2-1$. Then $g(x,y)$ is convex, and $h(x,y)$ is strongly convex. The optimal solution $y^\ast(x)=0$ for any $x$. However, the hypergradient
    \begin{align*}
        \nabla_x f(x,y^\ast(x))-\nabla_{xy}^2 g(x,y^\ast(x))(\nabla_{yy}^2g(x,y^\ast(x))^{-1}\nabla_y f(x,y^\ast(x))
    \end{align*}
    is not well-defined since $\nabla_{yy}^2 g(x,y)=0$ is not invertible for any $x$ and $y$. In this context, how to calculate the hypergradient remains unexplored.
\end{example}

\begin{remark}
    For the strong convexity of the lower-level problem in the barrier reformulation, we need at least one constraint $h_i(x,y)$ to be strongly convex in $y$. However, by Assumption \ref{assumption:general4}(\ref{assumption:general4(3)}), we can add a constraint $h_{k+1}(x,y)=\|y\|^2\leq R^2$ which is strongly convex in $y$. After adding this constraint, it remains equivalent to the original problem. Therefore, we can say Assumption \ref{assumption:general4}(\ref{assumption:general4(3)}) implies Assumption \ref{assumption:convex}.
\end{remark}

The most important result, Theorem \ref{lem:local_M}, still holds in the general convex case. Therefore, Algorithm \ref{algo:fo} and Algorithm \ref{algo:fo_lower_2} still work in the general convex case, and we can get the same convergence rate $\widetilde{\mathcal{O}}(1/(\epsilon^{2}t^{4.5}))$ in Theorem \ref{thm:total_rate}. However, we do not have the convergence for hypergradient. The Algorithm \ref{algo:fo_asymptotic} does not work for this case.

\subsection{Proof of Proposition \ref{prop:diff_hyperfun}}\label{proof:prop:diff_hyperfun}
Through direct computing the Hessian of $\widetilde{g}_t(x, y^*_t(x))$, we get
    \begin{align*}
    \nabla^2_{yy}\widetilde{g}_t(x,y^\ast_t(x))&=\nabla^2_{yy}g(x,y^\ast_t(x))+t\sum_{i=1}^k\left(\frac{\nabla^2_{yy}h_i(x,y^\ast_t(x))}{-h_i(x,y^\ast_t(x))}+\frac{\nabla_y h_i(x,y^\ast_t(x))\nabla_y h_i(x,y^\ast_t(x))^\top}{h_i^2(x,y^\ast_t(x))}\right).
    \end{align*}
    If Assumption \ref{assumption:nonlinear} holds, then $\nabla^2_{yy}\widetilde{g}_t(x,y)\succeq \mu_g I$, which means $y^\ast_t(x)$ is unique for any $t>0$ and $x\in\mathcal{X}$. Thus the hypergradient can be computed directly by using implicit function Theorem (see equation (2.8) in \cite{ghadimi2018approximationmethodsbilevelprogramming})
    \begin{align}\label{eq:hypergradient}
        \nabla_x\widetilde{\phi}_t(x)=\nabla_x f(x,y_t^\ast(x))-\nabla_{xy}^2\widetilde{g}_t(x,y_t^\ast(x))(\nabla_{yy}^2\widetilde{g}_t(x,y_t^\ast(x)))^{-1}\nabla_y f(x,y_t^\ast(x)).
    \end{align}
    This also illustrates that $\widetilde{\phi}_t(x)$ is differentiable.

    If Assumption \ref{assumption:linear} holds (i.e., $h_i$'s and $g$ are linear), the Hessian can be bounded below by 
    \begin{align*}
        \nabla^2_{yy}\widetilde{g}_t(x,y^\ast_t(x))&=t\sum_{i=1}^k\frac{\nabla_y h_i(x,y^\ast_t(x))\nabla_y h_i(x,y^\ast_t(x))^\top}{h_i^2(x,y^\ast_t(x))}\\
        &\geq \frac{t}{\min_i h_i^2(x,y^\ast_t(x))}\sum_{i=1}^k\nabla_y h_i(x,y^\ast_t(x))\nabla_y h_i(x,y^\ast_t(x))^\top.
    \end{align*}
    It is noteworthy that $h_i(x,y^\ast_t(x))<0$ for any $i$. This is because, if $h_i(x,y^\ast_t(x))=0$ for some $i$, then $\widetilde{g}_t(x,y^\ast(x))=\infty$ due to the log-barrier, which contradicts that $y^\ast_t(x)$ minimize $\widetilde{g}_t(x,y)$. Since $h_i(x,y)$ is linear, we can set
    \begin{align*}
        h_i(x,y)&=a_i(x)^\top y-b_i(x).
    \end{align*}
    Define $A(x)=(a_1(x),\cdots,a_k(x))^\top$. Note that the feasible set $\mathcal{Y}(x)$ is compact, we get $A(x)^\top A(x)$ is positive definitive for any $x$. Otherwise, if there exists a vector $v$ such that $v^\top A(x)^\top A(x) v=0$, then $v^\top a_i(x)=0$ for any $i$. However, in this case, if $h_i(x,y)\le 0$, it implies that $h_i(x,y+\alpha v)\le 0$ for any $\alpha$,  which contradicts the compactness of feasible set. Thus, we know $\nabla^2_{yy}\widetilde{g}_t(x,y^\ast_t(x))\succ 0$, which implies (\ref{eq:hypergradient}) still holds. We obtain that $\widetilde{\phi}_t(x)$ is differentiable.

\section{Convergence of Hypergradient to the original problem}

\subsection{Proof of Proposition \ref{prop:relation}}\label{proof:prop:relation}
\textbf{(1)} Under the strongly convex setting, we already have that $y^\ast(x)$ is unique. By  \cite{Giorgi2018ATO}, it is known that $\nabla_x y^\ast(x)$ exists. As for the converse, it does not hold. Consider the following counterexample: $g(x,y)=y^2$ and $h(x,y)=y$, where $y\in\RR$. Then $y^\ast(x)=0$ and $\nabla_x y^\ast(x)=0$ for any $x$, and also $h(x,y)$ is always active. However, by KKT condition we have
$$0=\nabla_y g(x,0)+\lambda(x)\nabla_y h(x,0)=\lambda(x)\nabla_y h(x,0)=\lambda(x).$$
Thus the multiplier $\lambda(x)$ always equals $0$.\\
    
\noindent\textbf{(2)} Under the linear setting, we fist prove \textbf{(a)} is equivalent to \textbf{(b)}, and then prove \textbf{(b)} is equivalent to \textbf{(c)}.\\

\noindent\textbf{(a)$\Rightarrow$(b): }Assuming $x$ is an \textbf{SCSC} point, we want to show that $y^\ast(x)$ is unique. We will use proof by contradiction to demonstrate that if this is not true, then there must exist a point $(x,y)$ at which the optimal Lagrange multiplier $\lambda_i(x,y)$ is zero for some $i$. 

Suppose $y^\ast(x)$ is not unique, then the optimal solution set must be the face $F$ of a polyhedron. Denote $\partial F$ as the relative boundary of $F$ (for specific definition, see Definition 2.11 on Page 28 in \cite{drusvyatskiy2020convex}). For any two points $y$ and $y'$ in $F\setminus\partial F$, they have the same active constraints. Due to linearity of $g$ and $h$, we have $\nabla_y g(x,y)=\nabla_y g(x,y')$ and $\nabla_y h_i(x,y)=\nabla_y h_i(x,y')$ for any $i$. Combining these facts with the LICQ assumption, we get $\lambda_i(x,y)=\lambda_i(x,y')$. For any point $y''$ in $\partial F$, the active index set of $y\in F\setminus\partial F$ is necessarily a proper subset of the active index set of $y''$. By LICQ assumption, we can prove $\lambda_i(x,y)=\lambda_i(x,y’')$ for any active $i$ as follows: denote $\mathcal{I}^\ast(x,y)$, $\mathcal{I}^\ast(x,y'')$ as the set of active indices at $(x,y)$ and $(x,y'')$, respectively.  Consider the following KKT condition at $(x,y)$ and $(x,y'')$
    \begin{align}
        \nabla_y g(x,y)+\sum_{i\in\mathcal{I}^\ast(x,y)}\lambda_i(x,y)\nabla_y h_i(x,y)&=0\label{eq:kkt_condition_at_y}\\
        \nabla_y g(x,y'')+\sum_{i\in\mathcal{I}^\ast(x,y'')}\lambda_i(x,y'')\nabla_y h_i(x,y'')&=0\label{eq:kkt_condition_at_y''}
    \end{align}
    Note that $\mathcal{I}^\ast(x,y)$ is a proper subset of $\mathcal{I}^\ast(x,y'')$, and $\nabla_y g(x,y)=\nabla_y g(x,y'')$, $\nabla_y h_i(x,y)=\nabla_y h_i(x,y'')$ since $g(x,y)$, $h_i(x,y)$ is linear in $y$ for any $i$, by subtracting (\ref{eq:kkt_condition_at_y}) from (\ref{eq:kkt_condition_at_y''}), we have:
    \begin{align*}
        \sum_{i\in\mathcal{I}^\ast(x,y'')}(\lambda_i(x,y'')-\lambda_i(x,y))\nabla_y h_i(x,y'')=0.
    \end{align*}
    By LICQ assumption, the gradients $\nabla_y h_i(x,y'')$ are independent for $i\in\mathcal{I}^\ast(x,y'')$. Thus, we obtain that $\lambda_i(x,y)=\lambda_i(x,y’')$ for any $i\in \mathcal{I}^*(x,y'')$. Since $\mathcal{I}^\ast(x,y)$ is a proper subset of $\mathcal{I}^\ast(x,y'')$, there exists an index $i\in \mathcal{I}^\ast(x,y'')$ but $i\notin  \mathcal{I}^\ast(x,y)$. Therefore, we know $\lambda_i(x,y)=0$, which further means $\lambda_i(x,y'')=\lambda_i(x,y)=0$. However, $i$ is active at $(x,y'')$. This contradicts the assumption that $x$ is an \textbf{SCSC} point, thus proving the uniqueness of $y^*(x)$. By the argument above, we can also conclude $\lambda_i(x,y)$ is not a function of $y\in y^\ast(x)$.\\

\noindent\textbf{(b)$\Rightarrow$(a): }
    Assuming $y^\ast(x)$ is unique, we want to show $x$ is an \textbf{SCSC} point. From the proof of \textbf{(a)$\Rightarrow$(b)}, we know that $\lambda_i(x,y)$ is independent of the choice of $y\in y^\ast(x)$. With out loss of generality, we denote $\lambda_i(x,y)=\lambda_i(x)$. We will use proof by contradiction to demonstrate that if this is not true, we can construct a vector $v$ such that we can prove $h_i(x,y^\ast(x)+sv)\leq 0$ for any index $i$ and $s$ small enough. This means $(x,y^\ast(x)+sv)$ is feasible. We can also prove $g(x,y^\ast(x)+sv)=g(x,y^\ast(x))$. This contradicts the uniqueness of $y^\ast(x)$.
    
    Given that $y^\ast(x)$ is the unique optimal solution, it must be an extreme point of a polyhedron. Without loss of generality, we assume $y^\ast(x)$ is determined by $h_i(x,y^\ast(x))=0$, $i=1,...,p$, i.e. the first $p$ constraints are active. Since $y^\ast(x)$ is unique, we have $p\ge m$. By assuming the LICQ, we also know $p\leq m$. Thus $p$ exactly equals $m$. By complementary slackness, the multipliers satisfy $\lambda_i(x)=0$ for $i=m+1,...,k$. We now prove that $\lambda_i(x)>0$ for each $i\in\{1,...,m\}$. Suppose, for contradiction, that $\lambda_i(x)=0$ for some $i\in\{1,...,m\}$. Without loss of generality, assume $\lambda_i(x)=0$ for $i\in\{1,...,l\}$, where $l\leq m$. According to KKT conditions, we have
    \begin{align}
        &\nabla_y g(x,y^\ast(x))+\sum_{i=1}^k\lambda_i(x) \nabla_y h_i(x,y^\ast(x))\nonumber\\
        =&\begin{cases}
            \nabla_y g(x,y^\ast(x))+\sum_{i=l+1}^m\lambda_i(x) \nabla_y h_i(x,y^\ast(x)) & \text{ if }l<m\\
            \nabla_y g(x,y^\ast(x)) & \text{ if }l=m
        \end{cases}\label{linearkkt}\\
        =&0.\nonumber
    \end{align}
    For notational convenience, we suppose that $l<m$. The case that $l=m$ is similar. Therefore,
    \begin{align*}
        \nabla_y g(x,y^\ast(x))\in \mathrm{span}\{\nabla_y h_i(x,y^\ast(x)):i=l+1,...,m\},
    \end{align*}
    where $\mathrm{span}\{\nabla_y h_i(x,y^\ast(x)):i=l+1,...,m\}$ is the subspace spanned by $\nabla_y h_i(x,y^\ast(x))$ for $i=l+1,...,m$. Now consider the subspace $\mathrm{span}\{\nabla_y h_i(x,y^\ast(x)):i=1,...,l\}$. The LICQ assumption means these vectors are independent, thus the dimension of this subspace is $l$. Let us denote the matrix formed by these gradients as
    \begin{align*}
        A=[\nabla_y h_1(x,y^\ast(x)),\cdots,\nabla_y h_l(x,y^\ast(x))]^\top.
    \end{align*}
    By Gordan's theorem, either $Av<0$ has a solution $v$, or $A^\top w=0$ has a nonzero solution $w$ with $w\geq 0$. Note that by LICQ assumption, $\nabla_y h_1(x,y^\ast(x)),\cdots,\nabla_y h_l(x,y^\ast(x))$ are linear independent, so $A^\top w=0$ has only the zero solution. Consequently, there exists a vector $v\ne 0$ such that $v^\top \nabla_y h_i(x,y^\ast(x))<0$ for any $i\in\{1,...,m\}$. Treating $\mathrm{span}\{\nabla_y h_i(x,y^\ast(x)):i=1,...,m\}$ as the entire space, we consider the orthogonal complement of the subspace $\mathrm{span}\{\nabla_y h_i(x,y^\ast(x)):i=1,...,l\}$, denoted as $\left(\mathrm{span}\{\nabla_y h_i(x,y^\ast(x)):i=1,...,l\}\right)^\perp$. Denote the sum of this space and the space spanned by $v$ as $\left(\mathrm{span}\{\nabla_y h_i(x,y^\ast(x)):i=1,...,l\}\right)^\perp+\mathrm{span}\{v\}$.
    The dimension of $\left(\mathrm{span}\{\nabla_y h_i(x,y^\ast(x)):i=1,...,l\}\right)^\perp+\mathrm{span}\{v\}$ is $m-l+1$. Treating $\RR^m$ as the entire space, we consider the orthogonal complement of the subspace $\mathrm{span}\{\nabla_y h_i(x,y^\ast(x)):i=l+1,...,m\}$, denoted as $\left(\mathrm{span}\{\nabla_y h_i(x,y^\ast(x)):i=l+1,...,m\}\right)^\perp$. Its dimension is $l$. Then, the intersection of two subspaces
    \begin{align*}
        \left\{\left(\mathrm{span}\{\nabla_y h_i(x,y^\ast(x)):i=1,...,l\}\right)^\perp+\mathrm{span}\{v\}\right\}\cap\left(\mathrm{span}\{\nabla_y h_i(x,y^\ast(x)):i=l+1,...,m\}\right)^\perp\ne\emptyset
    \end{align*}
    since the sum of their dimensions is greater that $m$. Choose a vector $\hat{v}\ne 0$ in the intersection of the two subspaces. We decompose $\hat{v}=v_1+v_2$, where $v_1\in \left(\mathrm{span}\{\nabla_y h_i(x,y^\ast(x)):i=1,...,l\}\right)^\perp$ and $v_2\in \mathrm{span}\{v\}$, so there exist a constant $\alpha$ such that $v_2=\alpha v$. We can assume $\alpha\ge 0$, otherwise consider $-\hat{v}$. Therefore, we have 
    \begin{align*}
        \hat{v}^\top\nabla_y h_i(x,y^\ast(x))=v_2^\top\nabla_y h_i(x,y^\ast(x))=\alpha v^\top\nabla_y h_i(x,y^\ast(x))\leq 0
    \end{align*}
    for any $i\in\{1,...,l\}$, which implies 
    \begin{align}\label{eq:condition1}
        h_i(x,y^\ast(x)+s\hat{v})\leq 0\text{ for any }i\in\{1,...,l\}\text{, and }s>0
    \end{align}
    by the linearity of $h_i(x,y)$. Note that $\hat{v}\in\left(\mathrm{span}\{\nabla_y h_i(x,y^\ast(x)):i=l+1,...,m\}\right)^\perp$, we also have $\hat{v}^\top\nabla_y h_i(x,y^\ast(x))=0$ for any $i\in\{l+1,...,m\}$, which means 
    \begin{align}\label{eq:condition2}
        h_i(x,y^\ast(x)+s\hat{v})= 0\text{ for any }i\in\{l+1,...,m\}\text{ and }s\in\RR.
    \end{align} 
    Since $i\in\{m+1,...,k\}$ is inactive, we know $h_i(x,y^\ast(x))<0$ for any $i\in\{m+1,...,k\}$. By the continuity of $h_i(x,y)$ we can find a constant $S>0$ such that 
    \begin{align}\label{eq:condition3}
        h_i(x,y^\ast(x)+s\hat{v})<0\text{ for any }0<s\leq S.
    \end{align}
    Combining the fact $\hat{v}^\top\nabla_y h_i(x,y^\ast(x))=0$ and \eqref{linearkkt}, we obtain $\hat{v}^\top\nabla_y g(x,y^\ast(x))=0$, and thus 
    \begin{align}\label{eq:condition4}
        g(x,y^\ast(x)+s\hat{v})=g(x,y^\ast(x))\text{ for any }s\in\RR.
    \end{align}
    (\ref{eq:condition1}), (\ref{eq:condition2}), (\ref{eq:condition3}) means $y^\ast(x)+s\hat{v}$ remains feasible for any $0<s\leq S$. (\ref{eq:condition4}) implies $y^\ast(x)+s\hat{v}$ is still optimal for any $0<s\leq S$. These contradicts the uniqueness of $y^\ast(x)$. Therefore, we conclude that $\lambda_i(x)>0$ for any $i\in\{1,...,m\}$, i.e. $x$ is an \textbf{SCSC} point.\\

\noindent\textbf{(b)$\Rightarrow$(c): }
    Assuming $y^\ast(x)$ is unique, we will prove $\nabla_x y^\ast(x)$ exists. Since we have proved that $(b)$ is equivalent to $(a)$, $x$ is also an \textbf{SCSC} point. By Lemma \ref{lem:con_of_multi}, i.e. the continuity of $\lambda_i(x)$, there exists a neighborhood $U$ of $x$ such that any point in $U$ is an \textbf{SCSC} point, and their active indices are the same. Without loss of generality, suppose $y^\ast(x)$ is the solution of following equations
    \begin{align*}
        \begin{pmatrix}
        h_1(x,y^\ast(x))\\
        \vdots\\
        h_p(x,y^\ast(x))
    \end{pmatrix}=0.
    \end{align*}
    for any $x$ in the neighborhood. Since $y^\ast(x)$ is unique, we have $p\geq m$. According to LICQ assumption, we also know $p\leq m$. Therefore, $p$ exactly equals $m$, and $(\nabla_y h_1(x,y^\ast(x)),...,h_m(x,y^\ast(x)))$ is invertivble. By the Implicit Function Theorem, $y^\ast(x)$ is differentiable at $x$.\\

\noindent\textbf{(c)$\Rightarrow$(b): }
    Assuming $\nabla_x y^\ast(x)$ exists, we will show $y^\ast(x)$ is unique. The existence of the gradient \(\nabla_x y^\ast(x)\) implicitly assumes differentiability at the point \(y^\ast(x)\), which in turn requires that \(y^\ast(x)\) be well-defined and single-valued. Therefore, \(y^\ast(x)\) is unique.

\subsection{Proof of Theorem \ref{thm:mainlinear}}\label{proof:thm:mainlinear}

We first establish the following lemma:

\begin{lemma}\label{lem_lowerbound}
   Suppose Assumption \ref{assumption:general1} and \ref{assumption:linear} hold and $y^\ast(x)$ is unique, then we have $t/(-h_i(x,y^\ast_t(x)))$ converges to $\lambda_i(x)$.
\end{lemma}

\begin{proof}
    The KKT condition of the original problem is
    \begin{align}\label{eq:kkt4}
        \nabla_y g(x,y^\ast(x))+\sum_{i=1}^k\lambda_i(x) \nabla_y h_i(x,y^\ast(x))=0.
    \end{align} 
    We first claim that under linear case, the notations in (\ref{eq:kkt4}) is also meaningful because
    \begin{enumerate}
        \item By Remark \ref{remark:welldefined}, $\lambda_i(x,y)$ is independent of the choice of $y\in y^\ast(x)$. Without loss of generality we denote $\lambda_i(x)=\lambda_i(x,y)$;
        \item Although $y^\ast(x)$ may represent a set, given that $h_i(x,y)$ and $g(x,y)$ are all linear in $y$, $\nabla_y h_i(x,y)$ is independent of the choice of $y$. Without loss of generality we use the notation $\nabla_y h_i(x,y^\ast(x))$ and $\nabla_y g(x,y^\ast(x))$.
    \end{enumerate}
    By subtracting (\ref{eq:kkt4}) from optimal condition equation of barrier reformulation 
    \begin{align*}
        \nabla_y g(x,y^\ast_t(x))+\sum_{i=1}^k\frac{t}{-h_i(x,y_t^\ast(x))} \nabla_y h_i(x,y_t^\ast(x))=0,
    \end{align*} 
    we derive the following expression
    \begin{align*}
        \nabla_y g(x,y^\ast(x))-\nabla_y g(x,y^\ast_t(x))+\sum_{i=1}^k\left(\lambda_i(x) \nabla_y h_i(x,y^\ast(x))-\frac{t}{-h_i(x,y_t^\ast(x))} \nabla_y h_i(x,y_t^\ast(x))\right)=0.
    \end{align*}
    By adding and subtracting $\sum_{i=1}^k\lambda_i(x) \nabla_y h_i(x,y_t^\ast(x)$, we obatin
    \begin{align*}
        0=&\nabla_y g(x,y^\ast(x))-\nabla_y g(x,y^\ast_t(x))\\
        +&\sum_{i=1}^k\left(\lambda_i(x) \nabla_y h_i(x,y^\ast(x))-\lambda_i(x) \nabla_y h_i(x,y_t^\ast(x))\right)\\
        +&\sum_{i=1}^k\left(\lambda_i(x)\nabla_y h_i(x,y_t^\ast(x))-\frac{t}{-h_i(x,y_t^\ast(x))} \nabla_y h_i(x,y_t^\ast(x))\right)\\
        =&\nabla_y g(x,y^\ast(x))-\nabla_y g(x,y^\ast_t(x))\\
        &+\sum_{i=1}^k\lambda_i(x) \left(\nabla_y h_i(x,y^\ast(x))-\nabla_y h_i(x,y_t^\ast(x))\right)\\
        &+\sum_{i=1}^k\left(\lambda_i(x)-\frac{t}{-h_i(x,y_t^\ast(x))}\right) \nabla_y h_i(x,y_t^\ast(x)).
    \end{align*}
    It follows that
    \begin{align}
        &\left\|\sum_{i=1}^k\left(\lambda_i(x)-\frac{t}{-h_i(x,y_t^\ast(x))}\right) \nabla_y h_i(x,y_t^\ast(x))\right\|\nonumber\\
        \leq&\left\|\nabla_y g(x,y^\ast(x))-\nabla_y g(x,y^\ast_t(x))\right\|+\left\|\sum_{i=1}^k\lambda_i(x) \left(\nabla_y h_i(x,y^\ast(x))-\nabla_y h_i(x,y_t^\ast(x))\right)\right\|\nonumber\\
        \leq&\left\|\nabla_y g(x,y^\ast(x))-\nabla_y g(x,y^\ast_t(x))\right\|+\sum_{i=1}^k\lambda_i(x) \left\|\nabla_y h_i(x,y^\ast(x))-\nabla_y h_i(x,y_t^\ast(x))\right\|\label{ineq:t/-hconverges}\\
        \overset{\text{(i)}}{=}&0,\nonumber
    \end{align} 
    where (i) comes from the following equations due to the linearity of $g(x,y)$ and $h_i(x,y)$ for any $i$
    \begin{align*}
        \left\|\nabla_y g(x,y^\ast(x))-\nabla_y g(x,y^\ast_t(x))\right\|&=0,\\
        \left\|\nabla_y h_i(x,y^\ast(x))-\nabla_y h_i(x,y_t^\ast(x))\right\|&=0.
    \end{align*}
    Next we prove that $t/(-h_i(x,y^\ast_t(x)))$ converges to $\lambda_i(x)$ for any $i$, i.e. for any $\epsilon>0$, there exists $T$ such that for any $0<t\leq T$, we have
    \begin{align*}
        \left|\lambda_i(x)-\frac{t}{-h_i(x,y_t^\ast(x))}\right|\leq\epsilon.
    \end{align*}
    Without loss of generality, we prove the case $i=1$. If this claim were false, there exists $\epsilon_0>0$ such that for any $T>0$, we can find $0<t\leq T$ satisfies
    \begin{align*}
        \left|\lambda_1(x)-\frac{t}{-h_1(x,y_t^\ast(x))}\right|\geq\epsilon_0.
    \end{align*}
    This implies we can find a sequence $t_j$, such that 
    \begin{align*}
        \left|\lambda_1(x)-\frac{t_j}{-h_1(x,y_{t_j}^\ast(x))}\right|\geq\epsilon_0.
    \end{align*}
    for any $j$ and $\lim_{j\to\infty}t_j=0$. Set
    \begin{align*}
        \mu^i_j=\lambda_i(x)-\frac{t_j}{-h_i(x,y_{t_j}^\ast(x))}
    \end{align*}
    and let $i^\ast_j=\arg\max_{i}|\mu^i_j|$, we can choose a subsequence of $\{t_j\}_{j=1}^\infty$, still denoted as $\{t_j\}_{j=1}^\infty$, such that $i^\ast_j=i^\ast$ for any $j$, where index $i^\ast$ is a fixed index. It holds that $|\mu^{i^\ast}_j|\geq|\mu^{1}_j|\geq\epsilon_0$. Moreover, we have $|\mu^{i}_j|/|\mu^{i^\ast}_j|\leq 1$ for any index $i$. Therefore, we can select a subsequence, still denoted as $\{t_j\}_{j=1}^\infty$, such that $\lim_{j\to\infty}|\mu^{i}_j|/|\mu^{i^\ast}_j|=:\hat{\mu}^i$ exists for any $i$. By (\ref{ineq:t/-hconverges}), we obtain
    \begin{align}\label{eq:linearkkt}
        \left\|\sum_{i=1}^k\left(\lambda_i(x)-\frac{t_j}{-h_i(x,y_{t_j}^\ast(x))}\right) \nabla_y h_i(x,y_{t_j}^\ast(x))\right\|=0
    \end{align}
    Let $j$ go to infinity, we obtain
    \begin{align*}
        \sum_{i=1}^k\hat{\mu}^i \nabla_y h_i(x_0,y^\ast(x_0))=0.
    \end{align*}
    Recall (\ref{eq:lineargap}) says
    \begin{align*}
        \|y-y^\ast(x)\|\leq\frac{1}{\tau(x)\|\nabla_y g(x,y^\ast(x))\|}\left(g(x,y)-g(x,y^\ast(x))\right).
    \end{align*}
    According to Lemma \ref{lem:optimalitygapforpoint}, which shows $g(x,y^\ast_{t})-g(x,y^\ast(x))\leq kt$, we have 
    \begin{align}\label{eq:finiteterm}
        \|y_{t_j}^\ast(x)-y^\ast(x)\|\leq\frac{1}{\tau(x)\|\nabla_y g(x,y^\ast(x))\|}\left(g(x,y_{t_j}^\ast(x))-g(x,y^\ast(x))\right)\leq\frac{kt_j}{\tau(x)\|\nabla_y g(x,y^\ast(x))\|},
    \end{align}
    where $\tau(x)$ is the lower bound of the cosine of the angle between $y-y^\ast(x)$ and the $\nabla_y g(x,y^\ast(x))$. The right hands of (\ref{eq:finiteterm}) is finite because when $y^\ast(x)$ is unique, $\tau(x)\|\nabla_y g(x,y^\ast(x))\|\ne 0$. Therefore, $\lim_{j\to\infty}y^\ast_{t_j}(x)=y^\ast(x)$. This implies 
    $$|\hat{\mu}^i|=\lim_{j\to\infty}\frac{|\mu^{i}_j|}{|\mu^{i^\ast}_j|}\leq\frac{|\lambda_i(x)-\frac{0}{-h_i(x,y^\ast(x))}|}{\epsilon_0}=0$$
    for any $i\notin\mathcal{I}^\ast(x_0)$. Dividing both sides of (\ref{eq:linearkkt}) with $\mu_j^{i^\ast}$ and let $j$ go to infinity, we have
    \begin{align*}
        0&=\lim_{j\to\infty}\frac{\left\|\sum_{i=1}^k\left(\lambda_i(x)-\frac{t_j}{-h_i(x,y_{t_j}^\ast(x))}\right) \nabla_y h_i(x,y_{t_j}^\ast(x))\right\|}{\mu_j^{i^\ast}}\\
        &=\sum_{i=1}^k\hat{\mu}^i \nabla_y h_i(x,y^\ast(x))\\
        &\overset{\text{(i)}}{=}\sum_{i\in\mathcal{I}^\ast(x_0)}\hat{\mu}^i \nabla_y h_i(x,y^\ast(x))
    \end{align*}
    In (i) we utilize the face that $\hat{\mu}^i=0$ for any $i\notin\mathcal{I}^\ast(x_0)$. Since $\hat{\mu}^{i^\ast}=\lim_{j\to\infty}|\mu^{i^\ast}_j|/|\mu^{i^\ast}_j|=1$, we also have $i^\ast\in\mathcal{I}^\ast(x)$. However, by the LICQ assumption, $\nabla_y h_i(x,y^\ast(x))$ are linear independent for $i\in\mathcal{I}^\ast(x_0)$, which means $\hat{\mu}^i$ for any $i\in\mathcal{I}^\ast(x_0)$. Here is a contradiction, so $t/(-h_i(x,y^\ast_t(x)))$ converges to $\lambda_i(x)$ for any $i$.
\end{proof}

\begin{lemma}\label{lem:linearcase}
    Suppose Assumption \ref{assumption:general1} and Assumption \ref{assumption:linear} hold. If $x$ is an \textbf{SCSC} point, then there are exact $m$ constraints active and the following matrix is invertible:
    \begin{align*}
        \sum_{i=1}^k\lambda^2_i(x)\nabla_y h_i(x,y^\ast(x))\left(\nabla_y h_i(x,y^\ast(x))\right)^\top.
    \end{align*}
\end{lemma}

\begin{proof}
  By Proposition \ref{prop:relation}, $x$ is an \textbf{SCSC} point means $y^\ast(x)$ is unique. Since the lower-level problem is a linear program, $y^\ast(x)$ is a vertex of a polyhedron, so there are at least $m$ constraints active. According to LICQ assumption, there are at most $m$ constraints active. Without loss of generality, assume $h_i(x,y^\ast(x))$ is active for $i=1,...,m$. We also know $\nabla_y h_i(x,y^\ast(x))$ are linear independent for $i=1,...,m$ by LICQ assumption. Therefore, the following matrix
    \begin{align*}
        \sum_{i=1}^k\lambda^2_i(x)\nabla_y h_i(x,y^\ast(x))\left(\nabla_y h_i(x,y^\ast(x))\right)^\top=\sum_{i=1}^m\lambda^2_i(x)\nabla_y h_i(x,y^\ast(x))\left(\nabla_y h_i(x,y^\ast(x))\right)^\top.
    \end{align*}
     is invertible.
\end{proof}

Next, we prove the relation of Jacobians in linear case when $t$ approaches $0$.

\noindent\textbf{Proof of Theorem \ref{thm:mainlinear}:} We fix an \textbf{SCSC} point $x_0$ and prove this theorem at this point. We intend to present our proof in two steps. First, we will compute the limit of $\nabla_x y^\ast_t(x_0)$ when $t$ goes to $0$. Second, we will compute $\nabla_x y^\ast(x_0)$ directly, and then show that they are equal.\\
\noindent\textbf{Step 1.} In this step, we will use the standard process to compute the Jacobian $\nabla_x y^\ast_t(x_0)$ as follows
\begin{align}
    \nabla_x y^\ast_t(x_0)&=(\nabla_{yy}\widetilde{g}_t(x_0,y^\ast_t(x_0)))^{-1}\nabla_{yx}\widetilde{g}_t(x_0,y^\ast_t(x_0))\label{eq:Jacobianofbarrierreformulation}\\
    &=(t\nabla_{yy}\widetilde{g}_t(x_0,y^\ast_t(x_0)))^{-1}t\nabla_{yx}\widetilde{g}_t(x_0,y^\ast_t(x_0))\nonumber.
\end{align}
where $0<t\leq T$ with some constant $T$. Next we will leverage Lemma \ref{lem_lowerbound} and Lemma \ref{lem:linearcase} to compute the limit 
$$\lim_{t\to 0}(t\nabla_{yy}\widetilde{g}_t(x_0,y^\ast_t(x_0)))^{-1}t\nabla_{yx}\widetilde{g}_t(x_0,y^\ast_t(x_0)).$$

Since $\widetilde{g}_t(x,y)$ is strongly convex in $y$ for any $x$, by the optimal condition of the barrier reformulated lower-level problem, we have
$$
    \nabla_y \widetilde{g}_t(x,y^\ast_t(x))=0\quad\forall x.
$$
Note that $\widetilde{g}_t$ is two times continuously differentiable by Assumption \ref{assumption:general1}(\ref{assumption:general1(1)}), to get the expression of $\nabla_x y^\ast_t(x_0)$, we take the gradient with respect to $x$ at point $x_0$ on both sides, yielding
    \begin{align*}
    0=&\nabla_{xy} \widetilde{g}_t(x_0,y^\ast_t(x_0))\\
    \overset{\text{(i)}}{=}&\nabla_x\left(\nabla_y g(x_0,y^\ast_t(x_0))+\sum_{i=1}^k\frac{t\nabla_y h_i(x_0,y_t^\ast(x_0))}{-h_i(x_0,y_t^\ast(x_0))}\right)\\
    \overset{\text{(ii)}}{=}&\nabla^2_{xy} g(x_0,y^\ast_t(x_0))+\left(\nabla_x y_t^\ast(x_0)\right)^\top\nabla^2_{yy} g(x_0,y^\ast_t(x_0))\\
    &+\sum_{i=1}^k\frac{t\nabla_x h_i(x_0,y_t^\ast(x_0))\left(\nabla_y h_i(x_0,y_t^\ast(x_0))\right)^\top}{h^2_i(x_0,y_t^\ast(x_0))}+\left(\nabla_x y_t^\ast(x_0)\right)^\top\sum_{i=1}^k\frac{t\nabla_y h_i(x_0,y_t^\ast(x_0))\left(\nabla_y h_i(x_0,y_t^\ast(x_0))\right)^\top}{h^2_i(x_0,y_t^\ast(x_0))} \\
    &+\sum_{i=1}^k\frac{t\nabla^2_{xy} h_i(x_0,y_t^\ast(x_0))}{-h_i(x_0,y_t^\ast(x_0))} +\left(\nabla_x y_t^\ast(x_0)\right)^\top\sum_{i=1}^k\frac{t\nabla^2_{yy} h_i(x_0,y_t^\ast(x_0))}{-h_i(x_0,y_t^\ast(x_0))},
    \end{align*}
    where (i) is by the definition of $\widetilde{g}_t(x,y)$ and (ii) is by direct computation. Rearranging terms, we obtain 
    \begin{align}\label{eq:jacobianofbarrierreformulation}
        &\left(\nabla^2_{yy} g(x_0,y^\ast_t(x_0))+\sum_{i=1}^k\frac{t\nabla_y h_i(x_0,y_t^\ast(x_0))\left(\nabla_y h_i(x_0,y_t^\ast(x_0))\right)^\top}{h^2_i(x_0,y_t^\ast(x_0))}+\sum_{i=1}^k\frac{t\nabla^2_{yy} h_i(x_0,y_t^\ast(x_0))}{-h_i(x_0,y_t^\ast(x_0))}\right)\nabla_x y_t^\ast(x_0)\nonumber\\
        =&-\left(\nabla^2_{yx} g(x_0,y^\ast_t(x_0))+\sum_{i=1}^k\frac{t\nabla_y h_i(x_0,y_t^\ast(x_0))\left(\nabla_x h_i(x_0,y_t^\ast(x_0))\right)^\top}{h^2_i(x_0,y_t^\ast(x_0))}+\sum_{i=1}^k\frac{t\nabla^2_{yx} h_i(x_0,y_t^\ast(x_0))}{-h_i(x_0,y_t^\ast(x_0))}\right).
    \end{align}
    Lemma \ref{lem_lowerbound} and Lemma \ref{lem:optimalitygapforpoint} tell us $\lim_{t\to 0}t/(-h_i(x_0,y^\ast_t(x_0)))=\lambda_i(x_0)$ for any $i$ and $\lim_{t\to 0}y^\ast_t(x_0)=y^\ast(x_0)$. Thus
    \begin{align}
        &\lim_{t\to 0}\left(t\nabla^2_{yy} g(x_0,y^\ast_t(x_0))+\sum_{i=1}^k\frac{t^2\nabla_y h_i(x_0,y_t^\ast(x_0))\left(\nabla_y h_i(x_0,y_t^\ast(x_0))\right)^\top}{h^2_i(x_0,y_t^\ast(x_0))}+t\sum_{i=1}^k\frac{t\nabla^2_{yy} h_i(x_0,y_t^\ast(x_0))}{-h_i(x_0,y_t^\ast(x_0))}\right)\nonumber\\
        \overset{\text{(i)}}{=}&\sum_{i=1}^k\lambda^2_i(x_0)\nabla_y h_i(x_0,y^\ast(x_0))\left(\nabla_y h_i(x_0,y^\ast(x_0))\right)^\top,\label{eq:nablaynablay}
        \end{align}
        and
        \begin{align}
        &\lim_{t\to 0}\left(t\nabla^2_{yx} g(x_0,y^\ast_t(x_0))+\sum_{i=1}^k\frac{t^2\nabla_y h_i(x_0,y_t^\ast(x_0))\left(\nabla_x h_i(x_0,y_t^\ast(x_0))\right)^\top}{h^2_i(x_0,y_t^\ast(x_0))}+t\sum_{i=1}^k\frac{t\nabla^2_{yx} h_i(x_0,y_t^\ast(x_0))}{-h_i(x_0,y_t^\ast(x_0))}\right)\nonumber\\
        \overset{\text{(ii)}}{=}&\sum_{i=1}^k\lambda^2_i(x_0)\nabla_y h_i(x_0,y^\ast(x_0))\left(\nabla_x h_i(x_0,y^\ast(x_0))\right)^\top,\label{eq:nablaynablax}
    \end{align}
    where (i), (ii) are because
    \begin{align*}
        \lim_{t\to 0}t\nabla_{yy}^2 g(x_0,y_t^\ast(x_0))&=\lim_{t\to 0}t\cdot\nabla_{yy}^2 g(x_0,y^\ast(x_0))=0,\\
        \lim_{t\to 0}t\nabla_{yx}^2 g(x_0,y_t^\ast(x_0))&=\lim_{t\to 0}t\cdot\nabla_{yx}^2 g(x_0,y^\ast(x_0))=0,\\
        \lim_{t\to 0}t\frac{t\nabla^2_{yy} h_i(x_0,y_t^\ast(x_0))}{-h_i(x_0,y_t^\ast(x_0))}&=\lim_{t\to 0}t\cdot\lambda_i(x_0)\nabla^2_{yy} h_i(x_0,y^\ast(x_0))=0,\\
        \lim_{t\to 0}t\frac{t\nabla^2_{yx} h_i(x_0,y_t^\ast(x_0))}{-h_i(x_0,y_t^\ast(x_0))}&=\lim_{t\to 0}t\cdot\lambda_i(x_0)\nabla^2_{yx} h_i(x_0,y^\ast(x_0))=0.
    \end{align*}
    By Lemma \ref{lem:linearcase}, the following matrix is invertible
    \begin{align*}
        \sum_{i=1}^k\lambda^2_i(x_0)\nabla_y h_i(x_0,y^\ast(x_0))\left(\nabla_y h_i(x_0,y^\ast(x_0))\right)^\top.
    \end{align*}
    Therefore, again applying
    $$\lim_{t\to 0}\sum_{i=1}^k\frac{t^2\nabla_y h_i(x_0,y_t^\ast(x_0))\left(\nabla_y h_i(x_0,y_t^\ast(x_0))\right)^\top}{h^2_i(x_0,y_t^\ast(x_0))}=\sum_{i=1}^k\lambda^2_i(x_0)\nabla_y h_i(x_0,y^\ast(x_0))\left(\nabla_y h_i(x_0,y^\ast(x_0))\right)^\top,$$
    we can find a constant $T$ small enough such that 
    \begin{align*}
        \sum_{i=1}^k\frac{t^2\nabla_y h_i(x_0,y_t^\ast(x_0))\left(\nabla_y h_i(x_0,y_t^\ast(x_0))\right)^\top}{h^2_i(x_0,y_t^\ast(x_0))}
    \end{align*}
    is invertible for any $0<t\leq T$. This implies the invertibility of
    \begin{align*}
        t\nabla^2_{yy} g(x_0,y^\ast_t(x_0))+\sum_{i=1}^k\frac{t^2\nabla_y h_i(x_0,y_t^\ast(x_0))\left(\nabla_y h_i(x_0,y_t^\ast(x_0))\right)^\top}{h^2_i(x_0,y_t^\ast(x_0))}+t\sum_{i=1}^k\frac{t\nabla^2_{yy} h_i(x_0,y_t^\ast(x_0))}{-h_i(x_0,y_t^\ast(x_0))}.
    \end{align*}
    for any $0<t\leq T$. By (\ref{eq:jacobianofbarrierreformulation}), we derive the expression of $\nabla_x y_t^\ast(x_0)$ for any $0<t\leq T$

    \resizebox{\textwidth}{!}{
    \begin{minipage}{\textwidth}
    \begin{align*}
        \nabla_x y_t^\ast(x_0)=&-\left[t\nabla^2_{yy} g(x_0,y^\ast_t(x_0))+\sum_{i=1}^k\frac{t^2\nabla_y h_i(x_0,y_t^\ast(x_0))\left(\nabla_y h_i(x_0,y_t^\ast(x_0))\right)^\top}{h^2_i(x_0,y_t^\ast(x_0))}+t\sum_{i=1}^k\frac{t\nabla^2_{yy} h_i(x_0,y_t^\ast(x_0))}{-h_i(x_0,y_t^\ast(x_0))}\right]^{-1}\\
        &\times\left[t\nabla^2_{yx} g(x_0,y^\ast_t(x_0))+\sum_{i=1}^k\frac{t^2\nabla_y h_i(x_0,y_t^\ast(x_0))\left(\nabla_x h_i(x_0,y_t^\ast(x_0))\right)^\top}{h^2_i(x_0,y_t^\ast(x_0))}+t\sum_{i=1}^k\frac{t\nabla^2_{yx} h_i(x_0,y_t^\ast(x_0))}{-h_i(x_0,y_t^\ast(x_0))}\right]
    \end{align*}
    \end{minipage}
    }
    
From (\ref{eq:nablaynablay}) and (\ref{eq:nablaynablax}), we have

    \resizebox{\textwidth}{!}{
    \begin{minipage}{\textwidth}
    \begin{align*}
        \lim_{t\to 0}\nabla_x y_t^\ast(x_0)=&\lim_{t\to 0}-\left(\sum_{i=1}^k\left(\frac{t}{-h_i(x_0,y_t^\ast(x_0))}\right)^2\nabla_y h_i(x_0,y_t^\ast(x_0))\left(\nabla_y h_i(x_0,y_t^\ast(x_0))\right)^\top\right)^{-1}\\
        &\times\left(\sum_{i=1}^k\left(\frac{t}{-h_i(x_0,y_t^\ast(x_0))}\right)^2\nabla_y h_i(x_0,y_t^\ast(x_0))\left(\nabla_x h_i(x_0,y_t^\ast(x_0))\right)^\top\right)\\
        =&-\left(\sum_{i=1}^k\lambda^2_i(x_0)\nabla_y h_i(x_0,y_t^\ast(x_0))\left(\nabla_y h_i(x_0,y_t^\ast(x_0))\right)^\top\right)^{-1}\\
        &\times\left(\sum_{i=1}^k\lambda^2_i(x_0)\nabla_y h_i(x_0,y_t^\ast(x_0))\left(\nabla_x h_i(x_0,y_t^\ast(x_0))\right)^\top\right)
    \end{align*}
    \end{minipage}
    }

\noindent\textbf{Step 2.} In this step, we will establish that
    \begin{align*}
        \nabla_x y^\ast(x_0)=&-\left(\sum_{i=1}^k\lambda^2_i(x_0)\nabla_y h_i(x_0,y^\ast(x_0))\left(\nabla_y h_i(x_0,y^\ast(x_0))\right)^\top\right)^{-1}\\
        &\times\left(\sum_{i=1}^k\lambda^2_i(x_0)\nabla_y h_i(x_0,y^\ast(x_0))\left(\nabla_x h_i(x_0,y^\ast(x_0))\right)^\top\right)\\
        =&-\left(\sum_{i=1}^k\lambda_i(x_0)\nabla_y h_i(x_0,y^\ast(x_0))\left(\lambda_i(x_0)\nabla_y h_i(x_0,y^\ast(x_0))\right)^\top\right)^{-1}\\
        &\times\left(\sum_{i=1}^k\lambda_i(x_0)\nabla_y h_i(x_0,y^\ast(x_0))\left(\lambda_i(x_0)\nabla_x h_i(x_0,y^\ast(x_0))\right)^\top\right)
    \end{align*}
By the continuity of $h_i(x,y)$ and Lemma \ref{lem:con_of_multi}, i.e. the continuity of $\lambda_i(x)$, we can select a neighborhood $U$ of $x_0$ such that the active index set for any point in $U$ is the same as that of $x_0$. According to Proposition \ref{prop:relation}, $x$ is an \textbf{SCSC} point, which means $y^\ast(x)$ is unique. By Lemma \ref{lem:linearcase}, there are exact $m$ constrains active at point $x_0$, we assume $h_i(x_0,y^\ast(x_0))=0$ for $i=1,...,m$, and $i=m+1,...,k$ are inactive. Therefore, $y^\ast(x)$ is the solution of the following $m$ equations
\begin{align*}
    \begin{pmatrix}
    h_1(x,y)\\
    \vdots\\
    h_m(x,y)
\end{pmatrix}=0.
\end{align*} 
for any $x$ in $U$. Since $\lambda_i(x_0)>0$ for any active index at point $x_0$ with $i=1,...,m$, these equations are equivalent to
\begin{align*}    
\begin{pmatrix}
    \lambda_1(x_0) h_1(x,y)\\
    \vdots\\
    \lambda_m(x_0) h_m(x,y)
\end{pmatrix}=0.
\end{align*}
Replace $y$ by $y^\ast(x)$ in these equations, we have 
\begin{align*}    
\begin{pmatrix}
    \lambda_1(x_0) h_1(x,y^\ast(x))\\
    \vdots\\
    \lambda_m(x_0) h_m(x,y^\ast(x))
\end{pmatrix}=0.
\end{align*}
These hold for any $x$ in $U$. Take gradient for $x$ on both sides, we obtain
\begin{align*}
    \begin{pmatrix}
    \lambda_1(x_0) \nabla_x^\top h_1(x,y^\ast(x))\\
    \vdots\\
    \lambda_m(x_0) \nabla_x^\top h_m(x,y^\ast(x))
\end{pmatrix}
+
\begin{pmatrix}
    \lambda_1(x_0) \nabla_y^\top h_1(x,y^\ast(x))\\
    \vdots\\
    \lambda_m(x_0) \nabla_y^\top h_m(x,y^\ast(x))
\end{pmatrix}\nabla_x y^\ast(x)=0
\end{align*}
for any $x$ in $U$. Multiply the matrix $\left(
    \lambda_1(x_0) \nabla_y h_1(x,y^\ast(x)),\cdots,\lambda_m(x_0) \nabla_y h_m(x,y^\ast(x))
\right)$ to both sides and moving terms. Note that the index $i=m+1,...,k$ are inactive at point $x_0$, then $\lambda_i(x_0)=0$ for $i=m+1,...,k$. We finalize
\begin{align}
        \nabla_x y^\ast(x_0)=&-\left(\sum_{i=1}^k\lambda_i(x_0)\nabla_y h_i(x_0,y^\ast(x_0))\left(\lambda_i(x_0)\nabla_y h_i(x_0,y^\ast(x_0))\right)^\top\right)^{-1}\nonumber\\
        &\times\left(\sum_{i=1}^k\lambda_i(x_0)\nabla_y h_i(x_0,y^\ast(x_0))\left(\lambda_i(x_0)\nabla_x h_i(x_0,y^\ast(x_0))\right)^\top\right)\label{eq:localexpression}\\
        =&\lim_{t\to 0}\nabla_x y^\ast_t(x_0).\nonumber
    \end{align}
    This completes the proof.

\subsection{Proof of Theorem \ref{thm:mainnonlinear}}\label{proof:thm:mainnonlinear}

\begin{lemma}\label{lem_lowerbound2}
   Suppose Assumption \ref{assumption:general1} and \ref{assumption:nonlinear} hold, then we have $t/(-h_i(x,y^\ast_t(x)))$ converges to $\lambda_i(x)$ uniformly for any $i$.
\end{lemma}

\begin{proof}
    The KKT condition of original problem is
    \begin{align}\label{eq:kkt5}
        \nabla_y g(x,y^\ast(x))+\sum_{i=1}^k\lambda_i(x) \nabla_y h_i(x,y^\ast(x))=0.
    \end{align} 
    By subtracting (\ref{eq:kkt5}) from optimal condition equation of barrier reformulation, given below,
    \begin{align*}
        \nabla_y g(x,y^\ast_t(x))+\sum_{i=1}^k\frac{t}{-h_i(x,y_t^\ast(x))} \nabla_y h_i(x,y_t^\ast(x))=0,
    \end{align*} 
    we derive the following expression
    \begin{align*}
        \nabla_y g(x,y^\ast(x))-\nabla_y g(x,y^\ast_t(x))+\sum_{i=1}^k\left(\lambda_i(x) \nabla_y h_i(x,y^\ast(x))-\frac{t}{-h_i(x,y_t^\ast(x))} \nabla_y h_i(x,y_t^\ast(x))\right)=0.
    \end{align*}
    It follows that
    \begin{align}
        0=&\nabla_y g(x,y^\ast(x))-\nabla_y g(x,y^\ast_t(x))\nonumber\\
        &+\sum_{i=1}^k\left(\lambda_i(x) \nabla_y h_i(x,y^\ast(x))-\lambda_i(x) \nabla_y h_i(x,y_t^\ast(x))\right)\nonumber\\
        &+\sum_{i=1}^k\left(\lambda_i(x)\nabla_y h_i(x,y_t^\ast(x))-\frac{t}{-h_i(x,y_t^\ast(x))} \nabla_y h_i(x,y_t^\ast(x))\right)\nonumber\\
        =&\nabla_y g(x,y^\ast(x))-\nabla_y g(x,y^\ast_t(x))\nonumber\\
        &+\sum_{i=1}^k\lambda_i(x) \left(\nabla_y h_i(x,y^\ast(x))-\nabla_y h_i(x,y_t^\ast(x))\right)\nonumber\\
        &+\sum_{i=1}^k\left(\lambda_i(x)-\frac{t}{-h_i(x,y_t^\ast(x))}\right) \nabla_y h_i(x,y_t^\ast(x))\nonumber\\
        =&\left\|\sum_{i=1}^k\left(\lambda_i(x)-\frac{t}{-h_i(x,y_t^\ast(x))}\right) \nabla_y h_i(x,y_t^\ast(x))\right\|\nonumber\\
        \leq&\left\|\nabla_y g(x,y^\ast(x))-\nabla_y g(x,y^\ast_t(x))\right\|+\left\|\sum_{i=1}^k\lambda_i(x) \left(\nabla_y h_i(x,y^\ast(x))-\nabla_y h_i(x,y_t^\ast(x))\right)\right\|\nonumber\\
        \leq&\left\|\nabla_y g(x,y^\ast(x))-\nabla_y g(x,y^\ast_t(x))\right\|+\sum_{i=1}^k|\lambda_i(x)|\left\|\nabla_y h_i(x,y^\ast(x))-\nabla_y h_i(x,y_t^\ast(x))\right\|\label{ineq:t/-hconverges3}
    \end{align}
    
    Next we prove that $t/(-h_i(x,y^\ast_t(x)))$ converges to $\lambda_i(x)$ uniformly for any $i$, i.e. for any $\epsilon>0$, there exists $T$ such that for any $x\in\mathcal{X}$ and $0<t\leq T$, we have
    \begin{align*}
        \left|\lambda_i(x)-\frac{t}{-h_i(x,y_t^\ast(x))}\right|\leq\epsilon.
    \end{align*}
    Without loss of generality, we consider the case $i=1$. We prove this by contradiction. If there exists $\epsilon_0>0$ such that for any $T>0$, we can find $x\in\mathcal{X}$  and $0<t\leq T$ satisfies
    \begin{align*}
        \left|\lambda_1(x)-\frac{t}{-h_1(x,y_t^\ast(x))}\right|\geq\epsilon_0.
    \end{align*}
    This implies we can find a sequence $(x_j,t_j)$, such that the following holds for any $j$
    \begin{align*}
        \left|\lambda_1(x_j)-\frac{t_j}{-h_1(x_j,y_{t_j}^\ast(x_j))}\right|\geq\epsilon_0,
    \end{align*}
    and also we have $\lim_{j\to\infty}t_j=0$. Since $\mathcal{X}$ is compact, we can assume $\lim_{j\to\infty}x_j=x_0$. Set
    \begin{align*}
        \mu^i_j=\lambda_i(x_j)-\frac{t_j}{-h_i(x_j,y_{t_j}^\ast(x_j))}
    \end{align*}
    and let $i^\ast_j=\arg\max_{i}|\mu^i_j|$, we can choose a subsequence of $\{(x_j,t_j)\}_{j=1}^\infty$, still denoted as $\{(x_j,t_j)\}_{j=1}^\infty$, such that $i^\ast_j=i^\ast$ for any $j$, where index $i^\ast$ is a fixed index. It holds that $|\mu^{i^\ast}_j|\geq|\mu^{1}_j|\geq\epsilon_0$. Moreover, we have $|\mu^{i}_j|/|\mu^{i^\ast}_j|\leq 1$ for any index $i$. Therefore, we can select a subsequence, still denoted as $\{(x_j,t_j)\}_{j=1}^\infty$, such that $\lim_{j\to\infty}|\mu^{i}_j|/|\mu^{i^\ast}_j|=:\hat{\mu}^i$ exists for any $i$. Apply (\ref{ineq:t/-hconverges3}) to $(x_j,t_j)$, we obtain
    \begin{align}
        &\left\|\sum_{i=1}^k\left(\lambda_i(x_j)-\frac{t_j}{-h_i(x_j,y_{t_j}^\ast(x_j))}\right) \nabla_y h_i(x_j,y_{t_j}^\ast(x_j))\right\|\nonumber\\
        \leq&\left\|\nabla_y g(x_j,y^\ast(x_j))-\nabla_y g(x_j,y^\ast_{t_j}(x_j))\right\|+\sum_{i=1}^k|\lambda_i(x_j)| \left\|\nabla_y h_i(x_j,y^\ast(x_j))-\nabla_y h_i(x_j,y_{t_j}^\ast(x_j))\right\|\label{ineq:t/-hconverges2}
    \end{align}
    It is worth noting that 
    \begin{align}
        \lim_{j\to\infty}(\nabla_y h_i(x_j,y_{t_j}^\ast(x_j))-\nabla_y h_i(x_0,y^\ast(x_0)))&=0\label{eq:convergenceofgradient1}\\
        \lim_{j\to\infty}(\nabla_y g(x_j,y_{t_j}^\ast(x_j))-\nabla_y g(x_0,y^\ast(x_0)))&=0\label{eq:convergenceofgradient2}.
    \end{align}
    This is because, by Lemma \ref{lem:optimalitygapforpoint}, $y_{t}^\ast(x)$ uniformly converges to $y^\ast(x)$. Under strongly convex setting, we also know $y^\ast(x)$ continuously depends on $x$. According to Lemma \ref{lem:limitofsequence}, $y_{t_j}^\ast(x_j)$ converges to $y^\ast(x_0)$, which implies $(x_j,y_{t_j}^\ast(x_j))$ converges to $(x_0,y^\ast(x_0))$.

    Note that $\lambda_i(x_j)$ is uniform bounded by Lemma \ref{lem:con_of_multi}, the right-hand side of (\ref{ineq:t/-hconverges2}) converges to $0$. Dividing both sides of (\ref{ineq:t/-hconverges2}) with $\mu_j^\ast$ and let $j$ go to infinity, it follows that
    \begin{align*}
        \sum_{i=1}^k\hat{\mu}^i \nabla_y h_i(x_0,y^\ast(x_0))=0.
    \end{align*}
    Applying Lemma \ref{lem:optimalitygapforpoint} to $(x_j,t_j)$, we have
    $$\|y^\ast_{t_j}(x_j)-y^\ast(x_j)\|\leq\sqrt{\frac{2}{\mu_g}kt_j}.$$
    This means $\lim_{j\to\infty}y^\ast_{t_j}(x_j)=y^\ast(x_0)$. Therefore, the following holds
    $$|\hat{\mu}^i|=\lim_{j\to\infty}\frac{|\mu^{i}_j|}{|\mu^{i^\ast}_j|}\leq\frac{|\lambda_i(x_0)-\frac{0}{-h_i(x_0,y^\ast(x_0))}|}{\epsilon_0}=0.$$
    Thus, we have
    \begin{align*}
        \sum_{i\in\mathcal{I}^\ast(x_0)}\hat{\mu}^i \nabla_y h_i(x_0,y^\ast(x_0))=0.
    \end{align*}
    Since $\hat{\mu}^{i^\ast}=\lim_{j\to\infty}|\mu^{i^\ast}_j|/|\mu^{i^\ast}_j|=1$, we also have $i\in\mathcal{I}^\ast(x_0)$. However, by the LICQ assumption, $\nabla_y h_i(x_0,y^\ast(x_0))$ are linear independent for $i\in\mathcal{I}^\ast(x_0)$, which means $\hat{\mu}^i=0$ for any $i\in\mathcal{I}^\ast(x_0)$. Here is a contradiction, so $t/(-h_i(x,y^\ast_t(x)))$ converges to $\lambda_i(x)$ uniformly for any $i$.
    
\end{proof}

\begin{lemma}\label{lem_existence}
    Suppose Assumption \ref{assumption:general1} and Assumption \ref{assumption:nonlinear} hold, and $x$ is an \textbf{SCSC} point. Then $\lim_{t\to 0}\nabla_x y_t^\ast(x)$ exists.
\end{lemma}

\begin{proof}
    Denote $\mathcal{I}^\ast(x)$ be the set of active index at $y^\ast(x)$. Since LICQ assumption is satisfied at $y^\ast(x)$, $\{\nabla_y h_{i}(x,y^\ast(x))\}_{i\in\mathcal{I}^\ast(x)}$ are linearly independent. By Lemma \ref{lem:optimalitygapforpoint} and Lemma \ref{lem:continuityofyastt}, $y^\ast_t(x)$ continuously converges to $y^\ast(x)$, so we can assume that $t$ is sufficiently small such that $\{\nabla_y h_{i}(x,y^\ast_t(x))\}_{i\in\mathcal{I}^\ast(x)}$ are linearly independent. Following (\ref{eq:Jacobianofbarrierreformulation}) in the proof of Theorem \ref{thm:mainlinear}, we derive

    \resizebox{\textwidth}{!}{
    \begin{minipage}{\textwidth}
    \begin{align*}
        \nabla_x y_t^\ast(x)=&-\left(\nabla^2_{yy} g(x,y^\ast_t(x))+\sum_{i=1}^k\frac{1}{t}\frac{t^2}{h^2_i(x,y_t^\ast(x))}\nabla_y h_i(x,y_t^\ast(x))\left(\nabla_y h_i(x,y_t^\ast(x))\right)^\top+\sum_{i=1}^k\frac{t\nabla^2_{yy} h_i(x,y_t^\ast(x))}{-h_i(x,y_t^\ast(x))}\right)^{-1}\\
        &\times\left(\nabla^2_{yx} g(x,y^\ast_t(x))+\sum_{i=1}^k\frac{1}{t}\frac{t^2}{h^2_i(x,y_t^\ast(x))}\nabla_y h_i(x,y_t^\ast(x))\left(\nabla_x h_i(x,y_t^\ast(x))\right)^\top+\sum_{i=1}^k\frac{t\nabla^2_{yx} h_i(x,y_t^\ast(x))}{-h_i(x,y_t^\ast(x))}\right)
    \end{align*}
    \end{minipage}}
    
    Unlike the linear case, handling the strongly convex case becomes extremely challenging. The reason is that in the linear case, the limit of a few terms in the above expression (e.g., $A_t^k(x)$ and $D_t(x)$ to be defined later) are invertible, but this is no longer the case without linearity. Therefore, it is necessary to develop a new set of tools.

    Denote the following terms as
    \begin{align}
        v^j_t(x)&=\frac{t}{-h_j(x,y_t^\ast(x))}\nabla_y h_j(x,y_t^\ast(x))\label{notation:vjt}\\
        v^j(x)&=\lambda_{j}(x)\nabla_y h_j(x,y^\ast(x))\label{notation:vj}\\
        V^j_t(x)&=\mathrm{span}\{\frac{t}{-h_i(x,y^\ast_t(x))}\nabla_y h_i(x,y_t^\ast(x)):i\in\{1,...,j\}\cap \mathcal{I}^\ast(x)\}\label{notation:space}\\
        V^{j}&=\mathrm{span}\{\lambda_i\nabla_y h_i(x,y^\ast(x)):i\in\{1,...,j\}\cap\mathcal{I}^\ast(x)\}\nonumber\\
        A_t^j(x)&=\sum_{i=1}^j\frac{1}{t}(v^j_t(x))(v^j_t(x))^\top\nonumber\\
        &=\sum_{i=1}^j\frac{1}{t}\frac{t^2}{h^2_i(x,y_t^\ast(x))}\nabla_y h_i(x,y_t^\ast(x))\left(\nabla_y h_i(x,y_t^\ast(x))\right)^\top,\ j=1,...,k\nonumber\\
        B_t(x)&=\nabla^2_{yy} g(x,y^\ast_t(x))+\sum_{i=1}^k\frac{t\nabla^2_{yy} h_i(x,y_t^\ast(x))}{-h_i(x,y_t^\ast(x))}\nonumber\\
        B_t^j(x)&=A_t^j(x)+B_t(x)\label{notation:Btj}\\
        B(x)&=\nabla^2_{yy} g(x,y^\ast(x))+\sum_{i=1}^k\lambda_i(x)\nabla^2_{yy} h_i(x,y^\ast(x))\nonumber\\
        C_t(x)&=\nabla^2_{yx} g(x,y^\ast_t(x))+\sum_{i=1}^k\frac{t\nabla^2_{yx} h_i(x,y_t^\ast(x))}{-h_i(x,y_t^\ast(x))}\nonumber\\
        C(x)&=\nabla^2_{yx} g(x,y^\ast(x))+\sum_{i=1}^k\lambda_i(x)\nabla^2_{yx} h_i(x,y^\ast(x))\nonumber\\
        D_t(x)&=\sum_{i=1}^k\frac{1}{t}\frac{t^2}{h^2_i(x,y_t^\ast(x))}\nabla_y h_i(x,y_t^\ast(x))\left(\nabla_x h_i(x,y_t^\ast(x))\right)^\top,\label{notation:Dt}
    \end{align}
    where $\mathcal{I}^\ast(x)$ denotes the set of \textbf{active index} at the point $(x,y^\ast(x))$. 
    
    By the above notations, we have
    \begin{align}\label{eq:jacobiannotation}
        \nabla_x y_t^\ast(x)=\left(A^k_t(x)+B_t(x)\right)^{-1}\left(C_t(x)+D_t(x)\right).
    \end{align}

    We will demonstrate the convergence of the above two terms separately.
    
    \noindent\textbf{Step 1.} To prove that the following limit exists, 
    $$\lim_{t\to 0}\left(A^k_t(x)+B_t(x)\right)^{-1}C_t(x)$$ we will show $\lim_{t\to 0}C_t(x)$ exists and $\lim_{t\to 0}\left(A^k_t(x)+B_t(x)\right)^{-1}$ exists as a linear transformation. 
    
    By Lemma \ref{lem_lowerbound2}, i.e. $\lim_{t\to 0}t/(-h_i(x,y^\ast_t(x)))=\lambda_i(x)$, and Lemma \ref{lem:optimalitygapforpoint}, which shows $\lim_{t\to 0}y^\ast_t(x)=y^\ast(x)$, we have 
    \begin{align*}
        \lim_{t\to 0}C_t(x)&=\lim_{t\to 0}\nabla^2_{yx} g(x,y^\ast_t(x))+\sum_{i=1}^k\frac{t\nabla^2_{yx} h_i(x,y_t^\ast(x))}{-h_i(x,y_t^\ast(x))}\\
        &=\nabla^2_{yx} g(x,y^\ast(x))+\sum_{i=1}^k\lambda_i(x)\nabla^2_{yx} h_i(x,y^\ast(x)).
    \end{align*}
    
    Now we prove $\lim_{t\to 0}(A_t^k(x)+B_t(x))^{-1}$ exists, which is the key point of this proof. We proceed by induction. 
    
    If $j=1$ is an active index, note that
    $$A_t^j(x)=\sum_{i=1}^j\frac{1}{t}v^i_t(x)(v^i_t(x))^\top.$$
    Applying the Sherman-Morrison formula, we obtain
        \begin{align*}
            (A_t^1(x)+B_t(x))^{-1}&=\left(\frac{1}{t}v_t^1(x) \left(v_t^1(x)\right)^\top+B_t(x)\right)^{-1}\\
            &={B_t(x)}^{-1}-\frac{{B_t(x)}^{-1}v_t^1(x){v_t^1(x)}^\top {B_t(x)}^{-1}}{t^2+{v_t^1(x)}^{\top}{B_t(x)}^{-1}v_t^1(x)}.
        \end{align*}
        To prove $\lim_{t\to 0}\left(A^k_t(x)+B_t(x)\right)^{-1}$ exists, it is sufficient to show that $\lim_{t\to 0}B_t(x)$ and $\lim_{t\to 0}v_t^1(x)$ exist, and $\lim_{t\to 0}{v_t^1(x)}^{\top}{B_t(x)}^{-1}v_t^1(x)>0$. Applying Lemma \ref{lem_lowerbound2}, and Lemma \ref{lem:optimalitygapforpoint} again, we obtain 
        \begin{align*}
            \lim_{t\to 0}v^1_t(x)=\lim_{t\to 0}\frac{t}{-h_1(x,y^\ast_t(x))}\cdot\lim_{t\to 0}\nabla_y h_1(x,y^\ast_t(x))=\lambda_1(x)\nabla_y h_1(x,y^\ast(x))=v^1(x).
        \end{align*} 
        Since $B(x)\succeq\nabla^2_{yy}g(x,y^\ast(x))\succ0$, it follows that
        \begin{align*}
            \lim_{t\to 0}B_t(x)^{-1}=\left(\nabla^2_{yy} g(x,y^\ast(x))+\sum_{i=1}^k\lambda_i(x)\nabla^2_{yy} h_i(x,y^\ast(x))\right)^{-1}={B(x)}^{-1}\succ0,
        \end{align*}
        Under the LICQ assumption and given that we have assumed index $1$ is active, we have $v^1(x)\ne 0$. Therefore, $\lim_{t\to 0}{v_t^1(x)}^{\top}{B_t(x)}^{-1}v_t^1(x)={v^1}(x)^{\top}{B(x)}^{-1}v^1(x)>0$, implying
        \begin{align*}
            \lim_{t\to 0}(A_t^1(x)+B_t(x))^{-1}={B(x)}^{-1}-\frac{{B(x)}^{-1}v^1(x) {v^1(x)}^\top {B(x)}^{-1}}{{v^1}(x)^{\top}{B(x)}^{-1}v^1(x)}.
        \end{align*}
        If $j=1$ is an inactive index, we obtain
        \begin{align*}
            &\lim_{t\to 0}A^1_t(x)\\
            =&\lim_{t\to 0}\frac{1}{t}\frac{t^2}{h^2_1(x,y^\ast_t(x))}\nabla_yh_1(x,y^\ast_t(x))\left(\nabla_yh_1(x,y^\ast_t(x))\right)^\top\\
            =&\lim_{t\to 0}\frac{t}{h^2_1(x,y^\ast_t(x))}\nabla_yh_1(x,y^\ast_t(x))\left(\nabla_yh_1(x,y^\ast_t(x))\right)^\top\\
            =&0.
        \end{align*}
        Therefore,
        \begin{align*}
            \lim_{t\to 0}(A_t^1(x)+B_t(x))^{-1}={B(x)}^{-1}.
        \end{align*}
        
        Next, we prove the existence of $\lim_{t\to 0}(A_t^k(x)+B_t(x))^{-1}$ by induction. For $j=1$, we have proved the existence of $\lim_{t\to 0}(A^j_t(x)+B_t(x))^{-1}$. Suppose when $j=\widetilde{j}$, $\lim_{t\to 0}(A^{\widetilde{j}}_t(x)+B_t(x))^{-1}=:G^{\widetilde{j}}(x)$ exists, where $G^{\widetilde{j}}(x)$ is some quantity to be specified soon. We will prove that for $j=\widetilde{j}+1$, $\lim_{t\to 0}(A^{\widetilde{j}+1}_t(x)+B_t(x))^{-1}=\lim_{t\to 0}(B^{\widetilde{j}}_t(x))^{-1}=:G^{\widetilde{j}+1}(x)$ also exists, by considering two cases, depending on whether $\widetilde{j}+1$ is active or not.
        
        \noindent\textbf{Case 1: $j=\widetilde{j}+1$ is an active index.} Again using the Sherman-Morrison formula, we find
        \begin{align}
            (A_t^{\widetilde{j}+1}(x)+B_t(x))^{-1}&=(\frac{1}{t}v_t^{\widetilde{j}+1}(x){v_t^{\widetilde{j}+1}(x)}^\top+A_t^{\widetilde{j}}(x)+B_t(x))^{-1}\nonumber\\
            &\overset{\text{(i)}}{=}(\frac{1}{t}v_t^{\widetilde{j}+1}(x){v_t^{\widetilde{j}+1}(x)}^\top+B_t^{\widetilde{j}}(x))^{-1}\nonumber\\
            &={B_t^{\widetilde{j}}(x)}^{-1}-\frac{{B_t^{\widetilde{j}}(x)}^{-1}v_t^{\widetilde{j}+1}(x) {v_t^{\widetilde{j}+1}(x)}^\top {B_t^{\widetilde{j}}(x)}^{-1}}{t^2+{v_t^{\widetilde{j}+1}(x)}^\top{B_t^{\widetilde{j}}(x)}^{-1}v_t^{\widetilde{j}+1}(x)}. \label{eq:inversion}
        \end{align}
        (i) is because the notation in (\ref{notation:Btj}). Note that $\lim_{t\to 0}v_t^{\widetilde{j}+1}(x)=\lambda_{\widetilde{j}+1}(x)\nabla_y h_{\widetilde{j}+1}(x,y^\ast(x))$ exists, and by the induction hypothesis, $\lim_{t\to 0}{B_t^{\widetilde{j}}(x)}^{-1}=G^{\widetilde{j}}(x)$ exists. Then to show that the limit of \eqref{eq:inversion} exists, it is sufficient to prove that 
        \begin{align}\label{eq:positive}
            \lim_{t\to 0}{v_t^{\widetilde{j}+1}(x)}^\top{B_t^{\widetilde{j}}(x)}^{-1}v_t^{\widetilde{j}+1}(x)={v^{\widetilde{j}+1}(x)}^\top G^{\widetilde{j}}(x)v^{\widetilde{j}+1}(x)>0.
        \end{align}
        Towards this end, let us first define 
        $w^{\widetilde{j}+1}_t(x):={{B_t^{\widetilde{j}}}(x)}^{-1}v_t^{\widetilde{j}+1}(x)$, then it suffices to first show that $w^{\widetilde{j}+1}(x)= \lim_{t\to 0}w^{\widetilde{j}+1}_t(x)$  exists, then  ${v^{\widetilde{j}+1}(x)}^\top w^{\widetilde{j}+1}(x)>0$ as desired.
        
        To proceed, let us first observe that the following two limits exist:
        \begin{align}
            &\lim_{t\to 0}(A^{\widetilde{j}}_t(x)+B_t(x))^{-1}=\lim_{t\to 0}{B_t^{\widetilde{j}}(x)}^{-1} :=G^{\widetilde{j}}(x)\label{eq:twoeq}\\
            &\lim_{t\to 0}v_t^{\widetilde{j}+1}(x)=\lambda_{\widetilde{j}+1}(x)\nabla_y h_{\widetilde{j}+1}(x,y^\ast(x)),\nonumber
        \end{align}
        which represents that $\lim_{t\to 0}w^{\widetilde{j}+1}_t(x)=\lim_{t\to 0}(A^{\widetilde{j}}_t(x)+B_t(x))^{-1}v_t^{\widetilde{j}+1}(x)$ also exists. 
        
        Next we calculate $w^{\widetilde{j}+1}(x)=\lim_{t\to 0}w^{\widetilde{j}+1}_t(x)$ indirectly. 
        We decompose $v_t^{\widetilde{j}+1}(x)$ and $w^{\widetilde{j}+1}_t(x)$ as $v_t^{\widetilde{j}+1}(x)=\hat{v}_t+\hat{v}_t^\perp$ and $w^{\widetilde{j}+1}_t(x)=\hat{w}_t+\hat{w}_t^\perp$ respectively, with $\hat{v}_t,\hat{w}_t\in V^{\widetilde{j}}_t(x)$ and $\hat{v}_t^\perp,\hat{w}_t^\perp\in (V^{\widetilde{j}}_t(x))^\perp$. Here $V^{\widetilde{j}}_t(x)$ defined at (\ref{notation:space}) is
        $$V^j_t(x)=\mathrm{span}\{\frac{t}{-h_i(x,y^\ast_t(x))}\nabla_y h_i(x,y_t^\ast(x)):i\in\{1,...,j\}\cap \mathcal{I}^\ast(x)\}.$$
        A key observation is, since the term $A^{\widetilde{j}}_t(x)$ in $B_t^{\widetilde{j}}(x)$ has the potential to approach infinity, $B_t^{\widetilde{j}}(x)$ will stretch vectors in the subspace $V^{\widetilde{j}}_t(x)$ to infinite length, hence its inverse will compress vectors in $V^{\widetilde{j}}_t(x)$ to zero. We will demonstrate that the limit of $\hat{w}_t$ is zero, and then compute the limit of $\hat{w}_t^\perp$.
        
        We will first show that the decomposition vectors $\hat{v}_t$, $\hat{w}_t$, $\hat{v}_t^\perp$, $\hat{w}_t^\perp$ also have limit when $t$ goes to $0$. Note that the projection map has a matrix representation:
        $$\hat{w}_t=\proj_{V^{\widetilde{j}}_t(x)}w^{\widetilde{j}+1}_t(x)=F_t(x)({F_t(x)}^\top F_t(x))^{-1}{F_t(x)}^\top w^{\widetilde{j}+1}_t(x),$$
        where $F_t(x):=\{[v^i_t(x)]:i\in\{1,...,\widetilde{j}\}\cap\mathcal{I}^\ast(x)\}$ is the matrix with columns as the vectors $v^i_t(x)$ for $i\in\{1,...,\widetilde{j}\}\cap\mathcal{I}^\ast(x)$. According to Lemma \ref{lem_lowerbound2}, and Lemma \ref{lem:optimalitygapforpoint}, $v_t^{i}(x)$ converges to $v^i(x)$ for $i=1,...,\widetilde{j}$. By LICQ assumption, $F(x)^\top F(x)$ is invertible, where $F(x):=\{[v^i(x)]:i\in\{1,...,\widetilde{j}\}\cap\mathcal{I}^\ast(x)\}$ is the matrix with columns as the vectors $v^i(x)$ for $i\in\{1,...,\widetilde{j}\}\cap\mathcal{I}^\ast(x)$. Note that $\lim_{t\to 0} F_t(x)=F(x)$, we confirm $$\lim_{t\to 0}F_t(x)({F_t(x)}^\top F_t(x))^{-1}{F_t(x)}^\top=F(x)({F(x)}^\top F(x))^{-1}{F(x)}^\top$$ exists. Therefore $\hat{w}_t$ converges, and the convergence of $\hat{w}_t^\perp=w^{\widetilde{j}+1}_t(x)-\hat{w}_t$ follows. The convergence of $\hat{v}$ and $\hat{v}^\perp_t$ can also be obtained by a similar process. Denote $\hat{w}=\lim_{t\to 0}\hat{w}_t$, $\hat{w}^\perp=\lim_{t\to 0}\hat{w}^\perp_t$ and $\hat{v}=\lim_{t\to 0}\hat{v}_t$, $\hat{v}^\perp=\lim_{t\to 0}\hat{v}^\perp_t$. 
        
        Note that
        \begin{align}
            \hat{v}_t+\hat{v}_t^\perp&=v_t^{\widetilde{j}+1}(x)\nonumber\\
            &=\left(A^{\widetilde{j}}_t(x)+B_t(x)\right)w^{\widetilde{j}+1}_t(x)\nonumber\\
            &=\left(A^{\widetilde{j}}_t(x)+B_t(x)\right)\left(\hat{w}_t+\hat{w}_t^\perp\right)\nonumber\\
            &=A^{\widetilde{j}}_t(x)\hat{w}_t+A^{\widetilde{j}}_t(x) \hat{w}_t^\perp+B_t(x)\hat{w}_t+B_t(x)\hat{w}_t^\perp,\label{eq:decompose}
        \end{align} 
        We will compute convergence of some terms in (\ref{eq:decompose}) which are needed to prvove (\ref{eq:positive})
        \begin{enumerate}
            \item \textbf{Claim: $\lim_{t\to 0}A^{\widetilde{j}}_t(x)\hat{w}_t^\perp=0$.} 
            
            \textbf{Proof.} This is because 
            \begin{align*}
            \lim_{t\to 0}A^{\widetilde{j}}_t(x)\hat{w}_t^\perp&=\lim_{t\to 0}\sum_{i=1}^{\widetilde{j}}\frac{1}{t}v_t^i(x) {v_t^i(x)}^\top\hat{w}_t^\perp\\
            &\overset{\text{(i)}}{=}\lim_{t\to 0}\sum_{i\in\{1,...,\widetilde{j}\}\setminus\mathcal{I}^\ast(x)}\frac{1}{t}v_t^i(x) {v_t^i(x)}^\top\hat{w}_t^\perp\\
            &=\lim_{t\to 0}\sum_{i\in\{1,...,\widetilde{j}\}\setminus\mathcal{I}^\ast(x)}\frac{t}{h^2_i(x,y^\ast_t(x))}\nabla_y h_i(x,y_t^\ast(x)){\nabla_y h_i(x,y_t^\ast(x))}^\top\hat{w}_t^\perp\\
            &\overset{\text{(ii)}}{=}0,
            \end{align*}
            where (i) is because $v^i_t(x)\in V^{\widetilde{j}}_t(x)$, $\hat{w}_t^\perp\in(V^{\widetilde{j}}_t(x))^\perp$, and $v^i_t(x)\hat{w}_t^\perp=0$ for any $i\in\{1,...,\widetilde{j}\}\cap\mathcal{I}^\ast(x)$, and (ii) is because $\lim_{t\to 0}[t/h^2_i(x,y^\ast_t(x))=0/h^2_i(x,y^\ast(x))]=0$ for any $i\notin\mathcal{I}^\ast(x)$;
            \item \textbf{Claim: $\lim_{t\to 0}\hat{w}_t=0$.}
            
            \textbf{Proof.} If $\lim_{t\to 0}\hat{w}_t=:\zeta\ne0$, it must be that $\zeta\in V^{\widetilde{j}}(x)$ since $\hat{w}_t\in V_t^{\widetilde{j}}(x)$ for any $t$. Consequently, 
            \begin{align}\label{eq:divergence}
                \lim_{t\to 0}\|A^{\widetilde{j}}_t(x)\hat{w}_t\|\overset{\text{(i)}}{=}\lim_{t\to 0}\left\|\sum_{i\in\{1,...,\widetilde{j}\}\cap\mathcal{I}^\ast(x)}\frac{1}{t}v_t^i(x) {v_t^i(x)}^\top\hat{w}_t\right\|=\infty,
            \end{align} 
            where (i) is because $\lim_{t\to 0}[t/h^2_i(x,y^\ast_t(x))]=0$ for any $i\notin\mathcal{I}^\ast(x)$ and also $\lim_{t\to 0}\hat{w}^\perp_t=\hat{w}^\perp$ exists. (\ref{eq:divergence}) implies that $\lim_{t\to 0}\|A^{\widetilde{j}}_t(x)\hat{w}_t+A^{\widetilde{j}}_t(x) \hat{w}_t^\perp+B_t(x)\hat{w}_t+B_t(x)\hat{w}_t^\perp\|=\infty$, since $\lim_{t\to 0}\|A^{\widetilde{j}}_t(x) \hat{w}_t^\perp+B_t(x)\hat{w}_t+B_t(x)\hat{w}_t^\perp\|$ exists as all three terms in $\|A^{\widetilde{j}}_t(x) \hat{w}_t^\perp+B_t(x)\hat{w}_t+B_t(x)\hat{w}_t^\perp\|$ have limits. However, $\lim_{t\to 0} \|\hat{v}_t+\hat{v}_t^\perp\|=\|v^{\widetilde{j}}(x)\|<\infty$. Combining these facts with (\ref{eq:decompose}), we have
            \begin{align*}
            \infty>\lim_{t\to 0}\|\hat{v}_t+\hat{v}_t^\perp\|=\lim_{t\to 0}\|A^{\widetilde{j}}_t(x)\hat{w}_t+A^{\widetilde{j}}_t(x) \hat{w}_t^\perp+B_t(x)\hat{w}_t+B_t(x)\hat{w}_t^\perp\|=\infty.
            \end{align*}
            Here is a contradiction, so we conclude $\lim_{t\to 0}\hat{w}_t=0$;
            \item \textbf{Claim: $\lim_{t\to 0}B_t(x)\hat{w}_t=0$.}

            \textbf{Proof.}  According to Lemma \ref{lem_lowerbound2}, and Lemma \ref{lem:optimalitygapforpoint}, it is easy to see $\lim_{t\to 0}B_t(x)=B(x)$. Note that $\lim_{t\to 0}\hat{w}_t=0$, we can directly get $\lim_{t\to 0}B_t(x)\hat{w}_t=0$.
        \end{enumerate}
         
        Based on the above results, we are able to compute $w^{\widetilde{j}+1}(x)=\lim_{t\to 0}w^{\widetilde{j}+1}_t(x)=\lim_{t\to 0}\hat{w}^\perp_t=\hat{w}^\perp$ as follows:
        \begin{enumerate}
            \item Consider a linear transformation $\mathscr{B}(x)$ restricted to $(V^{\widetilde{j}}(x))^\perp$ as follows
            $$\mathscr{B}(x)(v):=\proj_{(V^{\widetilde{j}}(x))^\perp}B(x)v$$
            for any $v\in (V^{\widetilde{j}}(x))^\perp$. We show that $\mathscr{B}(x)$ is positive definite. Suppose there exists a non-zero vector $\eta\in(V^{\widetilde{j}}(x))^\perp$ such that $\eta^\top\mathscr{B}(x)(\eta)\leq 0$. Decomposing  $B(x)\eta=\eta_1+\eta_2$, where $\eta_1\in V^{\widetilde{j}}(x)$ and $\eta_2\in(V^{\widetilde{j}}(x))^\perp$, then $\eta^\top \mathscr{B}(x)(\eta)=\eta^\top \proj_{(V^{\widetilde{j}}(x))^\perp}B(x)\eta=\eta^\top\eta_2\leq 0$. Hence $\eta^\top B(x)\eta=\eta^\top(\eta_1+\eta_2)=\eta^\top\eta_2\leq 0$, contradicting the fact that $B(x)\succeq\nabla_{yy}g(x,y^\ast(x))$. Therefore, $\mathscr{B}(x)$ must be a positive definite linear transformation from $(V^{\widetilde{j}}(x))^\perp$ to $(V^{\widetilde{j}}(x))^\perp$;
            \item Consider the inverse of $\mathscr{B}(x)$ on $(V^{\widetilde{j}}(x))^\perp$, denoted as $(\mathscr{B}(x))^{-1}$. We extend the linear transformation $(\mathscr{B}(x))^{-1}:(V^{\widetilde{j}}(x))^\perp\to (V^{\widetilde{j}}(x))^\perp$ to a linear transformation $\widetilde{\mathscr{B}}(x)$ defined on $\RR^m$ as follows
            \begin{equation}
                \widetilde{\mathscr{B}}(x)(v)=\begin{cases}\label{eq:linearmap}
                    0 &\text{ if } v\in V^{\widetilde{j}}(x) \\
                    (\mathscr{B}(x))^{-1}(v) &\text{ if }  v\in \left(V^{\widetilde{j}}(x)\right)^\perp.
                 \end{cases}
            \end{equation}
            and $\widetilde{\mathscr{B}}(x)$ remains positive definite on $(V^{\widetilde{j}}(x))^\perp$, as $(\mathscr{B}(x))^{-1}$ is positive definite;
            \item We will show that $\lim_{t\to 0}w^{\widetilde{j}+1}_t(x)=\widetilde{\mathscr{B}}(x)(v^{\widetilde{j}+1})$. By calculation, we have
            \begin{align*}
                \hat{v}+\hat{v}^\perp&=\lim_{t\to 0}\hat{v}_t+\hat{v}_t^\perp\\
                &=\lim_{t\to 0}v_t^{\widetilde{j}+1}(x)\\
                &=\lim_{t\to 0}\left(A_t^{\widetilde{j}}(x)+B_t(x)\right)w_t^{\widetilde{j}+1}(x)\\
                &=\lim_{t\to 0}\left(A_t^{\widetilde{j}}(x)+B_t(x)\right)\left(\hat{w}_t+\hat{w}_t^\perp\right)\\
                &=\lim_{t\to 0}\left(A_t^{\widetilde{j}}(x)\hat{w}_t+A_t^{\widetilde{j}}(x)\hat{w}_t^\perp+B_t(x)\hat{w}_t+B_t(x)\hat{w}_t^\perp\right)\\
                &\overset{\text{(i)}}{=}\lim_{t\to 0}A_t^{\widetilde{j}}(x)\hat{w}_t+B(x)\hat{w}^\perp\\\
                &=\lim_{t\to 0}\sum_{i\in\{1,...,\widetilde{j}\}}\frac{t}{h^2_i(x,y^\ast_t(x))}\nabla_y h_j(x,y_t^\ast(x)){\nabla_y h_j(x,y_t^\ast(x))}^\top\hat{w}_t+B(x)\hat{w}^\perp\\
                &\overset{\text{(ii)}}{=}\lim_{t\to 0}\sum_{i\in\{1,...,\widetilde{j}\}\cap\mathcal{I}^\ast(x)}\frac{t}{h^2_i(x,y^\ast_t(x))}\nabla_y h_j(x,y_t^\ast(x)){\nabla_y h_j(x,y_t^\ast(x))}^\top\hat{w}_t+B(x)\hat{w}^\perp\\
                &=\sum_{i\in\{1,...,\widetilde{j}\}\cap\mathcal{I}^\ast(x)}\left(\lim_{t\to 0}\frac{t}{h^2_i(x,y^\ast_t(x))}{\nabla_y h_i(x,y_t^\ast(x))}^\top\hat{w}_t^\perp\right)\nabla_y h_i(x,y^\ast(x))+B(x)\hat{w}^\perp,\\
            \end{align*}
            where (i) is because $\lim_{t\to 0}A_t^{\widetilde{j}}(x)\hat{w}_t^\perp=0$, $\lim_{t\to 0}B_t(x)\hat{w}_t=0$, $\lim_{t\to 0}B_t(x)=B(x)$ which have been shown in the claims we proved before, and (ii) is because $\lim_{t\to 0}t/h^2_i(x,y^\ast_t(x))=0$ for any $i\notin\mathcal{I}^\ast(x)$. Note that the vector 
            $$\sum_{i\in\{1,...,\widetilde{j}\}\cap\mathcal{I}^\ast(x)}\left(\lim_{t\to 0}\frac{t}{h^2_i(x,y^\ast_t(x))}{\nabla_y h_i(x,y_t^\ast(x))}^\top\hat{w}_t^\perp\right)\nabla_y h_i(x,y^\ast(x))$$
            is in the subspace $V^{\widetilde{j}}(x)$, then $\hat{v}^\perp=\proj_{(V^{\widetilde{j}}(x))^\perp}B(x)\hat{w}^\perp=\mathscr{B}(x)(\hat{w}^\perp)$, implying $\lim_{t\to 0}w^{\widetilde{j}+1}_t(x)=\hat{w}^\perp=(\mathscr{B}(x))^{-1}(\hat{v}^\perp)=\widetilde{\mathscr{B}}(x)(v^{\widetilde{j}+1})$.
        \end{enumerate}
        To conclude, we get the result that $w^{\widetilde{j}+1}(x)=\widetilde{\mathscr{B}}(x)(v^{\widetilde{j}+1})$. Under the LICQ assumption and the assumption that $\widetilde{j}+1$ is active, we have $\hat{v}^\perp=\proj_{\left(V^{\widetilde{j}}(x)\right)^\perp}{v^{\widetilde{j}+1}}(x)\ne 0$. Otherwise, if $\proj_{\left(V^{\widetilde{j}}(x)\right)^\perp}{v^{\widetilde{j}+1}}(x)=0$, then $v^{\widetilde{j}+1}(x)\in V^{\widetilde{j}}(x)$, which implies $v^{\widetilde{j}+1}(x)$ can be expressed as the linear combination of $v^{i}(x)$ where $i\in\mathcal{I}(x)\cap\{1,...,\widetilde{j}\}$. Note that $v^{i}(x)=\lambda_i(x)\nabla_y h_i(x,y^\ast(x))$, and $\lambda_i(x)>0$ for any $i\in\mathcal{I}(x)$. Thus $\nabla_y h_{\widetilde{j}+1}(x,y^\ast(x))$ can be expressed as the linear combination of $\nabla_y h_{i}(x,y^\ast(x))$ where $i\in\mathcal{I}(x)\cap\{1,...,\widetilde{j}\}$. This contradicts the LICQ assumption. 
        
        Hence, 
        \begin{align}\label{eq:limpositive}
            \lim_{t\to 0}{v_t^{\widetilde{j}+1}(x)}^\top{B_t^{\widetilde{j}}(x)}^{-1}v_t^{\widetilde{j}+1}(x)&=(\hat{v}+\hat{v}^\perp)^\top(\hat{w}+\hat{w}^\perp)\nonumber\\
            &=(\hat{v}+\hat{v}^\perp)^\top\hat{w}^\perp\nonumber\\
            &\overset{\text{(i)}}{=}{\hat{v}^\perp}^\top(\mathscr{B}(x))^{-1}(\hat{v}^\perp)\\
            &\overset{\text{(i)}}{>}0,
        \end{align} 
        where (i) is because $\hat{v}^\top\hat{w}^\perp=0$, and (ii) is because ${B_t^{\widetilde{j}}(x)}^{-1}$ is positive definite in $\left(V^{\widetilde{j}}(x)\right)^\perp$. It follows that
        \begin{align*}
           \lim_{t\to 0}(A^{\widetilde{j}+1}_t(x)+B_t(x))^{-1}&=\lim_{t\to 0}{B_t^{\widetilde{j}}(x)}^{-1}-\frac{{B_t^{\widetilde{j}}(x)}^{-1}v_t^{\widetilde{j}+1}(x) {v_t^{\widetilde{j}+1}(x)}^\top {B_t^{\widetilde{j}}(x)}^{-1}}{t^2+{v_t^{\widetilde{j}+1}(x)}^\top{B_t^{\widetilde{j}}(x)}^{-1}v_t^{\widetilde{j}+1}(x)}\\
           &=\widetilde{\mathscr{B}}(x)-\frac{\widetilde{\mathscr{B}}(x)\left(\hat{v}^\perp(x)\right) \left(\widetilde{\mathscr{B}}(x)\left(\hat{v}^\perp(x)\right)\right)^\top}{{\hat{v}^\perp}^\top(\mathscr{B}(x))^{-1}(\hat{v}^\perp)}
        \end{align*}
        exists. We have proved the existence of $\lim_{t\to 0}(A^{\widetilde{j}+1}_t(x)+B_t(x))^{-1}$ when $j=\widetilde{j}+1$ is an active index. 
        
        \noindent\textbf{Case 2: $j=\widetilde{j}+1$ is an inactive index.} In this case, we have
        \begin{align*}
            \lim_{t\to 0}(A_t^{\widetilde{j}+1}(x)+B_t(x))^{-1}&=\lim_{t\to 0}\left(\frac{t}{h^2_{\widetilde{j}+1}(x,y^\ast_t(x))}\nabla_yh_{\widetilde{j}+1}(x,y^\ast_t(x))\left(\nabla_yh_{\widetilde{j}+1}(x,y^\ast_t(x))\right)^\top+B_t^{\widetilde{j}}(x)\right)^{-1}\\
            &=\lim_{t\to 0}\left(tE^{\widetilde{j}+1}_t(x)+B_t^{\widetilde{j}}(x)\right)^{-1},
        \end{align*}
         where $$E^j_t(x):=\frac{1}{h^2_{j}(x,y^\ast_t(x))}\nabla_yh_{j}(x,y^\ast_t(x))\left(\nabla_yh_{j}(x,y^\ast_t(x))\right)^\top.$$
         Note that the following limit exists
         $$\lim_{t\to 0}E^{\widetilde{j}+1}_t(x)=\frac{1}{h^2_{\widetilde{j}+1}(x,y^\ast(x))}\nabla_yh_{\widetilde{j}+1}(x,y^\ast(x))\left(\nabla_yh_{\widetilde{j}+1}(x,y^\ast(x))\right)^\top.$$ 
         Further, by inductive assumption, the following limit exists
         $$\lim_{t\to 0}{B_t^{\widetilde{j}}}^{-1}(x)=\lim_{t\to 0}(A^{\widetilde{j}}_t(x)+B_t(x))^{-1}.$$
         Further, assuming $t$ is sufficiently small, applying Taylor expansion
        \begin{align*}
            (B_t^{\widetilde{j}}(x)+tE^{\widetilde{j}+1}_t(x))^{-1}=&{B_t^{\widetilde{j}}(x)}^{-1}-t{B_t^{\widetilde{j}}(x)}^{-1}E^{\widetilde{j}+1}_t(x){B_t^{\widetilde{j}}(x)}^{-1}\\
            &+t^2{B_t^{\widetilde{j}}(x)}^{-1}E^{\widetilde{j}+1}_t(x){B_t^{\widetilde{j}}(x)}^{-1}E^{\widetilde{j}+1}_t(x){B_t^{\widetilde{j}}(x)}^{-1}-...\\
            =&{B_t^{\widetilde{j}}}^{-1}(x)+\mathcal{O}(t).
        \end{align*}
        Thus, $\lim_{t\to 0}(A_t^{\widetilde{j}+1}(x)+B_t(x))^{-1}=\lim_{t\to 0}\left({B_t^{\widetilde{j}}(x)}^{-1}+\mathcal{O}(t)\right)=G^{\widetilde{j}}(x)$ which is defined in (\ref{eq:twoeq}).
        
        To conclude, we have proved $$\lim_{t\to 0}\left(A^k_t(x)+B_t(x)\right)^{-1}C_t(x)$$ exists regardless of whether $j$ is active or inactive.
    
    \noindent\textbf{Step 2.} Now go back to (\ref{eq:jacobiannotation}), we will show that $\lim_{t\to 0}(A^k_t(x)+B_t(x))^{-1}D_t(x)$ exists. By the notation (\ref{notation:Dt}), we have
        \begin{align*}
            D_t(x)&=\sum_{i=1}^k\frac{1}{t}\left(\frac{t}{-h_j(x,y^\ast_t(x))}\right)^2\nabla_y h_j(x,y_t^\ast(x))\nabla_x h_j(x,y_t^\ast(x))\\
            &=\sum_{i=1}^k\frac{1}{t}\frac{t}{-h_j(x,y^\ast_t(x))}\nabla_y h_j(x,y_t^\ast(x))\left(\frac{t}{-h_j(x,y^\ast_t(x))}\nabla_x h_j(x,y_t^\ast(x))\right)^\top
        \end{align*}
        Note that by Lemma \ref{lem_lowerbound2}
        \begin{align*}
            \lim_{t\to 0}\left(\frac{t}{-h_j(x,y^\ast_t(x))}\nabla_x h_j(x,y_t^\ast(x))\right)^\top=\left(\lambda_i(x)\nabla_x h_j(x,y^\ast(x))\right)^\top,
        \end{align*}
        then it is sufficient to prove that the following limit exists for any $j$:
        \begin{align*}
            \lim_{t\to 0}(A^k_t(x)+B_t(x))^{-1}\frac{1}{t}\frac{t}{-h_j(x,y^\ast_t(x))}\nabla_y h_j(x,y_t^\ast(x)).
        \end{align*}
         Without loss of generality, assume $j=k$. Otherwise, simply swap the positions of the $j$-th and $k$-th constraints. 
         
         If $k$ is an inactive index, then we have
        \begin{align*}
            \lim_{t\to 0}\frac{1}{t}\frac{t}{-h_k(x,y^\ast_t(x))}\nabla_y h_k(x,y^\ast_t(x))=\lim_{t\to 0}\frac{1}{-h_k(x,y^\ast_t(x))}\nabla_y h_k(x,y^\ast_t(x))=\frac{1}{-h_k(x,y^\ast(x))}\nabla_y h_k(x,y^\ast(x)).
        \end{align*}
        It follows that
        \begin{align*}
            &\lim_{t\to 0}(A^k_t(x)+B_t(x))^{-1}\frac{1}{t}\frac{t}{-h_k(x,y^\ast_t(x))}\nabla_y h_k(x,y_t^\ast(x))\\
            =&\lim_{t\to 0}(A^k_t(x)+B_t(x))^{-1}\cdot\lim_{t\to 0}\frac{1}{t}\frac{t}{-h_k(x,y^\ast_t(x))}\nabla_y h_k(x,y_t^\ast(x))\\
            =&\lim_{t\to 0}(A^k_t(x)+B_t(x))^{-1}\frac{1}{-h_k(x,y^\ast(x))}\nabla_y h_k(x,y^\ast(x))\\
            =&G^k(x)\frac{1}{-h_k(x,y^\ast(x))}\nabla_y h_k(x,y^\ast(x)),
        \end{align*}
        where $G^k(x)$ is defined in (\ref{eq:twoeq}).
        
        If $k$ is an active index, then we have
        \begin{align}
            &(A^k_t(x)+B_t(x))^{-1}\frac{1}{t}\frac{t}{-h_k(x,y^\ast_t(x))}\nabla_y h_k(x,y_t^\ast(x))\nonumber\\
            \overset{\text{(i)}}{=}&(A^k_t(x)+B_t(x))^{-1}\frac{1}{t}v^k_t(x)\nonumber\\
            =&\left(\frac{1}{t}v^k_t(x){v^k_t(x)}^\top+B^{k-1}_t(x)\right)^{-1}\frac{1}{t}v^k_t(x)\nonumber\\
            =&\left(v^k_t(x){v^k_t(x)}^\top+tB^{k-1}_t(x)\right)^{-1}v^k_t(x)\nonumber\\
            \overset{\text{(i)}}{=}&\left(tB_t^{k-1}(x)\right)^{-1}v^k_t(x)-\frac{\left(tB_t^{k-1}(x)\right)^{-1}v^k_t(x){v^k_t(x)}^\top \left(tB_t^{k-1}(x)\right)^{-1}v^k_t(x)}{1+{v^k_t(x)}^\top \left(tB_t^{k-1}(x)\right)^{-1}v^k_t(x)}\nonumber\\
            =&\frac{\left(tB_t^{k-1}(x)\right)^{-1}v^k_t(x)}{1+{v^k_t(x)}^\top \left(tB_t^{k-1}(x)\right)^{-1}v^k_t(x)}\nonumber\\
            =&\frac{\left(A_t^{kƒƒ-1}(x)+B_t(x)\right)^{-1}v^k_t(x)}{t+{v^k_t(x)}^\top\left(A_t^{k-1}(x)+B_t(x)\right)^{-1}v^k_t(x)}\label{eq:(A+B)D}.
        \end{align}
        In (i) we plug in $\frac{t}{-h_k(x,y^\ast_t(x))}\nabla_y h_k(x,y_t^\ast(x))=v^k_t(x)$, and (ii) is by the Sherman-Morrison formula. This converges is because: as $t$ converge to $0$, by replacing $\widetilde{j}$ with $k-1$ in (\ref{eq:limpositive}), we obtain that $\lim_{t\to 0}{v^k_t(x)}^\top\left(A_t^{k-1}(x)+B_t(x)\right)^{-1}v^k_t(x)$ is positive.\\

        Finally $\lim_{t\rightarrow0}\nabla_x y^*_t(x)$ exists since
        $$
        \nabla_x y^*_t(x) = (A_t^{k}(x)+B_t(x))^{-1}(C_t(x)+D_t(x))
        $$
        and we proved the convergence of the two parts in the above two steps.
\end{proof}

The following two lemmas are used in the proof of Lemma \ref{lem:uniformconvergeofgradient}.

\begin{lemma}\label{lem:heinethm}
    For a family of vector-valued function $\{f_t(x)\}_{t\in\RR^+}$ defined on a compact set $U$, if $f_t(x)$ point-wisely, but not uniformly, converges to $f(x)$ when $t$ goes to $0$, then there exists a sequence $\{(x_j,t_j)\}_{j=1}^\infty$ such that $x_j$ converges to $\hat{x}$ and $t_j$ converges to $0$ but $\lim_{j\to\infty}f_{t_j}(x_j)$ does not exist.
\end{lemma}

\begin{proof}
    By $\epsilon$-$\delta$ Language, $f_t(x)$ does not converges to $f(x)$ uniformly means there exists a constant $\epsilon_0$ such that for any $T$, there exists a point $x$ and $t_1$, $t_2\leq T$ satisfying $\|f_{t_1}(x)-f_{t_2}(x)\|\geq\epsilon_0$. We construct the sequence $(t_j,x_j)$ as follows:
    \begin{itemize}
        \item For any $j\geq 1$, set $T_j=1/j$. There exists $t'_{j,1}$, $t'_{j,2}\leq T_j$ and a point $x'_j$ such that $$\|f_{t'_{j,1}}(x'_j)-f_{t'_{j,2}}(x'_j)\|\geq\epsilon_0;$$
        \item Since $U$ is compact, choose a subsequence of $\{(x'_j,t'_{j,1},t'_{j,2})\}_{j=1}^\infty$, still denoted as $\{(x'_j,t'_{j,1},t'_{j,2})\}_{j=1}^\infty$, such that $\lim_{j\to\infty}x'_j=\hat{x}$ for some point $\hat{x}$;
        \item For any $j\geq 1$, set $x_{2j-1}=x_{2j}=x'_j$ and $t_{2j-1}=t'_{j,1}$, $t_{2j}=t'_{j,2}$. Then we have
        \begin{align}\label{ineq:nonCauchy}
            \|f_{t_{2j-1}}(x_{2j-1})-f_{t_{2j}}(x_{2j})\|\geq\epsilon_0,\quad\forall j\geq 1.
        \end{align}
    \end{itemize}
    It is not hard to see that $\lim_{j\to \infty}x_j=\lim_{j\to \infty}x'_j=\hat{x}_j$. However, $\{f_{t_j}(x_j)\}_{j=1}^\infty$ is not a Cauchy sequence by (\ref{ineq:nonCauchy}). We conclude that $\lim_{j\to\infty}f_{t_j}(x_j)$ does not exist.    
\end{proof}

\begin{lemma}\label{lem:limitofsequence}
    Suppose $\{f_t(x)\}_{t\in\RR^+}$ is a family of matrix-valued functions defined on $U$. If $f_t(x)$ uniformly converges to a continuous function $f(x)$, and $\{(x_j,t_j)\}_{j=1}^\infty$ is a sequence such that $\lim_{j\to\infty}t_j=0$ and $\lim_{j\to\infty}x_j=\hat{x}$ for some $\hat{x}$, then $\lim_{j\to\infty}f_{t_j}(x_j)=f(\hat{x})$. 
\end{lemma}

\begin{proof}
    By uniformly convergence, for any $\epsilon/2>0$, there exists $T$ such that $\|f_t(x)-f(x)\|\leq\epsilon/2$ for any $0<t\leq T$ and $x\in U$. Since $\lim_{j\to\infty}t_j=0$, we can find $J_1$ such that $t_j\leq T$ for any $j\geq J_1$. Note that $f(x)$ is continuous, for any $\epsilon/2>0$, there exists $\delta$ such that $\|f(x)-f(\hat{x})\|\leq\epsilon/2$ for any $x$ satisfies $\|x-\hat{x}\|\leq\delta$. According to $\lim_{j\to\infty}x_j=\hat{x}$, there exists $T_2$ such that $\|x_j-\hat{x}\|\leq\delta$ for any $0<t\leq T_2$. To conclude, for any $\epsilon>0$, we can find $J=\max\{J_1,J_2\}$ such that for any $j\geq J$ we have
    $$\|f_{t_j}(x_j)-f(\hat{x})\|\leq\|f_{t_j}(x_j)-f(x_j)\|+\|f(x_j)-f(\hat{x})\|\leq\frac{\epsilon}{2}+\frac{\epsilon}{2}=\epsilon.$$
    The proof is complete.
\end{proof}

\begin{lemma}\label{lem:uniformconvergeofgradient}
    Suppose Assumption \ref{assumption:general1} and Assumption \ref{assumption:nonlinear} hold, and $x$ is an \textbf{SCSC} point. Then there exists a compact neighborhood $U$ of $x$ such that $\nabla_x y_t^\ast(x)$ converges uniformly in $U$. Compact neighborhood is a compact set containing an open ball centered at $x$.
\end{lemma}

\begin{proof}
     We fix an \textbf{SCSC} point $x_0$ and prove this result at this point. We select a compact neighborhood $U$ of $x_0$ satisfying the following properties:
     \begin{itemize}
         \item The active index set for any point in $U$ is the same as that of $x_0$, and $\lambda_i(x)>0$ for each active index as long as $U$ is small enough, because $\lambda_i(x)$ is continuous, and $x_0$ is an \textbf{SCSC} point;
         \item There exists $H>0$ such that the inequality $h_i(x,y^\ast(x))\leq -H<0$ holds for any inactive index $i$ and any $x$ in $U$. This can be satisfied when $H$ and $U$ are small enough because $\lambda_i(x)$, $h_i(x,y^\ast(x))$ are continuous in $x$, and $x_0$ is an \textbf{SCSC}  point;
         \item There exists a constant $T$ such that the vectors $\nabla_y h_i(x,y^\ast_t(x))$, where $i\in\{1,...,k\}\cap\mathcal{I}^\ast(x_0)$, are linear independent for any $0<t\leq T$. This because the LICQ assumption is satisfied at point $(x,y^\ast(x))$ and $\nabla_y h_i(x,y^\ast_t(x))$ converges to  $\nabla_y h_i(x,y^\ast(x))$ uniformly by lemma \ref{lem_lowerbound2}.
     \end{itemize} 
     Using the notation in lemma \ref{lem_existence}, we have
     $$\nabla_x y^*_t(x) = (A_t^{k}(x)+B_t(x))^{-1}(C_t(x)+D_t(x)).$$
     To prove $\nabla_x y^*_t(x)$ uniformly converges in $U$, it is sufficient to prove that $\left(A^k_t(x)+B_t(x)\right)^{-1}C_t(x)$ and $(A_t(x)+B_t(x))^{-1}D_t(x)$ converge uniformly in $U$ separately. 
     
     Before detailed proof, we give the following uniform convergence which can all be directly obtained from Lemma \ref{lem:optimalitygapforpoint} and \ref{lem_lowerbound2}, i.e. $y^\ast_t(x)$ converges to $y^\ast(x)$ uniformly, and $t/(-h_i(x,y^\ast_t(x)))$ converges to $\lambda_i(x)$ uniformly
         \begin{align}
             \lim_{t\to 0}v_t^j(x)&=\lim_{t\to 0}\frac{t}{-h_j(x,y^\ast_t(x))}\nabla_y h_j(x,y^\ast_t(x))=\lambda_j(x)\nabla_y h_j(x,y^\ast(x));\label{eq:uniformofv}\\
             \lim_{t\to 0}B_t(x)&=\lim_{t\to 0}\nabla^2_{yy} g(x,y_t^\ast(x))+\sum_{i=1}^k\frac{t}{-h_j(x,y^\ast_t(x))}\nabla^2_{yy} h_i(x,y_t^\ast(x))\nonumber\\
             &=\nabla^2_{yy} g(x,y^\ast(x))+\sum_{i=1}^k\lambda_i(x)\nabla^2_{yy} h_i(x,y^\ast(x));\label{eq:uniformofB}\\
             \lim_{t\to 0}C_t(x)&=\lim_{t\to 0}\nabla^2_{yx} g(x,y_t^\ast(x))+\sum_{i=1}^k\frac{t}{-h_j(x,y^\ast_t(x))}\nabla^2_{yx} h_i(x,y_t^\ast(x))\nonumber\\
             &=\nabla^2_{yx} g(x,y^\ast(x))+\sum_{i=1}^k\lambda_i(x)\nabla^2_{yx} h_i(x,y^\ast(x))\label{eq:uniformofC}.
         \end{align}
     
     \noindent\textbf{Step 1:} We will prove $\left(A^k_t(x)+B_t(x)\right)^{-1}C_t(x)$ converges uniformly in $U$ by contradiction. If it does not uniformly converge, we will identify a subsequence $(x_j,t_j)$ such that $x_j$ converges to $\hat{x}\in U$, and $t_j$ converges to $0$, but $\left(A^k_{t_j}(x_j)+B_{t_j}(x_j)\right)^{-1}C_{t_j}(x_j)$ does not converge to $\widetilde{\mathscr{B}}(\hat{x})(C(\hat{x}))$, where $\widetilde{\mathscr{B}}$ is defined in (\ref{eq:linearmap}). However, we will show that for that kind of sequence, we can always have $\lim_{j\to\infty}\left(A^k_{t_j}(x_j)+B_{t_j}(x_j)\right)^{-1}C_{t_j}(x_j)=\widetilde{\mathscr{B}}(\hat{x})(C(\hat{x}))$, which leads a contradiction.
     
     Denoting $C_t(x):=[C_{t,1}(x),...,C_{t,n}(x)]$. Note that
     $$\left(A^k_t(x)+B_t(x)\right)^{-1}C_t(x)=\left[\left(A^k_t(x)+B_t(x)\right)^{-1}C_{t,1}(x),\cdots,\left(A^k_t(x)+B_t(x)\right)^{-1}C_{t,n}(x)\right],$$
     it is sufficient to prove that $\lim_{t\to 0}\left(A^k_t(x)+B_t(x)\right)^{-1}C_{t,i}(x)$ converges uniformly for any $i$. Since $C_t(x)$ converges uniformly by (\ref{eq:uniformofC}), we denote $C_i(x):=\lim_{t\to 0}C_{t,i}(x).$
    
    By our design of $U$, the active index set $\mathcal{I}^\ast(x)=\mathcal{I}^\ast(x_0)$ remains unchanged for any $x$ in $U$. According to (\ref{eq:linearmap}) in the proof of Lemma \ref{lem_existence}, we know $\lim_{t\to 0}\left(A^k_t(x)+B_t(x)\right)^{-1}C_i(x)$ can be computed as follows:
    \begin{align*}
            \lim_{t\to 0}\left(A^k_t(x)+B_t(x)\right)^{-1} C_i(x)=\widetilde{\mathscr{B}}(x)(C_i(x))=\begin{cases}
                0 &\text{ if } C_i(x)\in V^{k}(x_0) \\
                {\mathscr{B}(x)}^{-1}(C_i(x)) &\text{ if }  C_i(x)\in \left(V^{k}(x_0)\right)^\perp.
            \end{cases}
    \end{align*}
    Now we prove that $\left(A^k_t(x)+B_t(x)\right)^{-1}C_{t,i}(x)$ converges to $\widetilde{\mathscr{B}}(x)(C_i(x))$ uniformly for any $i$. If this does not hold for some $i$, by Lemma \ref{lem:heinethm}, there exists a sequence $\{(x_j,t_j)\}_{j=1}^\infty$ such that $0<t_j\leq T$ for any $j$, and $\lim_{j\to\infty}t_j=0$, $\lim_{j\to\infty}x_j=\hat{x}$ for some point $\hat{x}\in U$, but
    \begin{align*}
    \lim_{t\to 0}\left(A^k_{t_j}(x_j)+B_{t_j}(x_j)\right)^{-1}C_{t_j,i}(x_j)
    \end{align*}
    does not exist. We will demonstrate that $\left(A^k_{t_j}(x_j)+B_{t_j}(x_j)\right)^{-1}C_{t_j,i}(x_j)$ has a uniform bound for $j$ so that we can find a subsequence, still denoted as $\{(x_j,t_j)\}_{j=1}^\infty$, such that
    \begin{align}\label{eq:limitnotequal}
        \lim_{j\to\infty}\left(A^k_{t_j}(x_j)+B_{t_j}(x_j)\right)^{-1}C_{t_j,i}(x_j)
    \end{align}
    exists but not equal $\widetilde{\mathscr{B}}(\hat{x})(C_i(\hat{x}))$. Then we will show that the limit of (\ref{eq:limitnotequal}) must be $\widetilde{\mathscr{B}}(\hat{x})(C_i(\hat{x}))$, induces a contradiction.
    
    Note that 
    $$A^k_{t_j}(x_j)+B_{t_j}(x_j)\succeq B_{t_j}(x_j)\succeq \nabla_{yy}^2g(x_j,y_{t_j}^\ast(x_j))$$
    and $g(x,y)$ is $\mu_g$-strongly convex. Thus, we have
    $$\left(A^k_{t_j}(x_j)+B_{t_j}(x_j)\right)^{-1}\preceq \frac{1}{\mu_g}I$$
    is uniformly bounded. Note that $C_{t,i}(x)$ converges to $C_{i}(x)$ uniformly, $C_{t,i}(x)$ can be seen as a continuous map defined on $[0,T]\times\mathcal{X}$. Therefore, $C_{t,i}(x)$ also has a uniform bound for any $t$ and $x$. Therefore, $\left(A^k_{t_j}(x_j)+B_{t_j}(x_j)\right)^{-1}C_{t_j,i}(x_j)$ has a uniform bound for any $j$. 
    
    We choose a subsequence, still denoted as $\{(x_j,t_j)\}_{j=1}^\infty$, such that
    \begin{align*}
        \lim_{j\to\infty}w_{t_j}(x_j):=\lim_{j\to\infty}\left(A^k_{t_j}(x_j)+B_{t_j}(x_j)\right)^{-1}C_{t_j,i}(x_j)=w
    \end{align*}
    but $w\ne\widetilde{\mathscr{B}}(\hat{x})(C_i(\hat{x}))$. 
    
    Decompose $C_{t_j,i}(x_j)=\hat{C}_j+\hat{C}_j^\perp$ and $w_{t_j}(x_j)=\hat{w}_j+\hat{w}_j^\perp$, where $\hat{C}_j$, $\hat{w}_j\in V^k_{t_j}(x_j)$ and $\hat{C}_j^\perp$, $\hat{w}_j^\perp\in\left(V^k_{t_j}(x_j)\right)^\perp$. Here by notation given in (\ref{notation:space})
    $$V^j_t(x)=\mathrm{span}\{\frac{t}{-h_i(x,y^\ast_t(x))}\nabla_y h_i(x,y_t^\ast(x)):i\in\{1,...,j\}\cap \mathcal{I}^\ast(x)\}$$
    is from (\ref{notation:space}). We want to show that $\hat{C}_j$, $\hat{C}_j^\perp$, $\hat{w}_j$ and $\hat{w}_j^\perp$ all converges when $j$ goes to infinity. Note that
    \begin{align*}
        \hat{C}_j&=\proj_{V^k_{t_j}(x_j)}C_{t_j,i}(x_j)\\
        &=F_{t_j}(x_j)\left(F_{t_j}(x_j)^\top F_{t_j}(x_j)\right)^{-1}F_{t_j}(x_j)^\top C_{t_j,i}(x_j)
    \end{align*}
    where $F_{t_j}(x_j)$ is the matrix with columns as the vectors $v^i_{t_j}(x_j)=(t_j/(-h_i(x_j,y^\ast_{t_j}(x_j))))\nabla_y h_i(x_j,y^\ast_{t_j}(x_j))$ for $i\in\mathcal{I}^\ast(x)$. Since we have the uniform convergence by (\ref{eq:uniformofC}), and clearly $C_i(x)$ is continuous in $x$, by Lemma \ref{lem:limitofsequence}, we get $\lim_{j\to\infty}C_{t_j,i}(x_j)=\hat{C}_i(\hat{x})$. It is sufficient to prove 
    $$F_{t_j}(x_j)\left(F_{t_j}(x_j)^\top F_{t_j}(x_j)\right)^{-1}F_{t_j}(x_j)^\top $$
    also converges. According to (\ref{eq:uniformofv}), we know $v^i_{t}(x)$ converges to $\lambda_i(x)\nabla_y h_i(x,y^\ast(x))$ uniformly. Note that $\lambda_i(x)\nabla_y h_i(x,y^\ast(x))$ is continuous in $x$. By Lemma \ref{lem:limitofsequence}, $\lim_{j\to\infty}v^i_{t_j}(x_j)=v^i(\hat{x})$. This implies $\lim_{j\to\infty}F_{t_j}(x_j)=F(\hat{x})$ where $F(\hat{x})$ is the matrix with columns as the vectors $\lambda_i(\hat{x})\nabla_y h_i(\hat{x},y^\ast(\hat{x}))$ for $i\in\mathcal{I}^\ast(x)$. By LICQ assumption $F(\hat{x})^\top F(\hat{x})$ is invertible, we obtain $F_{t_j}(x_j)\left(F_{t_j}(x_j)^\top F_{t_j}(x_j)\right)^{-1}F_{t_j}(x_j)^\top$ converges to $F(\hat{x})\left(F(\hat{x})^\top F(\hat{x})\right)^{-1}F(\hat{x})^\top$ when $j$ goes to infinity. Therefore, 
    \begin{align*}
        \lim_{j\to\infty}\hat{C}_j&=\lim_{j\to\infty}\left(\proj_{V^k_{t_j}(x_j)}C_{t_j,i}(x_j)\right)\\
        &=\lim_{j\to\infty}\left(F_{t_j}(x_j)\left(F_{t_j}(x_j)^\top F_{t_j}(x_j)\right)^{-1}F_{t_j}(x_j)^\top C_{t_j,i}(x_j)\right)\\
        &=\lim_{j\to\infty}\left(F_{t_j}(x_j)\left(F_{t_j}(x_j)^\top F_{t_j}(x_j)\right)^{-1}F_{t_j}(x_j)^\top \right)\cdot\lim_{j\to\infty}\left(C_{t_j,i}(x_j)\right)\\
        &=F(\hat{x})\left(F(\hat{x})^\top F(\hat{x})\right)^{-1}F(\hat{x})^\top C_i(\hat{x})\\
        &=\proj_{V^k(\hat{x})}C_i(\hat{x})\\
        &=:\hat{C}.
    \end{align*}
    By the same method, $\lim_{j\to\infty}\hat{w}_j=:\hat{w}$, $\lim_{j\to\infty}\hat{C}_j^\perp=:\hat{C}^\perp=C_i(\hat{x})-\hat{C}$, and $\lim_{j\to\infty}w_j^\perp=:\hat{w}^\perp=w-\hat{w}$ all exist. 
    
    Next, we will show that
    \begin{align*}
        \lim_{j\to\infty}\left(A^k_{t_j}(x_j)+B_{t_j}(x_j)\right)^{-1}(C_{t_j,i}(x_j))=\lim_{j\to\infty}w_{t_j}(x_j)=\widetilde{\mathscr{B}}(\hat{x})(C_i(\hat{x})),
    \end{align*}
    to contradict our assumption that
     \begin{align*}
        \lim_{j\to\infty}\left(A^k_{t_j}(x_j)+B_{t_j}(x_j)\right)^{-1}(C_{t_j,i}(x_j))\ne \widetilde{\mathscr{B}}(\hat{x})(C_i(\hat{x})).
    \end{align*}
    Note that
    \begin{align*}
        \hat{C}_j+\hat{C}_j^\perp&=C_{t_j,i}(x_j)\\
        &=\left(A^k_{t_j}(x_j)+B_{t_j}(x_j)\right)w_{t_j}(x_j)\\
        &=\left(A^k_{t_j}(x_j)+B_{t_j}(x_j)\right)\left(w_j+w_j^\perp+\right)\\
        &=A^k_{t_j}(x_j)\hat{w}_j+A^k_{t_j}(x_j)\hat{w}j^\perp+B_{t_j}(x_j)\hat{w}_j+B_{t_j}(x_j)\hat{w}_j^\perp.
    \end{align*}
    Similar to the proof of Lemma \ref{lem_existence}, we will prove the following three results:
    \begin{enumerate}
        \item \textbf{Claim: $\lim_{j\to\infty}A^k_{t_j}(x_j) \hat{w}_j^\perp=0$.} 
        
            \textbf{Proof.} This is because 
            \begin{align*}
            \lim_{j\to\infty}A^{k}_{t_j}(x_j)\hat{w}_j^\perp&=\lim_{j\to\infty}\sum_{i=1}^{k}\frac{1}{t}v_{t_j}^i(x_j) {v_{t_j}^i(x_j)}^\top \hat{w}_j^\perp\\
            &\overset{\text{(i)}}{=}\lim_{j\to\infty}\sum_{i\in\{1,...,k\}\setminus\mathcal{I}^\ast(x_0)}\frac{1}{t}v_{t_j}^i(x_j) {v_{t_j}^i(x_j)}^\top \hat{w}_j^\perp\\
            &=\lim_{j\to\infty}\sum_{i\in\{1,...,k\}\setminus\mathcal{I}^\ast(x_0)}\frac{t}{h^2_i(x_j,y_{t_j}^\ast(x_j))}\nabla_y h_i(x_j,y_{t_j}^\ast(x_j)){\nabla_y h_i(x_j,y_{t_j}^\ast(x_j))}^\top \hat{w}_j^\perp\\
            &\overset{\text{(ii)}}{=}0,
            \end{align*}
            where (i) is because $v^j_{t_j}(x_j)\in V^j_t(x)$, $\hat{w}_j^\perp\in (V^j_t(x))^\perp$, and then $v^j_{t_j}(x_j)\hat{w}_j^\perp=0$ for any $i\in\{1,...,k\}\cap\mathcal{I}^\ast(x_0)$, and (ii) is because 
            $$\lim_{j\to\infty}(t_j/h^2_i(x_j,y^\ast_{t_j}(x_j)))=(\lim_{j\to\infty}t_j)/h^2_i(\hat{x},y^\ast(\hat{x}))\leq 0/H^2=0,\quad \forall i\notin\mathcal{I}^\ast(x_0).$$
            Here $H$, we defined before, is the lower bound of $-h_i(x,y^\ast(x))$ with $x$ in $U$ and inactive index $i$;
        \item \textbf{Claim: $\lim_{j\to\infty}\hat{w}_j=0$.}

        \textbf{Proof.} We prove this by contradiction. If $\lim_{j\to \infty}\hat{w}_j=:\zeta\ne0$, it must be that $\zeta\in V^{k}(\hat{x})$ since $\hat{w}_j\in V_{t_j}^{k}(x_j)$ for any $j$. Consequently, 
            $$\lim_{j\to\infty}\|A^{k}_{t_j}(x_j)\hat{w}_j\|\overset{\text{(i)}}{=}\lim_{j\to\infty}\left\|\sum_{i\in\{1,...,k\}\cap\mathcal{I}^\ast(x_0)}\frac{1}{t_j}v_{t_j}^i(x_j) {v_{t_j}^i(x_j)}^\top\hat{w}_j\right\|=\infty,$$
            where (i) is because $\lim_{j\to\infty}(t_j/h^2_i(x_j,y^\ast_{t_j}(x_j)))=0$ for any $i\notin\mathcal{I}^\ast(x_0)$ which we have proven in the first Claim.
            
            By (\ref{eq:uniformofB}), we have $B_t(x)$ converges to $B(x)$ uniformly. Therefore, by the continuity of $B(x)$ in $x$ and Lemma \ref{lem:limitofsequence}, $\lim_{j\to\infty}B_{t_j}(x_j)=B(\hat{x})$,
            implying that $\lim_{j\to\infty}\|A^{k}_{t_j}(x_j) \hat{w}_j^\perp+B_{t_j}(x_j)\hat{w}_j+B_{t_j}(x_j)\hat{w}_j^\perp\|<\infty$. Thus $\lim_{j\to\infty}\|A^{k}_{t_j}(x_j)\hat{w}_j+A^{k}_{t_j}(x_j) \hat{w}_j^\perp+B_{t_j}(x_j)\hat{w}_j+B_{t_j}(x_j)\hat{w}_j^\perp\|=\infty$. However, $\lim_{j\to\infty} \|\hat{C}_j+\hat{C}_j^\perp\|=\|C_i(\hat{x})\|<\infty$ since $C_i(x)$ is a continuous map defined on $\mathcal{X}$. This induces the following contradiction
            \begin{align*}
            \infty>\lim_{j\to\infty} \|\hat{C}_j+\hat{C}_j^\perp\|=\lim_{j\to\infty}\|A^k_{t_j}(x_j)\hat{w}_j+A^k_{t_j}(x_j) \hat{w}_j^\perp+B_{t_j}(x_j)\hat{w}_j+B_{t_j}(x_j)\hat{w}_j^\perp\|=\infty.
            \end{align*}
            Therefore, we conclude $\lim_{j\to\infty}\hat{w}_j=0$;
        \item \textbf{Claim: $\lim_{j\to\infty}B_{t_j}(x_j)\hat{w}_j=0$.}

        \textbf{Proof.} Note that $\lim_{j\to\infty}B_{t_j}(x_j)=B(\hat{x})$ and $\lim_{j\to\infty}\hat{w}_j=0$, we can directly get $\lim_{j\to\infty}B_{t_j}(x_j)\hat{w}_j=0$.
    \end{enumerate}
    Given that $A^k_{t_j}(x_j)\hat{w}_j\in V^k_{t_j}(x_j)$ for any $j$, we can obtain
    \begin{align*}
        \hat{C}^\perp&=\lim_{j\to\infty}\hat{C}_j^\perp\\
        &=\lim_{j\to\infty}\proj_{\left(V^k_{t_j}(x_j)\right)^\perp}\left(A^k_{t_j}(x_j)\hat{w}_j+A^k_{t_j}(x_j)\hat{w}_j^\perp+B_{t_j}(x_j)\hat{w}_j+B_{t_j}(x_j)\hat{w}_j^\perp\right)\\
        &=\proj_{\left(V^k(\hat{x})\right)^\perp}\left(B(\hat{x})\hat{w}^\perp\right),
    \end{align*}
    where $\hat{x}$ is the limit of $x_j$. It follows that
    \begin{align*}
        \lim_{j\to\infty}\left(A^k_{t_j}(x_j)+B_{t_j}(x_j)\right)^{-1}(C_{t_j,i}(x_j))=\lim_{j\to\infty}w_{t_j}(x_j)=\lim_{j\to\infty}\hat{w}_j+\hat{w}_j^\perp=\hat{w}^\perp=\widetilde{\mathscr{B}}(\hat{x})(C_i(\hat{x})).
    \end{align*}
    This contradicts our assumption that
     \begin{align*}
        \lim_{j\to\infty}\left(A^k_{t_j}(x_j)+B_{t_j}(x_j)\right)^{-1}(C_{t_j,i}(x_j))\ne \widetilde{\mathscr{B}}(\hat{x})(C_i(\hat{x})).
    \end{align*}
    So we proved the first part that $\left(A^k_t(x)+B_t(x)\right)^{-1}C_t(x)$ converges uniformly.\\ 

    \noindent\textbf{Step 2:} Now we prove the second part that $(A_t(x)+B_t(x))^{-1}D_t(x)$ converges uniformly. Recall that
    \begin{align*}
        D_t(x)=\sum_{i=1}^k\frac{1}{t}\frac{t^2}{h^2_i(x,y_t^\ast(x))}\nabla_y h_i(x,y_t^\ast(x))\left(\nabla_x h_i(x,y_t^\ast(x))\right)^\top.
    \end{align*}
    We can obtain the uniform convergence of $(t/(-h_i(x,y_t^\ast(x))))\left(\nabla_x h_i(x,y_t^\ast(x))\right)^\top$ similar to (\ref{eq:uniformofv}). $(A_t(x)+B_t(x))^{-1}D_t(x)$ can be rewritten as
    \begin{align*}
        &\left(A^k_t(x)+B_t(x)\right)^{-1}D_t(x)\\
        =&\left((A^k_t(x)+B_t(x))^{-1}\frac{1}{t}\frac{t}{-h_j(x,y^\ast_t(x))}\nabla_y h_j(x,y_t^\ast(x))\right)\left(\frac{t}{-h_j(x,y^\ast_t(x))}\nabla_x h_i(x,y^\ast_t(x))\right)^\top
    \end{align*}
    Then it is sufficient to prove that 
    $$(A^k_t(x)+B_t(x))^{-1}\frac{1}{t}\frac{t}{-h_j(x,y^\ast_t(x))}\nabla_y h_j(x,y_t^\ast(x))$$
    converges uniformly for any $j$. Without loss of generality, in the following proof, we only consider the case $j=k$. Otherwise, simply swap the positions of the $j$-th and $k$-th constraints. 
    
    If $k$ is an inactive index, by the assumption at the beginning that $h_k(x,y^\ast(x))\leq -H$, we have
    $$\frac{1}{t}\frac{t}{-h_k(x,y^\ast_t(x))}\nabla_y h_k(x,y_t^\ast(x))=\frac{1}{-h_k(x,y^\ast_t(x))}\nabla_y h_k(x,y_t^\ast(x))$$
    converges to $(1/(-h_k(x,y^\ast(x))))\nabla_y h_k(x,y^\ast(x))$ uniformly. We can obtain the uniform convergence of the following term
    $$(A^k_t(x)+B_t(x))^{-1}\frac{1}{t}\frac{t}{-h_k(x,y^\ast_t(x))}\nabla_y h_k(x,y_t^\ast(x))$$ 
    by replacing $C_{t,i}(x)$ with $(1/(-h_k(x,y^\ast_t(x))))\nabla_y h_k(x,y_t^\ast(x))$ in Step 1, 
    
    If $k$ is an active index, following (\ref{eq:(A+B)D}) in the proof of lemma \ref{lem_existence}, we have
    \begin{align*}
        &(A^k_t(x)+B_t(x))^{-1}\frac{1}{t}\frac{t}{-h_k(x,y^\ast_t(x))}\nabla_y h_k(x,y_t^\ast(x))\\
        =&\frac{\left(A_t^{k-1}(x)+B_t(x)\right)^{-1}v^k_t(x)}{t+{v^k_t(x)}^\top\left(A_t^{k-1}(x)+B_t(x)\right)^{-1}v^k_t(x)},
    \end{align*}
    where we define $v^k_t(x)$ in (\ref{notation:vjt}) as
    $$v^k_t(x)=\frac{t}{-h_k(x,y^\ast_t(x))}\nabla_y h_k(x,y^\ast_t(x)).$$ 
    We just need to prove $\left(A_t^{k-1}(x)+B_t(x)\right)^{-1}v^k_t(x)$ converges uniformly, and 
    \begin{align*}
        \lim_{t\to 0}{v^k_t(x)}^\top\left(A_t^{k-1}(x)+B_t(x)\right)^{-1}v^k_t(x)\ne 0
    \end{align*}
    for any $x$ in $U$. The first result has already been proven in Step 1 just by replacing $k$ with $k-1$ and $C_{t,i}(x)$ with $v^k_t(x)$. The second result has also been proven in (\ref{eq:limpositive}) by replacing $\widetilde{j}$ with $k-1$. 
    
   To conclude, we proved that $\nabla_x y_t^\ast(x)$ converges uniformly in $U$.
\end{proof}

Next, we prove the relation of Jacobians under the strongly convex setting when $t$ approaches $0$.

\noindent\textbf{Proof of Theorem \ref{thm:mainnonlinear}:} We fix an \textbf{SCSC} point $x_0$ and prove this theorem at this point. It is sufficient to prove that $\lim_{t\to 0}\frac{\partial\left(y_t^\ast(x)\right)_i}{\partial x_j}=\frac{\partial\left(y^\ast(x)\right)_i}{\partial x_j}$ at \textbf{SCSC} point $x_0$ for any $i,j$. By Lemma \ref{lem:optimalitygapforpoint}, we have $y_t^\ast(x_0)$ converges to $y^\ast(x_0)$, and by Lemma \ref{lem:uniformconvergeofgradient} there exists a neighborhood $U$ of $x_0$ such that $\nabla_x y_t^\ast(x)$ converges uniformly in $U$, we can interchange the order of limit and derivatives
    \begin{align*}
        \lim_{t\to 0}\frac{\partial\left(y_t^\ast(x_0)\right)_i}{\partial x_j}=\frac{\partial\lim_{t\to 0}\left(y_t^\ast(x_0)\right)_i}{\partial x_j}=\frac{\partial\left(y^\ast(x_0)\right)_i}{\partial x_j}.
    \end{align*}
    Therefore, we completed the proof.

\section{Convergence Analysis of Algorithms}

\subsection{Proof of Proposition \ref{prop:strongly_convex}}\label{proof:prop:strongly_convex}
Through directly computing the Hessian, we get
    \begin{align*}
    \nabla^2_{yy}\widetilde{g}_t(x,y)&=\nabla^2_{yy}g(x,y)+t\sum_{i=1}^k\left(\frac{\nabla^2_{yy}h_i(x,y)}{-h_i(x,y)}+\frac{\nabla_y h_i(x,y)\nabla_y h_i(x,y)^\top}{h_i^2(x,y)}\right).
    \end{align*}
    
    If $g(x,y)$ is $\mu_g$-strongly convex in $y$, then 
    \begin{align*}
    \nabla^2_{yy}\widetilde{g}_t(x,y)\succeq \mu_g I,
    \end{align*}
    which means $\widetilde{g}_t(x,y)$ is $\mu_g$-strongly convex. 
    
    If $h_i(x,y)$ are linear in $y$ for any $i$, define $A(x)=(\nabla_y h_1(x,y),\cdots,\nabla_y h_k(x,y))$. In the proof in Proposition \ref{prop:diff_hyperfun}, we have proven that $A(x)A(x)^\top$ is positive for any $x$. The smallest eigenvalue $\sigma(x)$ is continuous depends on $x$, then we obtain that $\sigma=\min_{x\in\mathcal{X}}\sigma(x)>0$. It is also easy to see $H=\sup_{i\in\{1,...,k\},x\in\mathcal{X},y\in\mathcal{Y}_x}-h_i(x,y)<\infty$. Therefore
    \begin{align*}
        \nabla^2_{yy}\widetilde{g}_t(x,y)&\succeq t\sum_{i=1}^k\frac{\nabla_y h_i(x,y)\nabla_y h_i(x,y)^\top}{h_i^2(x,y)}\\
        &\succeq t\frac{\sum_{i=1}^k\nabla_y h_i(x,y)\nabla_y h_i(x,y)^\top}{H^2}\\
        &=t\frac{A(x)^\top A(x)}{H^2}\\
        &\succeq t\frac{\sigma}{H^2}I
    \end{align*}
which means $\widetilde{g}_t(x,y)$ is $t\frac{\sigma}{H^2}$-strongly convex.

\subsection{Proof of Theorem \ref{lem:local_M}}\label{proof:lem:local_M}

Denote $\overline{y}=\arg\min_{y\in\mathcal{Y}(x)}\{\max_{i\in\{1,...,k\}}h_i(x,y)\}$. From our assumption 
$$\arg\min_{y\in\mathcal{Y}(x)}\{\max_{i\in\{1,...,k\}}h_i(x,y)\}\leq-d,$$ 
it is clear that $h_i(x,\overline{y})\leq -d$ for any $i$. Our main idea of the proof is to reduce the problem to a one-dimensional case. 

If $y_t^\ast(x)=\overline{y}$, then let $m=d$. Then we directly have $h_i(x,y_t^\ast(x))\leq -d$ for any i.

If $y_t^\ast(x)\ne \overline{y}$. 
Consider the line $y_s=\overline{y}+s\boldsymbol{v}$, where $\boldsymbol{v}=\frac{y_t^\ast(x)-\overline{y}}{\|y_t^\ast(x)-\overline{y}\|}$ and define
\begin{align*}
        \widetilde{g}_x(s)&:=\widetilde{g}_t(x,y_s)\\
        &=g(x,\overline{y}+s\boldsymbol{v})-t\sum_{i=1}^k\log\left(-h_{i}(x,\overline{y}+s\boldsymbol{v})\right).
\end{align*}
 It is not hard to see $s^\ast:=\|y_t^\ast(x)-\overline{y}\|$ is the minimizer of $\widetilde{g}_x(s)$. We evaluate the bound of $h_j(x,y^\ast_t(x))$ for any index $j$ based on the range of values for $\frac{d}{d\boldsymbol{v}}h_{j}(x,y)|_{y=y_{s^\ast}}$, where $\frac{d}{d\boldsymbol{v}}h_{j}(x,y)|_{y=y_{s^\ast}}$ is the directional derivative of $h_j(x,\cdot)$ along direction $\boldsymbol{v}$ at point $y_{s^\ast}$.
 
 \noindent\textbf{Case 1: }If $\frac{d}{d\boldsymbol{v}}h_{j}(x,y)|_{y=y_{s^\ast}}\leq 0$, by convexity of $h_{j}(x,y_s)$ in $s$, $\frac{d}{d\boldsymbol{v}}h_{j}(x,y)|_{y=y_{s}}$ is monotonically increasing in $s$, thus we can get $\frac{d}{d\boldsymbol{v}}h_{j}(x,y)|_{y=y_{s}}\leq 0$ for any $s\in[0,s^\ast]$. Therefore, we have
\begin{align*}
    h_{j}(x,y_{s^\ast})-h_{j}(x,\overline{y})=\int_{0}^{s^\ast}\frac{d}{ds}h_{j}(x,\overline{y}+s\boldsymbol{v})ds\leq 0,
\end{align*}
implying $h_{j}(x,y_{s^\ast})\leq h_j(x,\overline{y})\leq-d$.

\noindent\textbf{Case 2: }If $0<\frac{d}{d\boldsymbol{v}}h_{j}(x,y)|_{y=y_{s^\ast}}\leq \frac{d}{4R}$, where $R$ is from Assumption \ref{assumption:general3}(\ref{assumption:general3(3)}), denote $s^\ast_j=\argmin_{\{s:y_s\in \mathcal{Y}(x)\}} h_j(x,y_{s})$. From the convexity of $h_j(x,y_s)$ in $s$, we know $s^\ast\geq s^\ast_j$. Otherwise, by convexity of $h_{j}(x,y_s)$ in $s$, $\frac{d}{d\boldsymbol{v}}h_{j}(x,y)|_{y=y_{s}}$ is monotonically increasing in $s$, $\frac{d}{d\boldsymbol{v}}h_{j}(x,y)|_{y=y_{s}}\geq 0$ for any $s\in[s^\ast,s^\ast_j]$, implying $h_j(x,y_{s^\ast})\leq h_j(x,y_{s^\ast_j})$. This contradicts $s^\ast_j=\argmin_{\{s:y_s\in \mathcal{Y}(x)\}} h_j(x,y_{s})$. Therefore, we obtain $$h_j(x,y_{s^\ast})-h_j(x,y_{s^\ast_j})=\int_{s_j^\ast}^{s^\ast}\frac{d}{ds}h_{j}(x,\overline{y}+s\boldsymbol{v})ds\overset{\text{(i)}}{\leq}(s^\ast-s^\ast_j)\frac{d}{d\boldsymbol{v}}h_{j}(x,y)|_{y=y_{s^\ast}}\overset{\text{(ii)}}{\leq}\frac{d}{2}. $$ In (i) we used the result that $\frac{d}{ds}h_{j}(x,\overline{y}+s\boldsymbol{v})\leq\frac{d}{d\boldsymbol{v}}h_{j}(x,y)|_{y=y_{s^\ast}}$ for any $s^\ast_j\leq s\leq s^\ast$. This is because, by convexity of $h_j(x,y_s)$ in $s$, $\frac{d}{d\boldsymbol{v}}h_{j}(x,y)|_{y=y_{s}}$ is monotonically increasing in $s$. In (ii) we use the fact that $s^\ast-s^\ast_j\leq2R$ from Assumption \ref{assumption:general3}(\ref{assumption:general3(3)}) and $\frac{d}{d\boldsymbol{v}}h_{j}(x,y)|_{y=y_{s^\ast}}\leq \frac{d}{4R}$. Note that $h_j(x,y_{s^\ast_j})\leq h_j(x,\overline{y})\leq -d$, we have $h_j(x,y^\ast_t(x))=h_j(x,y_{s^\ast})\leq -\frac{d}{2}$.

\noindent\textbf{Case 3: }If $\frac{d}{4R}<\frac{d}{d\boldsymbol{v}}h_{j}(x,y)|_{y=y_{s^\ast}}$, we divide the index set $I=\{1,...,k\}$ into two subsets, $I_1=\{i_1\in I:\frac{d}{d\boldsymbol{v}}h_{i_1}(x,y)|_{y=y_{s^\ast}}>0\}$ and $I_2=\{i_2\in I:\frac{d}{d\boldsymbol{v}}h_{i_2}(x,y)|_{y=y_{s^\ast}}\leq 0\}$. Note that the optimality condition of $\widetilde{g}_x(s)$ is $\frac{d}{ds} \widetilde{g}_x(s)|_{s=s^\ast}=0$, which can be expressed as
\begin{align}\label{eq:balanceequation}
    \frac{d}{d\boldsymbol{v}}g(x,y)|_{y=y_{s^\ast}}+t\sum_{i=1}^k\frac{\frac{d}{d\boldsymbol{v}}h_{i}(x,y)|_{y=y_{s^\ast}}}{-h_{i}(x,y_{s^\ast})}=0.
\end{align} 
Rearranging the terms of (\ref{eq:balanceequation}), note that $-h_j(x,y_{s^\ast})=-h_j(x,y^\ast_t(x))> 0$, we obtain
    \begin{align*}
        \frac{td}{-4Rh_j(x,y_{s^\ast})}<t\frac{\frac{d}{d\boldsymbol{v}}h_{j}(x,y)|_{y=y_{s^\ast}}}{-h_{j}(x,y_{s^\ast})}&=-\frac{d}{d\boldsymbol{v}}g(x,y)|_{y=y_{s^\ast}}-t\sum_{i\in I\setminus\{j\}}\frac{\frac{d}{d\boldsymbol{v}}h_{i}(x,y)|_{y=y_{s^\ast}}}{-h_{i}(x,y_{s^\ast})}\\
        &\overset{\text{(i)}}{\leq} L_g-t\sum_{i_1\in I_1\setminus\{j\}}\frac{\frac{d}{d\boldsymbol{v}}h_{i_1}(x,y)|_{y=y_{s^\ast}}}{-h_{i_1}(x,y_{s^\ast})}-t\sum_{i_2\in I_2\setminus\{j\}}\frac{\frac{d}{d\boldsymbol{v}}h_{i_2}(x,y)|_{y=y_{s^\ast}}}{-h_{i_2}(x,y_{s^\ast})}\\
        &\overset{\text{(ii)}}{\leq} L_g-t\sum_{i_2\in I_2\setminus\{j\}}\frac{\frac{d}{d\boldsymbol{v}}h_{i_2}(x,y)|_{y=y_{s^\ast}}}{-h_{i_2}(x,y_{s^\ast})}\\
        &\overset{\text{(iii)}}{\leq} L_g+\frac{tkL_h}{d}.
    \end{align*}
    In (i) we utilize $-\frac{d}{d\boldsymbol{v}}g(x,y)|_{y=y_{s^\ast}}\leq L_g$ by Assumption \ref{assumption:general2}(\ref{assumption:general2(3)}). (ii) is due to 
    $$-t\sum_{i_1\in I_1\setminus\{j\}}\frac{\frac{d}{d\boldsymbol{v}}h_{i_1}(x,y)|_{y=y_{s^\ast}}}{-h_{i_1}(x,y_{s^\ast})}\leq 0.$$
    In (iii) we use the fact that $I_2$ is designed such that for all $i_2\in I_2$, Case 1 is satisfied, so $-h_{i_2}(x,y_{s^\ast})\leq d$; further we have used the fact that for $i_2\in I_2$, $0<-\frac{d}{d\boldsymbol{v}}h_{i_2}(x,y)|_{y=y_{s^\ast}}\leq L_h$ where the second inequality comes from Assumption \ref{assumption:general2}(\ref{assumption:general2(6)}).  Note that $t\leq T$, we have
    \begin{align*}
        h_j(x,y_{s^\ast})\leq-t\frac{d^2}{4dRL_g+4RTkL_h}.
    \end{align*}

To conclude, the following holds for any $j$
\begin{align*}
    h_j(x,y_{s^\ast})\leq-\min\{t\frac{d^2}{4dRL_g+4RTkL_h},\frac{d}{2}\}.
\end{align*}

It is noteworthy that we do not need the convexity of $g(x,y)$ and $\widetilde{g}_t(x,y)$ in this proof.

\subsection{Proof of Proposition \ref{prop:lipschitz_smooth}}\label{proof:prop:lipschitz_smooth}

Through directly computing the Hessian, we get
    \begin{align*}
    \nabla^2_{yy}\widetilde{g}_t(x,y)&=\nabla^2_{yy}g(x,y)+t\sum_{i=1}^k\left(\frac{\nabla^2_{yy}h_i(x,y)}{-h_i(x,y)}+\frac{\nabla_y h_i(x,y)\nabla_y h_i(x,y)^\top}{h_i^2(x,y)}\right).
    \end{align*}
    Note that we only consider $y\in\mathcal{Y}_m(x)$, i.e. $h_i(x,y)\leq-m$ for any $i$. Therefore, we have the following estimate
    \begin{align*}
    \left\|\nabla^2_{yy}\widetilde{g}_t(x,y)\right\|&\leq\left\|\nabla^2_{yy}g(x,y)\right\|+t\sum_{i=1}^k\left(\frac{\left\|\nabla^2_{yy}h_i(x,y)\right\|}{-h_i(x,y)}+\frac{\left\|\nabla_y h_i(x,y)\right\|^2}{h_i^2(x,y)}\right)\\
    &\overset{\text{(i)}}{\leq}\overline{L}_g+\frac{tk\overline{L}_h}{m}+\frac{tkL_h^2}{m^2},
    \end{align*}
    where (i) is from the assumption that $h_i(x,y)\leq-m$, and Assumption \ref{assumption:general2} that $\left\|\nabla^2_{yy}g(x,y)\right\|\leq\overline{L}_g$, $\left\|\nabla^2_{yy}h_i(x,y)\right\|\leq\overline{L}_h$ and $\left\|\nabla_yh_i(x,y)\right\|\leq L_h$ for any $i$.
\subsection{Proof of Lemma \ref{lem:lipHessian}}\label{proof:lem:lipHessian}

By direct computing

    \resizebox{\textwidth}{!}{
    \begin{minipage}{\textwidth}
    \begin{align}
        &\left\|\nabla^2_{yy}\widetilde{g}_t(x_1,y_1)-\nabla^2_{yy}\widetilde{g}_t(x_2,y_2)\right\|\nonumber\\
        =&\left\|\nabla^2_{yy}g(x_1,y_1)+t\sum_{i=1}^k\left(\frac{\nabla^2_{yy}h_i(x_1,y_1)}{-h_i(x_1,y_1)}+\frac{\nabla_y h_i(x_1,y_1)\nabla_y h_i(x_1,y_1)^\top}{h_i^2(x_1,y_1)}\right)\right.\nonumber\\
        &-\left.\left(\nabla^2_{yy}g(x_2,y_2)+t\sum_{i=1}^k\left(\frac{\nabla^2_{yy}h_i(x_2,y_2)}{-h_i(x_2,y_2)}+\frac{\nabla_y h_i(x_2,y_2)\nabla_y h_i(x_2,y_2)^\top}{h_i^2(x_2,y_2)}\right)\right)\right\|\nonumber\\
        =&\left\|(\nabla^2_{yy}g(x_1,y_1)-\nabla^2_{yy}g(x_2,y_2))+t\sum_{i=1}^k\left(\frac{\nabla^2_{yy}h_i(x_1,y_1)}{-h_i(x_1,y_1)}-\frac{\nabla^2_{yy}h_i(x_2,y_2)}{-h_i(x_2,y_2)}\right)\right.\nonumber\\
        &\left.+t\sum_{i=1}^k\left(\frac{\nabla_y h_i(x_1,y_1)\nabla_y h_i(x_1,y_1)^\top}{h_i^2(x_1,y_1)}-\frac{\nabla_y h_i(x_2,y_2)\nabla_y h_i(x_2,y_2)^\top}{h_i^2(x_2,y_2)}\right)\right\|\nonumber\\
        =&\left\|(\nabla^2_{yy}g(x_1,y_1)-\nabla^2_{yy}g(x_2,y_2))+t\sum_{i=1}^k\frac{h_i(x_1,y_1)\nabla^2_{yy}h_i(x_2,y_2)-h_i(x_2,y_2)\nabla^2_{yy}h_i(x_1,y_1)}{h_i(x_1,y_1)h_i(x_2,y_2)}\right.\nonumber\\
        &\left.+t\sum_{i=1}^k\frac{h_i^2(x_2,y_2)\nabla_y h_i(x_1,y_1)\nabla_y h_i(x_1,y_1)^\top-h_i^2(x_1,y_1)\nabla_y h_i(x_2,y_2)\nabla_y h_i(x_2,y_2)^\top}{h_i^2(x_1,y_1)h_i^2(x_2,y_2)}\right\|.\label{eq:threeterms}
        \end{align}
        \end{minipage}}

        We estimate three terms of (\ref{eq:threeterms}) separately. For the first term, we utilize Assumption \ref{assumption:general2}(\ref{assumption:general2(5)}) and obtain
        \begin{align*}
            \left\|\nabla^2_{yy}g(x_1,y_1)-\nabla^2_{yy}g(x_2,y_2)\right\|\leq\overline{\overline{L}}_{g}\|(x_1,y_1)-(x_2,y_2)\|.
        \end{align*}
        
        For the second, we have

        \resizebox{\textwidth}{!}{
        \begin{minipage}{\textwidth}
        \begin{align}
            &\left\|t\sum_{i=1}^k\frac{h_i(x_1,y_1)\nabla^2_{yy}h_i(x_2,y_2)-h_i(x_2,y_2)\nabla^2_{yy}h_i(x_1,y_1)}{h_i(x_1,y_1)h_i(x_2,y_2)}\right\|\nonumber\\
            \leq&t\sum_{i=1}^k\left\|\frac{h_i(x_1,y_1)\nabla^2_{yy}h_i(x_2,y_2)-h_i(x_2,y_2)\nabla^2_{yy}h_i(x_1,y_1)}{h_i(x_1,y_1)h_i(x_2,y_2)}\right\|\nonumber\\
            =&t\sum_{i=1}^k\left\|\frac{h_i(x_1,y_1)\nabla^2_{yy}h_i(x_2,y_2)-h_i(x_1,y_1)\nabla^2_{yy}h_i(x_1,y_1)+h_i(x_1,y_1)\nabla^2_{yy}h_i(x_1,y_1)-h_i(x_2,y_2)\nabla^2_{yy}h_i(x_1,y_1)}{h_i(x_1,y_1)h_i(x_2,y_2)}\right\|\nonumber\\
            \leq&t\sum_{i=1}^k\left(\left\|\frac{\nabla^2_{yy}h_i(x_2,y_2)-\nabla^2_{yy}h_i(x_1,y_1)}{h_i(x_2,y_2)}\right\|+\left\|\frac{\nabla^2_{yy}h_i(x_1,y_1)\left(h_i(x_1,y_1)-h_i(x_2,y_2)\right)}{h_i(x_1,y_1)h_i(x_2,y_2)}\right\|\right)\nonumber\\
            \overset{\text{(i)}}{\leq}&tk\left(\frac{\overline{\overline{L}}_h}{m}+\frac{\overline{L}_hL_h}{m^2}\right)\|(x_1,y_1)-(x_2,y_2)\|,\label{eq:secondterm}
        \end{align}
        \end{minipage}}
        
        where (i) is because $h_i(x_1,y_1)$, $h_i(x_2,y_2)\leq -m$, $\|\nabla^2_{yy}h_i(x_2,y_2)-\nabla^2_{yy}h_i(x_1,y_1)\|\leq\overline{\overline{L}}_h\|(x_1,y_1)-(x_2,y_2)\|$, $\|\nabla^2_{yy}h_i(x_1,y_1)\|\leq\overline{L}_h$, and $|h_i(x_1,y_1)-h_i(x_2,y_2)|\leq L_h\|(x_1,y_1)-(x_2,y_2)\|$.
        
        For the final term, we have
        \begin{align}
            &\left\|t\sum_{i=1}^k\frac{h_i^2(x_2,y_2)\nabla_y h_i(x_1,y_1)\nabla_y h_i(x_1,y_1)^\top-h_i^2(x_1,y_1)\nabla_y h_i(x_2,y_2)\nabla_y h_i(x_2,y_2)^\top}{h_i^2(x_1,y_1)h_i^2(x_2,y_2)}\right\|\nonumber\\
            \leq&t\sum_{i=1}^k\left\|\frac{h_i^2(x_2,y_2)\nabla_y h_i(x_1,y_1)\nabla_y h_i(x_1,y_1)^\top-h_i^2(x_1,y_1)\nabla_y h_i(x_2,y_2)\nabla_y h_i(x_2,y_2)^\top}{h_i^2(x_1,y_1)h_i^2(x_2,y_2)}\right\|\nonumber\\
            =&t\sum_{i=1}^k\left(\left\|\frac{h_i^2(x_2,y_2)\nabla_y h_i(x_1,y_1)\nabla_y h_i(x_1,y_1)^\top-h_i^2(x_1,y_1)\nabla_y h_i(x_1,y_1)\nabla_y h_i(x_1,y_1)^\top}{h_i^2(x_1,y_1)h_i^2(x_2,y_2)}\right\|\right.\nonumber\\
            &+\left\|\frac{h_i^2(x_1,y_1)\nabla_y h_i(x_1,y_1)\nabla_y h_i(x_1,y_1)^\top-h_i^2(x_1,y_1)\nabla_y h_i(x_1,y_1)\nabla_y h_i(x_2,y_2)^\top}{h_i^2(x_1,y_1)h_i^2(x_2,y_2)}\right\|\nonumber\\
            &+\left.\left\|\frac{h_i^2(x_1,y_1)\nabla_y h_i(x_1,y_1)\nabla_y h_i(x_2,y_2)^\top-h_i^2(x_1,y_1)\nabla_y h_i(x_2,y_2)\nabla_y h_i(x_2,y_2)^\top}{h_i^2(x_1,y_1)h_i^2(x_2,y_2)}\right\|\right)\label{eq:complicatedterm}\\
            \overset{\text{(i)}}{\leq}&\left(2tk\frac{L_h^3}{m^3}+2tk\frac{L_h\overline{L}_h}{m^2}\right)\|(x_1,y_1)-(x_2,y_2)\|.\nonumber
        \end{align}
        In (i), when evaluating the first term, we use the following estimate
        \begin{align*}
            &\left\|\frac{h_i^2(x_2,y_2)\nabla_y h_i(x_1,y_1)\nabla_y h_i(x_1,y_1)^\top-h_i^2(x_1,y_1)\nabla_y h_i(x_1,y_1)\nabla_y h_i(x_1,y_1)^\top}{h_i^2(x_1,y_1)h_i^2(x_2,y_2)}\right\|\\
            =&\left\|\frac{(h_i(x_2,y_2)+h_i(x_1,y_1))(h_i(x_2,y_2)-h_i(x_1,y_1))\nabla_y h_i(x_1,y_1)\nabla_y h_i(x_1,y_1)^\top}{h_i^2(x_1,y_1)h_i^2(x_2,y_2)}\right\|\\
            =&\left\|\frac{(h_i(x_2,y_2)-h_i(x_1,y_1))\nabla_y h_i(x_1,y_1)\nabla_y h_i(x_1,y_1)^\top}{h_i^2(x_1,y_1)h_i(x_2,y_2)}+\frac{(h_i(x_2,y_2)-h_i(x_1,y_1))\nabla_y h_i(x_1,y_1)\nabla_y h_i(x_1,y_1)^\top}{h_i(x_1,y_1)h_i^2(x_2,y_2)}\right\|\\
            \leq&2tk\frac{L_h^3}{m^3}\|(x_1,y_1)-(x_2,y_2)\|.
        \end{align*}
        In conclusion, we have

        \resizebox{\textwidth}{!}{
        \begin{minipage}{\textwidth}
        \begin{align*}
        &\left\|\nabla^2_{yy}\widetilde{g}_t(x_1,y_1)-\nabla^2_{yy}\widetilde{g}_t(x_2,y_2)\right\|\\
        \leq&\left(\overline{\overline{L}}_{g}+tk\left(\frac{\overline{\overline{L}}_h}{m}+\frac{\overline{L}_hL_h}{m^2}+\frac{2L^3_h}{m^3}+\frac{2L_h\overline{L}_h}{m^2}\right)\right)\|(x_1,y_1)-(x_2,y_2)\|.
        \end{align*}
        \end{minipage}}
        
    We can also derive that $\|\nabla^2_{xy}\widetilde{g}_t(x_1,y_1)-\nabla^2_{xy}\widetilde{g}_t(x_2,y_2)\|\leq \overline{\overline{L}}_{\widetilde{g}_t,m}\|(x_1,y_1)-(x_2,y_2)\|$ using the same process.

\subsection{Proof of Lemma \ref{lem:local_lip}}\label{proof:lem:local_lip}
By Assumption \ref{assumption:general2}(\ref{assumption:general2(6)}), i.e. the Lipschitz continuity of $h_i(x,y)$, we have $h_i(x+\Delta x,y)\leq -\frac{d}{2}$ for any $i$ and $x+\Delta x$, where $\|\Delta x\|\leq\frac{d}{2L_h}$. This implies $\min_{y\in\mathcal{Y}(x)}\{\max_{i\in\{1,...,k\}}h_i(x+\Delta x,y)\}\}\leq -\frac{d}{2}$ for any point in the ball $B_{x}(\frac{d}{2L_h})$. According to Theorem \ref{lem:local_M}, we know $h_i(x+\Delta x,y^\ast_t(x+\Delta x))$ has a local upper bound denoted as $-m^{loc}$ in the ball $B_{x}(\frac{d}{2L_h})$. $m^{loc}$ can be computed by replacing $d$ with $\frac{d}{2}$ in Theorem \ref{lem:local_M}. This completes the first part of the proof.

Before proceeding with the proof, we provide the following estimate
\begin{align}
        \left\|\nabla^2_{yx}\widetilde{g}_t(x,y)\right\|&\leq\left\|\nabla^2_{yx}g(x,y)\right\|+t\sum_{i=1}^k\left(\frac{\left\|\nabla^2_{yx}h_i(x,y)\right\|}{-h_i(x,y)}+\frac{\left\|\nabla_y h_i(x,y)\right\|\left\|\nabla_x h_i(x,y)\right\|}{h_i^2(x,y)}\right)\nonumber\\
    &\leq\overline{L}_g+\frac{tk\overline{L}_h}{m^{loc}}+\frac{tkL_h^2}{\left(m^{loc}\right)^2}\label{eq:mixhessian}
    \end{align}
    similar as the proof in Proposition \ref{prop:lipschitz_smooth}.

By direct computation, for any $x_1$, $x_2$ in the ball $B_{x}(\frac{d}{2L_h})$, we have

    \resizebox{\textwidth}{!}{
    \begin{minipage}{\textwidth}
    \begin{align}
        &\left\|\nabla_x \widetilde{\phi}_t(x_1)-\nabla_x \widetilde{\phi}_t(x_2)\right\|\nonumber\\
        =&\left\|\nabla_x f(x_1,y_t^\ast(x_1))-\nabla_{xy}^2\widetilde{g}_t(x_1,y_t^\ast(x_1))(\nabla_{yy}^2\widetilde{g}_t(x_1,y_t^\ast(x_1)))^{-1}\nabla_y f(x_1,y_t^\ast(x_1))\right.\nonumber\\
        &\left.-\nabla_x f(x_2,y_t^\ast(x_2))+\nabla_{xy}^2\widetilde{g}_t(x_2,y_t^\ast(x_2))(\nabla_{yy}^2\widetilde{g}_t(x_2,y_t^\ast(x_2)))^{-1}\nabla_y f(x_2,y_t^\ast(x_2))\right\|\nonumber\\
        \leq&\left\|(\nabla_x f(x_1,y_t^\ast(x_1))-\nabla_x f(x_2,y_t^\ast(x_2)))\right\|\nonumber\\
        &+\left\|\nabla_{xy}^2\widetilde{g}_t(x_2,y_t^\ast(x_2))(\nabla_{yy}^2\widetilde{g}_t(x_2,y_t^\ast(x_2)))^{-1}(\nabla_y f(x_2,y_t^\ast(x_2))-\nabla_y f(x_1,y_t^\ast(x_1)))\right\|\nonumber\\
        &+\left\|\nabla_{xy}^2\widetilde{g}_t(x_2,y_t^\ast(x_2))((\nabla_{yy}^2\widetilde{g}_t(x_2,y_t^\ast(x_2)))^{-1}-(\nabla_{yy}^2\widetilde{g}_t(x_1,y_t^\ast(x_1)))^{-1})\nabla_y f(x_1,y_t^\ast(x_1))\right\|\nonumber\\
        &+\left\|(\nabla_{xy}^2\widetilde{g}_t(x_2,y_t^\ast(x_2))-\nabla_{xy}^2\widetilde{g}_t(x_1,y_t^\ast(x_1)))(\nabla_{yy}^2\widetilde{g}_t(x_1,y_t^\ast(x_1)))^{-1}\nabla_y f(x_1,y_t^\ast(x_1))\right\|\nonumber\\
        \overset{\text{(i)}}{\leq}&\left(\overline{L}_f+\overline{L}_f\frac{1}{\mu_{\widetilde{g}_t}}\left(\overline{L}_g+\frac{tk\overline{L}_h}{m^{loc}}+\frac{tkL_h^2}{(m^{loc})^2}\right)+L_f\left(\frac{1}{\mu_{\widetilde{g}_t}}\right)^2\overline{\overline{L}}_{\widetilde{g}_t,m^{loc}}\left(\overline{L}_g+\frac{tk\overline{L}_h}{m^{loc}}+\frac{tkL_h^2}{(m^{loc})^2}\right)+L_f\frac{1}{\mu_{\widetilde{g}_t}}\overline{\overline{L}}_{\widetilde{g}_t,m^{loc}}\right)\nonumber\\
        &\times\|(x_1,y_t^\ast(x_1))-(x_2,y_t^\ast(x_2))\|\label{eq:finalresult}.
    \end{align}
    \end{minipage}}
    
    In (i) we estimate $\|(\nabla_{yy}^2\widetilde{g}_t(x_2,y_t^\ast(x_2)))^{-1}-(\nabla_{yy}^2\widetilde{g}_t(x_1,y_t^\ast(x_1)))^{-1}\|$ using the following inequality
    \begin{align*}
        \|A^{-1}-B^{-1}\|=\|A^{-1}(A-B)B^{-1}\|\leq\|A^{-1}\|\cdot\|B^{-1}\|\cdot\|A-B\|,
    \end{align*}
    and $\|(\nabla_{yy}^2\widetilde{g}_t(x_2,y_t^\ast(x_2)))^{-1}\|\leq 1/\mu_{\widetilde{g}_t}$ from Proposition \ref{prop:strongly_convex}. 
    
    Now we need to bound $\|(x_1,y_t^\ast(x_1))-(x_2,y_t^\ast(x_2))\|$. Note that $\|y^\ast_t(x_1)-y^\ast_t(x_2)\|\leq\|\nabla_x y_t^\ast(x)\|_{loc}\|x_1-x_2\|$, where $\|\nabla_x y_t^\ast(x)\|_{loc}$ is the upper bound of norm of the Jacobian in the ball $B_x(\frac{d}{2L_h})$. We need to evaluate $\|\nabla_x y_t^\ast(x)\|_{loc}$. By optimal condition of $\widetilde{g}_t(x,y)$, we know $\nabla_y\widetilde{g}_t(x,y^\ast_t(x))$=0. This implies 
    $$0=\nabla_x(\nabla_y\widetilde{g}_t(x,y^\ast_t(x)))=\nabla^2_{xy}\widetilde{g}_t(x,y^\ast_t(x))+(\nabla_x y^\ast_t(x))^\top\nabla^2_{yy} \widetilde{g}_t(x,y^\ast_t(x)).$$
    This following
    \begin{align*}
        \|\nabla_x y_t^\ast(x)\|=&\|\nabla^2_{xy}\widetilde{g}_t(x,y)(\nabla^2_{yy}\widetilde{g}_t(x,y))^{-1}\|\\
        =&\|\nabla^2_{xy}\widetilde{g}_t(x,y)\|\|\nabla^2_{yy}\widetilde{g}_t(x,y)\|^{-1}\\
        \overset{\text{(i)}}{\leq}&\frac{1}{\mu_{\widetilde{g}_t}}\left(\overline{L}_g+\frac{tk\overline{L}_h}{{m^{loc}}}+\frac{tkL_h^2}{(m^{loc})^2}\right),
    \end{align*}
    where in (i) we use $\|(\nabla_{yy}^2\widetilde{g}_t(x_2,y_t^\ast(x_2)))^{-1}\|\leq 1/\mu_{\widetilde{g}_t}$ and 
    \begin{align*}
        \left\|\nabla^2_{yx}\widetilde{g}_t(x,y)\right\|\leq\overline{L}_g+\frac{tk\overline{L}_h}{m^{loc}}+\frac{tkL_h^2}{\left(m^{loc}\right)^2}
    \end{align*}
    Therefore, we obtain
    \begin{align*}
        \|y^\ast_t(x_1)-y^\ast_t(x_2)\|\leq\frac{1}{\mu_{\widetilde{g}_t}}\left(\overline{L}_g+\frac{tk\overline{L}_h}{m^{loc}}+\frac{tkL_h^2}{(m^{loc})^2}\right)\|x_1-x_2\|.
    \end{align*}
    To conclude, we obtain

    \resizebox{\textwidth}{!}{
    \begin{minipage}{\textwidth}
    \begin{align*}
        &\|\nabla_x \widetilde{\phi}_t(x_1)-\nabla_x \widetilde{\phi}_t(x_2)\|\\
        \leq&\left(\overline{L}_f+\overline{L}_f\frac{1}{\mu_{\widetilde{g}_t}}\left(\overline{L}_g+\frac{tk\overline{L}_h}{m^{loc}}+\frac{tkL_h^2}{(m^{loc})^2}\right)+L_f\left(\frac{1}{\mu_{\widetilde{g}_t}}\right)^2\overline{\overline{L}}_{\widetilde{g}_t,m^{loc}}\left(\overline{L}_g+\frac{tk\overline{L}_h}{m^{loc}}+\frac{tkL_h^2}{(m^{loc})^2}\right)+L_f\frac{1}{\mu_{\widetilde{g}_t}}\overline{\overline{L}}_{\widetilde{g}_t,m^{loc}}\right)\nonumber\\
        &\times\|(x_1,y_t^\ast(x_1))-(x_2,y_t^\ast(x_2))\|\\
        \leq&\left(\overline{L}_f+\overline{L}_f\frac{1}{\mu_{\widetilde{g}_t}}\left(\overline{L}_g+\frac{tk\overline{L}_h}{{m^{loc}}}+\frac{tkL_h^2}{(m^{loc})^2}\right)+L_f\left(\frac{1}{\mu_{\widetilde{g}_t}}\right)^2\overline{\overline{L}}_{\widetilde{g}_t,m^{loc}}\left(\overline{L}_g+\frac{tk\overline{L}_h}{m^{loc}}+\frac{tkL_h^2}{(m^{loc})^2}\right)+L_f\frac{1}{\mu_{\widetilde{g}_t}}\overline{\overline{L}}_{\widetilde{g}_t,m^{loc}}\right)\\
        &\times\left(1+\frac{1}{\mu_{\widetilde{g}_t}}\left(\overline{L}_g+\frac{tk\overline{L}_h}{{m^{loc}}}+\frac{tkL_h^2}{(m^{loc})^2}\right)\right)\|x_1-x_2\|.
    \end{align*}
    \end{minipage}}
    
    The proof is completed.

\subsection{Proof of Lemma \ref{lem:lip_global_bound}}\label{proof:lem:lip_global_bound}
Now we apply Lemma \ref{lem:local_lip} to the reference point $x_s$. Replacing $d$ in Lemma \ref{lem:local_lip} with $d_s$, it follows that for any $x$ in the ball $B_{x_s}(d_s/(2L_h))$, we have $h_i(x,y^\ast_t(x))\leq -m^{loc}$. As stated in the first paragraph of proof in Lemma \ref{lem:local_lip}, the constant $-m^{loc}$ can be computed by replacing $d$ in Theorem \ref{lem:local_M} with $d_s/2$. Here we use the notation $m(d_s/2):=m^{loc}$ to emphasize the dependency of $m$ on $d_s$. Since the feasibility check must terminate when $d$ is between $\frac{D}{2}$ to $D$, $d_s$ is greater than $D/2$ for any $s$. From Theorem \ref{lem:local_M}, we know that $m(d)$ is monotonically increasing with respect to $d$. It follows that $m^{loc}=m(d_s/2)\geq m(D/4)=:M^\ast$. Note that
$$\overline{\overline{L}}_{\widetilde{g}_t,m}=\overline{\overline{L}}_{g}+tk\left(\frac{\overline{\overline{L}}_h}{m}+\frac{\overline{L}_hL_h}{m^2}+\frac{2L^3_h}{m^3}+\frac{2L_h\overline{L}_h}{m^2}\right),$$
which is monotonically decreasing with respect to $m$, we have $\overline{\overline{L}}_{\widetilde{g}_t,m^{loc}}\leq \overline{\overline{L}}_{\widetilde{g}_t,M^\ast}$. Therefore, the following holds

    \resizebox{\textwidth}{!}{
    \begin{minipage}{\textwidth}
    \begin{align*}
        \overline{L}^{loc}_{\widetilde{\phi}_t}=&\left(\overline{L}_f+\overline{L}_f\frac{1}{\mu_{\widetilde{g}_t}}\left(\overline{L}_g+\frac{tk\overline{L}_h}{{m^{loc}}}+\frac{tkL_h^2}{(m^{loc})^2}\right)+L_f\left(\frac{1}{\mu_{\widetilde{g}_t}}\right)^2\overline{\overline{L}}_{\widetilde{g}_t,m^{loc}}\left(\overline{L}_g+\frac{tk\overline{L}_h}{m^{loc}}+\frac{tkL_h^2}{(m^{loc})^2}\right)+L_f\frac{1}{\mu_{\widetilde{g}_t}}\overline{\overline{L}}_{\widetilde{g}_t,m^{loc}}\right)\\
        &\times\left(1+\frac{1}{\mu_{\widetilde{g}_t}}\left(\overline{L}_g+\frac{tk\overline{L}_h}{{m^{loc}}}+\frac{tkL_h^2}{(m^{loc})^2}\right)\right)\\
        \leq&\left(\overline{L}_f+\overline{L}_f\frac{1}{\mu_{\widetilde{g}_t}}\left(\overline{L}_g+\frac{tk\overline{L}_h}{{M^\ast}}+\frac{tkL_h^2}{(M^\ast)^2}\right)+L_f\left(\frac{1}{\mu_{\widetilde{g}_t}}\right)^2\overline{\overline{L}}_{\widetilde{g}_t,M^\ast}\left(\overline{L}_g+\frac{tk\overline{L}_h}{M^\ast}+\frac{tkL_h^2}{(M^\ast)^2}\right)+L_f\frac{1}{\mu_{\widetilde{g}_t}}\overline{\overline{L}}_{\widetilde{g}_t,M^\ast}\right)\\
        &\times\left(1+\frac{1}{\mu_{\widetilde{g}_t}}\left(\overline{L}_g+\frac{tk\overline{L}_h}{{M^\ast}}+\frac{tkL_h^2}{(M^\ast)^2}\right)\right)\\
        =:&\overline{L}_{\widetilde{\phi}_t}.
    \end{align*}
    \end{minipage}}
    
    By the design of $m$ in Theorem \ref{lem:local_M}, we have $M^\ast=m(D/4)=\mathcal{O}(t)$. Note that 
    \begin{align}\label{eq:Mast}
        \overline{\overline{L}}_{\widetilde{g}_t,M^\ast}=\overline{\overline{L}}_{g}+tk\left(\frac{\overline{\overline{L}}_h}{M^\ast}+\frac{\overline{L}_hL_h}{\left(M^\ast\right)^2}+\frac{2L^3_h}{\left(M^\ast\right)^3}+\frac{2L_h\overline{L}_h}{\left(M^\ast\right)^2}\right)=\mathcal{O}(1/t^2),
    \end{align}
    we conclude $\overline{L}_{\widetilde{\phi}_t}=\mathcal{O}(1/t^4)$.

\subsection{Proof of Lemma \ref{lemma:approx_error_hypergrad}}\label{proof:lemma:approx_error_hypergrad}

According to Theorem \ref{thm:convergence_rage_alg_2}, we have $\|y^\ast_t(x_s)-\hat{y}_s\|\leq\epsilon$. Note that $h_i(x,y^\ast_t(x))\leq -m_s$ and $h_i(x,\hat{y}_s)\leq -m_s$ for any $i$. By direct computation

    \resizebox{\textwidth}{!}{
    \begin{minipage}{\textwidth}
    \begin{align}
        &\left\|\nabla_x \widetilde{\phi}_t(x_s)-\hat{\nabla}_x\widetilde{\phi}_t(x_s)\right\|\nonumber\\
        =&\left\|\nabla_x f(x_s,y_t^\ast(x_s))-\nabla_{xy}^2\widetilde{g}_t(x_s,y_t^\ast(x_s))(\nabla_{yy}^2\widetilde{g}_t(x_s,y_t^\ast(x_s)))^{-1}\nabla_y f(x_s,y_t^\ast(x_s))\right.\nonumber\\
        &\left.-\nabla_x f(x_s,\hat{y}_s)+\nabla_{xy}^2\widetilde{g}_t(x_s,\hat{y}_s)(\nabla_{yy}^2\widetilde{g}_t(x_s,\hat{y}_s))^{-1}\nabla_y f(x_s,\hat{y}_s)\right\|\nonumber\\
        \leq&\left\|(\nabla_x f(x_s,y_t^\ast(x))-\nabla_x f(x_s,\hat{y}_s))\right\|\nonumber\\
        &+\left\|\nabla_{xy}^2\widetilde{g}_t(x_s,\hat{y}_s)(\nabla_{yy}^2\widetilde{g}_t(x_s,\hat{y}_s))^{-1}(\nabla_y f(x_s,\hat{y}_s)-\nabla_y f(x_s,y_t^\ast(x_s)))\right\|\nonumber\\
        &+\left\|\nabla_{xy}^2\widetilde{g}_t(x_s,\hat{y}_s)((\nabla_{yy}^2\widetilde{g}_t(x_s,\hat{y}_s))^{-1}-(\nabla_{yy}^2\widetilde{g}_t(x_s,y_t^\ast(x_s)))^{-1})\nabla_y f(x_s,y_t^\ast(x_s))\right\|\nonumber\\
        &+\left\|(\nabla_{xy}^2\widetilde{g}_t(x_s,\hat{y}_s)-\nabla_{xy}^2\widetilde{g}_t(x_s,y_t^\ast(x_s)))(\nabla_{yy}^2\widetilde{g}_t(x_s,\hat{y}_s))^{-1}\nabla_y f(x_s,y_t^\ast(x_s))\right\|\nonumber\\
        \leq&\left(\overline{L}_f+\overline{L}_f\frac{1}{\mu_{\widetilde{g}_t}}\left(\overline{L}_g+\frac{tk\overline{L}_h}{m_s}+\frac{tkL_h^2}{m_s^2}\right)+L_f\left(\frac{1}{\mu_{\widetilde{g}_t}}\right)^2\overline{\overline{L}}_{\widetilde{g}_t,m_s}\left(\overline{L}_g+\frac{tk\overline{L}_h}{m_s}+\frac{tkL_h^2}{m_s^2}\right)+L_f\frac{1}{\mu_{\widetilde{g}_t}}\overline{\overline{L}}_{\widetilde{g}_t,m_s}\right)\nonumber\\
        &\times\|(x,y^\ast_t(x_s))-(x,\hat{y}_s)\|\nonumber\\
        \leq&\left(\overline{L}_f+\overline{L}_f\frac{1}{\mu_{\widetilde{g}_t}}\left(\overline{L}_g+\frac{tk\overline{L}_h}{m_s}+\frac{tkL_h^2}{m_s^2}\right)+L_f\left(\frac{1}{\mu_{\widetilde{g}_t}}\right)^2\overline{\overline{L}}_{\widetilde{g}_t,m_s}\left(\overline{L}_g+\frac{tk\overline{L}_h}{m_s}+\frac{tkL_h^2}{m_s^2}\right)+L_f\frac{1}{\mu_{\widetilde{g}_t}}\overline{\overline{L}}_{\widetilde{g}_t,m_s}\right)\epsilon_s\nonumber\\
        =&\overline{L}'_{\widetilde{\phi}_t,m_s}\epsilon_s.\label{eq:approximated gradient}
    \end{align}
    \end{minipage}}
    
\subsection{Proof of Theorem \ref{thm:total_rate}}\label{proof:thm:total_rate}
Before proceeding, we establish the following result 

\begin{lemma}\label{lem:upperboundofapproxgradient}
    The coefficient $\overline{L}'_{\widetilde{\phi}_t,m_s}$ defined in Lemma \ref{lemma:approx_error_hypergrad} has an upper bound 
    
    \resizebox{\textwidth}{!}{$\overline{L}'_{\widetilde{\phi}_t}=\overline{L}_f+\overline{L}_f\frac{1}{\mu_{\widetilde{g}_t}}\left(\overline{L}_g+\frac{tk\overline{L}_h}{M}+\frac{tkL_h^2}{M^2}\right)+L_f\left(\frac{1}{\mu_{\widetilde{g}_t}}\right)^2\overline{\overline{L}}_{\widetilde{g}_t,M}\left(\overline{L}_g+\frac{tk\overline{L}_h}{M}+\frac{tkL_h^2}{M^2}\right)+L_f\frac{1}{\mu_{\widetilde{g}_t}}\overline{\overline{L}}_{\widetilde{g}_t,M}.$}
    Here $M$ is from Corollary  \ref{remark:upper_bound_of_h_at_approximate_point}, and  $\overline{\overline{L}}_{\widetilde{g}_t,M}$ can be computed by replacing $m$ from Lemma \ref{lem:lipHessian} with $M$.
\end{lemma}

\begin{proof}
    From Corollary \ref{remark:upper_bound_of_h_at_approximate_point}, we have $m_s\geq M$. In view of Lemma \ref{lemma:approx_error_hypergrad}, $\overline{\overline{L}}_{\widetilde{g}_t,m}$ is monotonically increases with respect to $m$, so $\overline{\overline{L}}_{\widetilde{g}_t,m_s}\leq\overline{\overline{L}}_{\widetilde{g}_t,M}$. Therefore, we obtain

    \resizebox{\textwidth}{!}{
    \begin{minipage}{\textwidth}
    \begin{align*}
        \overline{L}'_{\widetilde{\phi}_t,m_s}&=\overline{L}_f+\overline{L}_f\frac{1}{\mu_{\widetilde{g}_t}}\left(\overline{L}_g+\frac{tk\overline{L}_h}{m_s}+\frac{tkL_h^2}{m_s^2}\right)+L_f\left(\frac{1}{\mu_{\widetilde{g}_t}}\right)^2\overline{\overline{L}}_{\widetilde{g}_t,m_s}\left(\overline{L}_g+\frac{tk\overline{L}_h}{m_s}+\frac{tkL_h^2}{m_s^2}\right)+L_f\frac{1}{\mu_{\widetilde{g}_t}}\overline{\overline{L}}_{\widetilde{g}_t,m_s}\\
        &\leq\overline{L}_f+\overline{L}_f\frac{1}{\mu_{\widetilde{g}_t}}\left(\overline{L}_g+\frac{tk\overline{L}_h}{M}+\frac{tkL_h^2}{M^2}\right)+L_f\left(\frac{1}{\mu_{\widetilde{g}_t}}\right)^2\overline{\overline{L}}_{\widetilde{g}_t,M}\left(\overline{L}_g+\frac{tk\overline{L}_h}{M}+\frac{tkL_h^2}{M^2}\right)+L_f\frac{1}{\mu_{\widetilde{g}_t}}\overline{\overline{L}}_{\widetilde{g}_t,M}\\
        &=\overline{L}'_{\widetilde{\phi}_t}.
    \end{align*}
    \end{minipage}}
    
    This complete the proof.
\end{proof}

Next, we provide the proof of Theorem \ref{thm:total_rate}. For any point $x_s\in\mathcal{X}$, the lower-level algorithm will guarantee (by Theorem \ref{thm:convergence_rage_alg_2}):
    \begin{align}\label{eq:approxerror}
        \|\hat{y}_s - y_{t}^*(x_s)\|\leq\epsilon_s
    \end{align}
    within $\mathcal{O}(\kappa\log(1/\epsilon_s))$ numbers of lower-level gradient oracles, where $\hat{y}_s$ is the output of the lower-level algorithm. According to Lemma \ref{lemma:approx_error_hypergrad}, $1/\epsilon_s=4\overline{L}'_{\widetilde{\phi}_t,m_s}/\epsilon\leq 4\overline{L}'_{\widetilde{\phi}_t}/\epsilon$, the iteration numbers of lower-level solver is at most $\mathcal{O}(\kappa\log(4\overline{L}'_{\widetilde{\phi}_t}/\epsilon))$.

    To proceed with the proof, we first give an upper bound of $\|\hat{\nabla}_x\widetilde{\phi}_t(x_s)\|$. Note that
    \begin{align*}
        \hat{\nabla}_x\widetilde{\phi}_t(x_s)=\nabla_x f(x_s,\hat{y}_s)-\nabla_{xy}^2\widetilde{g}_t(x_s,\hat{y}_s)(\nabla_{yy}^2\widetilde{g}_t(x_s,\hat{y}_s))^{-1}\nabla_y f(x_s,\hat{y}_s),
    \end{align*}
    and we have (by Assumption \ref{assumption:general2} and Proposition \ref{prop:strongly_convex})
    \begin{equation*}
    \begin{aligned}
        &\|\nabla_x f(x_s,\hat{y}_s)\|\leq L_f,\ \|\nabla_y f(x_s,\hat{y}_s)\|\leq L_f,\ \|\nabla_y h(x_s,\hat{y}_s)\|\leq L_h; \\
        &\|\nabla_{xy}^2\widetilde{g}_t(x_s,\hat{y}_s)\|\leq\overline{L}_g,\ \|\nabla_{xy}^2\widetilde{h}_t(x_s,\hat{y}_s)\|\leq\overline{L}_h,\ \|(\nabla_{yy}^2\widetilde{g}_t(x_s,\hat{y}_s))^{-1}\|\leq 1/\mu_{\widetilde{g}_t}.
    \end{aligned}
    \end{equation*}
    Further, note that
    $$
\nabla^2_{xy}\widetilde{g}_t(x_s,\hat{y}_s)=\nabla^2_{xy}g(x_s,\hat{y}_s)+t\sum_{i=1}^k\left(\frac{\nabla^2_{xy}h_i(x_s,\hat{y}_s)}{-h_i(x_s,\hat{y}_s)}+\frac{\nabla_x h_i(x_s,\hat{y}_s)\nabla_y h_i(x_s,\hat{y}_s)^\top}{h_i^2(x_s,\hat{y}_s)}\right).
    $$
    In view of Corollary  \ref{remark:upper_bound_of_h_at_approximate_point} we know $h_i(x_s,\hat{y}_s)\leq -M$, thus
    $$
    \|\nabla^2_{xy}\widetilde{g}_t(x_s,\hat{y}_s)\|\leq \overline{L}_g+ t k \frac{\overline{L}_h}{M} + t k \frac{L_h^2}{M^2}.
    $$
    Combining the previous estimates, we reach to the conclusion that
    \begin{equation*}
        \|\hat{\nabla}_x\widetilde{\phi}_t(x_s)\| \leq L_f + \frac{L_f}{\mu_{\widetilde{g}_t}}\left(\overline{L}_g+ t k \frac{\overline{L}_h}{M} + t k \frac{L_h^2}{M^2}\right).
    \end{equation*}
    
    From Lemma \ref{lem:local_lip}, we know that $\widetilde{\phi}_t$ is locally Lipschitz smooth with parameter $\overline{L}_{\widetilde{\phi}_t,s}$ for any $x$ such that $\|x - x_{s}\|\leq d_s/(2 L_h)$. Since we take $\eta_s=\min\{ d_s/(2L_h) \cdot 1/\|\hat{\nabla}_x\widetilde{\phi}_t(x_s)\|, 1/\overline{L}_{\widetilde{\phi}_t,s}\}$ so that $\|x_{s+1} - x_{s}\|\leq d_s/(2 L_h)$, therefore we know from Lipschitzness that
    \begin{equation*}
        \widetilde{\phi}_t(x_{s+1}) - \widetilde{\phi}_t(x_{s}) \leq \nabla \widetilde{\phi}_t(x_s)^\top (x_{s+1} - x_{s}) + \frac{\overline{L}_{\widetilde{\phi}_t, s}}{2}\|x_{s+1} - x_{s}\|^2
    \end{equation*}
    where $\overline{L}_{\widetilde{\phi}_t, s}$ is specified by Lemma \ref{lem:local_lip}, i.e.
    \begin{align*}
    \overline{L}_{\widetilde{\phi}_t, s} =& (\overline{L}_f+\overline{L}_f\frac{1}{\mu_{\widetilde{g}_t}}(\overline{L}_g+\frac{tk\overline{L}_h}{m_{s}}+\frac{tkL_h^2}{{m_{s}}^2})+L_f(\frac{1}{\mu_{\widetilde{g}_t}})^2\overline{\overline{L}}_{\widetilde{g}_t,m_{s}}(\overline{L}_g+\frac{tk\overline{L}_h}{m_{s}}+\frac{tkL_h^2}{{m_{s}}^2})+L_f\frac{1}{\mu_{\widetilde{g}_t}}\overline{\overline{L}}_{\widetilde{g}_t,m_{s}})\\
    &\times(1+\frac{1}{\mu_{\widetilde{g}_t}}(\overline{L}_g+\frac{tk\overline{L}_h}{{m_{s}}}+\frac{tkL_h^2}{{m_{s}}^2}))
    \end{align*}
    where $m_{s}$ and $d_s$ are outputs of Algorithm \ref{algo:fo_lower_2}, also $\overline{L}_{\widetilde{\phi}_t, s}$ is upper bounded by $\overline{L}_{\widetilde{\phi}_t}$ (see Lemma \ref{lem:lip_global_bound}).
    
    Therefore we have
    \begin{equation}\label{eq:thm_main_temp1}
    \begin{aligned}
        \widetilde{\phi}_t(x_{s+1}) - \widetilde{\phi}_t(x_{s}) &\leq \nabla_x \widetilde{\phi}_t(x_s)^\top (x_{s+1} - x_{s}) + \frac{\overline{L}_{\widetilde{\phi}_t, s}}{2}\|x_{s+1} - x_{s}\|^2 \\
        &\overset{\text{(i)}}{\leq} (\frac{\overline{L}_{\widetilde{\phi}_t, s}}{2} - \frac{1}{\eta_s}) \|x_{s+1} - x_{s}\|^2 + (\nabla \widetilde{\phi}_t(x_s) - \hat{\nabla} \widetilde{\phi}_t(x_s))^\top (x_{s+1} - x_{s}) \\
        &\leq (\frac{\overline{L}_{\widetilde{\phi}_t, s}}{2} - \frac{1}{\eta_s}) \|x_{s+1} - x_{s}\|^2 + \eta_s\|\nabla \widetilde{\phi}_t(x_s) - \hat{\nabla} \widetilde{\phi}_t(x_s)\|^2 + \frac{1}{4 \eta_s}\|x_{s+1} - x_{s}\|^2 \\
        &\overset{\text{(ii)}}{\leq} -\frac{1}{4 \eta_s}\|x_{s+1} - x_{s}\|^2 + \eta_s\left(\frac{\epsilon}{4}\right)^2
    \end{aligned}
    \end{equation}
    where (i) is by the property of projection, namely
    $$
    (x_s-\eta_s\hat{\nabla}_x\widetilde{\phi}_t(x_s)-x_{s+1})^\top (x-x_{s+1})\leq 0,\ \forall x\in\mathcal{X}
    $$
    and taking $x=x_s$ gives
    $$
    \hat{\nabla}_x\widetilde{\phi}_t(x_s)^\top (x_{s+1}-x_{s})\leq -\frac{1}{\eta_s}\|x_s-x_{s+1}\|^2.
    $$
    In (ii) we utilize the following result from Lemma \ref{lemma:approx_error_hypergrad}
    $$\|\nabla_x \widetilde{\phi}_t(x_s) - \hat{\nabla} \widetilde{\phi}_t(x_s)\|\leq \epsilon_s \overline{L}'_{\widetilde{\phi}_t,m_s}\leq\frac{\epsilon}{4}.$$
    
    Now suppose 
    $$
    \frac{1}{\eta_s^2}\|x_{s+1} - x_{s}\|^2 \leq \frac{\epsilon^2}{4},
    $$
    then the proof is already finished. Otherwise we have $-\frac{1}{4 \eta_s}\|x_{s+1} - x_{s}\|^2 + \eta_s \left(\frac{\epsilon}{4}\right)^2\leq 0$, and by \eqref{eq:thm_main_temp1} we know that the algorithm is a descent algorithm. Re-arranging the terms we get:
    $$
    \frac{1}{\eta_s^2}\|x_{s+1} - x_{s}\|^2\leq \frac{4(\widetilde{\phi}_t(x_{s}) - \widetilde{\phi}_t(x_{s+1}))}{\eta_s} + \frac{\epsilon^2}{4}.
    $$
    
    It remains to lower bound $\eta_s$ so that the right-hand side of the above inequality is upper bounded. Note that we have defined
    $$
    \eta_s=\min\{1, \frac{d_s}{2L_h} \frac{1}{\|\hat{\nabla}_x\widetilde{\phi}_t(x_s)\|}, \frac{1}{\overline{L}_{\widetilde{\phi}_t, s}}\}
    $$
    where
    $$
    \frac{d_s}{2L_h} \frac{1}{\|\hat{\nabla}_x\widetilde{\phi}_t(x_s)\|}\geq \frac{D}{4L_h}\frac{1}{L_f + \frac{L_f}{\mu_{\widetilde{g}_t}}\left(\overline{L}_g+ t k \frac{\overline{L}_h}{M} + t k \frac{L_h^2}{M^2}\right)}
    $$
    ($D$ is a constant given in Remark \ref{remark:non-emptyness}, and $d_s\geq D/2$ from Corollary  \ref{remark:upper_bound_of_h_at_approximate_point}) and
    $$
    \frac{1}{\overline{L}_{\widetilde{\phi}_t, s}}\geq \frac{1}{\overline{L}_{\widetilde{\phi}_t}}
    $$
    where $\overline{L}_{\widetilde{\phi}_t}$ is given in Lemma \ref{lem:lip_global_bound}. Thus we have
    $$
    \frac{1}{\eta_s^2}\|x_{s+1} - x_{s}\|^2\leq \frac{4(\widetilde{\phi}_t(x_{s}) - \widetilde{\phi}_t(x_{s+1}))}{\zeta} + \frac{\epsilon^2}{4}.
    $$
    where 
    \begin{equation}\label{eq:eta_lower_bound}
        \zeta = \min\left\{1, \frac{D}{4L_h}\frac{1}{L_f + \frac{L_f}{\mu_{\widetilde{g}_t}}\left(\overline{L}_g+ t k \frac{\overline{L}_h}{M} + t k \frac{L_h^2}{M^2}\right)}, \frac{1}{\overline{L}_{\widetilde{\phi}_t}}\right\}.
    \end{equation}
    By setting $\max_{x\in\mathcal{X}}\widetilde{\phi}_t(x)=\widetilde{\Phi}<+\infty$ and $S=\frac{64\widetilde{\Phi}}{5\zeta\epsilon^2}=\mathcal{O}(\frac{1}{\zeta\epsilon^2})$, we derive $$
    \min_{s=0,...,S-1}\frac{1}{\eta_s}\|x_s - x_{s+1}\|\leq\frac{3\epsilon}{4}
    $$ by telescoping sum for $s=0,...,S$.
    Then 
    \begin{align*}
        &\min_{s=0,...,S-1}\frac{1}{\eta_s}\|x_{s} - \proj_{\mathcal{X}}(x_s-\eta_s \nabla_x\widetilde{\phi}_t(x_s))\|\\
        \leq& \min_{s=0,...,S-1}\frac{1}{\eta_s}\left(\|x_{s} - \proj_{\mathcal{X}}(x_s-\eta_s \nabla_x\widetilde{\phi}_t(x_s))\|\right.\\
        &+\left.\|\proj_{\mathcal{X}}(x_s-\eta_s \nabla_x\widetilde{\phi}_t(x_s)) - \proj_{\mathcal{X}}(x_s-\eta_s \nabla_x\widetilde{\phi}_t(x_s))\|\right)\\
        \leq&\frac{3\epsilon}{4}+\max_{s=0,...,S-1}\frac{1}{\eta_s}\|\proj_{\mathcal{X}}(\eta_s \hat{\nabla}_x\widetilde{\phi}_t(x_s)-\eta_s \nabla_x\widetilde{\phi}_t(x_s))\|\\
        \leq&\frac{3\epsilon}{4}+\max_{s=0,...,S-1}\|\hat{\nabla}_x\widetilde{\phi}_t(x_s)-\nabla_x\widetilde{\phi}_t(x_s)\|\\
        \leq&\frac{3\epsilon}{4}+\frac{\epsilon}{4}\\
        =&\epsilon.
    \end{align*}
    Next, we prove the number of iterations for the inner loop is $\widetilde{\mathcal{O}}(t^{-0.5})$. Recall that, from the beginning of the proof, the number of iterations for the inner loop is at most $\mathcal{O}(\kappa\log(4\overline{L}'_{\widetilde{\phi}_t}/\epsilon))$ from Theorem \ref{thm:convergence_rage_alg_2}, where $\kappa=\sqrt{\overline{L}_{\widetilde{g}_t,m_s}/\mu_{\widetilde{g}_t}}$. Since $\overline{L}_{\widetilde{g}_t,m_s}$ can be computed by replacing $m$ in Proposition \ref{prop:lipschitz_smooth} with $m_s$, and $m_s=\mathcal{O}(t)$ in view of Theorem \ref{lem:local_M}, we have $\overline{L}_{\widetilde{g}_t,m_s}=\mathcal{O}(t^{-1})$, implying $\kappa=\mathcal{O}(t^{-0.5})$. According to Lemma \ref{lemma:approx_error_hypergrad}, $\overline{L}'_{\widetilde{\phi}_t}$ is polynomial in $t$. We conclude that $\mathcal{O}(\kappa\log(4\overline{L}'_{\widetilde{\phi}_t}/\epsilon))=\widetilde{\mathcal{O}}(t^{-0.5})$.

\subsection{Explaination for Remark \ref{remark:dependence}}\label{explainationforremark:dependence}

This is because, by the beginning of the proof for Theorem \ref{thm:total_rate}, the number of iterations for the inner loop is at most $\mathcal{O}(\kappa\log(4\overline{L}'_{\widetilde{\phi}_t}/\epsilon))$, where $\kappa=(\overline{L}_{\widetilde{g}_t,m_s}/\mu_{\widetilde{g}_t})^{1/2}$. Since $\overline{L}_{\widetilde{g}_t,m_s}$ can be computed by replacing $m$ in Proposition \ref{prop:lipschitz_smooth} with $m_s$, and $m_s=\mathcal{O}(t)$ by Theorem \ref{lem:local_M}, we have $\overline{L}_{\widetilde{g}_t,m_s}=\mathcal{O}(t^{-1})$, implying $\kappa=\mathcal{O}(t^{-0.5})$. According to Lemma \ref{lem:upperboundofapproxgradient}, $\overline{L}'_{\widetilde{\phi}_t}$ is polynomial in $t$. We conclude that $\mathcal{O}(\kappa\log(4\overline{L}'_{\widetilde{\phi}_t}/\epsilon))=\widetilde{\mathcal{O}}(t^{-0.5})$.

\subsection{Proof of Theorem \ref{thm:asymptotic}}\label{proof:thm:asymptotic}

By Theorem \ref{thm:total_rate} we know that for each given $t_i$ 
    \begin{equation}\label{eq:thm_asymptotic_temp1}
        \frac{1}{\eta_{t_i}}\|x_{i} - \proj_{\mathcal{X}}(x_i-\eta_{t_i} \nabla_x\widetilde{\phi}_{t_i}(x_i))\|\leq\epsilon_i
    \end{equation}
    where $\eta_{t_i}$ is the stepsize $\eta_s$ as specified in Theorem \ref{thm:total_rate}. Note that $\eta_{t_i}\leq 1$, also since \eqref{eq:thm_asymptotic_temp1} is decreasing when $\eta_{t_i}$ increases, see \citet[Lemma 2.3.1]{bertsekas1997nonlinear}, we conclude that:
    $$
    \|x_{i} - \proj_{\mathcal{X}}(x_i- \nabla_x\widetilde{\phi}_{t_i}(x_i))\|\leq\epsilon_i
    $$

    Now that $i\rightarrow \infty$ and $x_i\rightarrow x^*$, since $x^*$ is \textbf{SCSC} point, there exists a neighborhood of $x^*$ where all points are \textbf{SCSC} points (from Lemma \ref{lem:con_of_multi} we know the multiplier is continuous and from Definition \ref{def:scsc_point} we know \textbf{SCSC} means the multiplier is positive, which implies an open set). Therefore we could without loss of generality say that $x_i$ are also \textbf{SCSC} points. Consequently Theorem \ref{thm:mainlinear} and \ref{thm:mainnonlinear} imply
    $$
    \nabla_x\widetilde{\phi}_{t_i}(x_i)\rightarrow \nabla_x\phi(x^*).
    $$
    Now if we take $i\rightarrow\infty$, we know $x_i\rightarrow x^*$, $\epsilon_i\rightarrow 0$. By
    $$
    \|x_{i} - \proj_{\mathcal{X}}(x_i- \nabla_x\widetilde{\phi}_{t_i}(x_i))\|\leq\epsilon_i
    $$
    we conclude that
    $$
    x^* = \proj_{\mathcal{X}}(x^*- \nabla_x\phi(x^*))
    $$
    which means that $x^*$ is the stationary point of $\phi$.

\subsection{Proof of Lemma \ref{lem:boundednessofJabobian}}\label{proof:lem:boundednessofJabobian}

By direct computation, we derive

    \resizebox{\textwidth}{!}{
    \begin{minipage}{\textwidth}
    \begin{align*}
        \nabla_x y_t^\ast(x)=&-\left(\nabla^2_{yy} g(x,y^\ast_t(x))+\sum_{i=1}^k\frac{1}{t}\frac{t^2}{h^2_i(x,y_t^\ast(x))}\nabla_y h_i(x,y_t^\ast(x))\left(\nabla_y h_i(x,y_t^\ast(x))\right)^\top+\sum_{i=1}^k\frac{t\nabla^2_{yy} h_i(x,y_t^\ast(x))}{-h_i(x,y_t^\ast(x))}\right)^{-1}\\
        &\times\left(\nabla^2_{yx} g(x,y^\ast_t(x))+\sum_{i=1}^k\frac{1}{t}\frac{t^2}{h^2_i(x,y_t^\ast(x))}\nabla_y h_i(x,y_t^\ast(x))\left(\nabla_x h_i(x,y_t^\ast(x))\right)^\top+\sum_{i=1}^k\frac{t\nabla^2_{yx} h_i(x,y_t^\ast(x))}{-h_i(x,y_t^\ast(x))}\right).
    \end{align*}
    \end{minipage}}
    
    Denote
    \begin{align}
        A_t(x,y)=&\nabla^2_{yy} g(x,y)+\sum_{i=1}^k\frac{t\nabla^2_{yy} h_i(x,y)}{-h_i(x,y)},\ y\in \text{int}\mathcal{Y}(x)\label{notation:improve1}\\
        B_t(x,y)=&\nabla^2_{yx} g(x,y)+\sum_{i=1}^k\frac{t\nabla^2_{yx} h_i(x,y)}{-h_i(x,y)},\ y\in \text{int}\mathcal{Y}(x)\label{notation:improve2}\\
        v^i_t(x,y)=&\frac{t}{-h_i(x,y)}\nabla_y h_i(x,y),\ y\in \text{int}\mathcal{Y}(x)\label{notation:improve3}\\
        u^i_t(x,y)=&\frac{t}{-h_i(x,y)}\nabla_x h_i(x,y),\ y\in \text{int}\mathcal{Y}(x),\label{notation:improve4}
    \end{align}
    where $\text{int}\mathcal{Y}(x)$ is the interior of $\mathcal{Y}(x)$. These terms are well-defined only when $y\in \text{int}\mathcal{Y}(x)$ because $h_i(x,y)$ may be $0$ for some index $i$ when $y\in \mathcal{Y}(x)\setminus\text{int}\mathcal{Y}(x)$. Note that $y^\ast_t(x)\in\text{int}\mathcal{Y}(x)$, then
    \begin{align}
        \nabla_x y_t^\ast(x)=&\left(A_t(x,y_t^\ast(x))+\sum_{i=1}^k\frac{1}{t}v^i_t(x,y_t^\ast(x))\left(v^i_t(x,y_t^\ast(x))\right)^\top\right)^{-1}\nonumber\\
         &\times\left(B_t(x,y_t^\ast(x))+\sum_{i=1}^k\frac{1}{t}v^i_t(x,y_t^\ast(x))\left(u^i_t(x,y_t^\ast(x))\right)^\top\right).\label{eq:Jacobian!}
    \end{align}
    
    We first show that 
    $$\left\|\left(A_t(x,y_t^\ast(x))+\sum_{i=1}^k\frac{1}{t}v^i_t(x,y_t^\ast(x))\left(v^i_t(x,y_t^\ast(x))\right)^\top\right)^{-1}B_t(x,y_t^\ast(x))\right\|$$ 
    is uniformly bounded. By Corollary \ref{remark:upper_bound_of_h_at_approximate_point}, $h_i(x,y_t^\ast(x))\leq -M$, where $M=\min\{M_1t,M_2\}$ for some constant $M_1$ and $M_2$. This implies 
    \begin{align}\label{eq:boundoft/-h}
        t/(-h_i(x,y^\ast_t(x)))\leq\max\{1/M_1,1/M_2\}
    \end{align} 
    for any $i$ since $0<t\leq 1$. By Assumption \ref{assumption:general2}(\ref{assumption:general2(4)})(\ref{assumption:general2(7)}), we have $\|\nabla^2_{yx} g(x,y^\ast_t(x))\|\leq \overline{L}_g$ and $\|\nabla^2_{yx} h_i(x,y_t^\ast(x))\|\leq\overline{L}_h$. Thus 
    $$\left\|B_t(x,y^\ast_t(x))\right\|=\left\|\nabla^2_{yx} g(x,y^\ast_t(x))+\sum_{i=1}^k\frac{t\nabla^2_{yx} h_i(x,y^\ast_t(x))}{-h_i(x,y^\ast_t(x))}\right\|\leq\overline{L}_g+k\overline{L}_h\max\{1/M_1,1/M_2\}.$$
    Note that
    \begin{align}\label{eq:A+tvvleqmug}
        \left(A_t(x,y_t^\ast(x))+\sum_{i=1}^k\frac{1}{t}v^i_t(x,y_t^\ast(x))\left(v^i_t(x,y_t^\ast(x))\right)^\top\right)^{-1}\preceq(\nabla^2_{yy}g(x,y_t^\ast(x)))^{-1}\preceq\frac{1}{\mu_{g}}I,
    \end{align}
    we conclude that 
    
    \resizebox{\textwidth}{!}{$\left\|\left(A_t(x,y_t^\ast(x))+\sum_{i=1}^k\frac{1}{t}v^i_t(x,y_t^\ast(x))\left(v^i_t(x,y_t^\ast(x))\right)^\top\right)^{-1}B_t(x,y_t^\ast(x))\right\|\leq \frac{1}{\mu_{g}}\left(\overline{L}_g+k\overline{L}_h\max\{1/M_1,1/M_2\}\right).$}
    
    To prove the uniform boundedness of $\|\nabla_x y_t^\ast(x)\|$, it is sufficient to show that 
    
    \resizebox{\textwidth}{!}{
    \begin{minipage}{\textwidth}
    \begin{align}\label{equation:keyterm}
        \left\|\left(A_t(x,y_t^\ast(x))+\sum_{i=1}^k\frac{1}{t}v^i_t(x,y_t^\ast(x))\left(v^i_t(x,y_t^\ast(x))\right)^\top\right)^{-1}\frac{1}{t}v^{\hat{j}}_t(x,y_t^\ast(x))\left(u^{\hat{j}}_t(x,y_t^\ast(x))\right)^\top\right\|
    \end{align}
    \end{minipage}}
    is uniformly bounded for any fixed ${\hat{j}}$. We will prove this on the interval $(0, \widetilde{T}]$ and then on $[\widetilde{T}, 1]$ for some constant $\widetilde{T}$. The value of $\widetilde{T}$ will be determined at the end of the proof. We prove the uniform boundedness on $(0, \widetilde{T}]$ by two steps: denoting 
    \begin{align}
    \mathcal{X}_1=\{x\in\mathcal{X}:h_{\hat{j}}(x,y^\ast(x))\text{ active}\},    
    \end{align}
    we will first show that $\nabla_x y_t^\ast(x)$ is uniformly bounded on an open neighborhood of $\mathcal{X}_1$, and then demonstrate that $\nabla_x y_t^\ast(x)$ is also bounded outside this open neighborhood. 
    
    We declare that in this proof, all mentions of openness refer to openness in the sense of the subspace topology.\\

   \noindent\textbf{Step 1:} 
   We first show the uniform boundedness of (\ref{equation:keyterm}) on an open neighborhood of $\mathcal{X}_1$ as follows: we will prove that for any $x$ in $\mathcal{X}_1$, there exists an open neighborhood $W_x$ of $x$ such that (\ref{equation:keyterm}) is bounded above in this open neighborhood $W_x$. These open neighborhoods form an open cover of the compact set $\mathcal{X}_1$, so we can select a finite subcover $\{W_i\}_{i=1}^p$. We will show the uniform boundedness of (\ref{equation:keyterm}) on $\bigcup_{i=1}^p W_{i}$.\\
   
   Now we provide the detailed proof. Set $-h_{\hat{j}}(x,y_t^\ast(x))=\alpha$, then

   \resizebox{\textwidth}{!}{
    \begin{minipage}{\textwidth}
    \begin{align}
        &\left(A_t(x,y_t^\ast(x))+\sum_{i=1}^k\frac{1}{t}v^i_t(x,y_t^\ast(x))\left(v^i_t(x,y_t^\ast(x))\right)^\top\right)^{-1}\frac{1}{t}v^{\hat{j}}_t(x,y_t^\ast(x))\left(u^{\hat{j}}_t(x,y_t^\ast(x))\right)^\top\nonumber\\
        \overset{\text{(i)}}{=}&\left(A_t(x,y_t^\ast(x))+\sum_{i=1}^k\frac{1}{t}v^i_t(x,y_t^\ast(x))\left(v^i_t(x,y_t^\ast(x))\right)^\top\right)^{-1}\frac{1}{t}v^{\hat{j}}_t(x,y_t^\ast(x))\left(\frac{t}{-h_i(x,y^\ast_t(x))}\nabla_x h_{\hat{j}}(x,y_t^\ast(x))\right)^\top\nonumber\\
        \overset{\text{(ii)}}{=}&\left(A_t(x,y_t^\ast(x))+\sum_{i=1}^k\frac{1}{t}v^i_t(x,y_t^\ast(x))\left(v^i_t(x,y_t^\ast(x))\right)^\top\right)^{-1}\alpha^{-1}v^{\hat{j}}_t(x,y_t^\ast(x))\left(\nabla_x h_{\hat{j}}(x,y_t^\ast(x))\right)^\top.\label{eq:(A+vv)(tvu)}
    \end{align}
    \end{minipage}}
    In (i) we plug in $u^{\hat{j}}_t(x,y_t^\ast(x)=\frac{t}{-h_i(x,y^\ast_t(x))}\nabla_x h_{\hat{j}}(x,y_t^\ast(x))$, and in (ii) we utilize $-h_{\hat{j}}(x,y_t^\ast(x))=\alpha$. Note that $\|\nabla_x h_{\hat{j}}(x,y_t^\ast(x))\|$ is uniformly bounded above from Assumption \ref{assumption:general2}(\ref{assumption:general2(6)}), then it is sufficient to prove that the following term has uniform bound
    \begin{align*}
        \left\|\left(A_t(x,y_t^\ast(x))+\sum_{i=1}^k\frac{1}{t}v^i_t(x,y_t^\ast(x))\left(v^i_t(x,y_t^\ast(x))\right)^\top\right)^{-1}\alpha^{-1}v^{\hat{j}}_t(x,y_t^\ast(x))\right\|.
    \end{align*} 
    Denote 
    $$C_t(x,y_t^\ast(x))=A_t(x,y_t^\ast(x))+\sum_{i\ne {\hat{j}},i=1}^k\frac{1}{t}v^i_t(x,y_t^\ast(x))\left(v^i_t(x,y_t^\ast(x))\right)^\top.$$
    It is clear that $C_t(x,y_t^\ast(x))\succeq \mu_g I$ is positive definite. We have

    \resizebox{\textwidth}{!}{
    \begin{minipage}{\textwidth}
    \begin{align}
        &\left\|\left(A_t(x,y_t^\ast(x))+\sum_{i=1}^k\frac{1}{t}v^i_t(x,y_t^\ast(x))\left(v^i_t(x,y_t^\ast(x))\right)^\top\right)^{-1}\alpha^{-1}v^{\hat{j}}_t(x,y_t^\ast(x))\right\|\nonumber\\
        =&\left\|\left(C_t(x,y_t^\ast(x))+\frac{1}{t}v^{\hat{j}}_t(x,y_t^\ast(x))\left(v^{\hat{j}}_t(x,y_t^\ast(x))\right)^\top\right)^{-1}\alpha^{-1}v^{\hat{j}}_t(x,y_t^\ast(x))\right\|\nonumber\\
        \overset{\text{(i)}}{=}&\left\|\alpha^{-1}\left(C_t^{-1}(x,y_t^\ast(x)) v^{\hat{j}}_t(x,y_t^\ast(x))-\frac{C_t^{-1}(x,y_t^\ast(x))v^{\hat{j}}_t(x,y_t^\ast(x))\left(v^{\hat{j}}_t(x,y_t^\ast(x))\right)^\top C_t^{-1}(x,y_t^\ast(x)) v^{\hat{j}}_t(x,y_t^\ast(x))}{t+\left(v^{\hat{j}}_t(x,y_t^\ast(x))\right)^\top C_t^{-1}(x,y_t^\ast(x))v^{\hat{j}}_t(x,y_t^\ast(x))}\right)\right\|\nonumber\\
        =&\left\|\alpha^{-1}\frac{tC_t(x,y_t^\ast(x))^{-1} v^{\hat{j}}_t(x,y_t^\ast(x))}{t+\left(v^{\hat{j}}_t(x,y_t^\ast(x))\right)^\top C_t^{-1}(x,y_t^\ast(x))v^{\hat{j}}_t(x,y_t^\ast(x))}\right\|\nonumber\\
        \overset{\text{(ii)}}{=}&\left\|t^2\alpha^{-2}\frac{C_t^{-1}(x,y_t^\ast(x))\nabla_y h_{\hat{j}}(x,y^\ast_t(x))}{t+t^2\alpha^{-2}\left(\nabla_y h_{\hat{j}}(x,y^\ast_t(x))\right)^\top C_t^{-1}(x,y_t^\ast(x))\nabla_y h_{\hat{j}}(x,y^\ast_t(x))}\right\|\nonumber\\
        \leq&\frac{\|C_t^{-1}(x,y_t^\ast(x))\nabla_y h_{\hat{j}}(x,y^\ast_t(x))\|}{\left(\nabla_y h_{\hat{j}}(x,y^\ast_t(x))\right)^\top C_t^{-1}(x,y_t^\ast(x))\nabla_y h_{\hat{j}}(x,y^\ast_t(x))},\label{eq:C/hCh}
    \end{align}
    \end{minipage}}
    where (i) is from the Sherman-Morrison formula, and in (ii) we plug in 
    $$v^{\hat{j}}_t(x,y_t^\ast(x))=\frac{t}{-h_{\hat{j}}(x,y^\ast_t(x))}\nabla_y h_{\hat{j}}(x,y^\ast_t(x))=t\alpha^{-1}\nabla_y h_{\hat{j}}(x,y^\ast_t(x))$$
    Next we need to prove that $\left(\nabla_y h_{\hat{j}}(x,y^\ast_t(x))\right)^\top C_t^{-1}(x,y_t^\ast(x))\nabla_y h_{\hat{j}}(x,y^\ast_t(x))$ has a positive lower bound. Denote
   \begin{align}
       \mathcal{I}^\ast(x)&=\{i:h_i(x,y^\ast(x))\text{ active}\}\label{notation:activeset}\\
       V(x,y;z)&=\mathrm{span}\{\nabla_y h_i(x,y):i\in\mathcal{I}^\ast(z),\ i\ne {\hat{j}}\}.\label{notation:spanspace}
   \end{align}
    For further analysis, we decompose 
    \begin{align*}
        C_t^{-1}(x,y_t^\ast(x))\nabla_y h_{\hat{j}}(x,y^\ast_t(x))=:w_t(x)+w_t(x)^\perp,
    \end{align*}
    where $w_t(x)\in V(x,y^\ast_t(x);x_0)$ and $w_t(x)^\perp\in (V(x,y^\ast_t(x);x_0))^\perp$ for some point $x_0$. Then we have
    \begin{align}
        &\left(\nabla_y h_{\hat{j}}(x,y^\ast_t(x))\right)^\top C_t^{-1}(x,y^\ast_t(x))\nabla_y h_{\hat{j}}(x,y^\ast_t(x)\nonumber\\
        =&\left<\nabla_y h_{\hat{j}}(x,y^\ast_t(x)),w_t(x)+w_t(x)^\perp\right>\nonumber\\
        =&\left<\left(A_t(x,y^\ast_t(x))+\sum_{i\ne {\hat{j}},i=1}^k\frac{1}{t}v^{\hat{j}}_t\left(v^{\hat{j}}_t(x,y^\ast_t(x))\right)^\top\right)\left(w_t(x)+w_t(x)^\perp\right),w_t(x)+w_t(x)^\perp\right>\label{equation:(A+vv)w,w}.
    \end{align}
    According to (\ref{notation:improve1}), we have \begin{align}\label{equation:scofA}
        A_t(x,y^\ast_t(x))+\sum_{i\ne {\hat{j}},i=1}^k\frac{1}{t}v^{\hat{j}}_t\left(v^{\hat{j}}_t(x,y^\ast_t(x))\right)^\top\succeq \nabla^2_{yy}g(x,y^\ast_t(x))\succeq\mu_g I.
    \end{align}
    We aim to estimate the lower bound of $\|w_t(x)+w_t(x)^\perp\|$. It is difficult to obtain a global bound for $\|w_t(x)+w_t(x)^\perp\|$ in $\mathcal{X}_1$. This because that the active set $\mathcal{I}^\ast(x)$ changes as $x_0$ varies, and if we fix the point $x_0$, it's hard to analyze $C^{-1}_t(x,y^\ast_t(x))$. So we want to show that for any $x_0\in\mathcal{X}_1$, there exists an open neighborhood $W_{x_0}$ of $x_0$ such that $\|w_t(x)+w_t(x)^\perp\|$ is lower bounded on $W_{x_0}$.
    
    To construct this open neighborhood, we embed $\mathcal{X}_1$ into $\{(x,y):x\in\mathcal{X},\ y\in\mathcal{Y}(x)\}$ by the following embedding map
   \begin{align}
       \iota:\mathcal{X}_1&\to\{(x,y):x\in\mathcal{X},\ y\in\mathcal{Y}(x)\}\label{eq:embedmap}\\
       x&\mapsto(x,y^\ast(x)).\nonumber
   \end{align}
  
   Note that $\hat{j}\in\mathcal{I}^\ast(x)$ if $x\in\mathcal{X}_1$, by the LICQ assumption, we know the following holds for any $x\in\mathcal{X}_1$
   $$\|\proj_{V^\perp(x,y^\ast(x);x)}\nabla_y h_{\hat{j}}(x,y^\ast(x))\|>0.$$
   It is clear that $\|\proj_{V^\perp(x,y^\ast(x);x)}\nabla_y h_{\hat{j}}(x,y^\ast(x))\|$ is lower semicontinuous in $x$. Therefore, by compactness of $\mathcal{X}_1$, there exists a constant $\delta>0$ such that the following holds for any $x\in\mathcal{X}_1$
   $$\|\proj_{V^\perp(x,y^\ast(x);x)}\nabla_y h_{\hat{j}}(x,y^\ast(x))\|\geq\delta.$$
   For any $x_0$ in $\mathcal{X}_1$, we will show that we can find a neighborhood $U_{x_0}$ of $(x_0,y^\ast(x_0))$ in $\{(x,y):x\in\mathcal{X},\ y\in\mathcal{Y}(x)\}$ such that
   \begin{enumerate}
       \item For any $i\notin\mathcal{I}^\ast(x_0)$ and $(x,y)\in U_{x_0}$, we have $h_i(x,y)\leq -H_{x_0}$ for some constant $H_{x_0}>0$;
       \item The vectors $\{\nabla_y h_i(x,y):i\in\mathcal{I}^\ast(x_0)\}$ are linear independent;
       \item There exists a constant $\delta_{x_0}>0$ such that $\|\proj_{V^\perp(x,y;x_0)}\nabla_y h_{\hat{j}}(x,y)\|\geq\delta_{x_0}$ holds for any $(x,y)\in U_{x_0}$;
       \item There exists a constant $r_{x_0}$ such that for any $x\in\pi_x(U_{x_0})$, which is the projection of $U_{x_0}$ onto the $x$ component, we have $\{x\}\times\left(B_{y^\ast(x)}(r_{x_0})\cap\mathcal{Y}(x)\right)\subset U_{x_0}$, where $B_{y^\ast(x)}(r_{x_0})$ is a ball in $\mathcal{Y}(x)$ centered at $y^\ast(x)$ with radius $r_{x_0}$.
   \end{enumerate}
   \begin{figure}[htbp]
\centering

\tikzset{every picture/.style={line width=0.75pt}} 

\begin{tikzpicture}[x=0.75pt,y=0.75pt,yscale=-1,xscale=1]

\draw    (100,128) .. controls (142.27,148.07) and (222.27,148.07) .. (266.27,125.07) ;
\draw  [fill={rgb, 255:red, 0; green, 0; blue, 0 }  ,fill opacity=1 ] (169.27,142.77) .. controls (169.27,141.39) and (170.39,140.27) .. (171.77,140.27) .. controls (173.15,140.27) and (174.27,141.39) .. (174.27,142.77) .. controls (174.27,144.15) and (173.15,145.27) .. (171.77,145.27) .. controls (170.39,145.27) and (169.27,144.15) .. (169.27,142.77) -- cycle ;
\draw    (262.93,162.93) .. controls (254.77,156.69) and (251.68,153.68) .. (250.12,138.85) ;
\draw [shift={(249.93,136.93)}, rotate = 84.81] [color={rgb, 255:red, 0; green, 0; blue, 0 }  ][line width=0.75]    (10.93,-3.29) .. controls (6.95,-1.4) and (3.31,-0.3) .. (0,0) .. controls (3.31,0.3) and (6.95,1.4) .. (10.93,3.29)   ;
\draw  [draw opacity=0] (141.91,139.88) .. controls (143.43,123.97) and (157.25,111.95) .. (173.36,112.81) .. controls (188.92,113.64) and (201.08,126.18) .. (201.74,141.44) -- (171.77,142.77) -- cycle ; \draw   (141.91,139.88) .. controls (143.43,123.97) and (157.25,111.95) .. (173.36,112.81) .. controls (188.92,113.64) and (201.08,126.18) .. (201.74,141.44) ;  
\draw [color={rgb, 255:red, 255; green, 0; blue, 0 }  ,draw opacity=1 ][line width=0.75]  [dash pattern={on 4.5pt off 4.5pt}]  (149.93,121.93) .. controls (166.43,125.43) and (180.43,125.93) .. (194.43,123.43) ;
\draw [color={rgb, 255:red, 255; green, 0; blue, 0 }  ,draw opacity=1 ] [dash pattern={on 4.5pt off 4.5pt}]  (149.93,121.93) .. controls (148.93,127.43) and (147.93,132.93) .. (147.93,140.43) ;
\draw [color={rgb, 255:red, 255; green, 0; blue, 0 }  ,draw opacity=1 ] [dash pattern={on 4.5pt off 4.5pt}]  (194.43,123.43) .. controls (194.93,128.43) and (194.93,134.93) .. (195.93,141.93) ;
\draw  [dash pattern={on 0.84pt off 2.51pt}]  (99.5,117.5) .. controls (141.77,137.57) and (221.77,137.57) .. (265.77,114.57) ;
\draw    (258.43,95.93) .. controls (257.96,105.43) and (254.32,108.62) .. (246.67,117.91) ;
\draw [shift={(245.43,119.43)}, rotate = 308.99] [color={rgb, 255:red, 0; green, 0; blue, 0 }  ][line width=0.75]    (10.93,-3.29) .. controls (6.95,-1.4) and (3.31,-0.3) .. (0,0) .. controls (3.31,0.3) and (6.95,1.4) .. (10.93,3.29)   ;
\draw    (182.93,83.43) .. controls (171.93,103.43) and (169.43,123.93) .. (171.77,140.27) ;
\draw    (160.43,70.43) .. controls (159.97,79.78) and (166.93,84.33) .. (174.78,88.96) ;
\draw [shift={(176.43,89.93)}, rotate = 210.47] [color={rgb, 255:red, 0; green, 0; blue, 0 }  ][line width=0.75]    (10.93,-3.29) .. controls (6.95,-1.4) and (3.31,-0.3) .. (0,0) .. controls (3.31,0.3) and (6.95,1.4) .. (10.93,3.29)   ;
\draw   (25,149.72) .. controls (25,82.22) and (104.12,27.5) .. (201.72,27.5) .. controls (299.31,27.5) and (378.43,82.22) .. (378.43,149.72) .. controls (378.43,217.22) and (299.31,271.93) .. (201.72,271.93) .. controls (104.12,271.93) and (25,217.22) .. (25,149.72) -- cycle ;
\draw    (162.77,141.9) .. controls (162.37,134.93) and (162.37,129.43) .. (163.37,124.43) ;
\draw    (143.87,109.93) .. controls (145.27,118.82) and (151.92,123.33) .. (159.72,127.96) ;
\draw [shift={(161.37,128.93)}, rotate = 210.47] [color={rgb, 255:red, 0; green, 0; blue, 0 }  ][line width=0.75]    (10.93,-3.29) .. controls (6.95,-1.4) and (3.31,-0.3) .. (0,0) .. controls (3.31,0.3) and (6.95,1.4) .. (10.93,3.29)   ;
\draw  [line width=0.75]  (196.87,141.93) .. controls (199.41,141.72) and (200.58,140.35) .. (200.37,137.82) -- (200.37,137.82) .. controls (200.07,134.19) and (201.19,132.27) .. (203.73,132.07) .. controls (201.19,132.27) and (199.77,130.56) .. (199.48,126.93)(199.62,128.57) -- (199.48,126.93) .. controls (199.27,124.39) and (197.9,123.22) .. (195.37,123.43) ;

\draw (143,150.9) node [anchor=north west][inner sep=0.75pt]  [font=\scriptsize]  {$\left( x_{0} ,y^{\ast }( x_{0})\right)$};
\draw (245,165.4) node [anchor=north west][inner sep=0.75pt]  [font=\scriptsize]  {$\left\{\left( x,y^{\ast }( x)\right) :x\in \mathcal{X}_{1}\right\}$};
\draw (228.5,76.9) node [anchor=north west][inner sep=0.75pt]  [font=\scriptsize]  {$\left\{\left( x,y_{t}^{\ast }( x)\right) :x\in \mathcal{X}_{1}\right\}$};
\draw (116,54.4) node [anchor=north west][inner sep=0.75pt]  [font=\scriptsize]  {$\{( x_{0} ,y) :y\in \mathcal{Y}( x_{0})\}$};
\draw (146,232.9) node [anchor=north west][inner sep=0.75pt]    {$\mathbb{R}^{d_{x}} \times \mathbb{R}^{d_{y}}$};
\draw (86,96.4) node [anchor=north west][inner sep=0.75pt]  [font=\scriptsize]  {$\{x\}\times(B_{y^{\ast }( x)} r( x_{0}) \cap \mathcal{Y}( x))$};
\draw (205,123.4) node [anchor=north west][inner sep=0.75pt]  [font=\scriptsize]  {$r_{x_{0}}$};

\end{tikzpicture}

\caption{Find a spherical neighborhood that satisfies the first three properties, as shown by the solid hemisphere in the figure, then find a tubular neighborhood within it, as shown by the red dashed tube in the figure.}
\end{figure}
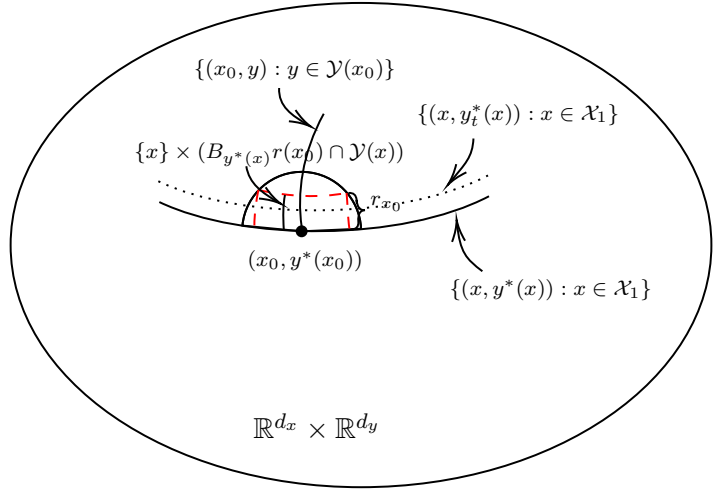
   The first property can be directly obtained from the continuity of the constraint functions $h_i(x,y)$ and the fact that $h_i(x,y^\ast(x))\ne 0$ for any $i\notin\mathcal{I}^\ast(x_0)$. The second property holds because $\{\nabla_y h_i(x_0,y^\ast(x_0)):i\in\mathcal{I}^\ast(x_0)\}$ are linear independent by LICQ assumption. This property ensures that $V(x,y;x_0)$ and $V^\perp(x,y;x_0)$ are non-degenerate on $U_{x_0}$, so $\proj_{V^\perp(x,y;x_0)}\nabla_y h_{\hat{j}}(x,y)$ is continuous with respect to $(x,y)$. Note that $$\|\proj_{V^\perp(x_0,y^\ast(x_0);x_0)}\nabla_y h_{\hat{j}}(x_0,y^\ast(x_0))\|\geq\delta,$$ 
   by continuity, the third property can also be satisfied. To show the fourth, note that we find a spherical neighborhood that satisfies the first three conditions. Therefore, within this spherical neighborhood, we can identify a tubular neighborhood, thus also satisfying the fourth property easily. 

   Denote $W_{x_0}=\pi_x(U_{x_0})$. The projection map is an open map, so $W_{x_0}$ is open. we will prove that $\|w_t(x)+w_t(x)^\perp\|$ is lower bounded on $W_{x_0}$. We consider the case that $t\leq(r_{x_0}^2\mu_g)/(2k)$. By Lemma \ref{lem:optimalitygapforpoint}, we have $\|y^\ast(x)-y_t^\ast(x)\|\leq \sqrt{2kt/\mu_g}$, which means $\|y^\ast(x)-y_t^\ast(x)\|\leq r_{x_0}$ and thus $(x,y^\ast_t(x))\in U_{x_0}$ by the fourth property of the open neighborhood. Then we obtain
    \begin{align*}
        &\nabla_y h_{\hat{j}}(x,y^\ast_t(x))\\
        =&\left(A_t(x,y_t^\ast(x))+\sum_{i\notin \mathcal{I}^\ast(x_0),i\ne {\hat{j}}}\frac{1}{t}v^i_t(x,y_t^\ast(x))\left(v^i_t(x,y_t^\ast(x))\right)^\top\right.\\
        &+\left.\sum_{i\in \mathcal{I}^\ast(x_0),i\ne {\hat{j}}}\frac{1}{t}v^i_t(x,y_t^\ast(x))\left(v^i_t(x,y_t^\ast(x))\right)^\top\right)(w_t(x)+w_t(x)^\perp)\\
        \overset{\text{(i)}}{=}&\left(A_t(x,y_t^\ast(x))+\sum_{i\notin \mathcal{I}^\ast(x_0),i\ne {\hat{j}}}\frac{1}{t}v^i_t(x,y_t^\ast(x))\left(v^i_t(x,y_t^\ast(x))\right)^\top\right)(w_t(x)+w_t(x)^\perp)\\
        &+\sum_{i\in \mathcal{I}^\ast(x_0),i\ne {\hat{j}}}\frac{1}{t}v^i_t(x,y_t^\ast(x))\left(v^i_t(x,y_t^\ast(x))\right)^\top w_t(x),
    \end{align*}
    where (i) is because 
    $$\sum_{i\in \mathcal{I}^\ast(x_0),i\ne {\hat{j}}}\frac{1}{t}v^i_t(x,y_t^\ast(x))\left(v^i_t(x,y_t^\ast(x))\right)^\top w_t(x)^\perp=0.$$ 
    On the other hand, we have $$\nabla_y h_{\hat{j}}(x,y^\ast_t(x))=\proj_{V(x,y_t^\ast(x);x_0)}\nabla_y h_{\hat{j}}(x,y^\ast_t(x))+\proj_{V^\perp(x,y_t^\ast(x);x_0)}\nabla_y h_{\hat{j}}(x,y^\ast_t(x)).$$ 
    Note that $\sum_{i\in \mathcal{I}^\ast(x_0),i\ne {\hat{j}}}\frac{1}{t}v^{i}_t(x,y_t^\ast(x))\left(v^{i}_t(x,y_t^\ast(x))\right)^\top w_t(x)\in V(x,y^\ast_t(x);x_0)$, we have 
    \begin{align*}
        &\proj_{V^\perp(x,y_t^\ast(x);x_0)}\nabla_y h_{\hat{j}}(x,y^\ast_t(x))\\
        =&\proj_{V^\perp(x,y_t^\ast(x);x_0)}\left[\left(A_t(x,y_t^\ast(x))+\sum_{i\notin \mathcal{I}^\ast(x_0),i\ne {\hat{j}}}\frac{1}{t}v^{\hat{j}}_t(x,y_t^\ast(x))\left(v^{\hat{j}}_t(x,y_t^\ast(x))\right)^\top\right)(w_t(x)+w_t(x)^\perp)\right],
    \end{align*}
    which means 
    \begin{align}
        &\left\|\left(A_t(x,y_t^\ast(x))+\sum_{i\notin \mathcal{I}^\ast(x_0),i\ne {\hat{j}}}\frac{1}{t}v^{\hat{j}}_t(x,y_t^\ast(x))\left(v^{\hat{j}}_t(x,y_t^\ast(x))\right)^\top\right)(w_t(x)+w_t(x)^\perp)\right\|\nonumber\\
        \geq&\left\|\proj_{V^\perp(x,y_t^\ast(x);x_0)}\left[\left(A_t(x,y_t^\ast(x))+\sum_{i\notin \mathcal{I}^\ast(x_0),i\ne {\hat{j}}}\frac{1}{t}v^{\hat{j}}_t(x,y_t^\ast(x))\left(v^{\hat{j}}_t\right)^\top\right)(w_t(x)+w_t(x)^\perp)\right]\right\|\nonumber\\
        =&\left\|\proj_{V^\perp(x,y_t^\ast(x);x_0)}\nabla_y h_{\hat{j}}(x,y^\ast_t(x))\right\|\nonumber\\
        \overset{\text{(i)}}{\geq}&\delta_{x_0}.\label{eq:(A+vv)(w_t(x)+w_t(x))}
    \end{align}
    where in (i) we use the third property of the open neighborhood. By (\ref{eq:boundoft/-h}), $t/(-h_i(x,y^\ast_t(x)))$ is uniformly bounded. Combining Assumption \ref{assumption:general2}(\ref{assumption:general2(4)})(\ref{assumption:general2(7)}), it is not hard to see that the maximum eigenvalue of $A_t(x,y^\ast_t(x))=\nabla^2_{yy} g(x,y^\ast_t(x))+\sum_{i=1}^k\frac{t\nabla^2_{yy} h_i(x,y_t^\ast(x))}{-h_i(x,y_t^\ast(x))}$ has upper bounded denoted as $\lambda_{max}$. According to the first property that for any $i\notin\mathcal{I}(x_0)$ and $(x,y)\in U_{x_0}$, $h_i(x,y)\leq -H_{x_0}$ holds for some constant $H_{x_0}>0$, we obtain that 
    \begin{align*}
        &\left\|A_t(x,y^\ast_t(x))+\sum_{i\notin \mathcal{I}^\ast(x_0),i\ne {\hat{j}}}\frac{1}{t}v^{\hat{j}}_t(x,y^\ast_t(x))\left(v^{\hat{j}}_t(x,y^\ast_t(x))\right)^\top\right\|\\
        =&\left\|A_t(x,y^\ast_t(x))+\sum_{i\notin \mathcal{I}^\ast(x_0),i\ne {\hat{j}}}\frac{1}{t}\frac{t}{-h_i(x,y^\ast_t(x))}\nabla_y h_i(x,y^\ast_t(x))\left(\frac{t}{-h_i(x,y^\ast_t(x))}\nabla_y h_i(x,y^\ast_t(x))\right)^\top\right\|\\
        =&\left\|A_t(x,y^\ast_t(x))+\sum_{i\notin \mathcal{I}^\ast(x_0),i\ne {\hat{j}}}\frac{t}{(-h_i(x,y^\ast_t(x)))^2}\nabla_y h_i(x,y^\ast_t(x))\left(\nabla_y h_i(x,y^\ast_t(x))\right)^\top\right\|\\
        \leq&\lambda_{max}+\frac{kL^2_h}{H_{x_0}^2}.
    \end{align*}
    Combing (\ref{eq:(A+vv)(w_t(x)+w_t(x))}), we obtain $\|w_t(x)+w_t(x)^\perp\|\geq \delta_{x_0}/(\lambda_{max}+\frac{kL^2_h}{H_{x_0}^2})$. Now we proved the lower boundedness of $\|w_t(x)+w_t(x)^\perp\|$ on $W_{x_0}$. 
    By (\ref{equation:(A+vv)w,w}) and (\ref{equation:scofA}), we obtain the lower boundedness of
    $$\left(\nabla_y h_{\hat{j}}(x,y^\ast_t(x))\right)^\top C_t^{-1}(x,y^\ast_t(x))\nabla_y h_{\hat{j}}(x,y^\ast_t(x).$$
    According to (\ref{eq:C/hCh}), this means the following term is upper bounded on $W_{x_0}$
    \begin{align}\label{eq:byproduct}
    \left\|\left(A_t(x,y_t^\ast(x))+\sum_{i=1}^k\frac{1}{t}v^i_t(x,y_t^\ast(x))\left(v^i_t(x,y_t^\ast(x))\right)^\top\right)^{-1}\alpha^{-1}v^{\hat{j}}_t(x,y_t^\ast(x))\right\|.
    \end{align}
    Recalling (\ref{eq:(A+vv)(tvu)}), this further implies the upper boundedness of the following term on $W_{x_0}$
    \begin{align}\label{eq:finalresult2}
        \left\|\left(A_t(x,y_t^\ast(x))+\sum_{i=1}^k\frac{1}{t}v^i_t(x,y_t^\ast(x))\left(v^i_t(x,y_t^\ast(x))\right)^\top\right)^{-1}\frac{1}{t}v^{\hat{j}}_t(x,y_t^\ast(x))\left(u^{\hat{j}}_t(x,y_t^\ast(x))\right)^\top\right\|.
    \end{align}
    Note that $\{W_{x_0}\}_{x_0\in\mathcal{X}_1}$ forms an open cover of $\mathcal{X}_1$ in $\mathcal{X}$. Since $\mathcal{X}_1$ is compact, we can choose a finite subcover, denoted as $\{W_i\}_{i=1}^p$ where $W_i=W_{x_0}$ for some $x_0$ in $\mathcal{X}_1$. Denote $T_i=(r_i^2\mu_g)/(2k)$, $r_i=r_{x_0}$ from the fourth property and $G_i$ be the upper bound of norm of (\ref{eq:Jacobian!}) when $0<t\leq T_i$. Take $T^\ast=\min\{T_i\}$,  $G^\ast=\max\{G_i\}$. It is not hard to see that the norm of (\ref{eq:Jacobian!}) is upper bounded by $G^\ast$ for any $0<t\leq T^\ast$ and $x$ in $\bigcup_{i=1}^p W_i$.\\

    \noindent\textbf{Step 2:} Consider $\mathcal{X}_2:=\mathcal{X}\setminus \bigcup_{i=1}^p W_i$ which is compact in $\mathcal{X}$. Since $\bigcup_{i=1}^p W_i$ has covered $\mathcal{X}_1$, it is easy to see that $h_{\hat{j}}(x,y^\ast(x))$ has a negative upper bound in $\mathcal{X}_2$, denoted as $-Q$. We assume that $t\leq (\mu_g Q^2)/(4kL_h^2)$. By Lemma \ref{lem:optimalitygapforpoint}, we have $\|y^\ast_t(x)-y^\ast(x)\|\leq\sqrt{2kt/\mu_g}\leq Q/(2L_h)$. Thus $|h_{\hat{j}}(x,y^\ast_t(x))-h_{\hat{j}}(x,y^\ast(x))|\leq Q/2$, which means $h_{\hat{j}}(x,y^\ast_t(x))\leq-Q/2$. Then the following term is uniformly bounded for any $x\in\mathcal{X}_2$ and $0<t\leq (\mu_g Q^2)/(4kL_h^2)$:

    \resizebox{\textwidth}{!}{
    \begin{minipage}{\textwidth}
    \begin{align}
        &\left\|\left(A_t(x,y^\ast_t(x))+\sum_{i=1}^k\frac{1}{t}v^i_t(x,y^\ast_t(x))\left(v^i_t(x,y^\ast_t(x))\right)^\top\right)^{-1}\frac{1}{t}v^{\hat{j}}_t(x,y^\ast_t(x))\left(u^{\hat{j}}_t(x,y^\ast_t(x))\right)^\top\right\|\nonumber\\
        =&\left\|\left(A_t(x,y^\ast_t(x))+\sum_{i=1}^k\frac{1}{t}v^i_t(x,y^\ast_t(x))\left(v^i_t(x,y^\ast_t(x))\right)^\top\right)^{-1}\frac{t}{h^2_{\hat{j}}(x,y^\ast_t(x))}\nabla_y h_{\hat{j}}(x,y^\ast_t(x))\left(\nabla_x h_{\hat{j}}(x,y^\ast_t(x))\right)^\top\right\|\nonumber\\
        \overset{\text{(i)}}{\leq}&\frac{1}{\mu_g}\frac{t}{h^2_{\hat{j}}(x,y_t^\ast(x))}\left\|\nabla_y h_{\hat{j}}(x,y_t^\ast(x))\right\|\cdot\left\|\nabla_x h_{\hat{j}}(x,y_t^\ast(x))\right\|\nonumber\\
        \overset{\text{(ii)}}{\leq}&\frac{4L_h^2}{\mu_g Q^2}.\label{eq:boundterm}
    \end{align}
    \end{minipage}}
    (i) is because (\ref{eq:A+tvvleqmug}) and $v^{\hat{j}}_t(x,y^\ast_t(x))=\frac{t}{-h_{\hat{j}}(x,y^\ast_t(x))}\nabla_y h_{\hat{j}}(x,y^\ast_t(x))$. In (ii), we utilize $t\leq 1$, $h_{\hat{j}}(x,y^\ast_t(x))\leq-Q/2$ and $\|\nabla_y h_{\hat{j}}(x,y^\ast_t(x))\|\leq L_h$. Furthermore, we obtain (\ref{eq:Jacobian!}) has bounded norm.\\
    
    To conclude, we have proved that $\|\nabla_x y^\ast_t(x)\|$ is upper bounded for any $x$ in $\mathcal{X}$ and $0<t\leq\widetilde{T}:=\min\{ T^\ast, (\mu_g Q^2)/(4kL_h^2)\}$. If $\widetilde{T}$ is less than $1$, see $\|\nabla_x y^\ast_t(x)\|$ as a continuous function defined of the compact set $\mathcal{X}\times [\widetilde{T},1]$, so it has an upper bound. Therefore, we proved that $\|\nabla_x y^\ast_t(x)\|$ is upper bounded for any $x\in\mathcal{X}$ and $0<t\leq 1$.

\subsection{Proof of Lemma \ref{cor:bound1}}\label{proof:cor:bound1}

Let $-h_{\hat{j}}(x,y^\ast_t(x))=\alpha$ and use the notations (\ref{notation:improve1})-(\ref{notation:improve4}), we obtain
\begin{align}
    &\left\|(\nabla_{yy} \widetilde{g}_t(x))^{-1}\frac{t\nabla_y h_{\hat{j}}(x,y^\ast_t(x))}{(-h_{\hat{j}} (x,y^\ast_t(x)))^2}\right\|\nonumber\\
    =&\left\|\left(A_t(x,y_t^\ast(x))+\sum_{i=1}^k\frac{1}{t}v^i_t(x,y_t^\ast(x))\left(v^i_t(x,y_t^\ast(x))\right)^\top\right)^{-1}\alpha^{-1}v^{\hat{j}}_t(x,y_t^\ast(x))\right\|.\label{eq:sameterm}
\end{align}
We have proved this term is uniformly bounded on $\bigcup_{i=1}^p W_i$, as detailed from (\ref{eq:byproduct}) to the end of Step 1 in Section \ref{proof:lem:boundednessofJabobian}. For $\mathcal{X}\setminus\bigcup_{i=1}^p W_i$, we can assume $h_{\hat{j}}(x,y^\ast(x))\leq -Q$. Then
\begin{align}
    &\left\|(\nabla_{yy} \widetilde{g}_t(x))^{-1}\frac{t\nabla_y h_{\hat{j}}(x,y^\ast_t(x))}{(-h_{\hat{j}} (x,y^\ast_t(x)))^2}\right\|\nonumber\\
    =&\left\|\left(A_t(x,y^\ast_t(x))+\sum_{i=1}^k\frac{1}{t}v^i_t(x,y^\ast_t(x))\left(v^i_t(x,y^\ast_t(x))\right)^\top\right)^{-1}\frac{t}{h^2_{\hat{j}}(x,y^\ast_t(x))}\nabla_y h_{\hat{j}}(x,y^\ast_t(x))\right\|\nonumber
\end{align}
is bounded by the same process of (\ref{eq:boundterm}).

\subsection{Proof of Lemma \ref{cor:bound2}}\label{proof:cor:bound2}

Following the argument at the end of Step 1 in proof of Lemma \ref{lem:boundednessofJabobian}, we select an open neighborhood $\{W_i\}_{i=1}^p$ of $\mathcal{X}_1$. 

If $x\in\bigcup_{i=1}^p W_i$, denote $\beta_1=\min_i\{r_i\}$. According to the fourth property of the open neighborhood, $\|y-y^\ast(x)\|\leq r_i$ for any $i$ means $y$ is in the $W_i$ for any $i$. By replicating the proof in Step 1 and simply substituting $y^\ast_t(x)$ with $y$, we can obtain $\|(\nabla_{yy}\widetilde{g}_t(x,y))^{-1}\nabla_{yx}\widetilde{g}_t(x,y)\|\leq G^\ast$ for any $y\in\text{int}\mathcal{Y}(x)$ such that $\|y-y^\ast(x)\|\leq\beta_1$. 

If $x\in\mathcal{X}\setminus\bigcup_{i=1}^p W_i$, denote $\beta_2=Q/(2L_h)$. Similar to Step 2 in proof of Lemma \ref{lem:boundednessofJabobian}, we can also find that $\|(\nabla_{yy}\widetilde{g}_t(x,y))^{-1}\nabla_{yx}\widetilde{g}_t(x,y)\|$ is upper bounded for any $y\in\text{int}\mathcal{Y}(x)$ such that $\|y-y^\ast(x)\|\leq\beta_2$. Define $\beta=\min\{\beta_1,\beta_2\}$, we complete the proof.

\subsection{Proof of Lemma \ref{lem:improvedglobalupperboundLip}}\label{proof:lem:improvedglobalupperboundLip}

The proof is similar to that of Lemma \ref{lem:local_lip} and Lemma \ref{lem:lip_global_bound}. First we aim to prove the local Lipschitz smoothness of $\widetilde{\phi}_t(x)$. Following (\ref{eq:finalresult}), for any $x_1$, $x_2$ satisfies $\|x_1-x\|$, $\|x_2-x\|\leq\frac{d}{2L_h}$, we need to estimate
\begin{align}
    &\|\nabla_x\widetilde{\phi}_t(x_1)-\nabla_x\widetilde{\phi}_t(x_2)\|\nonumber\\
    \leq&\left\|(\nabla_x f(x_1,y_t^\ast(x_1))-\nabla_x f(x_2,y_t^\ast(x_2)))\right\|\label{eq:term1}\\
    &+\left\|\nabla_{xy}^2\widetilde{g}_t(x_2,y_t^\ast(x_2))(\nabla_{yy}^2\widetilde{g}_t(x_2,y_t^\ast(x_2)))^{-1}(\nabla_y f(x_2,y_t^\ast(x_2))-\nabla_y f(x_1,y_t^\ast(x_1)))\right\|\label{eq:term2}\\
    &+\left\|\nabla_{xy}^2\widetilde{g}_t(x_2,y_t^\ast(x_2))((\nabla_{yy}^2\widetilde{g}_t(x_2,y_t^\ast(x_2)))^{-1}-(\nabla_{yy}^2\widetilde{g}_t(x_1,y_t^\ast(x_1)))^{-1})\nabla_y f(x_1,y_t^\ast(x_1))\right\|\label{eq:term3}\\
    &+\left\|(\nabla_{xy}^2\widetilde{g}_t(x_2,y_t^\ast(x_2))-\nabla_{xy}^2\widetilde{g}_t(x_1,y_t^\ast(x_1)))(\nabla_{yy}^2\widetilde{g}_t(x_1,y_t^\ast(x_1)))^{-1}\nabla_y f(x_1,y_t^\ast(x_1))\right\|.\label{eq:term4}
\end{align}
For (\ref{eq:term1}), by Assumption \ref{assumption:general2}(\ref{assumption:general2(2)}), we have
$$\left\|(\nabla_x f(x_1,y_t^\ast(x_1))-\nabla_x f(x_2,y_t^\ast(x_2)))\right\|\leq\overline{L}_f\|(x_1,y^\ast_t(x_1))-(x_2,y^\ast_t(x_2))\|.$$
We evaluate (\ref{eq:term2}), (\ref{eq:term3}) and (\ref{eq:term4}) separately.\\

\noindent\textbf{The term (\ref{eq:term2}): }
According to Lemma \ref{lem:boundednessofJabobian}, $\|\nabla_{xy}^2\widetilde{g}_t(x,y_t^\ast(x))(\nabla_{yy}^2\widetilde{g}_t(x,y_t^\ast(x)))^{-1}\|\leq J_1$. Combining the Lipschitz smoothness of $f(x,y)$ from Assumption\ref{assumption:general2}(\ref{assumption:general2(2)}).
Therefore 
\begin{align*}
    &\left\|\nabla_{xy}^2\widetilde{g}_t(x_2,y_t^\ast(x_2))(\nabla_{yy}^2\widetilde{g}_t(x_2,y_t^\ast(x_2)))^{-1}(\nabla_y f(x_2,y_t^\ast(x_2))-\nabla_y f(x_1,y_t^\ast(x_1)))\right\|\\
    \leq&J_1\overline{L}_f\|(x_1,y^\ast_t(x_1))-(x_2,y^\ast_t(x_2))\|.
\end{align*}\\

\noindent\textbf{The term (\ref{eq:term3}): } By Assumption \ref{assumption:general2}(\ref{assumption:general2(1)}), $\|\nabla_y f(x,y)\|\leq L_f$. To estimate $$\left\|\nabla_{xy}^2\widetilde{g}_t(x_2,y_t^\ast(x_2))((\nabla_{yy}^2\widetilde{g}_t(x_2,y_t^\ast(x_2)))^{-1}-(\nabla_{yy}^2\widetilde{g}_t(x_1,y_t^\ast(x_1)))^{-1})\nabla_y f(x_1,y_t^\ast(x_1))\right\|,$$
It is sufficient to evaluate the following term
\begin{align}
    &\left\|\nabla_{xy}^2\widetilde{g}_t(x_2,y_t^\ast(x_2))((\nabla_{yy}^2\widetilde{g}_t(x_2,y_t^\ast(x_2)))^{-1}-(\nabla_{yy}^2\widetilde{g}_t(x_1,y_t^\ast(x_1)))^{-1})\right\|\nonumber\\
    \overset{\text{(i)}}{=}&\left\|\nabla_{xy}^2\widetilde{g}_t(x_2,y_t^\ast(x_2))(\nabla_{yy}^2\widetilde{g}_t(x_2,y_t^\ast(x_2)))^{-1}(\nabla_{yy}^2\widetilde{g}_t(x_2,y_t^\ast(x_2))-\nabla_{yy}^2\widetilde{g}_t(x_1,y_t^\ast(x_1)))(\nabla_{yy}^2\widetilde{g}_t(x_1,y_t^\ast(x_1)))^{-1}\right\|\nonumber\\
     \overset{\text{(ii)}}{\leq}&J_1\|(\nabla_{yy}^2\widetilde{g}_t(x_1,y_t^\ast(x_1)))^{-1}(\nabla_{yy}^2\widetilde{g}_t(x_2,y_t^\ast(x_2))-\nabla_{yy}^2\widetilde{g}_t(x_1,y_t^\ast(x_1)))\|,\label{eq:similarterm}
\end{align}
where (i) is due to $A^{-1}-B^{-1}=-A^{-1}(A-B)B^{-1}$, and (ii) is from Lemma \ref{lem:boundednessofJabobian}, which says $\|\nabla_{xy}^2\widetilde{g}_t(x,y_t^\ast(x))(\nabla_{yy}^2\widetilde{g}_t(x,y_t^\ast(x)))^{-1}\|\leq J_1$. To estimate (\ref{eq:similarterm}), we directly compute as following 

\resizebox{\textwidth}{!}{
\begin{minipage}{\textwidth}
\begin{align}
    &\|(\nabla_{yy}^2\widetilde{g}_t(x_1,y_t^\ast(x_1)))^{-1}(\nabla_{yy}^2\widetilde{g}_t(x_2,y_t^\ast(x_2))-\nabla_{yy}^2\widetilde{g}_t(x_1,y_t^\ast(x_1)))\|\nonumber\\
    =&\left\|(\nabla_{yy}^2\widetilde{g}_t(x_1,y_t^\ast(x_1)))^{-1}\left[\nabla^2_{yy}g(x_1,y^\ast_t(x_1))-\nabla^2_{yy}g(x_2,y^\ast_t(x_2))\right.\right.\label{eq:threeterm1}\\
    &\left.\left.+t\sum_{i=1}^k\frac{h_i(x_1,y^\ast_t(x_1))\nabla^2_{yy}h_i(x_2,y^\ast_t(x_2))-h_i(x_2,y^\ast_t(x_2))\nabla^2_{yy}h_i(x_1,y^\ast_t(x_1))}{h_i(x_1,y^\ast_t(x_1))h_i(x_2,y^\ast_t(x_2))}\right.\right.\label{eq:threeterm2}\\
        &\left.\left.+t\sum_{i=1}^k\frac{h_i^2(x_2,y^\ast_t(x_2))\nabla_y h_i(x_1,y^\ast_t(x_1))\nabla_y h_i(x_1,y^\ast_t(x_1))^\top-h_i^2(x_1,y^\ast_t(x_1))\nabla_y h_i(x_2,y^\ast_t(x_2))\nabla_y h_i(x_2,y^\ast_t(x_2))^\top}{h_i^2(x_1,y^\ast_t(x_1))h_i^2(x_2,y^\ast_t(x_2))}\right]\right\|\label{eq:threeterm3}.
\end{align}
\end{minipage}}
We respectively evaluate (\ref{eq:threeterm1}), (\ref{eq:threeterm2}), and (\ref{eq:threeterm3}).

For (\ref{eq:threeterm1}) and (\ref{eq:threeterm2}), we use Lipschitz continuity of $\nabla_{yy}^2g(x,y)$ from Assumption \ref{assumption:general2}(\ref{assumption:general2(5)}), the result that $h_i(x_2,y^\ast_t(x_2))\leq -m^{loc}$ from Lemma \ref{lem:local_lip}, and the same process in (\ref{eq:secondterm}). Note that  $\nabla_{yy}^2\widetilde{g}(x,y)\succeq\mu_g I$, we get
\begin{align*}
    &\left\|(\nabla_{yy}^2\widetilde{g}_t(x_1,y_t^\ast(x_1)))^{-1}\left[\nabla^2_{yy}g(x_1,y^\ast_t(x_1))-\nabla^2_{yy}g(x_2,y^\ast_t(x_2))\right.\right.\\
    &\left.\left.+t\sum_{i=1}^k\frac{h_i(x_1,y^\ast_t(x_1))\nabla^2_{yy}h_i(x_2,y^\ast_t(x_2))-h_i(x_2,y^\ast_t(x_2))\nabla^2_{yy}h_i(x_1,y^\ast_t(x_1))}{h_i(x_1,y^\ast_t(x_1))h_i(x_2,y^\ast_t(x_2))}\right]\right\|\\
    \leq&\frac{1}{\mu_g}\left\|\nabla^2_{yy}g(x_1,y^\ast_t(x_1))-\nabla^2_{yy}g(x_2,y^\ast_t(x_2))\right\|\\
    &+\frac{1}{\mu_g}\left\|t\sum_{i=1}^k\frac{h_i(x_1,y^\ast_t(x_1))\nabla^2_{yy}h_i(x_2,y^\ast_t(x_2))-h_i(x_2,y^\ast_t(x_2))\nabla^2_{yy}h_i(x_1,y^\ast_t(x_1))}{h_i(x_1,y^\ast_t(x_1))h_i(x_2,y^\ast_t(x_2))}\right\|\\
    \leq&\frac{1}{\mu_g}\left(\overline{\overline{L}}_g+tk\left(\frac{\overline{\overline{L}}_h}{m^{loc}}+\frac{\overline{L}_hL_h}{\left(m^{loc}\right)^2}\right)\right)\|(x_1,y^\ast_t(x_1))-(x_2,y^\ast_t(x_2))\|
\end{align*}

For (\ref{eq:threeterm3}), Using the same method of adding and subtracting like terms as in (\ref{eq:complicatedterm}), and applying the triangle inequality, we obtain

 \begin{align}
            &\left\|(\nabla_{yy}^2\widetilde{g}_t(x_1,y_t^\ast(x_1)))^{-1}t\sum_{i=1}^k\left(\frac{h_i^2(x_2,y^\ast_t(x_2))\nabla_y h_i(x_1,y^\ast_t(x_1))\nabla_y h_i(x_1,y^\ast_t(x_1))^\top}{h_i^2(x_1,y^\ast_t(x_1))h_i^2(x_2,y^\ast_t(x_2))}\right.\right.\nonumber\\
            &\left.\left.-\frac{h_i^2(x_1,y^\ast_t(x_1))\nabla_y h_i(x_2,y^\ast_t(x_2))\nabla_y h_i(x_2,y^\ast_t(x_2))^\top}{h_i^2(x_1,y^\ast_t(x_1))h_i^2(x_2,y^\ast_t(x_2))}\right)\right\|\nonumber\\
            \leq&t\sum_{i=1}^k\left[\left\|(\nabla_{yy}^2\widetilde{g}_t(x_1,y_t^\ast(x_1)))^{-1}\left(\frac{h_i^2(x_2,y^\ast_t(x_2))\nabla_y h_i(x_1,y^\ast_t(x_1))\nabla_y h_i(x_1,y^\ast_t(x_1))^\top}{h_i^2(x_1,y^\ast_t(x_1))h_i^2(x_2,y^\ast_t(x_2))}\right.\right.\right.\nonumber\\
            &\left.\left.-\frac{h_i^2(x_2,y^\ast_t(x_2))\nabla_y h_i(x_1,y^\ast_t(x_1))\nabla_y h_i(x_2,y^\ast_t(x_2))^\top}{h_i^2(x_1,y^\ast_t(x_1))h_i^2(x_2,y^\ast_t(x_2))}\right)\right\|\nonumber\\
            &+\left\|(\nabla_{yy}^2\widetilde{g}_t(x_1,y_t^\ast(x_1)))^{-1}\left(\frac{h_i^2(x_2,y^\ast_t(x_2))\nabla_y h_i(x_1,y^\ast_t(x_1))\nabla_y h_i(x_2,y^\ast_t(x_2))^\top}{h_i^2(x_1,y^\ast_t(x_1))h_i^2(x_2,y^\ast_t(x_2))}\right.\right.\nonumber\\
            &\left.\left.-\frac{h_i^2(x_1,y^\ast_t(x_1))\nabla_y h_i(x_1,y^\ast_t(x_1))\nabla_y h_i(x_2,y^\ast_t(x_2))^\top}{h_i^2(x_1,y^\ast_t(x_1))h_i^2(x_2,y^\ast_t(x_2))}\right)\right\|\nonumber\\
            &+\left.\left\|(\nabla_{yy}^2\widetilde{g}_t(x_1,y_t^\ast(x_1)))^{-1}\left(\frac{h_i^2(x_1,y^\ast_t(x_1))\nabla_y h_i(x_1,y^\ast_t(x_1))\nabla_y h_i(x_2,y^\ast_t(x_2))^\top}{h_i^2(x_1,y^\ast_t(x_1))h_i^2(x_2,y^\ast_t(x_2))}\right.\right.\right.\nonumber\\
            &\left.\left.\left.-\frac{h_i^2(x_1,y^\ast_t(x_1))\nabla_y h_i(x_2,y^\ast_t(x_2))\nabla_y h_i(x_2,y^\ast_t(x_2))^\top}{h_i^2(x_1,y^\ast_t(x_1))h_i^2(x_2,y^\ast_t(x_2))}\right)\right\|\right]\nonumber\\
            \overset{\text{(i)}}{\leq}&\frac{tkL_h\overline{L}_h}{\mu_g\left(m^{loc}\right)^2}\|(x_1,y_t^\ast(x_1))-(x_2,y_t^\ast(x_2))\|.\nonumber\\
            &+t\sum_{i=1}^k\left\|(\nabla_{yy}^2\widetilde{g}_t(x_1,y_t^\ast(x_1)))^{-1}\left(\frac{h_i^2(x_2,y^\ast_t(x_2))\nabla_y h_i(x_1,y^\ast_t(x_1))\nabla_y h_i(x_2,y^\ast_t(x_2))^\top}{h_i^2(x_1,y^\ast_t(x_1))h_i^2(x_2,y^\ast_t(x_2))}\right.\right.\nonumber\\
            &-\left.\left.\frac{h_i^2(x_1,y^\ast_t(x_1))\nabla_y h_i(x_1,y^\ast_t(x_1))\nabla_y h_i(x_2,y^\ast_t(x_2))^\top}{h_i^2(x_1,y^\ast_t(x_1))h_i^2(x_2,y^\ast_t(x_2))}\right)\right\|\nonumber\\
            &+\frac{tkL_h\overline{L}_h}{\mu_g\left(m^{loc}\right)^2}\|(x_1,y_t^\ast(x_1))-(x_2,y_t^\ast(x_2))\|.\nonumber\\
            =&t\sum_{i=1}^k\left\|(\nabla_{yy}^2\widetilde{g}_t(x_1,y_t^\ast(x_1)))^{-1}\left(\frac{h_i^2(x_2,y^\ast_t(x_2))\nabla_y h_i(x_1,y^\ast_t(x_1))\nabla_y h_i(x_2,y^\ast_t(x_2))^\top}{h_i^2(x_1,y^\ast_t(x_1))h_i^2(x_2,y^\ast_t(x_2))}\right.\right.\nonumber\\
            &-\left.\left.\frac{h_i^2(x_1,y^\ast_t(x_1))\nabla_y h_i(x_1,y^\ast_t(x_1))\nabla_y h_i(x_2,y^\ast_t(x_2))^\top}{h_i^2(x_1,y^\ast_t(x_1))h_i^2(x_2,y^\ast_t(x_2))}\right)\right\|\label{eq:thefinalterm!}\\
            &+2\frac{tkL_h\overline{L}_h}{\mu_g\left(m^{loc}\right)^2}\|(x_1,y_t^\ast(x_1))-(x_2,y_t^\ast(x_2))\|.\nonumber
        \end{align}
    In (i) we used the result that $h_i(x_1,y^\ast_t(x_1)),\ h_i(x_2,y^\ast_t(x_2))\leq -m^{loc}$ from Lemma \ref{lem:local_lip}, Lipschitz smoothness of $h_i(x,y)$ from Assumption \ref{assumption:general2}(\ref{assumption:general2(7)}), boundedness of $\nabla_y h_i(x,y)$ from \ref{assumption:general2}(\ref{assumption:general2(6)}), and $\nabla_{yy}^2 \widetilde{g}_t(x,y^\ast_t(x))\succeq \mu_g I$. Finally, we need to evaluate (\ref{eq:thefinalterm!}), i.e.
    \begin{align}
        &t\sum_{i=1}^k\left\|\left(\nabla_{yy}\widetilde{g}_t(x_1,y^\ast_t(x_1))\right)^{-1}\left(\frac{(-h_i(x_2,y^\ast_t(x_2)))^2\nabla_yh_i(x_1,y^\ast_t(x_1))(\nabla_yh_i(x_2,y^\ast_t(x_2)))^\top}{(h_i(x_1,y^\ast_t(x_1)h_i(x_2,y^\ast_t(x_2)))^2}\right.\right.\nonumber\\
        &\left.\left.-\frac{(-h_i(x_1,y^\ast_t(x_1)))^2\nabla_yh_i(x_1,y^\ast_t(x_1))(\nabla_yh_i(x_2,y^\ast_t(x_2)))^\top}{(h_i(x_1,y^\ast_t(x_1)h_i(x_2,y^\ast_t(x_2)))^2}\right)\right\|\nonumber\\
        =&t\sum_{i=1}^k\left\|\left(\nabla_{yy}\widetilde{g}_t(x_1,y^\ast_t(x_1))\right)^{-1}\left(\frac{(h_i(x_1,y^\ast_t(x_1))-h_i(x_2,y^\ast_t(x_2)))\nabla_yh_i(x_1,y^\ast_t(x_1))(\nabla_yh_i(x_2,y^\ast_t(x_2)))^\top}{h^2_i(x_1,y^\ast_t(x_1)h_i(x_2,y^\ast_t(x_2))}\right.\right.\nonumber\\
        &\left.\left.+\frac{(h_i(x_1,y^\ast_t(x_1))-h_i(x_2,y^\ast_t(x_2)))\nabla_yh_i(x_1,y^\ast_t(x_1))(\nabla_yh_i(x_2,y^\ast_t(x_2)))^\top}{h_i(x_1,y^\ast_t(x_1)h^2_i(x_2,y^\ast_t(x_2))}\right)\right\|.\label{eq:thelastterm!!}
    \end{align}
    From Lemma \ref{cor:bound2}, we have
    \begin{align*}
        &t\sum_{i=1}^k\left\|\left(\nabla_{yy}\widetilde{g}_t(x_1,y^\ast_t(x_1))\right)^{-1}\left(\frac{(h_i(x_1,y^\ast_t(x_1))-h_i(x_2,y^\ast_t(x_2)))\nabla_yh_i(x_1,y^\ast_t(x_1))(\nabla_yh_i(x_2,y^\ast_t(x_2)))^\top}{h^2_i(x_1,y^\ast_t(x_1)h_i(x_2,y^\ast_t(x_2))}\right)\right\|\\
        \leq& J_2\frac{kL_h^2}{m^{loc}}\|(x_1,y_t^\ast(x_1))-(x_2,y_t^\ast(x_2))\|.
    \end{align*}
    For the second term in (\ref{eq:thelastterm!!}), we observe that
    \begin{align*}
        &\sum_{i=1}^k\left\|\left(\nabla_{yy}\widetilde{g}_t(x_1,y^\ast_t(x_1))\right)^{-1}t\left(\frac{(h_i(x_1,y^\ast_t(x_1))-h_i(x_2,y^\ast_t(x_2)))\nabla_yh_i(x_1,y^\ast_t(x_1))(\nabla_yh_i(x_2,y^\ast_t(x_2)))^\top}{h_i(x_1,y^\ast_t(x_1)h^2_i(x_2,y^\ast_t(x_2))}\right)\right\|\\
        =&\sum_{i=1}^k\left\|\left(\nabla_{yy}\widetilde{g}_t(x_1,y^\ast_t(x_1))\right)^{-1}t\left(\frac{(h_i(x_1,y^\ast_t(x_1))-h_i(x_2,y^\ast_t(x_2)))\nabla_yh_i(x_1,y^\ast_t(x_1))(\nabla_yh_i(x_2,y^\ast_t(x_2)))^\top}{h^2_i(x_1,y^\ast_t(x_1)h_i(x_2,y^\ast_t(x_2))}\right.\right.\\
        &\times\left.\left.\frac{h_i(x_1,y^\ast_t(x_1)}{h_i(x_2,y^\ast_t(x_2))}\right)\right\|\\
        \leq&J_2\frac{kL_h^2}{m^{loc}}\left|\frac{-h_i(x_1,y^\ast_t(x_1))}{-h_i(x_2,y^\ast_t(x_2))}\right|\|(x_1,y_t^\ast(x_1))-(x_2,y_t^\ast(x_2))\|\\
        \overset{\text{(i)}}{\leq}&J_2\frac{kL_h^2}{m^{loc}}\left|\frac{-h_i(x,y^\ast_t(x))+\frac{m}{2}}{-h_i(x,y^\ast_t(x))-\frac{m}{2}}\right|\|(x_1,y_t^\ast(x_1))-(x_2,y_t^\ast(x_2))\|\\
        =&J_2\frac{kL_h^2}{m^{loc}}\left|1+\frac{m}{-h_i(x,y^\ast_t(x))-\frac{m}{2}}\right|\|(x_1,y_t^\ast(x_1))-(x_2,y_t^\ast(x_2))\|\\
        \leq&J_2\frac{kL_h^2}{m^{loc}}\left(1+\left|\frac{m}{m-\frac{m}{2}}\right|\right)\|(x_1,y_t^\ast(x_1))-(x_2,y_t^\ast(x_2))\|\\
        \leq&J_2\frac{3kL_h^2}{m^{loc}}\|(x_1,y_t^\ast(x_1))-(x_2,y_t^\ast(x_2))\|,
    \end{align*}
    where (i) is because $h_i(x_1,y^\ast_t(x_1)),\ h_i(x_2,y^\ast_t(x_2)),\ h_i(x,y^\ast_t(x))\leq 0$, $h_i(x,y^\ast_t(x))$ is $L_h(1+J_1)$ Lipschitz continuous, and $\|x-x_1\|,\ \|x-x_2\|\leq m/(2L_h(1+J_1))$. To summarize, we get 
    \begin{align*}
    &\left\|\nabla_{xy}^2\widetilde{g}_t(x_2,y_t^\ast(x_2))((\nabla_{yy}^2\widetilde{g}_t(x_2,y_t^\ast(x_2)))^{-1}-(\nabla_{yy}^2\widetilde{g}_t(x_1,y_t^\ast(x_1)))^{-1})\nabla_y f(x_1,y_t^\ast(x_1))\right\|\\
    \leq&L_fJ_1\left(\frac{\overline{\overline{L}}_g}{\mu_g}+\frac{tk\overline{\overline{L}}_h}{\mu_gm^{loc}}+\frac{tk\overline{L_h}L_h}{\mu_g\left(m^{loc}\right)^2}+2\frac{tkL_h\overline{L}_h}{\mu_g\left(m^{loc}\right)^2}+J_2\frac{kL_h^2}{m^{loc}}+J_2\frac{3kL_h^2}{m^{loc}}\right)\|(x_1,y_t^\ast(x_1))-(x_2,y_t^\ast(x_2))\|.
    \end{align*}

\noindent\textbf{The term (\ref{eq:term4}): }This term can be evaluated directly by Lemma \ref{lem:lipHessian} and Lemma \ref{lem:boundednessofJabobian} as follows:
\begin{align*}
    &\left\|(\nabla_{xy}^2\widetilde{g}_t(x_2,y_t^\ast(x_2))-\nabla_{xy}^2\widetilde{g}_t(x_1,y_t^\ast(x_1)))(\nabla_{yy}^2\widetilde{g}_t(x_1,y_t^\ast(x_1)))^{-1}\nabla_y f(x_1,y_t^\ast(x_1))\right\|\\
    \leq&\overline{\overline{L}}_{\widetilde{g}_t,m^{loc}}J_1\|(x_1,y_t^\ast(x_1))-(x_2,y_t^\ast(x_2))\|.
\end{align*}

Note that $\|(x_1,y_t^\ast(x_1))-(x_2,y_t^\ast(x_2))\|\leq\|\nabla_x y^\ast_t(x)\|\|x_1-x_2\|\leq J_1\|x_1-x_2\|$ from Lemma \ref{lem:boundednessofJabobian}. Setting
\begin{align*}
    \overline{L}^{loc}_{\widetilde{\phi}_t}=&J_1\left(\overline{L}_f+\overline{L}_fJ_1+L_fJ_1\left(\frac{\overline{\overline{L}}_g}{\mu_g}+\frac{tk\overline{\overline{L}}_h}{\mu_gm^{loc}}+\frac{tk\overline{L_h}L_h}{\mu_g\left(m^{loc}\right)^2}\right.\right.\\
&\left.\left.+2\frac{tkL_h\overline{L}_h}{\mu_g\left(m^{loc}\right)^2}+J_2\frac{kL_h^2}{m^{loc}}+J_2\frac{3kL_h^2}{m^{loc}}\right)+\overline{\overline{L}}_{\widetilde{g}_t,m^{loc}}J_1\right),
\end{align*}
we conclude
    \begin{align*}
        \|\nabla_x\widetilde{\phi}_t(x_1)-\nabla_x\widetilde{\phi}_t(x_2)\|\leq\overline{L}^{loc}_{\widetilde{\phi}_t}\|x_1-x_2\|.
    \end{align*}

Following the proof of Lemma \ref{lem:lip_global_bound} around (\ref{eq:Mast}), $m^{loc}$ has a lower bound $M^\ast=m(D/4)=\mathcal{O}(t)$. By design of $\overline{\overline{L}}_{\widetilde{g}_t,m^{loc}}$ in Lemma \ref{eq:liphessian!}, we know that $\overline{\overline{L}}_{\widetilde{g}_t,m^{loc}}\leq\overline{\overline{L}}_{\widetilde{g}_t,M^\ast}=\mathcal{O}(1/t)$,  so $\overline{L}^{loc}_{\widetilde{\phi}_t}$ has an upper bound
\begin{align*}
    \overline{L}_{\widetilde{\phi}_t}=&J_1\left(\overline{L}_f+\overline{L}_fJ_1+L_fJ_1\left(\frac{\overline{\overline{L}}_g}{\mu_g}+\frac{tk\overline{\overline{L}}_h}{\mu_g M^\ast}+\frac{tk\overline{L_h}L_h}{\mu_g\left(M^\ast\right)^2}\right.\right.\\
&\left.\left.+2\frac{tkL_h\overline{L}_h}{\mu_g\left(M^\ast\right)^2}+J_2\frac{kL_h^2}{M^\ast}+J_2\frac{3kL_h^2}{M^\ast}\right)+\overline{\overline{L}}_{\widetilde{g}_t,M^\ast}J_1\right)\\
    =&\mathcal{O}(\frac{1}{t}).
\end{align*}

\subsection{Proof of Theorem \ref{thm:total_rate2}}\label{proof:thm:total_rate2}

The proof is almost identical to the proof of Theorem \ref{thm:total_rate}. The only noteworthy point is that $\hat{\nabla}_x\tilde{\phi}_t(x_s)$ now is upper bounded. This can be proved as follows: due to Lemma \ref{lem:optimalitygapforpoint}, $\|y^\ast(x_s)-y^\ast_t(x_s)\|\leq\beta/2$. The lower-level algorithm will guarantee (by Theorem \ref{thm:convergence_rage_alg_2}) $\|\hat{y}_s-y^\ast_t(x_s)\|\leq\beta/2$, which means $\|y^\ast(x_s)-\hat{y}_s\|\leq\beta$. By Lemma \ref{cor:bound2}, $\|(\nabla_{yy}\widetilde{g}_t(x_s,\hat{y}_s))^{-1}\nabla_{yx}\widetilde{g}_t(x_s,\hat{y}_s)\|\leq J_3$, so
    \begin{align*}
        \|\hat{\nabla}_x\tilde{\phi}_t(x_s)\|&=\|\nabla_x f(x_s,\tilde{y}_{s})-\nabla_{xy}^2\widetilde{g}_(x_s,\tilde{y}_{s})(\nabla_{yy}^2\widetilde{g}_(x_s,\tilde{y}_{s}))^{-1}\nabla_y f(x_s,\tilde{y}_{s})\|\\
        &\leq\|\nabla_x f(x_s,\tilde{y}_{s})\|+\|\nabla_{xy}^2\widetilde{g}_(x_s,\tilde{y}_{s})(\nabla_{yy}^2\widetilde{g}_(x_s,\tilde{y}_{s}))^{-1}\|\|\nabla_y f(x_s,\tilde{y}_{s})\|\\
        &\leq L_f(1+J_3).
    \end{align*}
    The remaining steps to complete the proof are analogous to those in Theorem \ref{thm:total_rate}.

\end{document}

%% file: _macros.tex
\usepackage{comment,url,graphicx,relsize}
\usepackage{amssymb,amsfonts,amsmath,amsthm,amscd,dsfont,mathrsfs,mathtools,nicefrac, bm}
\usepackage{float,psfrag,epsfig,color,xcolor,url}
\usepackage{epstopdf,bbm,mathtools,enumitem}
\usepackage{subfigure}
\usepackage[ruled,vlined]{algorithm2e}
\usepackage{tablefootnote}
\graphicspath{{Plots/}}
\usepackage{diagbox}
\usepackage{tikz}
\usetikzlibrary{calc,patterns,angles,quotes}
\usepackage{makecell}
\usepackage{multirow}
\usepackage{booktabs}

\newcommand{\RR}{\mathbb{R}}

\providecommand{\argmax}{\mathop\mathrm{arg max}} 
\providecommand{\argmin}{\mathop\mathrm{arg min}}

\newcommand{\proj}{\mathsf{proj}}

\ifdefined\nonewproofenvironments\else
\ifdefined\ispres\else
\renewenvironment{proof}{\noindent\textbf{Proof.}\hspace*{.3em}}{\qed\\}
\newenvironment{proof-sketch}{\noindent\textbf{Proof Sketch}
  \hspace*{0.em}}{\qed\bigskip\\}
\newenvironment{proof-idea}{\noindent\textbf{Proof Idea}
  \hspace*{0.em}}{\qed\bigskip\\}
\newenvironment{proof-of-lemma}[1][{}]{\noindent\textbf{Proof of Lemma {#1}.}
  \hspace*{0.em}}{\qed\\}
\newenvironment{proof-of-corollary}[1][{}]{\noindent\textbf{Proof of Corollary {#1}.}
  \hspace*{0.em}}{\qed\\}
\newenvironment{proof-of-theorem}[1][{}]{\noindent\textbf{Proof of Theorem {#1}.}
  \hspace*{0.em}}{\qed\\}
\newenvironment{proof-attempt}{\noindent\textbf{Proof Attempt}
  \hspace*{0.em}}{\qed\bigskip\\}

\fi
\newtheorem{theorem}{Theorem}[section]
\newtheorem{lemma}{Lemma}[section]
\newtheorem{corollary}{Corollary}[section]
\newtheorem{proposition}{Proposition}[section]
\newtheorem{assumption}{Assumption}[section]
\newtheorem{remark}{Remark}[section]
\newtheorem{example}{Example}[section]

\newtheorem{definition}{Definition}[section]

\usetikzlibrary{calc}